\documentclass{amsart}
\usepackage{amsmath}
\usepackage{amssymb}
\usepackage{mathtools}    
\usepackage{stmaryrd}     
\usepackage{xcolor}
\usepackage{booktabs}
\usepackage[hidelinks]{hyperref}
\usepackage{aliascnt}
\usepackage[nameinlink]{cleveref}

\usepackage{tikz}
\usetikzlibrary{positioning, arrows.meta, decorations.pathreplacing}

\usepackage[backend=biber,style=alphabetic,useprefix=true]{biblatex}
\usepackage{enumitem}
\setlist{leftmargin=8mm}
\newlist{enumabc}{enumerate}{1}
\setlist[enumabc]{label={\upshape(\alph*)}}
\newlist{enumalgo}{enumerate}{1}
\setlist[enumalgo]{label={\upshape\arabic*.}}

\usepackage{caption}
\usepackage{xspace}
\renewcommand{\th}{\nobreakdash-th\xspace}   

\allowdisplaybreaks

\numberwithin{equation}{section}
\numberwithin{figure}{section}

\usepackage[scr=dutchcal]{mathalpha}

\newcommand{\makesharedthm}[4]{%
  \newaliascnt{#1}{#4}            
  \newtheorem{#1}[#1]{#2}         
  \aliascntresetthe{#1}           
  \crefname{#1}{\MakeLowercase{#2}}{\MakeLowercase{#3}} 
  \Crefname{#1}{#2}{#3}           
}

\numberwithin{allthm}{section}
\makesharedthm{thm}{Theorem}{Theorems}{allthm}
\makesharedthm{prop}{Proposition}{Propositions}{allthm}
\makesharedthm{lem}{Lemma}{Lemmas}{allthm}
\makesharedthm{cor}{Corollary}{Corollaries}{allthm}
\makesharedthm{conj}{Conjecture}{Conjectures}{allthm}
\theoremstyle{definition}
\makesharedthm{defn}{Definition}{Definitions}{allthm}
\theoremstyle{remark}
\makesharedthm{rem}{Remark}{Remarks}{allthm}

\newcommand{\ZZ}{\mathbf{Z}}
\newcommand{\QQ}{\mathbf{Q}}
\newcommand{\FF}{\mathbf{F}}
\newcommand{\RR}{\mathbf{R}}
\newcommand{\eps}{\varepsilon}
\renewcommand{\leq}{\leqslant}
\renewcommand{\geq}{\geqslant}
\newcommand{\divides}{\mathrel{|}}
\newcommand{\ndivides}{\mathrel{\nmid}}
\DeclareMathOperator{\Mint}{\mathsf{M}_{\textrm{int}}}
\DeclareMathOperator{\trace}{Tr}
\DeclareMathOperator{\real}{Re}
\DeclareMathOperator{\wt}{wt}
\DeclareMathOperator{\ord}{ord}
\DeclareMathOperator{\li}{li}
\newcommand{\bbracket}[1]{\llbracket #1 \rrbracket}
\newcommand{\bangle}[1]{\langle #1 \rangle}
\newcommand{\cdotspace}[1]{\hspace{#1} \cdot \hspace{#1}}
\newcommand{\dset}{\mathcal{D}}
\newcommand{\Sset}{\mathscr{S}}
\newcommand{\Aset}{\mathcal{A}}
\newcommand{\Pset}{\mathscr{P}}
\newcommand{\Dset}{\mathscr{D}}
\newcommand{\Rset}{\mathscr{R}}
\newcommand{\Gset}{\mathscr{G}}
\newcommand{\Aalg}{\mathcal{A}}
\newcommand{\fwd}{\mathcal{F}}
\newcommand{\inv}{\mathcal{I}}
\newcommand{\Rring}{\mathscr{R}}
\newcommand{\FFext}{\FF_{2^\lambda}}

\newcommand{\tsp}{\mspace{1.5mu}}    
\newcommand{\textn}{\textnormal}
\newcommand{\step}[2]{\textit{Step #1: #2}}

\newcommand{\floatingbar}[1]{%
   \mathord{\smash{\bar{#1}}\vphantom{#1}}}
\newcommand{\makecost}[1]{%
   \expandafter\DeclareMathOperator\csname #1cost\endcsname{\mathsf{#1}}}
\newcommand{\makebarcost}[1]{%
   \expandafter\DeclareMathOperator\csname #1barcost\endcsname{\floatingbar{\mathsf{#1}}}}
\newcommand{\makestar}[1]{%
   \expandafter\newcommand\csname #1star\endcsname{%
      \csname #1cost\endcsname^{*\mspace{-1mu}}}}
\newcommand{\makebarstar}[1]{%
   \expandafter\newcommand\csname #1barstar\endcsname{%
      \csname #1barcost\endcsname^{*\mspace{-1mu}}}}

\makecost{M}
\makebarcost{M}
\makestar{M}
\makebarstar{M}
\makecost{R}
\makebarcost{R}
\makestar{R}
\makebarstar{R}
\makecost{T}
\makecost{C}
\makestar{C}
\makecost{F}
\makecost{G}
\makecost{K}
\makestar{K}

\makeatletter
\newcommand{\tpmod}[1]{{\@displayfalse\pmod{#1}}}
\makeatother

\DeclareRobustCommand{\restrict}{\mathbin{\protect\tikzrestrict}}
\newcommand{\tikzrestrict}{%
   \begin{tikzpicture}[baseline=0.05ex]
      \useasboundingbox (0,0) rectangle (1.7ex,1ex);
      \draw[line width=0.03em] (0,1ex) -- (1.4ex,1ex) -- (1.4ex,0) -- cycle;
   \end{tikzpicture}%
}

\begin{document}

\title{Faster enumeration of primes}
\author{David Harvey}

\begin{abstract}
We describe several new algorithms for finding all prime numbers
up to a given bound~$N$,
achieving the first ever speedup by a positive power of $\log N$
over the ancient sieve of Eratosthenes.
The fastest version, which is not fully rigorous, runs in
\[
N (\log \log N)^{1+o(1)}
\]
bit operations when analysed in the multitape Turing model.
This improves on the best existing algorithms
due to Pritchard (1981), Atkin--Bernstein (2004) and Sergeev (2016)
by a factor of almost $\log N$.
We also present a rigorous randomised (Las Vegas) variant
that is slower by a factor of $(\log \log N)^{1+o(1)}$,
and a rigorous deterministic variant
that is slower by a factor of $(\log N)^{1/2+o(1)}$.
The new algorithms make heavy use of
fast polynomial arithmetic over finite fields,
and also involve ideas from the theory of error-correcting codes.
\end{abstract}

\maketitle

\setcounter{tocdepth}{1}

\begin{center}
\textit{This work is dedicated to Terry Gagen, \\
who taught me that ``it always just fits''}
\end{center}
\bigskip

\makeatletter
\renewcommand{\tocsection}[3]{%
\indentlabel{\makebox[2em][l]{\@ifnotempty{#2}{\ignorespaces#1 #2.\quad}}}#3}
\makeatother
\tableofcontents

\section{Introduction}
\label{sec:introduction}


\newcommand{\sectionhyperref}[2]{%
   \hyperref[#1]{\begin{tabular}{@{}c@{}}\ref*{#1}. #2\end{tabular}}}

\begin{figure}[ht]
\centering

\begin{tikzpicture}[
    section/.style={align=center, draw=black, rounded corners=1pt,
    inner xsep=6pt, inner ysep=3pt},
   arrow/.style={-{Stealth[length=5pt,width=4pt]}}
]

\def\padtop{9pt}
\def\padbottom{7pt}
\def\padleft{3pt}
\def\padright{7pt}

\node[section]
   (intro)
   {\sectionhyperref{sec:introduction}%
   {Introduction}};

\node[section, right=1cm of intro.north east, anchor=north west]
   (turing)
   {\sectionhyperref{sec:turing}%
   {Turing machines}};

\node[section, below=0.6cm of intro, xshift=-3.8cm, anchor=north west]
   (genfunc)
   {\sectionhyperref{sec:generating-function}%
   {A generating function \\ for primes}};

\node[section, right=0.8cm of genfunc.north east, anchor=north west]
   (polyarith)
   {\sectionhyperref{sec:polynomial-arithmetic}%
   {Polynomial \\ arithmetic}};

\node[section, right=0.8cm of polyarith.north east, anchor=north west]
   (compression)
   {\sectionhyperref{sec:compression}%
   {Compression and \\ decompression}};

\node[section, below=0.9cm of genfunc, xshift=-0.4cm, anchor=north west]
   (core)
   {\sectionhyperref{sec:core}%
   {The core algorithm}};

\node[section, right=1.1cm of core.east, anchor=west]
   (sieve)
   {\sectionhyperref{sec:sieve}%
   {Sieve estimates for \\ square-rough integers}};

\node[section, below=0.9cm of core, xshift=-3.5cm, anchor=north west]
   (heuristic)
   {\sectionhyperref{sec:heuristic}%
   {The heuristic \\ algorithm}};

\node[section, right=0.8cm of heuristic.north east, anchor=north west]
   (probabilistic)
   {\sectionhyperref{sec:probabilistic}%
   {The probabilistic \\ algorithm}};

\node[section, right=0.8cm of probabilistic.north east, anchor=north west]
   (deterministic)
   {\sectionhyperref{sec:deterministic}%
   {The deterministic \\ algorithm}};

\draw[arrow] (turing) -- (polyarith);
\draw[arrow] (genfunc) -- (core);
\draw[arrow] (polyarith) -- (core);
\draw[arrow] (compression) -- (core);
\draw[arrow] (polyarith) -- (compression);
\draw[arrow] (core) -- (heuristic);
\draw[arrow] (core) -- (probabilistic);
\draw[arrow] (core) -- (deterministic);
\draw[arrow] (sieve) -- (probabilistic);
\draw[arrow] (sieve) -- (deterministic);
\draw[arrow,dashed] (heuristic) -- (probabilistic);
\draw[arrow,dashed] (probabilistic) -- (deterministic);


\end{tikzpicture}

\caption{Main dependencies between sections.
   Dashed arrows indicate dependence on minor auxiliary results.}
\end{figure}
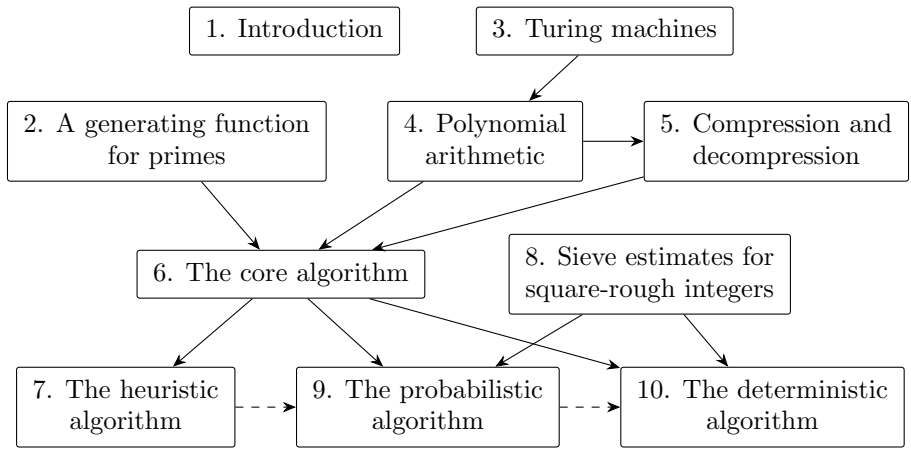

\subsection{Notation}
\label{sec:notation}

The notation $f = O(g)$, or equivalently $f \ll g$,
means that there is a constant $C > 0$ such that
$|f(x)| \leq C g(x)$ for all $x$ in the domain under consideration.
We use subscripts to indicate that $C$ depends on some other variable,
writing for instance $f = O_\eps(g)$ or $f \ll_\eps g$.

The notation $f = \Theta(g)$, or equivalently $f \asymp g$,
means that both $f \ll g$ and $g \ll f$ hold.
The notation $f \sim g$ means that
$f(x)/g(x) \to 1$ as $x \to \infty$,
and $f = o(g)$ means that $f(x)/g(x) \to 0$ as $x \to \infty$.
In particular, $o(1)$ stands for a function that
approaches zero as $x \to \infty$.

We write $\log x$ for the natural logarithm and
$\log_2 x$ for the base two logarithm.
We define $\lg x \coloneqq \max(1, \lceil \log_2 x \rceil)$,
so that $\lg x$ takes positive integer values for all $x > 0$.
We write $\log^* x$ for the iterated logarithm,
i.e., $\log^* x$ is the smallest integer $k \geq 0$ such that
$\log^{\circ k} x \leq 1$,
where $\log^{\circ k}$ denotes the $k$-fold composition of $\log$.

In many situations, especially when writing complexity bounds,
we observe the convention that $\log x$ should be read as $\log \max(x, 2)$.
For instance, if an algorithm is said to run in time $O(\lambda n \log n)$,
then when $n = 1$ the running time is $O(\lambda)$, not zero.
We rely on the reader's judgment to decide when this substitution
should be made; it should not cause any confusion.

Let $R$ be a ring (always assumed commutative with identity).
For $n \geq 0$ we write $R[x]_n$ for the set of polynomials in $R[x]$
of degree less than~$n$,
i.e., the $R$-module spanned by $\{1, x, \ldots, x^{n-1}\}$.
Similarly, $R[x^{-1}]_n$ denotes the $R$-module spanned by
$\{1, x^{-1}, \ldots, x^{-n+1}\}$.
We write $R\bbracket{x}$ and $R\bbracket{x^{-1}}$ for the rings of formal
power series in $x$ or $x^{-1}$ over $R$.
If $f \in R[x]$ or $R\bbracket{x}$,
we often write $f_i$ for the coefficient of $x^i$ in~$f$.
For $f \in R[x^{-1}]$ or $R\bbracket{x^{-1}}$,
by $f_i$ we mean instead the coefficient of $x^{-i}$.

\subsection{The sieve of Eratosthenes}
\label{sec:eratosthenes}

The object of this paper is to describe several new algorithms for
finding all prime numbers up to a prescribed bound~$N$.
This problem has a long history going back to Eratosthenes
(3rd--2nd century BCE),
whose sieve method is familiar to modern school students.
To perform the sieve, 
one first writes down all the integers from $2$ up to~$N$.
After crossing out the multiples of~$2$,
one discovers that the next prime is~$3$.
Continuing in this way, crossing out multiples of~$3$, then~$5$,
and so on for all primes up to~$\sqrt N$,
the integers remaining are exactly the primes up to~$N$.%
\footnote{Greaves \cite[\S1.1.1]{Greaves-sieves}
expresses some doubt that Eratosthenes' main goal was to locate primes,
suggesting instead that he may have been building a factor table.
His source for this claim is apparently Dickson
\cite[Ch.\,XIII]{Dickson-history-vol1}.
This is an interesting possibility, but appears to be a minority view.
Unfortunately Eratosthenes' original work has been lost.}

Eratosthenes' sieve is remarkably efficient.
To describe its asymptotic time complexity precisely,
we must specify a model of computation.
Most of the existing literature on prime enumeration algorithms
assumes some variant of the RAM (random access machine) model.
The exact model is not always clearly stated in these works,
but for the present discussion
let us assume that memory is made up of machine words,
each of which can store an integer of magnitude $N^{O(1)}$,
i.e., requiring $O(\log N)$ bits,
and let us define the time complexity of an algorithm to be
the number of ``word operations'' that it performs.
Such operations include basic arithmetic operations on single words
(addition, subtraction, comparison, and sometimes multiplication or division),
and memory accesses (reads or writes) to single words
via a word-sized address.
For a more detailed discussion of the sort of model we have in mind,
see \cite[\S1.2]{AHU-algorithms}.

To analyse the complexity of Eratosthenes' sieve,
observe that it marks off $N/p + O(1)$ multiples of $p$
for each prime $p < \sqrt N$.
By Mertens' theorem \cite[Thm.\,2.7(d)]{MV-mult-nt} we have
$\sum_{p < \sqrt N} N/p \sim N \log \log N$,
so the sieve runs in
\begin{equation}
\label{eq:eratosthenes-bound}
O(N \log \log N)
\end{equation}
word operations.

On the other hand,
according to the prime number theorem \cite[Eq.\,8.1]{MV-mult-nt}
there are about $N / \log N$ primes less than~$N$,
so it is conceivable that they can be generated
using only $O(N / \log N)$ word operations.
Over the past two thousand years,
the complexity bound \eqref{eq:eratosthenes-bound} has been modestly improved,
but not to anything approaching $O(N / \log N)$.
Prior to the present work, the best result was
\begin{equation}
\label{eq:pritchard-bound}
O\left(\frac{N}{\log \log N}\right)
\end{equation}
word operations,
which is faster than the basic sieve by a factor of $(\log \log N)^2$,
but still short of the ideal $O(N / \log N)$ by a factor of almost $\log N$.
The bound \eqref{eq:pritchard-bound} was first achieved
by Pritchard \cite{Pri-sublinear}.
Atkin and Bernstein later discovered a different approach that yields
the same bound \cite{AB-prime-sieves}.
We will give an overview of these algorithms in \Cref{sec:ram-results}.

More than a decade after establishing
the bound \eqref{eq:pritchard-bound},
with no further progress having been made,
Pritchard conjectured that \eqref{eq:pritchard-bound} is optimal
\cite[\S6]{Pri-incremental}.
In this paper we will see that this claim is incorrect.

\subsection{The Turing model}
\label{sec:turing-model}

Our new algorithms are analysed most naturally in a bit-oriented
model of computation.
For most of the paper we will work in the standard
\emph{multitape Turing model}.
In this model, time complexity refers to the number of steps
(or ``bit operations'') performed by the Turing machine.
For a detailed exposition of this model see
\cite[Ch.\,2]{Pap-complexity} or \cite[\S1.6]{AHU-algorithms}.

We emphasise that the Turing model is much more conservative regarding
the cost of memory access than the RAM model.
For example, to access the $i$\th bit of an array of~$n$ bits,
one potentially has to do $\Theta(n)$ work to move the tape head
to the $i$\th slot.
This is especially significant in sieving algorithms,
which frequently write to widely separated memory locations
in quick succession.
Thus, even though a RAM might appear to be doing
$O(\log N)$ ``bit operations'' of work per step,
one cannot simply multiply the number of word operations by $\log N$
to deduce the cost in the Turing model.

As an illustration,
consider the RAM bound \eqref{eq:pritchard-bound}
for enumerating the primes up to~$N$.
One might guess that this should correspond to
\begin{equation}
\label{eq:pritchard-bit}
O\left(\frac{N \log N}{\log \log N}\right)
\end{equation}
bit operations.
But in fact, prior to this paper, it was not known how to achieve
\eqref{eq:pritchard-bit} on a Turing machine.
Instead, the best bound known for this model was
\begin{equation}
\label{eq:sergeev-bound}
O(N \log N)
\end{equation}
bit operations,
which falls behind \eqref{eq:pritchard-bit} by a factor of $\log \log N$.
This relatively recent result is due to Sergeev \cite{Ser-prime-turing}
(see \cite{Ser-prime-turing-arxiv} for an English translation).

On the other hand,
it \emph{is} possible to convert a Turing bound
directly to a RAM bound via a simulation argument.
Suppose that the Turing machine has $t$ tapes
and an alphabet with $s$ symbols.
By inspecting an interval of width $w$ centred around each tape head,
the RAM can predict the next $\approx w/2$ steps
executed by the Turing machine.
Taking $w \approx (\log N) / (2 t \log s)$,
the number of possibilities is $s^{tw} \approx N^{1/2}$.
By building a precomputed lookup table of size roughly $N^{1/2}$
for each machine state (at a cost of $N^{1/2+o(1)}$ operations)
and encoding the contents of the Turing machine tapes
in the RAM memory using $\Theta(\log N)$ symbols per word,
the RAM can simulate $w/2 = \Theta(\log N)$ steps of the Turing machine
with only $O(1)$ word operations.
Thus, the bit complexity bounds stated in our main results below
(\Crefrange{thm:main-deterministic}{thm:main-heuristic})
may be converted to RAM word operation bounds
by simply dividing by $\Theta(\log N)$.

\subsection{Summary of new results}
\label{sec:new-results}

The main results of this paper are the following three theorems,
which give bit complexity bounds for finding all primes $p < N$
under various conditions.
Note that whenever we discuss an algorithm that generates all primes $p < N$,
it is expected that the primes are written in standard binary notation
and listed in increasing order.
The size of the output is therefore
$\sum_{p < N} \Theta(\log p) = \Theta(N)$ bits;
this observation also implies a lower bound of $\Omega(N)$
for the bit complexity.

\begin{rem}
Other output formats are of course possible.
One could write a bit array $(a_0, \ldots, a_{N-1})$
with $a_n = 1$ if and only if $n$ is prime.
This is actually equivalent to our format,
as one can convert back and forth in linear time
(\Cref{lem:convert-format}).
A more aggressive option would be to compress the data by storing
the \emph{differences} between consecutive primes.
As pointed out in \cite{Pri-compact,Pri-incremental},
under this scheme the total space required is only
$O(N \log \log N / \log N)$ bits.
It is conceivable that a Turing machine could generate
such a compressed table in (say)
$(N / \log N) (\log \log N)^{O(1)}$ bit operations,
but such a complexity bound seems hopelessly optimistic at present.
\end{rem}

Our first main result is as follows.
\begin{thm}[Deterministic enumeration of primes]
\label{thm:main-deterministic}
There is a Turing machine with the following properties.
Its input is a positive integer~$N$
and its output is the list of all primes $p < N$.
It runs in time
\[
N (\log N)^{1/2} (\log \log N)^{1+o(1)}.
\]
(See \Cref{rem:main-deterministic-precise} for a more precise bound.)
\end{thm}

\Cref{thm:main-deterministic} improves on
the previous best bit complexity bound \eqref{eq:sergeev-bound}
by a factor of almost $(\log N)^{1/2}$.
This result represents the first ever speedup by
a positive power of $\log N$ over Eratosthenes' sieve,
and closes almost half the gap between \eqref{eq:sergeev-bound}
and the trivial $\Omega(N)$ lower bound.
The proof is given in \Cref{sec:deterministic}.

To put this result into perspective,
observe that even \emph{writing down the integers from $1$ up to $N$}
takes more time; namely, $\Theta(N \log N)$ bit operations.
Another example is that
\emph{sorting an incorrectly ordered list of the primes $p < N$
into the correct order} requires $\Theta(N \log N)$ bit operations
via the standard merge sort algorithm (\Cref{lem:sort}).

Our second main result, proved in \Cref{sec:probabilistic},
shows that we can do even better if randomisation is permitted.
\begin{thm}[Probabilistic enumeration of primes]
\label{thm:main-probabilistic}
There is a probabilistic Turing machine with the following properties.
Its input is a positive integer~$N$.
Its output is either the list of all primes $p < N$, or ``FAIL''.
The second case occurs with probability at most~$\frac12$.
It runs in time
\begin{equation}
\label{eq:main-probabilistic}
N (\log \log N)^{2+o(1)}.
\end{equation}
(See \Cref{rem:main-probabilistic-precise} for a more precise bound.)
\end{thm}
By a probabilistic Turing machine,
we mean a Turing machine that is allowed to flip an independent,
unbiased coin at each step.
We stress that \emph{if} the algorithm in \Cref{thm:main-probabilistic}
returns a list of primes, then this output is guaranteed to be correct.
Therefore, by running the algorithm repeatedly until it succeeds,
with probability~$1$ we obtain a provably correct list of primes,
with the \emph{expected} running time given by
\eqref{eq:main-probabilistic}.
We will also see that the number of random bits needed by the algorithm
is very small, namely $O(\log^2 N)$.
(One might prefer to think of the algorithm as being deterministic,
but accepting an auxiliary input string of length $O(\log^2 N)$.
The theorem is then that the algorithm succeeds for a large fraction
of possible auxiliary inputs.)

Compared to \eqref{eq:sergeev-bound},
the bound \eqref{eq:main-probabilistic} entirely eliminates
the remaining positive power of $\log N$.
This result is especially striking when we consider the
\emph{average time} required to output each prime.
Whereas the best previous algorithms in the literature required
$O(\log^2 N)$ bit operations per prime,
and \Cref{thm:main-deterministic} improves this to
$(\log N)^{3/2+o(1)}$ bit operations,
in \Cref{thm:main-probabilistic} the work per prime falls to just
$\log N \tsp (\log \log N)^{2+o(1)}$ bit operations,
almost \emph{linear} in the bit size of each prime.
The difference is even more stark in the RAM model:
the best existing algorithms require $O(\log N / \log \log N)$
word operations per prime,
but \Cref{thm:main-probabilistic} reduces this to
$(\log \log N)^{2+o(1)}$ word operations,
i.e., from exponential in $\log \log N$ to only polynomial in $\log \log N$.

Finally, if one is willing to accept a reasonable hypothesis
concerning the number of ``square-primes''
(see \Cref{defn:square-prime})
in short intervals,
then our last main result shows that it is possible
to shave off another factor of $\log \log N$.
\begin{thm}[Heuristic enumeration of primes]
\label{thm:main-heuristic}
There is a (deterministic) Turing machine with the following properties.
Its input is a positive integer~$N$.
Its output is either the list of all primes $p < N$, or ``FAIL''.
If \Cref{conj:square-primes} is true, then the second case never occurs.
It runs in time
\begin{equation}
\label{eq:main-heuristic}
N (\log \log N)^{1+o(1)}.
\end{equation}
(See \Cref{rem:main-heuristic-precise} for a more precise bound.)
\end{thm}
In this result the average time per prime drops to
$\log N \tsp (\log \log N)^{1+o(1)}$ bit operations,
barely more than the $O(\log N \log \log N)$ cost of a
single multiplication of integers of bit size $\log N$ \cite{HvdH-nlogn}.
We emphasise that the algorithm never outputs an incorrect list of primes,
even if \Cref{conj:square-primes} is false.
(If the algorithm ever returns ``FAIL'', 
then we of course immediately obtain a disproof of the conjecture.)
The conjecture seems very likely to be true,
but also extremely difficult to prove.
See \Cref{sec:heuristic} for a detailed discussion,
and also for the proof of \Cref{thm:main-heuristic}.

The bound \eqref{eq:main-heuristic} is close to optimal
in the following weak sense.
The amount of ``entropy'' recovered by the core algorithm
is $\Theta(N \log \log N / \log N)$ bits
(see the discussion in \Cref{sec:overview}).
On the other hand,
standard conjectures concerning the complexity of polynomial multiplication
(see \Cref{rem:nlogn-conjecture})
suggest that our algorithms necessarily incur
an overhead of $\Theta(\log N)$ over the total bit size.
In combination, these observations suggest a lower bound
of $\Omega(N \log \log N)$ bit operations for any algorithm
along the same lines as ours.

\medskip
We have made no attempt to implement any of the new algorithms
on a real computer.
Our aim in this paper is to establish the theoretical complexity bounds
as directly as possible,
paying no attention to questions of practical efficiency.
The algorithms do not appear to be practical in their current form,
for several reasons.
\begin{itemize}
\item
The most serious problem is that the \emph{space} complexity (memory usage)
of these algorithms is very large.
We will not carry out a formal analysis,
but one can show that if the algorithms
are implemented in a straightforward way,
then in each of the main theorems,
the space complexity is smaller than the time complexity by
a factor of only $(\log \log N)^{O(1)}$,
and it does not seem easy to improve this by much.
By contrast, many existing algorithms admit ``segmented'' versions that
can generate primes in intervals (or arithmetic progressions)
of length $N^\beta$ using only $N^{\beta+o(1)}$ space for some $\beta < 1$
(see \Cref{sec:ram-results}).
The author has no idea how to achieve this for the new algorithms
without an unacceptable tradeoff in time complexity.
Therefore, at present we are stuck with needing roughly linear space
to generate the primes up to~$N$,
rendering the new algorithms useless for practical applications.

\item
The factors such as $(\log \log N)^{1+o(1)}$
appearing in the complexity bounds
are somewhat misleading from a practical point of view.
Some of these factors arise from quasilinear estimates
for the cost of operations on objects of size $(\log N)^{O(1)}$,
which are unlikely to be realistic for feasible values of~$N$.

\item
In our presentation we take a cavalier attitude towards constants,
preferring brevity over practicality in almost all cases.
For instance, in the very first step of the probabilistic algorithm
(see \Cref{sec:probabilistic-algorithm})
we implicitly assume that $N > \exp(80e^3) \approx 10^{698}$.
Many other equally outrageous examples could be given.
\end{itemize}

We conclude this section with a few open problems.
\begin{enumerate}
\item
Does there exist a ``segmented'' algorithm for finding the primes $p < N$
that runs in $O(N (\log N)^\alpha)$ bit operations,
and uses only $O(N^\beta)$ bits of space,
for some $\alpha < 1$ and $\beta < 1$?
The algorithms presented in this paper are the first to achieve $\alpha < 1$,
but only with $\beta = 1 + o(1)$,
or with a more careful analysis, possibly $\beta = 1 - o(1)$.
Previous algorithms have achieved $\beta < 1$ (see \Cref{sec:ram-results})
but only with $\alpha = 1 - o(1)$.

\item
Can the $(\log N)^{1/2}$ factor in the deterministic algorithm
(\Cref{thm:main-deterministic}) be improved?
In particular, can we achieve $N (\log N)^{o(1)}$ bit operations
deterministically and unconditionally?

\item
Can the $(\log \log N)^{2+o(1)}$ factor in the probabilistic algorithm
(\Cref{thm:main-probabilistic}) be improved to $(\log \log N)^{1+o(1)}$?
This would then match the heuristic algorithm (\Cref{thm:main-heuristic})
and would also agree with the exponent of $\log \log N$ appearing in the
weak $\Omega(N \log \log N)$ lower bound mentioned earlier.
For further discussion see \Cref{sec:probabilistic-complexity-gap}.

\item
Can the $(\log \log N)^{o(1)}$ factors be removed
in any of the three main theorems?
For Theorems \ref{thm:main-deterministic} and \ref{thm:main-probabilistic}
it would be enough to improve the complexity of polynomial multiplication
over $\FF_2$: see \Cref{rem:nlogn-conjecture},
\Cref{rem:main-probabilistic-precise} and
\Cref{rem:main-deterministic-precise}.
The situation for \Cref{thm:main-heuristic} is much more complicated: see
\Cref{rem:cyclotomic},
\Cref{sec:restricted-faster},
\Cref{rem:improve-compression},
\Cref{rem:varphi-W-speedup},
\Cref{rem:W-logloglogN} and
\Cref{rem:main-heuristic-precise}.
 
In the RAM model,
removing the $(\log \log N)^{o(1)}$ factors seems much more plausible,
at least for Theorems \ref{thm:main-deterministic}
and \ref{thm:main-probabilistic}: see \Cref{rem:RAM-multiply}.
Again, the situation for Theorem \ref{thm:main-heuristic} is more complicated:
see \Cref{rem:RAM-transform} and \Cref{rem:RAM-restricted}.
\end{enumerate}

In the rest of \Cref{sec:introduction},
we will survey previous results in the literature
for the RAM model (\Cref{sec:ram-results})
and the Turing model (\Cref{sec:turing-results}),
and then give an overview of the new algorithms (\Cref{sec:overview}).

\subsection{Previous results in the RAM model}
\label{sec:ram-results}

Eratosthenes' sieve visits each composite integer several times,
once for each prime divisor $p < \sqrt N$.
Historically, the first improvements over \eqref{eq:eratosthenes-bound}
arose by reorganising the sieve so that
each composite is visited exactly once.
The first such ``linear'' sieves, with complexity $O(N)$,
were described independently and roughly simultaneously in
\cite{Mai-primes} and \cite{Pratt-CGOL}%
\footnote{%
In the literature the paper \cite{Pratt-CGOL} is consistently
attributed to Gale and Pratt.
In fact, Pratt was the sole author.
The full version (currently available at the URL indicated in the bibliography)
has a ``SAMPLE PROGRAMS'' supplement that includes a linear-time
sieve of Eratosthenes explicitly labelled
``Due to Ross Gale and Vaughan Pratt''.
The miscitation apparently originated in \cite{GM-linear-sieve}
and then propagated through all of the subsequent literature.
}.
Other algorithms achieving linear time were given in
\cite{GM-linear-sieve}, \cite{Mis-program} and \cite{Ben-incremental}.
As explained in the survey paper \cite{Pri-familytree},
these methods may all be recognised as applying the same
general sieve strategy,
but in each case taking advantage of a different canonical form
for composite numbers.

The next asymptotic speedup was obtained by means of
the so-called ``wheel'' technique.
To illustrate the basic idea,
observe that by organising the sieve into blocks of size
$2 \times 3 \times 5 = 30$,
and ignoring those integers not coprime to~$30$,
we obtain some constant factor speedup.
By allowing the set of small primes to grow with $N$,
one can achieve a speedup proportional to $\log \log N$,
leading to the bound \eqref{eq:pritchard-bound},
i.e., $O(N / \log \log N)$ word operations.
This was first proved by Pritchard in \cite{Pri-sublinear}
and explained more succinctly in \cite{Pri-wheel}.
(An earlier incorrect attempt was made in \cite{Mai-primes}.)
As mentioned in \Cref{sec:eratosthenes},
this was the best complexity bound known for the RAM model
prior to the present work.

We remark in passing that much of the literature in the 1970s and 80s
is preoccupied with minimising the number of
multiplications relative to additions,
since this had an enormous payoff in practical implementations.
However, in the theoretical RAM model there is really no difference,
as one can build lookup tables to compute word-sized products in $O(1)$ time.
In the context of prime enumeration algorithms,
this fact seems to have been first pointed out explicitly
in \cite[\S2]{DJS-space}.

A different approach leading to the same $O(N / \log \log N)$ bound,
without using the sieve of Eratosthenes at all,
was given by Atkin and Bernstein \cite{AB-prime-sieves}.
(A preprint version appeared as early as 1999.)
The Atkin--Bernstein sieve plays an important role in the present paper,
so we explain it here in a little more detail.

Consider first the problem of finding all primes $p < N$,
$p \equiv 1 \pmod 4$.
It turns out that an integer $n \equiv 1 \pmod 4$ is prime
if and only if it is squarefree
and has an \emph{odd} number of representations
in the form $4u^2 + v^2$, with $u, v > 0$;
this follows from arithmetic properties
of the ring $\ZZ[\sqrt{-1}]$ \cite[Thm.\,6.1]{AB-prime-sieves}.
The Atkin--Bernstein algorithm proceeds by iterating over all pairs $(u,v)$
of positive integers in the elliptical region $4u^2 + v^2 < N$,
updating a bit array to keep track of the parity of how many times
each integer up to $N$ is represented by $4u^2 + v^2$.
After throwing away the non-squarefree integers,
the remaining integers with an odd number of such representations
are the primes.
The complexity bound follows from the fact that there are
$O(N)$ pairs $(u,v)$ to enumerate,
with a factor of $\log \log N$ being saved by
skipping useless congruence classes modulo a few small primes.

The integers $n \equiv 7 \pmod{12}$ are handled by a similar method
involving the quadratic form $3u^2 + v^2$;
the underlying ring here is $\ZZ[\tfrac12(-1 + \sqrt{-3})]$
\cite[Thm.\,6.2]{AB-prime-sieves}.
This leaves the case $n \equiv 11 \pmod{12}$,
which is handled via the \emph{indefinite} form $3u^2 - v^2$,
corresponding to the ring $\ZZ[\sqrt 3]$ \cite[Thm.\,6.3]{AB-prime-sieves}.
In this last case, one must iterate over points
in a region bounded by a hyperbola instead of an ellipse.

There are many other papers in the literature on prime enumeration
in a RAM-like model,
dealing with issues other than time complexity:
\begin{itemize}
\item
\textit{Space complexity}.
Early references for the ``segmented sieve'',
which uses only $O(N^{1/2})$ bits of space,
include \cite{Sin-357} and \cite{Bre-large-gaps}.
A more formal analysis was given in \cite{BH-segmented}.
(Of course, we assume here that the primes are processed in batches,
i.e., each interval of length $N^{1/2}$ is sieved separately,
with the primes being discarded before the next interval is processed.)

The Atkin--Bernstein algorithm \cite{AB-prime-sieves} was the first
to achieve $O(N^{1/2})$ space and $O(N / \log \log N)$ time simultaneously.
The space complexity of the Atkin--Bernstein sieve
was reduced to $O(N^{1/3})$ bits by Galway \cite{Gal-dissect}
(with full details provided in \cite[Ch.\,5]{Gal-thesis}).
Recently Helfgott adapted these ideas to achieve $N^{1/3+o(1)}$
space for the Eratosthenes sieve \cite{Hel-eratosthenes}.
Various other space-time tradeoffs are investigated in
\cite{Pri-compact}, \cite{DJS-space},
\cite{Sor-trading}, \cite{Sor-pseudosquares} and \cite{HS-space}.

\item
\textit{Incremental sieving}.
Instead of finding all primes up to a requested bound~$N$,
an ``incremental'' sieve uses a dynamic data structure to find
each successive prime with very little work.
This is perhaps not a very clearly defined notion,
but several authors have given interesting algorithms along these lines,
including \cite{Ben-incremental}, \cite{Pri-incremental},
and \cite{Sor-compact-incremental}.
Our new algorithms are in some sense
the opposite of an incremental sieve:
they provide almost no information about any primes
at all until the algorithm is near completion.

\item
\textit{Parallelism}.
The paper \cite{SP-parallel} gives a detailed analysis of sieve methods
in a parallel RAM model,
and \cite{BCCFGJMSS-IO} studies their I/O complexity.
Further references on these topics may be found in the recent survey paper
\cite[\S3.3]{GP-sieving}.
The latter also gives some pointers to a few state-of-the-art
practical implementations.
\end{itemize}

\subsection{Previous results in the Turing model}
\label{sec:turing-results}

The literature on prime enumeration in the Turing model is rather
sparse compared to the RAM model.
The reason for this is probably that
sieving is quite an unnatural approach in the Turing model,
due to the lack of random array access.
For example, in the sieve of Eratosthenes,
the cost in the Turing model of marking off the multiples of $p$ up to $N$
becomes \emph{linear} in~$N$,
so the total complexity over all $p < \sqrt N$ balloons to about $N^{3/2}$.

The first nontrivial result for the Turing model was given by
Sch\"onhage, Grotefeld and Vetter \cite[\S6.5]{SGV-turing}.
Their algorithm is very simple.
They first enumerate the primes $p < \sqrt N$ in time $N^{1/2+o(1)}$
using any reasonable method.
For each of these~$p$,
they use repeated addition to generate a list of the
multiples $kp$ (for $k \geq 2$) up to~$N$.
Next, they compute the union of these lists by repeatedly merging
pairs of lists,
noting that one can merge two sorted lists
on a Turing machine in linear time (\Cref{lem:merge}).
This produces a sorted list of all composite numbers up to~$N$.
Finally, they compute the complement to obtain the list of primes up to~$N$.
The complexity is dominated by the merging step, and is given by
\begin{equation}
\label{eq:SGV-bound}
O(N \log^2 N \log \log N).
\end{equation}
The authors also make a cryptic remark that
\eqref{eq:SGV-bound} is not optimal and can be improved substantially,
but no hints are given.

Farach-Colton and Tsai \cite{FT-prime-tables}
present two different algorithms.
First, they explain how to adapt the
list-and-merge strategy of \cite{SGV-turing}
to the Atkin--Bernstein sieve,
leading to an algorithm that runs in
\begin{equation}
\label{eq:FT-bound1}
O\left(\frac{N \log^2 N}{\log \log N}\right)
\end{equation}
bit operations \cite[\S3.1]{FT-prime-tables}.
This improves on \eqref{eq:SGV-bound} by a factor of $(\log \log N)^2$.

Their second algorithm \cite[\S3.2]{FT-prime-tables},
while being somewhat slower than the first,
contains the seeds of the approach taken in the present paper.
Recall that in the Atkin--Bernstein sieve,
for each $n \equiv 1 \pmod 4$
one must count the number of positive integer pairs $(u,v)$
such that $4u^2 + v^2 = n$.
Farach-Colton and Tsai propose doing this for all $n$
\emph{simultaneously} by considering the series product
\begin{equation}
\label{eq:FT-product}
\sum_{u \geq 1} x^{4u^2} \cdot
   \sum_{\substack{v \geq 1 \\ \textn{$v$ odd}}} x^{v^2}
   = (x^4 + x^{16} + x^{36} + \cdots) (x + x^9 + x^{25} + \cdots).
\end{equation}
For any $n \equiv 1 \pmod 4$,
the coefficient of $x^n$ in \eqref{eq:FT-product}
is exactly the number of representations of $n$
in the form $4u^2 + v^2$ with $u, v > 0$.
They suggest using Kronecker substitution \cite[Cor.\,8.27]{vzGG-compalg3}
to compute this product up to~$x^N$,
i.e., packing the coefficients into large integers
and then multiplying those integers.
Similar techniques can be applied to the forms $3u^2 + v^2$ and $3u^2 - v^2$
to handle the integers $n \equiv 3 \pmod 4$.
The authors show that this leads to the complexity bound
\begin{equation}
\label{eq:FT-bound2-mint}
O(\Mint(N \log N))
\end{equation}
for enumerating the primes $p < N$,
where $\Mint(n)$ denotes the cost of multiplying integers with $n$ bits.
Using the currently best available bound
$\Mint(n) = O(n \log n)$ \cite{HvdH-nlogn},
the overall running time becomes
\begin{equation}
\label{eq:FT-bound2}
O(N \log^2 N)
\end{equation}
bit operations.
This is inferior to \eqref{eq:FT-bound1},
and considerably worse than Sergeev's complexity bound
to be discussed shortly,
but as we will see in the next section,
it can be substantially improved.

\begin{rem}
\label{rem:FT-complications}
The analysis leading to \eqref{eq:FT-bound2-mint}
is more involved than one might expect.
The most straightforward estimate of the cost of computing the product
\eqref{eq:FT-product} modulo $x^N$ using Kronecker substitution
is indeed $O(\Mint(N \log N))$.
Moreover, Farach-Colton and Tsai show that this estimate
can be improved by a factor of $\log \log N$,
by proving nontrivial bounds on the size of the coefficients
in \eqref{eq:FT-product} \cite[Lem.\,5]{FT-prime-tables}.
Unfortunately, the counting problem for the $3u^2 - v^2$ case
cannot be encoded directly as a polynomial multiplication problem,
and their workaround for this issue loses a factor of $\log \log N$ again;
see \cite[Lem.\,8]{FT-prime-tables}.
\end{rem}

Shortly after \cite{FT-prime-tables} appeared,
Sergeev \cite{Ser-prime-turing} found a way to streamline the
original list-and-merge strategy of \cite{SGV-turing}
to obtain the current best Turing bound \eqref{eq:sergeev-bound}.
This was the first bound in the Turing model
for which the power of $\log N$ matches the original Eratosthenes bound
(after the usual translation from bit complexity to word operations).
Sergeev's improvements may be summarised as follows:
\begin{itemize}
\item
He incorporates Mairson's strategy \cite{Mai-primes},
so that each composite number is listed exactly once,
instead of once for each prime divisor $p < \sqrt N$.
\item
He merges the lists in a specific order,
taking into account the varying sizes of the lists.
\item
His most important innovation, which accounts for most of the speedup,
is that he represents the intermediate sorted lists of composites using
a ``compressed'' format.
This format takes advantage of the fact that in long runs of composites,
the leading bits in their binary representations are identical,
and do not need to be stored.
\end{itemize}

\subsection{Overview of the new algorithms}
\label{sec:overview}

In this section we explain the main ideas behind the new algorithms.
Our starting point is the second algorithm of Farach-Colton and Tsai
mentioned above \cite[\S3.2]{FT-prime-tables}.

At the outset, we make one modification that is very minor in principle
but affects many details throughout the paper:
instead of using the forms $4u^2 + v^2$ and $3u^2 \pm v^2$
inherited from \cite{AB-prime-sieves},
we will switch to $a^2 + 4b^2$ and $a^2 \pm 2b^2$.
Everything could be made to work with the original forms,
but a few technicalities become easier to manage with this choice of forms.
The constants in various complexity bounds may be affected by the switch,
but we have not investigated this.

The analogue of \eqref{eq:FT-product} in our setup runs as follows.
Define $H_{-1}(x) \coloneqq F(x) G_{-1}(x)$ where
\begin{align*}
F(x) & \coloneqq \sum_{\substack{a \geq 1 \\ \textn{$a$ odd}}} x^{a^2}
   = x + x^9 + x^{25} + \cdots, \\
G_{-1}(x) & \coloneqq \sum_{b \geq 1} x^{4b^2}
   = x^4 + x^{16} + x^{36} + \cdots.
\end{align*}
As in \cite{FT-prime-tables},
this series encodes information about the primes $p \equiv 1 \pmod 4$.
For the $p \equiv 3 \pmod 8$ case there is a corresponding product
$H_{-2}(x) \coloneqq F(x) G_{-2}(x)$ arising from the form $a^2 + 2b^2$.
For $a^2 - 2b^2$ we cannot use an ordinary power series product;
instead, we introduce an operation called the \emph{restricted product},
denoted by $F \restrict G$.
The primes $p \equiv 7 \pmod 8$ are then covered by
$H_2(x) \coloneqq F(x) \restrict G_2(x)$ for a suitable series $G_2(x)$.
See \Cref{sec:generating-function} for the definitions.

\medskip
To explain our first algorithmic improvement over \cite{FT-prime-tables},
recall that the Atkin--Bernstein primality condition is sensitive to only
the \emph{parity} of the number of representations of each~$n$.
It therefore suffices to compute the product $H_{-1} = F \cdot G_{-1}$
in $\FF_2[x]$ rather than $\ZZ[x]$.
By deploying polynomial multiplication algorithms
designed expressly for $\FF_2[x]$,
we avoid the coefficient growth inherent in
the Kronecker substitution method from \cite{FT-prime-tables}.

Let $\Mcost(n)$ denote the cost of multiplying
polynomials of degree $n$ in $\FF_2[x]$.
It is known that we may take $\Mcost(n) < n \log n \, (\log \log n)^{o(1)}$
(see \Cref{sec:poly-mult}).
Applying this result to compute $H_{-1} = F \cdot G_{-1}$ up to $x^N$
leads to the complexity bound
\begin{equation}
\label{eq:FT-improvement1}
N \log N \, (\log \log N)^{o(1)}
\end{equation}
for enumerating the primes $p < N$, $p \equiv 1 \pmod 4$.
This improves on \eqref{eq:FT-bound2} by a factor of almost $\log N$,
bringing us within striking distance of
Sergeev's bound \eqref{eq:sergeev-bound}.
The same bound \eqref{eq:FT-improvement1}
may be achieved for the primes $p \equiv 3 \pmod 4$,
but as in \Cref{rem:FT-complications},
the $p \equiv 7 \pmod 8$ case is more involved due to the
restricted product in the definition of $H_2(x)$;
we will not discuss this further here.

Incidentally, one can improve \eqref{eq:FT-improvement1}
by a further factor of $\log \log N$ without too much additional effort,
by again arranging to skip useless congruence classes
modulo small primes.
This leads to the complexity bound
\[
\frac{N \log N}{(\log \log N)^{1-o(1)}},
\]
which is comfortably below the current Turing machine record
\eqref{eq:sergeev-bound},
and almost matches the current RAM record \eqref{eq:pritchard-bound}.
We will not go into the details,
as this small speedup will be completely overwhelmed by the
much larger gains to be discussed in the rest of this section.

\medskip
Indeed, the bound \eqref{eq:FT-improvement1} is still almost a factor
of $\log N$ short of our target complexity bounds
\eqref{eq:main-probabilistic} and \eqref{eq:main-heuristic}.
The key to unlocking this additional speedup is
the simple observation that by the prime number theorem,
\emph{only a fraction of about $1/\log N$ of numbers less than $N$ are prime}.
Essentially, our plan is to take advantage of the \emph{sparsity}
of $H_{-1} = F \cdot G_{-1}$ to accelerate the polynomial multiplication.

Before describing how we do this,
we must confront an annoying technical detail,
which is that the Atkin--Bernstein condition for the primality of $n$
only applies when $n$ is \emph{squarefree}.
A positive proportion of integers are not squarefree,
and we cannot ignore their contribution to~$H_{-1}$.
Instead, we must analyse the behaviour of the coefficients of $x^n$
in $H_{-1}(x)$ for \emph{all}~$n$.
What we will prove is that
\[
H_{-1}(x) \equiv
   \sum_{\substack{m \geq 1 \\ \textn{$m$ odd}}} \;
   \sum_{\substack{p \equiv 1 \bmod 4 \\ l \geq 1}}
   x^{m^2 p^l} \pmod 2.
\]
Observe that this formula implies the Atkin--Bernstein criterion:
the exponent $m^2 p^l$ can only be squarefree when $m = l = 1$,
and the sum of such terms is exactly
$\sum_{p \equiv 1 \bmod 4} x^p$.
Similar identities can be proved for $H_{-2}(x)$ and $H_2(x)$,
and combining them we will find (see \Cref{thm:H-congruence}) that
\begin{equation}
\label{eq:H-sum-congruence-overview}
(H_{-1} + H_{-2} + H_2)(x) \equiv
   \sum_{\substack{m \geq 1 \\ \textn{$m$ odd}}} \;
   \sum_{\substack{p \geq 3 \\ l \geq 1}}
   x^{m^2 p^l} \pmod 2.
\end{equation}
The formula \eqref{eq:H-sum-congruence-overview} motivates
the following definition.
\begin{defn}
\label{defn:square-prime}
A \emph{square-prime} is an integer of the form $m^2 p$
where $m \geq 1$ is an integer and $p \geq 2$ is prime.
\end{defn}
In particular, an integer $n$ is an \emph{odd} square-prime if and only if
it is of the form $n = m^2 p$ where $m$ is odd and $p \geq 3$ is prime.
For most of the paper we will concentrate on the problem of enumerating
the odd square-primes $n < N$,
as this is what falls out naturally
from \eqref{eq:H-sum-congruence-overview}.
The first few odd square-primes are
\[
3, 5, 7, 11, 13, 17, 19, 23, \underline{27}, 29, 31, 37, 41, 43, 
   \underline{45}, 47, 53, 59, 61, \underline{63}, 67, 71, 73,
   \underline{75}, \ldots
\]
where we have underlined the integers that are not themselves prime.
After finding the odd square-primes up to~$N$,
it is straightforward to extract the actual primes
(\Cref{prop:square-primes-to-primes}).

\begin{rem}
These integers do not seem to have been explored much in the literature.
There is a closely related entry \href{https://oeis.org/A228056}{A228056}
in the Online Encyclopedia of Integer Sequences \cite{OEIS}
listing integers of the form $m^2 p$ where $m \geq 2$ and $p$ is prime,
i.e., the same as our definition but excluding the primes themselves.
Bhat \cite{Bhat-square-prime} calls such integers ``square-primes''
and investigates some of their properties,
including computing their density.
\end{rem}

Having shifted focus from primes to odd square-primes,
let us return to the question of how to accelerate
the computation of $H_{-1} = F \cdot G_{-1}$.
Let $T \geq 1$ be an integer parameter,
and define a sequence of vectors $a^r \in \FF_2^T$
representing the locations of the odd square-primes
in intervals of length~$T$:
for $r \geq 0$ and $0 \leq t < T$ we set $(a^r)_t \coloneqq 1$ if
the corresponding integer $n = rT + t$ is an odd square-prime,
and otherwise $(a^r)_t \coloneqq 0$.
Then by \eqref{eq:H-sum-congruence-overview} we have
\[
(H_{-1} + H_{-2} + H_2)(x) \equiv
   E_1(x) + \sum_{r \geq 0} \; \sum_{0 \leq t < T}
   (a^r)_t \, x^{rT + t} \pmod 2,
\]
where $E_1 \in \FF_2\bbracket{x}$ is an error term accounting for
the terms $x^{m^2 p^l}$ with $l \geq 2$.
We will ignore $E_1(x)$ for now;
it is very sparse and must be treated separately.

Given some large $N$, the odd square-primes up to $N$
are encoded by the vectors $a^r$ for $r < N/T$.
Let us adopt a statistical mindset and ask what these $a^r$ look like,
assuming that the distribution of odd square-primes
is ``random'' with respect to intervals of length~$T$.
We assume here that $T$ is much smaller than $N$ but not too small,
for now say $T = (\log N)^{\Theta(1)}$.
The prime number theorem implies that such intervals
contain on average about $T / \log N$ primes.
The situation for the odd square-primes is similar:
they have density $\pi^2/8 \approx 1.2337$ times that
of the ordinary primes,
and an interval of length $T$ contains on average
about $(\pi^2/8) \tsp T / \log N$ odd square-primes
(see \Cref{rem:odd-square-primes-asymptotic}).
This suggests that the ``entropy'' or ``information content'' of
a typical such $a^r$ is $O(T \log T / \log N)$ bits,
since we could represent $a^r$ by writing down a list
of the roughly $(\pi^2/8) \tsp T / \log N$ integers
$t \in \{0, \ldots, T-1\}$ for which $(a^r)_t = 1$.
Of course we cannot \emph{provably} represent $a^r$
using $O(T \log T / \log N)$ bits,
because we cannot rule out the existence of intervals containing
many more odd square-primes than expected.
Nevertheless it is true that $O(T \log T / \log N)$ bits
\emph{usually} suffice to represent~$a^r$.

The preceding discussion motivates the introduction of a function
$\kappa \colon \FF_2^T \to \FF_2^S$ that we call the \emph{compression map}.
The construction of $\kappa$,
which is inspired by the theory of Reed--Solomon codes \cite{RS-codes},
is described in detail in \Cref{sec:compression}.
It is a linear map, depending on $T$ and a ``threshold'' parameter~$R$,
with the following key property:
if $a \in \FF_2^T$ with $\wt(a) \leq R$,
then $a$ can be recovered unambiguously from~$\kappa(a)$.
Here $\wt(a)$ denotes the \emph{weight} of~$a$,
i.e., the number of nonzero entries in~$a$.
The process of recovering $a$ from $\kappa(a)$ is called \emph{decompression}.
Using the compression map,
we may represent any vector $a \in \FF_2^T$ of weight up to $R$
by the value $\kappa(a) \in \FF_2^S$, i.e., using only $S$ bits.
We will see in \Cref{sec:compression} that it is possible
to construct such a map $\kappa$ with $S \asymp R \log T$.

For the vectors~$a^r$,
to maintain a safe margin above the expected weight
$(\pi^2/8) \cdot T / \log N$,
we might take for instance $R \approx 2 T / \log N$.
In this case $S \asymp T \log T / \log N$,
which matches our earlier entropy estimate for $a^r$ (up to a constant).
In this sense, the compression map succeeds in representing each $a^r$
using only $O(T \log T / \log N)$ bits ---
provided of course that always $\wt(a^r) \leq R$.

\medskip
Returning to the product $H = F \cdot G_{-1}$,
recall that the textbook strategy for multiplying polynomials
is to first compute \emph{forward transforms} of the two input polynomials
(using some sort of fast Fourier transform),
then multiply the transforms pointwise,
and finally perform an \emph{inverse transform} to recover the product.
The complexity is typically dominated by the forward and inverse transforms,
with the pointwise multiplication step making a negligible contribution.

To take advantage of the compression map in the
computation of $H = F \cdot G_{-1}$, we must modify this strategy.
Let us begin by decomposing the input polynomials into~$T$ ``slices'',
writing
\begin{align*}
     F(x) & = \mspace{11.5mu} F^0(x^T) + \mspace{11.5mu} x F^1(x^T)
      + \cdots + x^{T-1} F^{T-1}(x^T), \\
G_{-1}(x) & = G_{-1}^0(x^T) + x G_{-1}^1(x^T)
   + \cdots + x^{T-1} G_{-1}^{T-1}(x^T).
\end{align*}
Expanding out the product $H_{-1} = F \cdot G_{-1}$ yields
\[
H_{-1}(x) = H_{-1}^0(x^T) + x H_{-1}^1(x^T)
   + \cdots + x^{T-1} H_{-1}^{T-1}(x^T)
\]
where
\begin{equation}
\label{eq:Ht-formula}
H_{-1}^t(y) \coloneqq
   \sum_{\substack{0 \leq i, j < T \\ i + j = t}} F^i(y) G_{-1}^j(y) +
   \sum_{\substack{0 \leq i, j < T \\ i + j = t + T}} y F^i(y) G_{-1}^j(y),
   \qquad 0 \leq t < T.
\end{equation}
One can also write down similar formulas for $H_{-2} = F \cdot G_{-2}$
and $H_2 = F \restrict G_2$.
To obtain $(H_{-1} + H_{-2} + H_2)(x)$ up to~$x^N$,
we need to compute $(H^t_{-1} + H^t_{-2} + H^t_2)(y)$
up to about $y^{N/T}$ for each~$t$.
As a first attempt at an algorithm for this problem,
we might try to adapt the textbook strategy as follows:
\begin{enumalgo}
\item
Compute forward transforms of the slices $F^i(y)$ and $G_d^j(y)$
(truncated to degree about $N/T$).
\item
Evaluate the sums of products on
the right hand side of \eqref{eq:Ht-formula},
and the analogous expressions for $H_{-2}$ and~$H_2$,
working in ``Fourier space''.
\item
Perform inverse transforms to recover the $T$ output polynomials
$H^t(y) \coloneqq (H_{-1}^t + H_{-2}^t + H_2^t)(y)$ up to $y^{N/T}$.
\end{enumalgo}
Comparing to the textbook strategy,
we see that so far nothing has really changed.
The number of transforms has increased from $\Theta(1)$ to $\Theta(T)$,
but their length has decreased from $\Theta(N)$ to $\Theta(N/T)$,
and these effects cancel each other out as fast Fourier transforms
run in quasi-linear time.
The cost of Step~2 remains negligible.
To achieve our target speedup,
we must somehow improve the complexity of both Steps~1 and~3
by a factor of roughly $\log N$.

For Step~1 this is actually quite easy.
The exponents in $F(x)$ and $G_{-1}(x)$ are always squares modulo~$T$,
so if we take $T$ to be a product of many small primes,
then a large proportion of the $F^i(y)$ and $G_{-1}^j(y)$
are automatically zero, and their transforms may be skipped.
Similar remarks apply to $H_{-2}$ and~$H_2$.
(We are being slightly dishonest here:
our choice $T = (\log N)^{\Theta(1)}$ is not quite
big enough to deliver savings comparable to $\log N$.
In the real algorithm we instead use a more complicated decomposition
involving a second slicing parameter~$W$.
See \Cref{rem:W-bigger-than-T}.)

The speedup for Step~3 is more involved,
and this is where the compression map
$\kappa \colon \FF_2^T \to \FF_2^S$ enters the game.
The key idea is that instead of performing the inverse transforms directly,
we first apply the compression map to the output of Step~2,
\emph{still working in Fourier space}.
Then in Step~3 we only need to perform $S$ inverse transforms instead of~$T$.
At this point what we have actually computed is
a list of compressed vectors $\kappa(a^r)$, or equivalently,
$S$ carefully chosen \emph{linear combinations} of the~$H^t$.
We finish up by applying the decompression routine to recover the~$a^r$,
or equivalently, to ``solve'' for the~$H^t$.

The net effect is that the number of transforms in Step 3 drops
by a factor equal to the ``compression ratio'' $T/S$.
For the choice $T = (\log N)^{\Theta(1)}$ and $R \asymp T / \log N$
mentioned earlier,
this results in a speedup relative to \eqref{eq:FT-improvement1}
by a factor of $T/S \asymp \log N / \log \log N$,
leading to the overall complexity bound $N (\log \log N)^{1+o(1)}$.
This explains more or less where the bound \eqref{eq:main-heuristic}
for the heuristic algorithm comes from,
although a complete proof must include a detailed analysis of
all steps, including compression, decompression and the ``pointwise''
operations in Step~2.
All of these steps are non-trivial,
and we will have to work quite hard to
stay within the target complexity bound.
We will refer to the procedure just sketched as the ``core algorithm'',
and it is worked out in detail in \Cref{sec:core}.

\begin{rem}
\label{rem:key-idea}
The trick in Step 3 of commuting the inverse transforms
through the compression map is the
principal technical innovation of the paper.
It would not be an exaggeration to say that every aspect
of the design of the main algorithms was ultimately driven
by the requirements of implementing this step.
The trick may well have other applications ---
any algorithm that involves multiplying polynomials to obtain
a sparse result could potentially benefit.
\end{rem}

Of course, there is a catch:
the core algorithm only correctly recovers a given $a^r$
if $\wt(a^r) \leq R$.
If $\wt(a^r) > R$, then we might be in trouble.
In some situations the decompression routine is able to report
that something has gone wrong:
in particular, it can recognise when there does not exist
\emph{any} $a' \in \FF_2^T$ with $\wt(a') \leq R$
such that $\kappa(a') = \kappa(a^r)$.
We then have a chance to test the square-primality of each
integer in the interval by some other means.
But if we are unlucky, such an ``impostor'' $a'$ may actually exist,
and then the decompression routine will simply return $a'$
with no indication of the mistake.

At this point we must consider several different approaches,
depending on how we want to handle these problematic intervals.
\begin{itemize}[itemsep=5pt]
\item
\textit{The heuristic algorithm} (\Cref{sec:heuristic}).
For this version we simply \emph{assume} that no problematic intervals exist.
More precisely, \Cref{conj:square-primes} proposes a specific upper bound
for the number of square-primes that can occur in a short interval.
In \Crefrange{sec:primes-very-short}{sec:square-primes-very-short}
we give some heuristic arguments and numerical evidence
in favour of the conjecture.
Then in \Cref{sec:heuristic-algorithm} we
present the main heuristic algorithm and prove \Cref{thm:main-heuristic}.

Note that if the conjecture is false and the algorithm consequently
errs at the decompression stage,
then the final list of odd square-primes will have too few elements.
The algorithm checks for this at the very end by means of a
fast prime-counting algorithm (\Cref{prop:count-square-primes}),
and so in this case correctly returns ``FAIL''.

\item
\textit{The probabilistic algorithm} (\Cref{sec:probabilistic}).
In this variant we replace $\kappa$ by a ``randomised'' compression map
$\kappa_\pi \colon \FF_2^T \to \FF_2^S$
depending on a randomly chosen parameter~$\pi$.
With this modification, and for suitable choices of various parameters,
we can prove that if there exist any intervals with $\wt(a^r) > R$,
then with high probability we will detect all of them
at the decompression stage,
i.e., with high probability no impostors $a'$ will exist for any $r < N/T$
(\Cref{sec:permuted-core}).
In the unlikely event that an impostor exists and we fail to spot it,
we can again use fast prime-counting techniques to recognise the mistake
at the very end.

To handle the intervals with $\wt(a^r) > R$,
we will use a modified version of the
AKS primality test to directly test the affected integers
for square-primality (\Cref{sec:modified-AKS}).
To estimate the cost of this step,
we need an upper bound for the number of intervals
containing an unusually large number of odd square-primes.
Fortunately, sieve methods are strong enough to prove such bounds;
the details are worked out in \Cref{sec:sieve}.
(\Cref{conj:square-primes} says that
such intervals should not exist at all,
but sieve methods are apparently too weak to prove this.)
Finally, the proof of \Cref{thm:main-probabilistic} is given
in \Cref{sec:probabilistic-algorithm}.

\item
\textit{The deterministic algorithm} (\Cref{sec:deterministic}).
If we forgo randomisation,
then there does not seem to be any rigorous way of avoiding impostors.
We must instead find some other means of detecting intervals with large
values of $\wt(a^r)$.
For this purpose we will deploy (a slight modification of)
Sergeev's algorithm to identify all integers $n < N$ having
no prime divisors less than a given bound~$Y$
(\Cref{sec:sergeev}).
This information can then be used to rule out
large values of $\wt(a^r)$ for most intervals.
As in the probabilistic case,
the remaining intervals must be checked directly via AKS,
and we again need good bounds for the number of such intervals,
which are obtained by sieve methods (\Cref{sec:sieve}).
There is a tradeoff between the cost of running Sergeev's algorithm,
which depends mainly on~$Y$,
and the cost of running the core algorithm
for a given $\wt(a^r)$ threshold,
which depends mainly on the ratio~$T/R$.
Optimising these parameters leads to \Cref{thm:main-deterministic},
whose proof is given in \Cref{sec:deterministic-algorithm}.
\end{itemize}

\subsection{Acknowledgments}

The author would like to thank Igor Sergeev and Daniel Johnston
for helpful conversations.
Many thanks to Anthropic's Claude Opus 4.8
for proofreading the paper and finding a number of embarrassing
but fortunately minor mistakes.
Any remaining errors are of course the sole responsibility of the author.

\section{A generating function for primes}
\label{sec:generating-function}

The main result of this section is \Cref{thm:H-congruence},
which gives a congruence for a certain generating function
involving the primes.
This result plays a central role in the core algorithm
presented in \Cref{sec:core}.
To state the theorem we must introduce some terminology and notation.

\begin{defn}
Let
\[
\dset \coloneqq \{-1, -2, 2\},
\]
and for each $d \in \dset$ define a character
\[
\chi_d(n) \coloneqq \left(\frac{d}{n}\right), \qquad \textn{$n$ odd},
\]
where $\left(\frac{\cdotspace{1pt}}{\cdot}\right)$
denotes the Jacobi symbol.
\end{defn}

These characters are given explicitly by
\begin{align*}
\chi_{-1}(n) & =
\begin{cases}
   \phantom{-} 1, & n \equiv 1,5 \tpmod 8, \\
   -1,            & n \equiv 3,7 \tpmod 8,
\end{cases} \\
\chi_{-2}(n) & =
\begin{cases}
   \phantom{-} 1, & n \equiv 1,3 \tpmod 8, \\
   -1,            & n \equiv 5,7 \tpmod 8,
\end{cases} \\
\chi_2(n) & =
\begin{cases}
   \phantom{-} 1, & n \equiv 1,7 \tpmod 8, \\
   -1,            & n \equiv 3,5 \tpmod 8.
\end{cases}
\end{align*}

\begin{defn}
\label{defn:restricted-product}
Let $R$ be a ring,
and let $f \in R\bbracket{x}$ and $g \in R\bbracket{x^{-1}}$, say
\[
f(x) = \sum_{i \geq 0} f_i x^i, \qquad
g(x) = \sum_{j \geq 0} g_j x^{-j}.
\]
Then the \emph{restricted product} of $f$ and $g$ is the series
\begin{equation}
\label{eq:restricted-product}
(f \restrict g)(x) \coloneqq
   \sum_{\substack{j \geq 0 \\ i \geq 2j}} f_i g_j x^{i-j} \in R\bbracket{x}.
\end{equation}
\end{defn}
The sum in \eqref{eq:restricted-product} is well defined,
as for each integer $n \geq 0$ there are only finitely many
integer pairs $(i,j)$ with $i \geq 2j \geq 0$ such that $i - j = n$.
One may think of these as corresponding to a subset of the
terms appearing in the nonsensical ``product''
$f(x) g(x) = \sum_{i,j \geq 0} f_i g_j x^{i-j}$.
See \Cref{fig:restricted-product}.

\newcommand{\lightshade}{black!8}
\newcommand{\darkshade}{black!20}
\newcommand{\gridshade}{black!50}
\newcommand{\gridshadelight}{black!35}

\begin{figure}[h]
\begin{tikzpicture}[scale=0.72]
\fill[fill=\lightshade] (0,0) rectangle (9.5,-1);
\fill[fill=\lightshade] (2,-1) rectangle (9.5,-2);
\fill[fill=\lightshade] (4,-2) rectangle (9.5,-3);
\fill[fill=\lightshade] (6,-3) rectangle (9.5,-4);
\fill[fill=\lightshade] (8,-4) rectangle (9.5,-5);
\foreach \i in {0,...,3}
   \draw (3.5 + \i, -0.5 - \i) node {$\bigstar$};
\foreach \i in {0,...,9}
   \draw [draw=\gridshade] (\i,0) -- (\i,-5.5);
\foreach \i in {0,...,5}
   \draw [draw=\gridshade] (0,-\i) -- (9.5,-\i);
\draw[very thick] (9.5,0) -- (0,0) -- (0,-1) -- (2,-1) -- (2,-2) -- 
   (4,-2) -- (4,-3) -- (6,-3) -- (6,-4) -- (8,-4) -- (8,-5) -- (9.5,-5);
\foreach \i in {0,...,4}
   \draw (-0.1, -\i - 0.5) node[left] {$g_{\i}$};
\draw (-0.3,-5.2) node[left] {$\vdots$};
\foreach \i in {0,...,8}
   \draw (\i + 0.5, 0.1) node[above] {$f_{\i}$};
\draw (9.4, 0.2) node[above] {$\cdots$};
\end{tikzpicture}
\caption{%
The restricted product $f \restrict g$.
The shaded region shows pairs $(i,j)$ such that $i \geq 2j \geq 0$.
The stars indicate products $f_i g_j$ contributing to the
coefficient of $x^3$ in $f \restrict g$, i.e., pairs for which $i - j = 3$.}
\label{fig:restricted-product}
\end{figure}
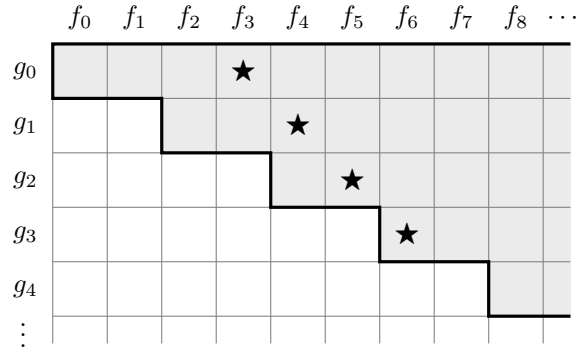

\begin{rem}
The definition of the restricted product is customised for
working with the quadratic form $x^2 - 2y^2$.
If we wanted to switch to a different form,
such as the form $3x^2 - y^2$ used in \cite{AB-prime-sieves,FT-prime-tables},
we would need to adjust the definition.
\end{rem}

\begin{defn}
\label{defn:series}
Throughout the paper,
the symbols~$F$, $G_d$ and $H_d$ (for $d \in \dset$)
denote the following power series:
\begin{alignat}{2}
\label{eq:defn-F}
F(x) & \coloneqq
   \sum_{\substack{a \geq 1 \\ \textn{$a$ odd}}} x^{a^2}
   = x + x^9 + x^{25} + \cdots & & \in \ZZ\bbracket{x}, \\
G_{-1}(x) & \coloneqq
   \sum_{b \geq 1} x^{4b^2}
   = x^4 + x^{16} + x^{36} + \cdots & & \in \ZZ\bbracket{x}, \\
G_{-2}(x) & \coloneqq
   \sum_{b \geq 1} x^{2b^2}
   = x^2 + x^8 + x^{18} + \cdots & & \in \ZZ\bbracket{x}, \\
\label{eq:defn-G2}
G_2(x) & \coloneqq
   \sum_{b \geq 1} x^{-2b^2}
   = x^{-2} + x^{-8} + x^{-18} + \cdots \; & & \in \ZZ\bbracket{x^{-1}},
\end{alignat}
\begin{alignat*}{2}
H_{-1} & \coloneqq F \cdot G_{-1} & & \in \ZZ\bbracket{x}, \\
H_{-2} & \coloneqq F \cdot G_{-2} & & \in \ZZ\bbracket{x}, \\
H_2 & \coloneqq F \restrict G_2 & & \in \ZZ\bbracket{x}.
\end{alignat*}
For legibility we will write $G_{d,i}$ instead of $(G_d)_i$
for the coefficient of $x^i$ in $G_d(x)$ (or $x^{-i}$ in the case $d = 2$),
and similarly for $H_d(x)$.
\end{defn}

\begin{rem}
The coefficients of each $H_d(x)$ up to $x^{125}$ are given by
\begin{align*}
H_{-1}(x) & =
   x^{5} + x^{13} + x^{17} + x^{25} + x^{29} + x^{37} +
   x^{41} + x^{45} + x^{53} + x^{61} + 2x^{65} +
   \phantom{x} \\ & \qquad
   x^{73} + 2x^{85} + x^{89} + x^{97} + x^{101} +
   x^{109} + x^{113} + x^{117} + 2x^{125} +
   \cdots, \\
H_{-2}(x) & =
   x^{3} + x^{9} + x^{11} + x^{17} + x^{19} +
   2x^{27} + 2x^{33} + x^{41} + x^{43} +
   \phantom{x} \\ & \qquad
   2x^{51} + 2x^{57} + x^{59} + x^{67} + x^{73} + x^{75} + 2x^{81} + x^{83} +
   \phantom{x} \\ & \qquad
   x^{89} + x^{97} + 3x^{99} + x^{107} + x^{113} + x^{121} + 2x^{123} +
   \cdots, \\
H_2(x) & =
   x^{7} + x^{17} + x^{23} + x^{31} + x^{41} +
   x^{47} + x^{49} + x^{63} + x^{71} + x^{73} +  
   \phantom{x} \\ & \qquad
   x^{79} + x^{89} + x^{97} + x^{103} + x^{113} + 2x^{119} + \cdots.
\end{align*}
\end{rem}

We may now state the main theorem of the section.
\begin{thm}
\label{thm:H-congruence}
For each $d \in \dset$,
\begin{equation}
\label{eq:H-congruence}
H_d(x) \equiv
   \sum_{\substack{m \geq 1 \\ \textn{$m$ odd}}} \;
   \sum_{\substack{p \geq 3 \\ \chi_d(p) = 1}} \;
   \sum_{l \geq 1}
   x^{m^2 p^l} \pmod 2.
\end{equation}
Consequently
\begin{equation}
\label{eq:H-sum-congruence}
(H_{-1} + H_{-2} + H_2)(x) \equiv
   \sum_{\substack{m \geq 1 \\ \textn{$m$ odd}}} \;
   \sum_{\substack{p \geq 3 \\ l \geq 1}} x^{m^2 p^l} \pmod 2.
\end{equation}
\end{thm}

Note that summation over the variable $p$ always indicates a sum over primes.
The generating function \eqref{eq:H-sum-congruence} is the main input
used by the core algorithm in \Cref{sec:core}.

\begin{rem}
It is possible for a given monomial $x^n$ to appear
more than once on the right hand side of \eqref{eq:H-congruence}
or \eqref{eq:H-sum-congruence}.
For example, $n = 27$ may be written as $m^2 p^l$ with
$(m,p^l) = (1,27)$ or $(3,3)$.
\end{rem}

\begin{rem}
\label{rem:mobius}
It is unfortunate that the series in \Cref{thm:H-congruence}
involve exponents of the form $m^2 p^l$ rather than simply~$p^l$,
but this seems to be in the nature of things.
One can apply M\"obius inversion to remove the $m^2$ factors, obtaining
\[
\sum_{\substack{m \geq 1 \\ \textn{$m$ odd, squarefree}}}
   \hspace{-17pt} (H_{-1} + H_{-2} + H_2)(x^{m^2}) \equiv
   \sum_{\substack{p \geq 3 \\ l \geq 1}} x^{p^l} \pmod 2,
\]
but this increases the number of terms on the left hand side,
and the author does not know how to use this identity to
obtain a prime enumeration algorithm whose complexity is competitive
with \Crefrange{thm:main-deterministic}{thm:main-heuristic}.
\end{rem}

The rest of this section is devoted to the
proof of \Cref{thm:H-congruence}.
As mentioned in \Cref{sec:overview},
this result modifies and extends
\cite[Thm.\,6.1--6.3]{AB-prime-sieves}:
it works with a different set of forms,
and it describes the coefficients of $x^n$ for all~$n$,
not just squarefree~$n$.
Following \cite{AB-prime-sieves}, the proof involves studying the arithmetic
of the rings of integers $\ZZ[\sqrt d]$ of the quadratic fields $\QQ(\sqrt d)$
for $d \in \dset$.

For odd $n \geq 1$, let $\psi_d(n)$ denote the number of ideals
in $\ZZ[\sqrt d]$ of norm~$n$,
and set $\psi_d(n) \coloneqq 0$ for even~$n$.
Define
\[
\Psi_d(x) \coloneqq \sum_{n \geq 1} \psi_d(n) x^n \in \ZZ\bbracket{x}.
\]
\begin{prop}
\label{prop:Psi-H-relation}
For $d \in \dset$ we have
\[
\Psi_d(x) = 2 H_d(x) + F(x).
\]
\end{prop}
The proof of \Cref{prop:Psi-H-relation} depends on the following three lemmas,
which describe a ``standard'' choice of generator for each ideal
of $\ZZ[\sqrt{d}]$.
These results are well known; for completeness we include quick proofs.
\begin{lem}
\label{lem:bijection-m1}
Let $n \geq 1$ be odd.
The map $(a,b) \mapsto \langle a + 2b \sqrt{-1} \rangle$
is a bijection from
\[
\{(a,b) \in \ZZ^2 : a^2 + 4 b^2 = n, \; a \geq 1\}
\]
to the set of ideals of $\ZZ[\sqrt{-1}]$ of norm~$n$.
\end{lem}
\begin{proof}
The ring $\ZZ[\sqrt{-1}]$ is a PID.
Since its unit group is $\{(\sqrt{-1})^k : k = 0, 1, 2, 3\}$,
every ideal of odd norm has a unique generator
$u + v \sqrt{-1}$ with $u$ odd, $v$ even and $u \geq 1$.
The correspondence is then $a = u$, $b = v/2$.
\end{proof}

\begin{lem}
\label{lem:bijection-m2}
Let $n \geq 1$ be odd.
The map $(a,b) \mapsto \langle a + b \sqrt{-2} \rangle$
is a bijection from
\[
\{(a,b) \in \ZZ^2 : a^2 + 2 b^2 = n, \; a \geq 1\}
\]
to the set of ideals of $\ZZ[\sqrt{-2}]$ of norm~$n$.
\end{lem}
\begin{proof}
The ring $\ZZ[\sqrt{-2}]$ is a PID with unit group $\{\pm 1\}$.
Any ideal of odd norm has exactly one generator
$a + b \sqrt{-2}$ with $a \geq 1$.
\end{proof}

\begin{lem}
\label{lem:bijection-2}
Let $n \geq 1$ be odd.
The map $(a,b) \mapsto \langle a + b \sqrt{2} \rangle$
is a bijection from
\[
\{(a,b) \in \ZZ^2 : a^2 - 2 b^2 = n, \; a \geq 2|b|\}
\]
to the set of ideals of $\ZZ[\sqrt{2}]$ of norm~$n$.
\end{lem}
\begin{proof}
The ring $\ZZ[\sqrt 2]$ is a PID
with unit group $\{\pm \omega^k : k \in \ZZ\}$,
where the fundamental unit $\omega \coloneqq 1 + \sqrt 2$
satisfies $N(\omega) = -1$.
Every ideal $I \subseteq \ZZ[\sqrt 2]$ of norm~$n$
has a unique generator $\theta$ such that $N(\theta) = n$ and
\begin{equation}
\label{eq:theta-interval}
\omega^{-1} \sqrt n \leq \theta < \omega \sqrt n.
\end{equation}
To prove this, let $I = \langle \theta \rangle$ for some~$\theta$.
Then $|N(\theta)| = N(I) = n$, so $N(\theta) = \pm n$.
Multiplying $\theta$ by $\omega$ if necessary,
we may ensure that $N(\theta) = n$.
Multiplying by $-1$ if necessary, we may further ensure that $\theta > 0$.
Multiplying finally by $\omega^{2k}$ for suitable $k \in \ZZ$,
we may ensure that \eqref{eq:theta-interval} holds.
The uniqueness of such $\theta$ follows from the structure of the unit group.

Writing $\theta = a + b \sqrt 2$,
the condition $a^2 - 2b^2 = n$ is equivalent to $N(\theta) = n$.
Assume now that $a^2 - 2b^2 = n$;
we must show that $a \geq 2|b|$ is equivalent to~\eqref{eq:theta-interval}.
Let $\bar\theta \coloneqq a - b \sqrt 2 = n/\theta$.
If \eqref{eq:theta-interval} holds,
then $\omega^{-1} \sqrt n < \bar\theta \leq \omega \sqrt n$.
Adding this inequality to \eqref{eq:theta-interval} shows that $a > 0$;
subtracting it from \eqref{eq:theta-interval} leads to $|b| \leq \sqrt{n/2}$.
Then $a^2 = 2b^2 + n \geq 4b^2$, so $a \geq 2|b|$.
Conversely, suppose that $a \geq 2|b|$.
Then $a > 2|b|$ as $a$ is odd.
From $a^2 - 2b^2 = n$ we get $|b| < \sqrt{n/2}$
and $\sqrt n \leq a < \sqrt{2n}$.
These imply that $0 < \theta < \omega\sqrt n$ and
$0 < \bar\theta < \omega\sqrt n$
and hence that \eqref{eq:theta-interval} holds.
\end{proof}

\begin{proof}[Proof of \Cref{prop:Psi-H-relation}]
For the $d = -1$ case, \Cref{lem:bijection-m1} implies
that $\Psi_{-1}(x) = \sum_{n \geq 1} \psi_{-1}(n) x^n$ is equal to
\[
\sum_{\substack{a \geq 1 \\ \textn{$a$ odd}}} \;
   \sum_{b \in \ZZ} x^{a^2 + 4b^2}
= \sum_{\substack{a \geq 1 \\ \textn{$a$ odd}}} x^{a^2}
   \left( 1 + \sum_{b \geq 1} x^{4b^2} + \sum_{b \leq -1} x^{4b^2} \right)
= F(x)(1 + 2 G_{-1}(x)).
\]
Similarly, for $d = -2$,
by \Cref{lem:bijection-m2} we find that $\Psi_{-2}(x)$ is equal to
\[
\sum_{\substack{a \geq 1 \\ \textn{$a$ odd}}} \; \sum_{b \in \ZZ} x^{a^2+2b^2}
= \sum_{\substack{a \geq 1 \\ \textn{$a$ odd}}} x^{a^2}
   \left( 1 + \sum_{b \geq 1} x^{2b^2} + \sum_{b \leq -1} x^{2b^2} \right)
= F(x)(1 + 2 G_{-2}(x)).
\]
For $d = 2$, by \Cref{lem:bijection-2} we have
\[
\Psi_2(x) =
\sum_{\substack{a \geq 1 \\ \textn{$a$ odd}}} \;
   \sum_{|b| \leq \frac12 a} x^{a^2 - 2b^2}
= \sum_{\substack{a \geq 1 \\ \textn{$a$ odd}}} x^{a^2}
   + 2 \sum_{\substack{a \geq 1 \\ \textn{$a$ odd}}} \;
      \sum_{1 \leq b \leq \frac12 a} x^{a^2 - 2b^2}.
\]
The condition $1 \leq b \leq \frac12 a$ is equivalent to
$2 \leq 2b^2 \leq \frac12 a^2$ (with $a, b > 0$), so this becomes
\[
\Psi_2(x) = F(x) + 2 \sum_{2 \leq j \leq \frac12 i} F_i G_{2,j} x^{i-j}.
\]
As $G_{2,0} = G_{2,1} = 0$,
we may replace the condition $2 \leq j \leq \frac12 i$
by $0 \leq j \leq \frac12 i$.
Thus
\[
\Psi_2(x) = F(x) + 2 \sum_{i \geq 2j \geq 0} F_i G_{2,j} x^{i-j}
   = F(x) + 2 F(x) \restrict G_2(x).   \qedhere
\]
\end{proof}

We may now give the proof of \Cref{thm:H-congruence}.
Consider the formal Dirichlet series $\sum_{n \geq 1} \psi_d(n) n^{-s}$.
This is exactly the Dedekind zeta function of $\QQ(\sqrt d)$
\cite[\S{}VIII.2]{FT-ANT}
with the Euler factor at $p = 2$ omitted. Therefore
\[
\sum_{n \geq 1} \psi_d(n) n^{-s} =
   \prod_{\chi_d(p) = 1} (1 - p^{-s})^{-2}
   \prod_{\chi_d(p) = -1} (1 - p^{-2s})^{-1}.
\]
(The first product corresponds to prime ideals of norm~$p$,
i.e., arising from primes that split in $\QQ(\sqrt d)$,
and the second product to prime ideals of norm~$p^2$,
i.e., arising from primes that remain inert.)
We may rewrite this as
\[
\prod_{p \geq 3} (1 - p^{-2s})^{-1}
   \prod_{\chi_d(p) = 1} \frac{1 - p^{-2s}}{(1 - p^{-s})^2}
= \sum_{\substack{m \geq 1 \\ \textn{$m$ odd}}} m^{-2s}
   \prod_{\chi_d(p) = 1} \frac{1 + p^{-s}}{1 - p^{-s}},
\]
and the last product becomes
\begin{align*}
\prod_{\chi_d(p) = 1} \frac{1 + p^{-s}}{1 - p^{-s}}
& = \prod_{\chi_d(p) = 1} (1 + 2 p^{-s} + 2 p^{-2s} + 2 p^{-3s} + \cdots) \\
& \equiv 1 + \sum_{\substack{\chi_d(p) = 1 \\ l \geq 1}} 2p^{-ls} \pmod 4.
\end{align*}
By the definition of $F(x)$, we also have
\[
\sum_{n \geq 1} F_n \tsp n^{-s} =
   \sum_{\substack{m \geq 1 \\ \textn{$m$ odd}}} m^{-2s}.
\]
Putting everything together, we obtain
\[
\sum_{n \geq 1} \frac{\psi_d(n) - F_n}{2} \cdot n^{-s} \equiv
   \sum_{\substack{m \geq 1 \\ \textn{$m$ odd}}} m^{-2s}
   \sum_{\substack{\chi_d(p) = 1 \\ l \geq 1}} p^{-ls} \pmod 2.
\]
Equating coefficients of $n^{-s}$,
we see that $(\psi_d(n) - F_n)/2$ is congruent modulo~$2$ to the
number of representations of $n$ in the form $n = m^2 p^l$
where $m \geq 1$ is odd, $\chi_d(p) = 1$ and $l \geq 1$.
By \Cref{prop:Psi-H-relation}
this quantity is exactly $H_{d,n}$,
so \eqref{eq:H-congruence} is proved.

To deduce \eqref{eq:H-sum-congruence} from \eqref{eq:H-congruence},
note that if $p \equiv 1 \pmod 8$
then $\chi_d(p) = 1$ for all three values of~$d$,
whereas if $p \equiv 3,5,7 \pmod 8$
then $\chi_d(p) = 1$ for exactly one value of~$d$.
This completes the proof of \Cref{thm:H-congruence}.

\section{Turing machines}
\label{sec:turing}

In this section we briefly recall
a few simple complexity results for the Turing model,
and discuss some issues relating to data layout on the Turing machine tape.

\subsection{Integer arithmetic}

We store integers in the usual binary representation.
We may add or subtract $n$-bit integers in time $O(n)$.
Operations such as computing products, quotients, remainders and square roots
may be carried out in time $n^{1+o(1)}$ for $n$-bit inputs.
We will sometimes need to evaluate formulas involving
logarithms and exponentials,
when choosing parameters for various algorithms.
For example, given an $n$-bit integer~$N$, we might need to compute
some integer $k = \exp((\log \log N)^{1/2}) + O(1)$;
this can also be done in time $n^{1+o(1)}$.
For details of all the algorithms just mentioned,
see for instance \cite{BZ-mca}.

\subsection{Counting}
\label{sec:counting}

Updating a $k$-bit counter from $n$ to $n+1$
requires $\Theta(k)$ bit operations in the worst case.
However, if we start at zero and increment repeatedly,
then the amortised time per increment is only $O(1)$,
because the $i$\th bit is only updated on every $2^i$\th increment
\cite[\S16.1]{CLRS-algorithms}.
The same remark applies when decrementing a counter from $N$ to zero.

As an application,
consider a vector $v \in \FF_2^n$.
We will sometimes need to convert between the \emph{binary array format},
i.e., representing $v$ as an array of $n$ bits,
and the \emph{sorted list format},
i.e., listing those indices $i$ such that $v_i = 1$ in increasing order.
For example, $v$ might represent the set of primes $p < n$.
\begin{lem}
\label{lem:convert-format}
Let $v \in \FF_2^n$ and let $\wt(v)$ denote
the number of $i \in \{0, \ldots, n-1\}$ such that $v_i = 1$.
We may convert $v$ between the binary array format and the sorted list format,
in either direction, in time $O(n + \wt(v) \log n)$.
\end{lem}
\begin{proof}
Given a binary array,
we simply increment a counter as we walk through the array,
copying its value to the output whenever we see a~$1$.
For the other direction, we compute the differences between successive
list elements,
and use them to work out the distance to travel
between each~$1$ in the output array.
\end{proof}

\subsection{Merging and sorting}

We will often need to deal with sorted lists.
Suppose that the elements of each list are represented
by bitstrings of a fixed length, say $b \geq 1$ bits.
Assume also that we have agreed on a total ordering for the elements,
and that we can compare two elements in time $O(b)$.
For example, the objects might simply be $b$-bit integers
with the usual ordering.
\begin{lem}
\label{lem:merge}
Given two sorted lists $S_1$ and $S_2$ of $b$-bit objects
of lengths respectively $n_1$ and $n_2$,
we may merge the lists to compute the union $S \coloneqq S_1 \cup S_2$,
as a sorted list, in time $O(b (n_1 + n_2))$.
\end{lem}
\begin{proof}
We walk through the lists in parallel,
copying elements to the output appropriately as we proceed.
\end{proof}
\begin{lem}
\label{lem:sort}
We may sort a list of $b$-bit objects of length $n$ in time $O(b n \log n)$.
\end{lem}
\begin{proof}
We apply the \emph{merge sort} algorithm,
i.e., recursively sort each half of the list
and then merge the results \cite[\S5.2.4]{Knuth-TAOCP3}.
\end{proof}

\begin{rem}
It is well known that merging and sorting are efficient in the Turing model,
but it is not so easy to find a formal treatment in the literature.
According to \cite{Petersen-sorting} there is a presentation in
\cite{Reischuk-complexity},
but the author was unable to obtain a copy of the latter.
See also \cite[Thm.\,1]{Diem-turing} and \cite[\S6.5.2]{SGV-turing}.
\end{rem}

\subsection{Arrays and transposition}

Let $A$ be an $m \times n$ array of $b$-bit objects.
We represent $A$ on the one-dimensional tape by listing
the entries $A_{i,j}$ in row-major order,
i.e., sorted first by $i$ and then by~$j$.
\begin{lem}
\label{lem:transpose}
We may transpose an $m \times n$ array of $b$-bit objects,
to obtain an $n \times m$ array, in time $O(b m n \log \min(m, n))$.
\end{lem}
\begin{proof}
The idea is to split the array in half along the shorter dimension,
transpose each half recursively and combine the results.
For details see \cite[Lem.\,18]{BGS-recurrences}.
\end{proof}

More generally, let $A$ be a $d$-dimensional array
of $b$-bit objects of size $n_1 \times \cdots \times n_d$.
Again we represent $A$ by listing the entries $A_{i_1,\ldots,i_d}$
lexicographically,
i.e., sorted first by~$i_1$, then by~$i_2$ and so on.
Complexity bounds similar to \Cref{lem:transpose}
may be proved for transposition problems involving such arrays
by decomposing into two-dimensional problems;
we will develop these in an \emph{ad hoc} manner as they arise.

\section{Polynomial arithmetic}
\label{sec:polynomial-arithmetic}

Polynomials in $\FF_2[x]$ will be represented by an array of bits,
i.e., as a sequence of coefficients in $\FF_2$
with respect to the usual monomial basis.
We may add (or subtract!) polynomials in $\FF_2[x]$
of degree at most $n$ in time $O(n)$.

\subsection{Polynomial multiplication and division}
\label{sec:poly-mult}

Fix some algorithm (Turing machine)
for multiplying polynomials in $\FF_2[x]$.
For $n \geq 1$ let $\Mbarcost(n)$ denote the worst-case running time
of the algorithm on inputs of degree less than~$n$.

It will be technically convenient to work with a ``smoothed'' version
of $\Mbarcost(n)$.
Let
\[
\Mbarstar(n) \coloneqq \frac{\Mbarcost(n)}{n \max(1, \log n)}, \qquad n \geq 1,
\]
i.e., so that $\Mbarstar(n) = \Mbarcost(n) / (n \log n)$ for $n \geq 3$.
Then set
\[
\Mstar(n) \coloneqq \max_{1 \leq m \leq n} \Mbarstar(m), \qquad n \geq 1.
\]
By construction $\Mstar(n)$ is \emph{non-decreasing}.
Finally put
\[
\Mcost(n) \coloneqq n \max(1, \log n) \cdot \Mstar(n), \qquad n \geq 1,
\]
so that $\Mcost(n) = n \log n \cdot \Mstar(n)$ for $n \geq 3$.
Clearly $\Mbarcost(n) \leq \Mcost(n)$ for all~$n$.
For the remainder of the paper we will work only
with $\Mcost(n)$ and $\Mstar(n)$,
and ignore the original $\Mbarcost(n)$.

The non-decreasing property of $\Mstar(n)$ implies that
$\Mcost(n)/n$ is also non-decreasing,
and hence that $\Mcost(n)$ is superadditive, i.e.,
\begin{equation}
\label{eq:superadditive}
\Mcost(n_1) + \cdots + \Mcost(n_r) \leq \Mcost(n_1 + \cdots + n_r).
\end{equation}
It also implies that
\[
\Mcost(n) \gg n \log n
\]
and that $\Mcost(n)$ is itself non-decreasing.

For technical reasons it will also be convenient to assume that
\begin{equation}
\label{eq:mult-constant}
\Mcost(kn) \ll_k \Mcost(n)
\end{equation}
for each $k \geq 1$.
This condition will be satisfied for all multiplication algorithms
considered below.

Let us give some examples.
The classical long multiplication algorithm achieves $\Mcost(n) \ll n^2$,
and Karatsuba's algorithm achieves $\Mcost(n) \ll n^\theta$ with
$\theta = \tfrac{\log 3}{\log 2} \approx 1.58$ \cite[\S8.1]{vzGG-compalg3}.
The polynomial version of the Sch\"onhage--Strassen algorithm \cite{Sch-char2}
achieves $\Mcost(n) \ll n \log n \log \log n$,
so that $\Mstar(n) \ll \log \log n$.
This bound held the asymptotic complexity record for several decades,
but was eventually superseded by an algorithm achieving
$\Mstar(n) \ll 8^{\log^* n}$ \cite{HvdHL-ffmul}.
(See \Cref{sec:notation} for the definition of $\log^* n$.)
The current best result \cite{HvdH-ffmul-cyclotomic}
is the further minuscule improvement
\begin{equation}
\label{eq:ffmul}
\Mstar(n) \ll 4^{\log^* n}.
\end{equation}

A large fraction of complexity bounds in this paper
depend in some way or another on $\Mcost(n)$.
We will generally try to write everything in terms of
$\Mcost(n)$ and $\Mstar(n)$
(and related functions to be introduced below)
so that our results can take advantage of future developments
in multiplication technology.
However, when this becomes too unwieldy,
or when we are dealing with less important secondary terms,
we will simplify matters by hard-wiring in a multiplication algorithm
such as \cite{HvdH-ffmul-cyclotomic},
whose complexity is given by \eqref{eq:ffmul}.
Since $\log^* n$ grows much more slowly than any fixed iterate of $\log$,
this allows us to assume that $\Mstar(n) < (\log n)^{o(1)}$
or even $\Mstar(n) < (\log \log n)^{o(1)}$.

\begin{rem}
The complexity bounds in
\Crefrange{thm:main-deterministic}{thm:main-heuristic}
rely crucially on the estimate $\Mstar(n) < (\log \log n)^{o(1)}$,
i.e., on \cite{HvdHL-ffmul} and subsequent work.
If we substituted instead the earlier Sch\"onhage--Strassen bound
$\Mstar(n) \ll \log \log n$,
the exponents of $\log \log N$ in those theorems would increase.
\end{rem}

\begin{rem}
\label{rem:nlogn-conjecture}
The complexity of polynomial multiplication over finite fields
is an area of active research.
The author suspects that $\Mcost(n) \ll n \log n$ is attainable,
i.e., $\Mstar(n) \ll 1$,
by analogy with the integer case \cite{HvdH-nlogn}.
If $O(n \log n)$ is feasible,
then the main results of the paper could be improved slightly:
see \Cref{rem:main-heuristic-precise}, \Cref{rem:main-probabilistic-precise}
and \Cref{rem:main-deterministic-precise}.
It was shown in \cite{HvdH-ffnlogn} that $O(n \log n)$ can be achieved
if we allow certain plausible (but as yet unproved)
number-theoretic hypotheses.
Moreover it is reasonable to guess that $O(n \log n)$ is the best possible
complexity for this problem,
again by analogy with the lower bound conjecture of Sch\"onhage and Strassen
for the integer case \cite{SS-multiply}.
\end{rem}

\begin{rem}
\label{rem:RAM-multiply}
The papers \cite{HvdHL-ffmul}, \cite{HvdH-ffmul-cyclotomic} and
\cite{HvdH-ffnlogn} all work primarily in the Turing model,
and the impact of these new methods in the RAM model
has not yet been studied.
On a RAM with words of bit size $O(\log N)$,
if one uses a packed encoding to represent polynomials in $\FF_2[x]$,
i.e., with $O(\log N)$ coefficients per word,
then it seems likely that one can multiply polynomials of degree $n \leq N$
in only $O((n/\log N) \log n) = O(n)$ word operations.
Namely, after running $O(1)$ recursion levels of the algorithm
of \cite{HvdH-ffmul-cyclotomic},
the resulting multiplication subproblems should be small enough
to be handled by lookup tables.
This may in effect allow us to take $\Mstar(n) \asymp 1$,
which could lead to small improvements in
\Crefrange{thm:main-deterministic}{thm:main-heuristic}
for the RAM model.
The author has not checked the details.
\end{rem}

\begin{prop}
\label{prop:ffdiv}
Let $f, g \in \FF_2[x]$ have degree at most~$n$, with $g \neq 0$.
Then we may perform the Euclidean division of $f$ by~$g$,
obtaining quotient and remainder, in time $O(\Mcost(n))$.
\end{prop}
\begin{proof}
This follows from \cite[Thm.\,9.6]{vzGG-compalg3}
via an algorithm that uses a Newton iteration
to reduce polynomial division to multiplication.

(The argument in \cite{vzGG-compalg3} is for an algebraic complexity model,
but the underlying algorithm can be easily adapted to the Turing model;
we omit the details.
A similar remark applies to all subsequent references to
algebraic algorithms from the literature.
We also mention that the argument relies on certain technical
assumptions concerning the algebraic analogue of $\Mcost(n)$,
which are described in the paragraphs following
\cite[Defn.\,8.26]{vzGG-compalg3}.
These are used to derive estimates such as
$\Mcost(n) + \Mcost(n/2) + \Mcost(n/4) + \cdots \ll \Mcost(n)$
that arise in the analysis of Newton's method.
In our setting a similar role is played by
the assumption \eqref{eq:mult-constant},
and the fact that $\Mcost(n)/n$ is non-decreasing.
Again we omit the details.)
\end{proof}

\subsection{Polynomial arithmetic over $\FFext$}
\label{sec:ffext}

In \Cref{sec:transform-schemes} and \Cref{sec:compression},
we will work extensively with polynomials over a finite extension of~$\FF_2$.
In this section we collect together a few basic complexity results
concerning such polynomials.

Let $\lambda \geq 1$.
To work with the finite field $\FFext$
we must first precompute a suitable irreducible polynomial.
\begin{lem}[\cite{Sho-irreducible}]
\label{lem:irreducible}
We may (deterministically) find an irreducible polynomial $h \in \FF_2[z]$
of degree $\lambda$ in time $O(\lambda^{O(1)})$.
\end{lem}
We then have $\FFext \cong \FF_2[z]/h(z)$,
and we represent elements of $\FFext$
by polynomials in $\FF_2[z]$ of degree less than~$\lambda$.
\begin{prop}[Arithmetic in $\FFext$]
\label{prop:ffext-arithmetic}
We may
\begin{enumabc}
\item add elements of $\FFext$ in time $O(\lambda)$,
\item multiply elements of $\FFext$ in time $O(\Mcost(\lambda))$, and
\item find the inverse of a nonzero element of $\FFext$ in time
$O(\Mcost(\lambda) \log \lambda)$.
\end{enumabc}
\end{prop}
\begin{proof}
(a) Clear.
(b) Follows from \Cref{prop:ffdiv}.
(c) Follows from \cite[Cor.\,11.11]{vzGG-compalg3}.
\end{proof}

\begin{lem}[\cite{Shp-primitive}]
\label{lem:generator}
We may (deterministically) find a primitive element in $\FFext^*$
in time $(2^\lambda)^{1/4+o(1)}$.
\end{lem}

We next consider polynomials over $\FFext$.
A polynomial in $\FFext[x]$ will be represented as a
sequence of coefficients in $\FFext$.
We may add two such polynomials of degree at most $n$
in time $O(\lambda n)$.
Let $\Mcost_\lambda(n)$ denote the cost of multiplying two polynomials
in $\FFext[x]$ of degree at most~$n$.

\begin{prop}
\label{prop:ffextmul}
There is a multiplication algorithm for $\FFext[x]$ achieving
\[
\Mcost_\lambda(n) \ll \Mcost(\lambda n).
\]
\end{prop}
\begin{proof}
The input polynomials lie in
$\FFext[x] \cong \FF_2[x,z]/h(z)$.
We lift the problem to $\FF_2[x,z]$
and then use Kronecker substitution \cite[Cor.\,8.28]{vzGG-compalg3}
to reduce to a univariate multiplication problem in $\FF_2[u]$
of degree $O(\lambda n)$, i.e., via the substitution
$x \mapsto u^{2\lambda}$, $z \mapsto u$.
After recovering the product in $\FF_2[x,z]$
we must divide the resulting coefficients in $\FF_2[z]$ by $h(z)$.
By \Cref{prop:ffdiv} the total cost is
$O(\Mcost(\lambda n) + n \Mcost(\lambda))$.
Using the assumption that $\Mcost(n)/n$ is non-decreasing,
this simplifies to $O(\Mcost(\lambda n))$.
\end{proof}
\begin{prop}
\label{prop:ffextdiv}
Let $f, g \in \FFext[x]$ have degree at most~$n$,
with $g$ nonzero and monic.
Then we may perform the Euclidean division of $f$ by~$g$,
obtaining quotient and remainder, in time $O(\Mcost(\lambda n))$.
\end{prop}
\begin{proof}
Again this follows from \cite[Thm.\,9.6]{vzGG-compalg3}
(and \Cref{prop:ffextmul}).
\end{proof}

\begin{prop}
\label{prop:ffextgcd}
Let $f, g \in \FFext[x]$ have degree at most~$n$.
Then we may compute $\gcd(f, g)$ in time
$O(\Mcost(\lambda n) \log \lambda n)$.
\end{prop}
\begin{proof}
According to \cite[Thm.\,11.7]{vzGG-compalg3}
the cost is $O(\Mcost_\lambda(n) \log n)$
plus $O(n)$ inversions in $\FFext$.
By \Cref{prop:ffextmul} and \Cref{prop:ffext-arithmetic}(c)
the bit complexity is
\[
O(\Mcost(\lambda n) \log n + n \Mcost(\lambda) \log \lambda)
   = O(\Mcost(\lambda n) \log \lambda n).  \qedhere
\]
\end{proof}

\begin{prop}
\label{prop:product-tree}
Let $g_1, \ldots, g_m \in \FFext[x]$
be nonconstant polynomials,
and let $n \coloneqq \sum_i \deg g_i$.
Then we may compute the product $g_1 \cdots g_m$ in time
$O(\Mcost(\lambda n) \log n)$.
\end{prop}
\begin{proof}
Using the standard product tree algorithm,
by \cite[Lem.\,10.4]{vzGG-compalg3} the complexity is
$O(\Mcost_\lambda(n) \log m) = O(\Mcost(\lambda n) \log n)$.
(The argument here makes use of the
superadditivity property \eqref{eq:superadditive},
and similarly for the next result.)
\end{proof}

\begin{prop}[Simultaneous modular reduction]
\label{prop:ffextmod}
Let $g_1, \ldots, g_m \in \FFext[x]$
be monic nonconstant polynomials,
and let $n \coloneqq \sum_i \deg g_i$.
Given $f \in \FFext[x]_n$,
we may compute the remainders $f \bmod g_i$
for $i = 1, \ldots, m$ in time $O(\Mcost(\lambda n) \log n)$.
\end{prop}
\begin{proof}
Using the standard remainder tree algorithm,
by \cite[Cor.\,10.17]{vzGG-compalg3} the complexity is again
$O(\Mcost_\lambda(n) \log m) = O(\Mcost(\lambda n) \log n)$.
\end{proof}

\begin{prop}[Evaluation and interpolation on a geometric progression]
\label{prop:geometric}
Let $n \geq 1$
and let $\beta \in \FFext^*$ have multiplicative order at least~$n$.
Both of the following problems may be solved in time
$O(\Mcost(\lambda n) + \Mcost(\lambda) \log \lambda)$:
\begin{enumabc}
\item
Given $\beta$ and $f \in \FFext[x]_n$,
evaluate $v_j \coloneqq f(\beta^j) \in \FFext$ for $j = 0, \ldots, n-1$.
\item
Given $\beta$ and $v_0, \ldots, v_{n-1} \in \FFext$,
find the unique polynomial $f \in \FFext[x]_n$ such that
$f(\beta^j) = v_j$ for $j = 0, \ldots, n-1$.
\end{enumabc}
\end{prop}
\begin{rem}
The result in part~(a) is a factor of $\log n$ faster than invoking
\Cref{prop:ffextmod} with the moduli $x - \beta^j$.
\end{rem}
\begin{proof}
Propositions~3,~4 and~5 of \cite{BS-special} describe algorithms for
these problems that amount to performing $O(1)$ multiplications
of polynomials in $\FFext[x]$ of degree $O(n)$,
together with $O(n)$ ``field operations'' in $\FFext$.
The latter include $O(n)$ additions and multiplications,
but also $O(n)$ inversions.
Using Montgomery's simultaneous inversion algorithm
(see for instance \cite[Alg.\,11.15]{CFA+-handbook}),
the latter may be replaced by $O(n)$ multiplications plus
a single field inversion.
By Propositions \ref{prop:ffext-arithmetic} and \ref{prop:ffextmul}
the bit complexity is
$O(\Mcost(\lambda n) + n \Mcost(\lambda) + \Mcost(\lambda) \log \lambda)
= O(\Mcost(\lambda n) + \Mcost(\lambda) \log \lambda)$.
\end{proof}

\begin{prop}[Rational function reconstruction]
\label{prop:reconstruction}
Let $n \geq 1$ and let
\[
c(x) = c_0 + c_1 x + \cdots + c_{2n-1} x^{2n-1} \in \FFext[x]_{2n}.
\]
Then there exists at most one pair of polynomials
$f, g \in \FFext[x]$,
with $f$ and $g$ relatively prime, $g$ monic and $x \ndivides g(x)$
such that
\[
\frac{f(x)}{g(x)} \equiv c(x) \pmod{x^{2n}},
\qquad \deg f < n, \quad \deg g \leq n.
\]
Given $n$ and $c(x)$ as input,
we may determine if such $f(x)$ and $g(x)$ exist,
and if so compute them, in time $O(\Mcost(\lambda n) \log \lambda n)$.
\end{prop}
The traditional Berlekamp--Massey algorithm
solves this problem in time quadratic in~$n$.
It is well known that this can be improved to quasilinear time,
but the author was unable to find a clear statement
in the literature in the form stated above.
We will briefly indicate how the result may be pieced together from
various ingredients in \cite{vzGG-compalg3}.
\begin{proof}
Consider applying \cite[Alg.\,3.14]{vzGG-compalg3}
(the Extended Euclidean Algorithm, or EEA) to the
input polynomials $x^{2n}$ and $c(x)$.
Let $(r_i,s_i,t_i)_{0 \leq i \leq \ell}$ be the list of ``rows''
generated by the EEA;
thus the $r_i$ are the successive remainders (normalised to be monic),
and $r_i(x) = s_i(x) x^{2n} + t_i(x) c(x)$ for all~$i$.

Let $j$ be the least index such that $\deg r_j < n$.
If $\gcd(r_j,t_j) = 1$ then \cite[Cor.\,5.21(i)]{vzGG-compalg3}
(with $n \coloneqq 2n$, $k \coloneqq n$)
implies that the pair $f \coloneqq r_j/\tau$, $g \coloneqq t_j/\tau$,
where $\tau$ denotes the leading coefficient of~$t_j$,
satisfies the desired conditions.
Conversely, if $f$ and $g$ satisfy the conditions,
then \cite[Cor.\,5.21(ii)]{vzGG-compalg3} implies that
$f = r_j/\tau$ and $g = t_j/\tau$;
therefore at most one such pair exists.
Moreover, this argument also shows that the desired $f$ and $g$
exist if and only if $\gcd(r_j,t_j) = 1$.

To actually compute $f$ and~$g$,
we first run \cite[Alg.\,11.6]{vzGG-compalg3} (Half-GCD),
with $n \coloneqq 2n$, $k \coloneqq n$, $h \coloneqq j$,
to compute the $j$\th row $(r_j,s_j,t_j)$.
The cost is $O(\Mcost(\lambda n) \log \lambda n)$,
by the same argument as in the proof of \Cref{prop:ffextgcd}.
We then compute $\gcd(r_j,t_j)$,
which again costs $O(\Mcost(\lambda n) \log \lambda n)$.
If $\gcd(r_j,t_j) = 1$, we return
$f \coloneqq r_j/\tau$, $g \coloneqq t_j/\tau$;
otherwise we return ``FAIL''.
\end{proof}

\begin{prop}[Root finding]
\label{prop:root-finding}
Given monic $g \in \FFext[x]$ of degree $n \geq 1$,
we may find all roots of $g(x)$
in $\FFext$ in (deterministic) time
$O(\lambda^2 \Mcost(\lambda n) \log \lambda n)$.
\end{prop}
For the application in \Cref{sec:compression}
the key issue is that we need the complexity to be quasilinear in~$n$.
We will use Berlekamp's classical strategy from
\cite[\S5]{Ber-factoring-large}.
\begin{proof}
\step{1}{remove extraneous factors.}
Let $f(x) \coloneqq \gcd(x^{2^\lambda} - x, g(x))$.
Then $f(x) = \prod_{i=1}^d (x - \alpha_i)$
where $\alpha_1, \ldots, \alpha_d$ ($d \leq n$)
are the distinct roots of $g(x)$ in $\FFext$.
We may compute $f(x)$ by first using repeated squaring
to compute $x^{2^\lambda} \pmod g$ in time
$O(\lambda \Mcost(\lambda n))$ (\Cref{prop:ffextdiv})
and then obtaining the gcd via \Cref{prop:ffextgcd}
at a cost of $O(\Mcost(\lambda n) \log \lambda n)$.

\step{2}{compute separating set.}
Let $\Rring \coloneqq \FFext[x]/f(x) \cong (\FFext)^d$
and consider the trace map $\trace \colon \Rring \to \Rring$ given by
$\trace(u) \coloneqq \sum_{m=0}^{\lambda-1} u^{2^m}$.
For $k = 0, \ldots, \lambda-1$
define $\sigma_k \coloneqq \trace(z^k x) \in \Rring$,
where $z \in \FFext$ is the variable defining $\FFext = \FF_2[z]/h(z)$,
and let $s_k \in \FFext[x]_d$ be the unique polynomial
representing $\sigma_k$.
Then $\{s_k\}_{k=0}^{\lambda-1}$ is a \emph{separating set} for~$f$,
in the sense that for any $i, j \in \{1, \ldots, d\}$, $i \neq j$,
there exists some index $k$ such that
$s_k(x)$ is divisible by exactly one of $x - \alpha_i$ and $x - \alpha_j$.
(For a proof of this statement, see \cite[\S5]{Ber-factoring-large}.)
We may compute each $s_k$ using repeated squaring in time
$O(\lambda \Mcost(\lambda d))$,
and thus all $s_k$ in time
$O(\lambda^2 \Mcost(\lambda d)) = O(\lambda^2 \Mcost(\lambda n))$.

\step{3}{refinement.}
We use the $s_k$ to carry out the following refinement procedure.
At the $k$\th step,
we begin with a partial factorisation $f = f_1 \cdots f_{m_k}$.
Using \Cref{prop:ffextmod},
we compute $s_k \bmod f_i$ for $i = 1, \ldots, m_k$
in time $O(\Mcost(\lambda d) \log d)$.
We then invoke \Cref{prop:ffextgcd} to compute
$t_i \coloneqq \gcd(f_i, s_k \bmod f_i)$ for each~$i$;
by \eqref{eq:superadditive} the cost over all $i$ is
$O(\Mcost(\lambda d) \log \lambda d)$.
If $t_i \neq 1$ and $t_i \neq f_i$
we refine the factorisation by replacing $f_i$ with $t_i$ and $f_i/t_i$.
After repeating this process for $k = 0, \ldots, \lambda-1$
at a total cost of
$O(\lambda \Mcost(\lambda d) \log \lambda d)
   = O(\lambda \Mcost(\lambda n) \log \lambda n)$,
the separation property guarantees that we arrive at
a complete factorisation of~$f$,
and hence obtain all the roots of~$g$.
\end{proof}

\subsection{Transform schemes}
\label{sec:transform-schemes}

In many of the algorithms discussed in this paper,
certain polynomials will occur as multiplicands repeatedly in many
separate multiplication problems.
To take advantage of this redundancy we introduce the following definition.
(For more on the history of this technique,
see the discussion of ``FFT caching'' and ``FFT addition''
in \cite[\S2.9]{Ber-fastmult}.)

\begin{defn}
\label{defn:transform-pair}
Let $n \geq 1$ be an integer.
By a \emph{transform pair of order $n$} (over~$\FF_2$)
we mean an integer $K \geq 1$ together with a pair of $\FF_2$-linear maps
\[
\fwd \colon \FF_2[x]_n \to \FF_2^K,
\qquad
\inv \colon \FF_2^K \to \FF_2[x]_{2n-1}
\]
with the property that
\[
\inv(\fwd(f) \cdot \fwd(g)) = fg
\]
for any $f, g \in \FF_2[x]_n$.
Here the dot operator means
pointwise multiplication in $\FF_2^K$.
This is simply a logical AND operation in each component,
and may be computed in time $O(K)$.
We call $\fwd$ the \emph{forward transform}
and $\inv$ the \emph{inverse transform}.
\end{defn}

Since $\inv$ is $\FF_2$-linear,
a transform pair may also be used to evaluate finite sums of products, i.e.,
\[
\inv\Bigl(\sum_i \fwd(f_i) \cdot \fwd(g_i)\Bigr)
   = \sum_i f_i g_i, \qquad f_i, g_i \in \FF_2[x]_n.
\]

\begin{defn}
\label{defn:transform-scheme}
A \emph{transform scheme} (over~$\FF_2$) consists of the following data:
\begin{itemize}
\item
a family of transform pairs $(K_n, \fwd_n, \inv_n)$,
one for each integer $n \geq 1$;
\item
a forward transform algorithm that takes as input $n \geq 1$
and $f \in \FF_2[x]_n$
and computes $\fwd_n(f) \in \FF_2^{K_n}$;
\item
an inverse transform algorithm that takes as input $n \geq 1$
and $u \in \FF_2^{K_n}$
and computes $\inv_n(u) \in \FF_2[x]_{2n-1}$.
\end{itemize}
We will omit the subscripts in $K_n$, $\fwd_n$ and $\inv_n$ when the value of $n$ is clear from context.
\end{defn}

We write $\Tcost(n)$ for the running time of the forward and inverse
transform algorithms (whichever is larger).
We also define
\[
\Kstar(n) \coloneqq \max_{1 \leq m \leq n} \frac{K_m}{m}, \qquad
n \geq 1,
\]
so that $\Kstar(n)$ is non-decreasing, and set
\[
\Kcost(n) \coloneqq n \cdot \Kstar(n).
\]
Thus $K_n \leq \Kcost(n)$ for all $n \geq 1$.
We also have automatically $\Kcost(n) \gg n$.

\begin{rem}
In the language of algebraic complexity theory,
$\Kcost(n)$ is analogous to the \emph{multiplicative complexity}
(or \emph{nonscalar complexity}),
and $\Tcost(n) + \Kcost(n)$ to the \emph{total complexity}
\cite[Defn.\,4.7]{BCS-complexity}.
We emphasise however that the concepts here are slightly different
as we are working in the Turing model.
\end{rem}

\begin{thm}
\label{thm:transform-scheme}
Given a multiplication algorithm for $\FF_2[x]$
with associated cost function $\Mcost(n)$,
there exists a transform scheme achieving
\begin{equation}
\label{eq:Kcost-bound}
\Kcost(n) \ll n \cdot 4^{\log^* n}
\end{equation}
(or equivalently, $\Kstar(n) \ll 4^{\log^* n}$) and
\begin{equation}
\label{eq:Tcost-bound}
\Tcost(n) \ll \Mcost(n).
\end{equation}
\end{thm}

\begin{rem}
All modern multiplication algorithms for $\FF_2[x]$
have a similar structure:
``evaluate'' the two input polynomials
(using some sort of fast Fourier transform,
possibly after re-encoding the polynomials in some way),
perform ``pointwise multiplications'' recursively,
and then ``interpolate'' to obtain the result.
In principle, by starting with such an algorithm
and unrolling the recursion all the way down to $\FF_2$,
one can construct a transform scheme in the sense of
\Cref{defn:transform-scheme}.
This approach has the advantage of producing forward and inverse transform
algorithms achieving \eqref{eq:Tcost-bound}
with the best possible big-$O$ constant,
but the downside is that the construction depends on the
details of the particular multiplication algorithm being used.
To avoid forcing the reader to digest copious technical details
of papers such as \cite{HvdH-ffmul-cyclotomic},
we have instead opted to formulate \Cref{thm:transform-scheme}
in such a way that we can treat the underlying multiplication algorithm
as a black box.
(In a real implementation one should of course structure the software
so as to take advantage of whatever FFT representation
is being used internally by the underlying multiplication algorithm.)
\end{rem}

\begin{proof}[Proof of \Cref{thm:transform-scheme}]
We will construct the desired transform scheme recursively,
reducing a transform problem of size $n$
to many problems of size $\lg n$.
The main recursive step only works for sufficiently large~$n$.
For small values of~$n$,
we may use any reasonable method;
for example, by adapting Karatsuba's multiplication algorithm
\cite[\S8.1]{vzGG-compalg3}
we obtain a transform scheme with $K_n = 3^{\lg n}$
(details omitted).
We now explain the main recursive construction for large~$n$.

\medskip
\textit{(1) Set up parameters.}
Define
\[
\lambda \coloneqq \lg n,
\qquad m \coloneqq \lceil \lambda/2 \rceil,
\qquad r \coloneqq \left\lceil \frac{n}{m} \right\rceil
   \asymp \frac{n}{\log n}.
\]
Fix an irreducible $h \in \FF_2[z]$ of degree~$\lambda$,
and put $\FFext \coloneqq \FF_2[z]/h(z)$.
Fix also a primitive element $\zeta \in \FFext^*$.
We have $\lambda < n$ for large $n$;
let $(K_\lambda,\fwd_\lambda,\inv_\lambda)$
be a transform pair of order $\lambda$ (constructed recursively).

\medskip
\textit{(2) Definition of $\fwd_n$ and $K_n$.}
Given an input polynomial $f \in \FF_2[x]_n$,
cut it up into $r$ chunks of length $m$
(noting that $rm \geq n$), say
\[
f(x) = \sum_{i=0}^{r-1} x^{im} f_i(x),
   \qquad f_i \in \FF_2[x]_m.
\]
Define a corresponding bivariate polynomial
$\tilde f(y,z) \coloneqq \sum_{i=0}^{r-1} f_i(z) y^i \in \FF_2[z]_m[y]_r$,
so that $f(x) = \tilde f(x^m, x)$.
Let $F(y)$ be the image of $\tilde f$ in $\FFext[y]_r$.
Define
\[
U_j \coloneqq F(\zeta^j) \in \FFext,
   \qquad 0 \leq j < 2r-1,
\]
and let $\tilde U_j \in \FF_2[z]_\lambda$ be polynomials such that
$\tilde U_j(z) = U_j$.
Finally set
\[
K_n \coloneqq (2r-1) K_\lambda,
   \qquad \fwd_n(f) \coloneqq \fwd_\lambda(\tilde U_0) \oplus \cdots
      \oplus \fwd_\lambda(\tilde U_{2r-2}) \in \FF_2^{K_n}.
\]

\textit{(3) Definition of $\inv_n$.}
We must construct a linear map
$\inv_n \colon \FF_2^{K_n} \to \FF_2[x]_{2n-1}$
in such a way that for any $f, g \in \FF_2[x]_n$,
it sends $\fwd_n(f) \cdot \fwd_n(g)$ to~$fg$.

Let $\tilde f$, $F$, $U_j$ and $\tilde U_j$ be as above,
and let $\tilde g$, $G$, $V_j$ and $\tilde V_j$
be the corresponding objects for~$g$.
First, applying
$\inv_\lambda \colon \FF_2^{K_\lambda} \to \FF_2[z]_{2\lambda-1}$
to each component of
\[
\fwd_n(f) \cdot \fwd_n(g) =
   \bigl(\fwd_\lambda(\tilde U_0) \cdot \fwd_\lambda(\tilde V_0)\bigr)
   \oplus \cdots \oplus
   \bigl(\fwd_\lambda(\tilde U_{2r-2})
      \cdot \fwd_\lambda(\tilde V_{2r-2})\bigr)
\]
recovers the polynomials $\tilde U_j \tilde V_j \in \FF_2[z]_{2\lambda-1}$.
On each component,
division by $h(z)$ then gives a linear map
$\FF_2[z]_{2\lambda-1} \to \FFext$
that yields the products $U_j V_j \in \FFext$.
The powers $\{\zeta^j\}_{j=0}^{2r-2}$ are distinct
(since $\ord \zeta = 2^\lambda - 1 \geq n - 1 \geq 2r - 1$ for large~$n$)
so the product $H \coloneqq FG \in \FFext[y]_{2r-1}$ is determined
by its values $H(\zeta^j) = F(\zeta^j) G(\zeta^j) = U_j V_j$ for these~$j$.
The corresponding interpolation map
$\oplus_{j=0}^{2r-2} \FFext \to \FFext[y]_{2r-1}$
is clearly linear (even over $\FFext$).
Since $2m-1 \leq \lambda$ the product
$\tilde f \tilde g \in \FF_2[z]_{2m-1}[y]_{2r-1}$
is determined in turn by~$H = FG$,
so we recover~$H$ by applying the linear map
$\FFext[y]_{2r-1} \to \FF_2[z]_{2m-1}[y]_{2r-1}$
that simply truncates with respect to~$z$.
Finally we deduce
$(fg)(x) = (\tilde f \tilde g)(x^m, x) \in \FF_2[x]_{2n-1}$
by first performing the substitution $(y,z) \mapsto (x^m,x)$,
which is a linear map
$\FF_2[z]_{2m-1}[y]_{2r-1} \to \FF_2[x]_{2rm-1}$,
followed by the truncation $\FF_2[x]_{2rm-1} \to \FF_2[x]_{2n-1}$.
(Note that the substitution automatically lands in $\FF_2[x]_{2n-1}$
if the input happens to be a sum of terms of the form
$\fwd_n(f) \cdot \fwd_n(g)$,
but this is not always true for arbitrary input in $\FF_2^{K_n}$.)

\medskip
\textit{(4) Cost of evaluating $\fwd_n$ and $\inv_n$.}
The parameters $\lambda$, $r$ and $m$
may be computed in time $(\log n)^{O(1)}$.
By \Cref{lem:irreducible} and \Cref{lem:generator}
we may (deterministically) find suitable $h(z)$ and $\zeta$
in time $(\log n)^{O(1)} + n^{1/4 + o(1)} = o(n)$.

Given $f \in \FF_2[x]_n$,
we may construct $\tilde f(y,z)$ in time $O(rm) = O(n)$
and then $F(y)$ in time $O(r\lambda) = O(n)$.
We may then compute the values $U_j = F(\zeta^j)$
via \Cref{prop:geometric}(a)
in time $O(\Mcost(\lambda r) + \Mcost(\lambda) \log \lambda) = O(\Mcost(n))$,
and carry out $2r-1$ applications of $\fwd_\lambda$
in time $O(r \Tcost(\lambda))$.
The cost of evaluating $\fwd_n(f)$ is thus
\begin{equation}
\label{eq:fwd-cost}
O(\Mcost(n) + r \Tcost(\lambda)).
\end{equation}
The complexity analysis for $\inv_n$ is similar;
using \Cref{prop:geometric}(b),
the cost of evaluating $\inv_n$ is given
by the same bound \eqref{eq:fwd-cost}.

Since $r \ll n / \lg n$, we deduce the recursive bound
$\Tcost(n) \ll \Mcost(n) + (n  / \lg n) \Tcost(\lg n)$
for all large~$n$, say for $n \geq n_0$.
Let $B > 0$ be a constant such that
\[
\Tcost(n) < B \Bigl( \Mcost(n) + \frac{n}{\lg n} \Tcost(\lg n) \Bigr)
\]
for all $n \geq n_0$.
Increasing $n_0$ if necessary,
we may ensure that $(\log \lg n) / \log n < (2B)^{-1}$
for all $n \geq n_0$.
Now choose $C > 2B$ such that
\begin{equation}
\label{eq:Tcost-bound-basecase}
\Tcost(n) < C \Mcost(n)
\end{equation}
for all $n < n_0$.
We will prove by induction that \eqref{eq:Tcost-bound-basecase}
also holds for $n \geq n_0$.
We may assume that $\lg n < n$ for such~$n$,
so the inductive hypothesis yields
\begin{align*}
\Tcost(n)
   & < B \Bigl(\Mcost(n) + \frac{C n}{\lg n} \Mcost(\lg n)\Bigr) \\
   & = B \bigl(\Mcost(n) + C n \log \lg n \cdot \Mstar(\lg n)\bigr) \\
   & \leq B \bigl(\Mcost(n) + C n \log \lg n \cdot \Mstar(n)\bigr) \\
   & = B \biggl(1 + \frac{C \log \lg n}{\log n} \biggr) \Mcost(n) \\
   & < B \biggl(1 + \frac{C}{2B} \biggr) \Mcost(n)
      < \biggl(\frac{C}{2} + \frac{C}{2}\biggr) \Mcost(n) = C \Mcost(n)
\end{align*}
as desired.
We conclude that $\Tcost(n) \ll \Mcost(n)$.

\medskip
\textit{(5) Estimating $K_n$.}
Since
\[
r \leq \frac{n}{m} + 1 \leq \frac{2n}{\lambda} + 1
   = \left(1 + O\left(\frac{\log n}{n}\right)\right) \frac{2n}{\lg n}
\]
we have
\[
\frac{K_n}{n} = \frac{(2r-1) K_\lambda}{n}
   \leq \left(1 + O\left(\frac{\log n}{n}\right)\right)
      \frac{4 K_\lambda}{\lambda}
\]
for large~$n$.
This shows that the ratio $K_n/n$ increases by
a factor of essentially $4$ at each recursion level,
and it is not difficult to deduce that $K_n \ll n \cdot 4^{\log^* n}$.

To make the argument precise, let $\Phi(x) \coloneqq \log_2 x + 1$.
Then $\Phi(x)$ is \emph{logarithmically slow} in the sense of
\cite[\S5]{HvdHL-mul},
so it admits a (non-decreasing) iterator $\Phi^*(x)$,
i.e., a function such that $\Phi^*(x) = \Phi^*(\Phi(x)) + 1$
for all large $x$ \cite[Lem.\,2]{HvdHL-mul}.
Moreover $\Phi^*(x) = \log^* x + O(1)$ \cite[Lem.\,3]{HvdHL-mul}.
We claim that for a suitable constant $A > 0$,
\begin{equation}
\label{eq:Kn-induction}
\frac{K_n}{n} < A \left(1 - \frac{1}{\sqrt n}\right) 4^{\Phi^*(n)}
\end{equation}
for all $n \geq 2$.
Indeed, if \eqref{eq:Kn-induction} holds with $\lambda$ in place of $n$,
then (for large $n$)
\begin{align*}
\frac{K_n}{n}
   & \leq \left(1 + O\left(\frac{\log n}{n}\right)\right)
      \frac{4 K_\lambda}{\lambda} \\
   & < \left(1 + O\left(\frac{\log n}{n}\right)\right)
       \left(1 - \frac{1}{\sqrt{\lg n}}\right)
       4A \cdot 4^{\Phi^*(\lg n)}
   < A \left(1 - \frac{1}{\sqrt n}\right) 4^{\Phi^*(\lg n) + 1}.
\end{align*}
Since $\Phi(n) \geq \lg n$ we have
$\Phi^*(\lg n) \leq \Phi^*(\Phi(n)) = \Phi^*(n) - 1$,
so \eqref{eq:Kn-induction} also holds for~$n$.
Increasing $n_0$ if necessary
so that the preceding argument is valid for $n \geq n_0$,
and choosing $A$ so that \eqref{eq:Kn-induction} holds for $n < n_0$,
we conclude by induction that \eqref{eq:Kn-induction} holds
for all $n \geq 2$.
Therefore $K_n \ll n \cdot 4^{\Phi^*(n)} \asymp n \cdot 4^{\log^* n}$,
so \eqref{eq:Kcost-bound} holds.
\end{proof}

\begin{rem}
\label{rem:cyclotomic}
The bound $\Kcost(n) \ll n \cdot 4^{\log^* n}$
in \Cref{thm:transform-scheme} is not optimal.
As mentioned earlier,
the paper \cite{HvdH-ffmul-cyclotomic} describes a polynomial multiplication
algorithm for $\FF_2[x]$ achieving
$\Mcost(n) \ll n \log n \cdot 4^{\log^* n}$.
It is pointed out in the paragraphs following
\cite[Thm.\,1.1]{HvdH-ffmul-cyclotomic} that for
an $\FF_2$-algebra~$\Aalg$,
the same approach leads to a multiplication algorithm for $\Aalg[x]$
whose multiplicative complexity is $O(n \cdot 2^{\log^* n})$.
Similarly, by adapting the strategy of \cite{HvdH-ffmul-cyclotomic}
to the setting of \Cref{thm:transform-scheme},
it is possible to improve \eqref{eq:Kcost-bound} to
\begin{equation}
\label{eq:Kcost-bound-improved}
\Kcost(n) \ll n \cdot 2^{\log^* n}.
\end{equation}
We do not have space to give the details here,
and suggest that the reader consult \cite[\S1.3]{HvdH-ffmul-cyclotomic}
for an overview of the construction.
The basic idea is to replace the coefficient field $\FFext = \FF_2[z]/h(z)$
by the ring $\FF_2[z]/\phi_\alpha(z)$
for a suitable cyclotomic polynomial $\phi_\alpha(z)$.
After dealing with some number-theoretic complications,
this enables a savings in zero-padding by a factor of two
at each recursive step.

The author does not know how to achieve simultaneously
$\Tcost(n) \ll \Mcost(n)$ and $\Kcost(n) \ll n \cdot C^{\log^* n}$
for any $C < 2$.
It \emph{is} possible to design a transform scheme
that achieves $\Kcost(n) \ll n$,
but currently not with $\Tcost(n)$ anywhere close to quasilinear in~$n$;
so far the best result seems to be $\Tcost(n) < n^{O(1)}$ \cite{BBP-linear}.
\end{rem}

\begin{rem}
\label{rem:RAM-transform}
As in \Cref{rem:RAM-multiply},
it may be possible to do better than \Cref{thm:transform-scheme}
in the RAM model.
Instead of mapping $\FF_2[x]_n$ to $\FF_2^K$,
we could map to $(\FF_{2^\mu})^L$ where $\mu \asymp \log n$
and $L \asymp n / \log n$ ---
in fact, this is already what is done by the first recursion level
of the algorithm in \Cref{thm:transform-scheme}.
Users of transform schemes elsewhere in the paper
(such as the core algorithm in \Cref{sec:core})
would then perform pointwise multiplications
directly in $\FF_{2^\mu}$ via lookup tables.
This may in effect allow us to take $\Kstar(n) \asymp 1$,
leading to a small improvement in \Cref{thm:main-heuristic}
for the RAM model (see \Cref{rem:main-heuristic-precise}).
\end{rem}

For the rest of the paper
we fix a transform scheme with $\Tcost(n) \ll \Mcost(n)$,
such as the scheme constructed in \Cref{thm:transform-scheme}.
Regarding $\Kcost(n)$,
we will adopt a similar policy as for $\Mcost(n)$:
we will express everything in terms of $\Kcost(n)$
until this becomes inconvenient,
at which point we will invoke \Cref{thm:transform-scheme},
using \eqref{eq:Kcost-bound} to replace $\Kstar(n)$
by $(\log n)^{o(1)}$ or $(\log \log n)^{o(1)}$.

\subsection{Restricted products}
\label{sec:restricted-products}

In this section we discuss efficient algorithms for computing
restricted products (see \Cref{defn:restricted-product})
over $\FF_2$.

Let $n \geq 1$.
If $f \in \FF_2[x]_{2n}$ and $g \in \FF_2[x^{-1}]_n$,
then $f \restrict g \in \FF_2[x]_{2n}$.
We call this a \emph{restricted product of size $n$} (over $\FF_2$).
Given an algorithm for performing this operation,
let $\Rbarcost(n)$ denote its worst-case running time.
By analogy with $\Mcost(n)$ we define
\[
\Rbarstar(n) \coloneqq \frac{\Rbarcost(n)}{n \max(1, \log n)}, \qquad n \geq 1,
\]
and then
\[
\Rstar(n) \coloneqq \max_{1 \leq m \leq n} \Rbarstar(m), \qquad
\Rcost(n) \coloneqq n \max(1, \log n) \cdot \Rstar(n),
\qquad n \geq 1.
\]
Then $\Rstar(n)$ is non-decreasing,
and $\Rbarcost(n) \leq \Rcost(n)$ for all $n \geq 1$.

To simplify certain calculations later,
it will be convenient to assume that
\begin{equation}
\label{eq:MR-reduction}
\Rcost(n) \gg \Mcost(n).
\end{equation}
This relation is satisfied for all restricted product algorithms
considered in this paper.
It is also a reasonable assumption more generally,
as it is possible to reduce a multiplication problem
of size $n$ to a restricted product of size $O(n)$
(we leave this as an exercise for the reader).

Our main result concerning the complexity of restricted products
is as follows.
\begin{thm}
\label{thm:restricted-simple}
There is a restricted product algorithm achieving
\[
\Rcost(n) < n \log n \cdot \exp((\log \log n)^{1/2+o(1)}),
\]
or equivalently $\Rstar(n) < \exp((\log \log n)^{1/2+o(1)})$.
\end{thm}
Note that $\exp((\log \log n)^{1/2+o(1)})$ grows more slowly
than any positive power of $\log n$,
and the estimate in \Cref{thm:restricted-simple} implies that
$\Rcost(n) < n (\log n)^{1+o(1)}$.
Before embarking on the proof,
we point out that it is relatively easy to design an algorithm
achieving the weaker bound $\Rcost(n) < n (\log n)^{2+o(1)}$.
This may be accomplished by reducing the problem to
two half-sized restricted products (handled recursively)
together with an ordinary product of polynomials of degree $O(n)$.
This leads to the recurrence
$\Rcost(n) < 2 \Rcost(n/2) + O(\Mcost(n))$
and hence to the bound
$\Rcost(n) \ll \Mcost(n) \log n < n (\log n)^{2+o(1)}$.
(This decomposition appears implicitly
in the proof of \cite[Lem.\,8]{FT-prime-tables},
although in that context the authors were able to
avoid recursing all the way to the bottom.)
Unfortunately this is not quite strong enough
to prove \Cref{thm:main-heuristic};
using this estimate for $\Rcost(n)$ would degrade the bound in
\Cref{thm:main-heuristic}
from $N (\log \log N)^{1+o(1)}$ to $N(\log \log N)^{2+o(1)}$
(see \Cref{rem:main-heuristic-precise}).

To prove the sharper estimate in \Cref{thm:restricted-simple}
we will generalise the idea sketched above,
splitting the polynomials into $r$ chunks rather than just two,
where $r$ is a slowly growing function of~$n$.
The following lemma describes the required decomposition
in a slightly more general context.
\begin{lem}[Block decomposition for restricted product]
\label{lem:restricted-decomposition}
Let $R$ be a ring and let $m \geq 1$ be an integer.
Then for $f \in R\bbracket{x}$ and $g \in R\bbracket{x^{-1}}$ we have
\begin{equation}
\label{eq:restricted-decomposition}
(f \restrict g)(x) =
   f g_0 + \sum_{k \geq 0} x^{km} (f_*^k \restrict g_*^k)(x)
   + (\tilde f \restrict \tilde g)(x^m),
\end{equation}
where
\begin{align*}
   f_*^k(x) & \coloneqq \sum_{0 \leq u < 2m} f_{2km+u} x^u \in R[x]_{2m}, \\
   g_*^k(x) & \coloneqq \sum_{1 \leq v < m} g_{km+v} x^{-v} \in R[x^{-1}]_m,
\end{align*}
and
\begin{align*}
\tilde f(y) & \coloneqq
   \sum_{i \geq 0} f^i(x) y^i \in R[x]\bbracket{y},
   & f^i(x) & \coloneqq
   \sum_{0 \leq s < m} f_{im+s} x^s \in R[x]_m, \\
\tilde g(y) & \coloneqq
   \sum_{j \geq 1} g^j(x) y^{-j} \in R[x]\bbracket{y^{-1}},
   & g^j(x) & \coloneqq
   \sum_{0 \leq t < m} g_{jm-t} x^t \in R[x]_m.
\end{align*}
\end{lem}

The expression $(\tilde f \restrict \tilde g)(y)$
means the restricted product with respect to $y$
over the coefficient ring $R[x]$.
Each coefficient of $\tilde f \restrict \tilde g$
is a polynomial in $R[x]_{2m-1}$.
Evaluating $(\tilde f \restrict \tilde g)(y)$ at $y = x^m$ amounts to taking
a suitable overlapping sum of these polynomials.

The decomposition of \Cref{lem:restricted-decomposition}
is illustrated in \Cref{fig:restricted-decomposition}.
The parameter $m$ indicates the block size.
The first term in \eqref{eq:restricted-decomposition}
corresponds to the horizontal strip along the top,
the second term to the dark triangular regions along the diagonal,
and the third term to the lightly shaded square blocks.

The reader may wonder why the vertical chunks $g_*^k(x)$ and $g^j(x)$
are offset from each other by one unit.
We have chosen to define the $g^j(x)$ this way to streamline the
discussion of ``slicing'' in \Cref{sec:core-slice}.

\begin{figure}[h]
\vspace{0.5em}
\begin{tikzpicture}[scale=0.32]
\newcommand{\m}{4}     
\renewcommand{\r}{4}   
\newcommand{\extright}{0.6}
\newcommand{\extdown}{0.6}

\pgfmathtruncatemacro{\rminusone}{\r-1}
\pgfmathtruncatemacro{\rminustwo}{\r-2}
\pgfmathtruncatemacro{\mminusone}{\m-1}


\fill[fill=\lightshade]
   ({2*\r*\m+\extright,0}) -- (0,0)
   -- (0,-1) -- ({2*\r*\m+\extright,-1}) -- cycle;

\foreach \i in {0,...,\rminusone}
{
   \fill[fill=\darkshade]
   ({2*(\i+1)*\m},{-\i*\m-1}) -- ({2*\i*\m+2},{-\i*\m-1})
   \foreach \s in {1,...,\mminusone}
   { -- ({2*\i*\m+2*\s},{-\i*\m-\s-1}) -- ({2*\i*\m+2*\s+2},{-\i*\m-\s-1}) }
   -- cycle;
}

\foreach \j in {0,...,\rminustwo}
{
   \fill[fill=\lightshade]
   ({(2*\j+2)*\m},{-\j*\m-1}) --
   ({2*\r*\m+\extright},{-\j*\m-1}) --
   ({2*\r*\m+\extright},{-(\j+1)*\m-1}) --
   ({(2*\j+2)*\m},{-(\j+1)*\m-1}) -- cycle;
}
\fill[fill=\lightshade]
({2*\r*\m},{-(\r-1)*\m-1}) --
({2*\r*\m+\extright},{-(\r-1)*\m-1}) --
({2*\r*\m+\extright},{-\r*\m-\extdown}) --
({2*\r*\m},{-\r*\m-\extdown}) -- cycle;


\pgfmathtruncatemacro{\ymax}{\r*\m}
\foreach \y in {0,...,\ymax}
   \draw[draw=\gridshadelight] (0,-\y) -- (2*\r*\m+\extright,-\y);

\pgfmathtruncatemacro{\xmax}{2*\r*\m}
\foreach \x in {0,...,\xmax}
   \draw[draw=\gridshadelight] (\x,0) -- (\x,-\r*\m-\extdown);


\draw[draw=black, thick]
   ({2*\r*\m+\extright,0}) -- (0,0) -- (0,-1) -- ({2*\r*\m+\extright,-1});

\foreach \j in {0,...,\rminustwo}
{
   \pgfmathtruncatemacro{\firsti}{2*\j+2}
   \pgfmathtruncatemacro{\lasti}{2*\r-1}
   \foreach \i in {\firsti,...,\lasti}
   {
      \draw[draw=black, thick]
      ({\i*\m},{-\j*\m-1}) -- ({(\i+1)*\m},{-\j*\m-1}) --
         ({(\i+1)*\m},{-(\j+1)*\m-1}) -- ({\i*\m},{-(\j+1)*\m-1}) -- cycle;
   }
   \draw[draw=black, thick]
   ({2*\r*\m+\extright},{-(\j+1)*\m-1}) -- ({2*\r*\m},{-(\j+1)*\m-1});
}
\draw[draw=black, thick]
({2*\r*\m},{-(\r-1)*\m-1}) -- ({2*\r*\m},{-\r*\m-\extdown});

\foreach \i in {0,...,\rminusone}
{
   \draw[draw=black, thick]
   ({2*(\i+1)*\m},{-\i*\m-1}) -- ({2*\i*\m+2},{-\i*\m-1})
   \foreach \s in {1,...,\mminusone}
   { -- ({2*\i*\m+2*\s},{-\i*\m-\s-1}) -- ({2*\i*\m+2*\s+2},{-\i*\m-\s-1}) }
   -- cycle;
}


\foreach \i in {0,1,2}
   \draw (\i + 0.5, 0.1) node[above] {$f_{\i}$};
\draw (4.0, 0.3) node[above] {$\cdots$};

\foreach \j in {0,1,2}
   \draw (-0.1, -\j - 0.55) node[left] {$g_{\j}$};
\draw (-0.4, -3.7) node[left] {$\vdots$};

\foreach \k in {0,1,2}
{
   \draw[decorate,decoration={brace,amplitude=5pt,raise=2pt,mirror}]
      (-1.5,{-0.1-\k*\m}) -- (-1.5,{0.1-(\k+1)*\m})
      node[midway,left=6pt] {$g_*^\k$};
}
\draw (-1.3, -3*\m-0.7) node[left] {$\vdots$};

\foreach \k in {0,1,2}
{
   \draw[decorate,decoration={brace,amplitude=5pt,raise=2pt}]
      ({2*\k*\m+0.1},1.4) -- ({2*(\k+1)*\m-0.1},1.4)
      node[midway,above=6pt] {$f_*^\k$};
}
\draw (2*3*\m+1.5, 1.1) node[above] {$\cdots$};

\foreach \j in {1,2,3}
{
   \draw[decorate,decoration={brace,amplitude=5pt,raise=2pt}]
      ({2*\r*\m+\extright},{-1.1-(\j-1)*\m})
      -- ({2*\r*\m+\extright},{-0.9-\j*\m})
      node[midway,right=6pt] {$g^\j$};
}
\draw (2*\r*\m+\extright+0.7, -3*\m-1.9) node[right] {$\vdots$};

\foreach \i in {0,...,4}
{
   \draw[decorate,decoration={brace,amplitude=5pt,raise=2pt,mirror}]
      ({\i*\m+0.1},{-\r*\m-\extdown}) -- ({(\i+1)*\m-0.1},{-\r*\m-\extdown})
      node[midway,below=6pt] {$f^\i$};
}
\draw (5*\m+1.5, -\r*\m-\extdown-0.9) node[below] {$\cdots$};

\end{tikzpicture}
\caption{Decomposition of \Cref{lem:restricted-decomposition} for $m = 4$.}
\label{fig:restricted-decomposition}
\end{figure}

\begin{proof}[Proof of \Cref{lem:restricted-decomposition}]
The result follows more or less directly from
\Cref{fig:restricted-decomposition},
but for the reader's convenience let us give an algebraic proof.
By definition
\[
(f \restrict g)(x) =
   \sum_{\substack{b \geq 0 \\ a \geq 2b}} f_a g_b x^{a-b} \in R\bbracket{x}.
\]
The sum of the terms with $b = 0$ is exactly $f g_0$.
For $b \geq 1$ we write $b = jm - t$ with $j \geq 1$ and $0 \leq t < m$;
the sum over these terms becomes
\[
\sum_{j \geq 1} \; \sum_{0 \leq t < m} \;
   \sum_{a \geq 2jm - 2t} f_a g_{jm-t} x^{a-jm+t}.
\]
We will consider separately the contributions from those $a$
lying in the intervals $2jm - 2t \leq a < 2jm$ and $a \geq 2jm$.

The contribution from the first interval is
\[
\sum_{j \geq 1} \; \sum_{1 \leq t < m} \;
   \sum_{2jm - 2t \leq a < 2jm} f_a g_{jm-t} x^{a-jm+t},
\]
where we have deleted the $t = 0$ terms as the interval is empty for $t = 0$.
Making the substitutions $j = k + 1$, $t = m - v$ and $a = 2km + u$,
the sum becomes
\[
\sum_{k \geq 0} \; \sum_{1 \leq v < m} \;
   \sum_{2v \leq u < 2m} f_{2km+u} g_{km+v} x^{km+u-v}.
\]
This is exactly $\sum_{k \geq 0} x^{km} (f_*^k \restrict g_*^k)(x)$,
i.e., the second term of \eqref{eq:restricted-decomposition}.

The contribution from the second interval is
\[
\sum_{j \geq 1} \; \sum_{0 \leq t < m} \;
   \sum_{a \geq 2jm} f_a g_{jm-t} x^{a-jm+t}.
\]
Writing $a = im + s$ with $i \geq 2j$ and $0 \leq s < m$, this becomes
\[
\sum_{\substack{j \geq 1 \\ i \geq 2j}} \;
\sum_{\substack{0 \leq t < m \\ 0 \leq s < m}} \;
   f_{im+s} g_{jm-t} x^{(i-j)m + s + t}
= \sum_{\substack{j \geq 1 \\ i \geq 2j}} \;
   f^i(x) g^j(x) x^{(i-j)m}.
\]
This is exactly $(\tilde f \restrict \tilde g)(x^m)$,
i.e., the third term of \eqref{eq:restricted-decomposition}.
\end{proof}

We are now ready for the proof of \Cref{thm:restricted-simple}.
\begin{proof}[Proof of \Cref{thm:restricted-simple}]
We are given as input polynomials
$f \in \FF_2[x]_{2n} \subseteq \FF_2\bbracket{x}$
and $g \in \FF_2[x^{-1}]_n \subseteq \FF_2\bbracket{x^{-1}}$.
For small~$n$, say $n < n_0$, we simply evaluate $f \restrict g$
directly from the definition.

For $n \geq n_0$,
we will describe a recursive algorithm that reduces the problem
to a collection of restricted products of size $r$ and size
$m \coloneqq \lceil n / r \rceil$, where
\[
r \coloneqq n^{1/u} + O(1), \qquad
u \coloneqq \exp\bigl( (\log \log n)^{1/2} \bigr).
\]
We may compute such $r$ in time $(\log n)^{O(1)}$.
By taking $n_0$ sufficiently large,
we may ensure that $3 \leq r < n$ and $3 \leq m < n$ for $n \geq n_0$.

We begin by applying \Cref{lem:restricted-decomposition}
with the given~$m$.
In the second term of \eqref{eq:restricted-decomposition},
we have $f_*^k = 0$ and $g_*^k = 0$ for $k \geq r$,
since $2rm \geq 2n$ and $rm \geq n$.
In the third term, we have $f^i = 0$ for $i \geq 2r$
and $g^j = 0$ for $j \geq r+1$;
furthermore, $g^r$ does not contribute to the sum,
because $\tilde f \restrict \tilde g$ involves the products $f^i g^j$
only for $i \geq 2j$,
and we just observed that $f^i = 0$ for $i \geq 2r$.
Therefore in the present situation \eqref{eq:restricted-decomposition} becomes
\begin{equation}
\label{eq:restricted-decomposition-2}
(f \restrict g)(x) =
   f g_0 + \sum_{0 \leq k < r} x^{km} (f_*^k \restrict g_*^k)(x)
   + (\hat f \restrict \hat g)(x^m)
\end{equation}
where
\[
\hat f(y) \coloneqq
   \sum_{0 \leq i < 2r} f^i(x) y^i \in \FF_2[x][y]_{2r},
\qquad
\hat g(y) \coloneqq
   \sum_{1 \leq j < r} g^j(x) y^{-j} \in \FF_2[x][y^{-1}]_r.
\]

We now estimate the complexity of evaluating
\eqref{eq:restricted-decomposition-2}.
The first term trivially requires time $O(n)$.
The second term may be computed recursively in time
$r \Rbarcost(m) + O(n)$.
For the third term, we must evaluate the sums
\[
S_h \coloneqq
   \sum_{\substack{j \geq 1, \, 2j \leq i < 2r \\ i - j = h}} f^i(x) g^j(x)
   \in \FF_2[x]_{2m-1},
   \qquad 0 \leq h < 2r.
\]
We will handle these via a transform pair of order~$m$
(see \Cref{defn:transform-pair}).
Let $\fwd \colon \FF_2[x]_m \to \FF_2^K$
and $\inv \colon \FF_2^K \to \FF_2[x]_{2m-1}$
be the corresponding forward and inverse transforms,
where $K \leq \Kcost(m)$.
We first compute $\fwd(f^i) \in \FF_2^K$ and $\fwd(g^j) \in \FF_2^K$
for $0 \leq i < 2r$ and $1 \leq j < r$ in time
$O(r \Tcost(m)) = O(r \Mcost(m))$.
We must then compute the pointwise sums
\[
\sum_{\substack{j \geq 1, \, 2j \leq i < 2r \\ i - j = h}}
   \fwd(f^i) \cdot \fwd(g^j) \in \FF_2^K,
   \qquad 0 \leq h < 2r.
\]
This involves computing a restricted product of size~$r$ (over $\FF_2$)
in each coordinate of $\FF_2^K$,
which may be done recursively in time $K \Rbarcost(r) + O(Kr)$.
We must also carry out $O(1)$ array transpositions of size
$K \times O(r)$ to access the data for these restricted products;
by \Cref{lem:transpose} the cost of this step is $O(K r \log r)$.
Finally, we must perform $2r$ inverse transforms to recover the $S_h$,
at a cost of $O(r \Tcost(m)) = O(r \Mcost(m))$,
and then accumulate the $S_h$ to obtain
$(\hat f \restrict \hat g)(x^m) = \sum_{0 \leq h < 2r} x^{hm} S_h(x)$
in time $O(n)$.
Combining these estimates, we obtain the recursive bound
\[
\Rbarcost(n) < r \Rbarcost(m) + \Kcost(m) \Rbarcost(r)
   + O(r \Mcost(m) + \Kcost(m) \, r \log r),
   \qquad n \geq n_0.
\]
Dividing by $n \log n$, and using the estimate
\[
rm = r \left\lceil \frac{n}{r} \right\rceil = n + O(r)
   = \left(1 + \frac{O(1)}{m}\right) n,
\]
the recurrence becomes
\begin{multline}
\label{eq:Rbarstar-recurrence}
\Rbarstar(n) < \left(1 + \frac{O(1)}{m}\right) \biggl(
   \frac{\log m}{\log n} (\Rbarstar(m) + A \Mstar(m)) \\
   + \frac{\log r}{\log n} \Kstar(m) (\Rbarstar(r) + B)
   \biggr), \qquad n \geq n_0,
\end{multline}
for suitable constants $A, B > 0$.

Fix $\eps \in (0, \tfrac12)$;
our goal is to prove that $\Rstar(n) < \exp((\log \log n)^{1/2+\eps})$
for all sufficiently large~$n$.
It is certainly enough to prove that
$\Rbarstar(n) < \exp((\log \log n)^{1/2+\eps})$ for large~$n$.
For this it suffices to find a constant $C_\eps > 0$ such that
\begin{equation}
\label{eq:Rstar-bound}
\Rbarstar(n) < C_\eps \exp((\log \log n)^\alpha)
\end{equation}
for all $n \geq 3$,
where for convenience we set $\alpha \coloneqq 1/2 + \eps/2$.
Indeed, it is clear that
$(\log \log n)^{1/2+\eps/2} + \log C_\eps < (\log \log n)^{1/2+\eps}$
for large enough $n$.

Let $n_\eps \geq n_0$ be a threshold to be chosen later.
We will take
\[
C_\eps \coloneqq 2 \max\left(A, B, \max_{3 \leq n < n_\eps}
      \frac{\Rbarstar(n)}{\exp((\log \log n)^\alpha)} \right),
\]
so that \eqref{eq:Rstar-bound} holds automatically for $3 \leq n < n_\eps$.
We will prove by induction on~$n$ that
\eqref{eq:Rstar-bound} also holds for $n \geq n_\eps$.
(Note that $n_\eps$ depends on $\eps$,
but $n_0$ must be an absolute constant,
as the algorithm does not ``know'' the value of $\eps$.)

Assume therefore that $n \geq n_\eps$.
Applying the inductive hypothesis to \eqref{eq:Rbarstar-recurrence} yields
\begin{multline*}
\Rbarstar(n) < \left(1 + \frac{O(1)}{m}\right) \biggl(
   \frac{\log m}{\log n} \, C_\eps
      \bigl( \exp((\log \log m)^\alpha) + \Mstar(m) \bigr) \\
   + \frac{\log r}{\log n} \, C_\eps \Kstar(m)
      \bigl( \exp((\log \log r)^\alpha) + 1 \bigr)
   \biggr).
\end{multline*}
Since $m = (1 + \frac{O(1)}{m}) (n/r)$ we have
$\log m = \log n - \log r + \frac{O(1)}{m}$, so
\begin{multline*}
\Rbarstar(n) < \left(1 + \frac{O(1)}{m}\right) \biggl(
   \left(1 + \frac{O(1)}{m \log n}\right) C_\eps
      \bigl( \exp((\log \log m)^\alpha) + \Mstar(m) \bigr) \\
   + \frac{\log r}{\log n} \, C_\eps \Bigl( \Kstar(m)
      \bigl( \exp((\log \log r)^\alpha) + 1 \bigr)
      - \exp((\log \log m)^\alpha) - \Mstar(m) \Bigr)
   \biggr) \\
   \shoveleft{\phantom{\Rbarstar(n)} {}< \left(1 + \frac{O(1)}{m}\right)
   C_\eps \bigl( \exp((\log \log m)^\alpha) + \Mstar(m) \bigr)} \\
   + \left(1 + \frac{O(1)}{m}\right) \frac{\log r}{\log n} \,
      C_\eps \Bigl( \Kstar(m) \bigl( \exp((\log \log r)^\alpha) + 1 \bigr)
       - \exp((\log \log m)^\alpha) \Bigr).
\end{multline*}
To establish \eqref{eq:Rstar-bound} for the given value of $n$,
it thus suffices to show that
\[
\Kstar(m) \bigl( \exp((\log \log r)^\alpha) + 1 \bigr)
   < \exp((\log \log m)^\alpha)
\]
and that
\[
\left(1 + \frac{O(1)}{m}\right)
   \bigl( \exp((\log \log m)^\alpha) + \Mstar(m) \bigr)
      < \exp((\log \log n)^\alpha).
\]
Taking logarithms,
and using the fact that $\log(1+x) < x$ for small $x > 0$,
to establish these two inequalities it suffices in turn to prove that
\begin{equation}
\label{eq:restrict-second-term}
\log \Kstar(m) + (\log \log r)^\alpha
   + \frac{1}{\exp((\log \log r)^\alpha)} < (\log \log m)^\alpha
\end{equation}
and
\begin{equation}
\label{eq:restrict-first-term}
\frac{O(1)}{m} + (\log \log m)^\alpha +
   \frac{\Mstar(m)}{\exp((\log \log m)^\alpha)} < (\log \log n)^\alpha.
\end{equation}

We may estimate $(\log \log r)^\alpha$ and $(\log \log m)^\alpha$ as follows.
From the definition of~$r$ we see that
$\log r = u^{-1} \log n + O(1)$ and hence
\begin{align*}
\log \log r & = \log \log n - (\log \log n)^{1/2} + O(1) \\
   & = \log \log n \left(1 - \frac{1}{(\log \log n)^{1/2}}
         + \frac{O(1)}{\log \log n} \right).
\end{align*}
Using the binomial expansion $(1 - x)^\alpha = 1 - \alpha x + O(x^2)$
(and increasing $n_0$ if necessary) we obtain
\begin{align*}
(\log \log r)^\alpha
   & = (\log \log n)^\alpha
         \left(1 - \frac{\alpha}{(\log \log n)^{1/2}}
            + \frac{O(1)}{\log \log n} \right) \\
   & = (\log \log n)^\alpha - \alpha (\log \log n)^{\eps/2} + O(1).
\end{align*}
Similarly for $m$ we have
\[
\log m = (1 - u^{-1}) \log n + O(1)
   = (1 - u^{-1}) \log n \left(1 + \frac{O(1)}{\log n} \right)
\]
so
\begin{align*}
\log \log m & = \log \log n + \log(1 - u^{-1}) + \frac{O(1)}{\log n} \\
   & = \log \log n - \frac{1 + O(u^{-1})}{u}
   = \log \log n \left(1 - \frac{1 + O(u^{-1})}{u \log \log n} \right),
\end{align*}
and then
\begin{align*}
(\log \log m)^\alpha
   & = (\log \log n)^\alpha
      \left(1 - \frac{\alpha + O(u^{-1})}{u \log \log n} \right) \\
   & = (\log \log n)^\alpha
         - \frac{\alpha + O(u^{-1})}{u (\log \log n)^{1/2-\eps/2}}.
\end{align*}

Using these estimates,
to prove \eqref{eq:restrict-second-term} it suffices to show that
\begin{equation}
\label{eq:restrict-second-term-2}
\log \Kstar(m) + O(1) < \alpha (\log \log n)^{\eps/2}.
\end{equation}
This holds provided that our transform scheme
has sufficiently low multiplicative complexity.
For instance, if we use the transform scheme constructed in
\Cref{thm:transform-scheme},
then the left hand side of \eqref{eq:restrict-second-term-2} is bounded by
$\log \Kstar(n) + O(1) \ll \log^* n$,
and then certainly \eqref{eq:restrict-second-term-2} holds
for $n \geq n_\eps$ for large enough $n_\eps$.

Similarly, to prove \eqref{eq:restrict-first-term} it suffices to show that
\[
\frac{O(1)}{m} + \frac{\Mstar(m)}{\exp((\log \log m)^\alpha)} <
   \frac{\alpha + O(u^{-1})}{u (\log \log n)^{1/2-\eps/2}}.
\]
Since $\Mstar(n)$ is non-decreasing,
and since $(\log \log m)^\alpha > \tfrac12 (\log \log n)^\alpha$
and $m > n^{1/2}$ for large $n$,
it is enough to prove that
\begin{equation}
\label{eq:restrict-first-term-2}
\frac{O(1)}{n^{1/2}}
   + \frac{\Mstar(n)}{\exp(\tfrac12(\log \log n)^{1/2+\eps/2})} <
   \frac{\Theta(1)}{\exp((\log \log n)^{1/2})(\log \log n)^{1/2-\eps/2}}.
\end{equation}
Again, assuming that we use a fast enough multiplication algorithm,
satisfying for instance \eqref{eq:ffmul},
then \eqref{eq:restrict-first-term-2} certainly holds for $n \geq n_\eps$
for large enough $n_\eps$.
This completes the proof of \eqref{eq:Rstar-bound}.
\end{proof}

\subsection{Faster restricted products}
\label{sec:restricted-faster}

The complexity bound in \Cref{thm:restricted-simple} is not optimal.
By pushing the techniques from \Cref{sec:restricted-products}
a bit further, one can obtain
\begin{equation}
\label{eq:restricted-advanced-complexity}
\Rcost(n) \ll \exp\bigl(D (\log \log n)^{1/2}\bigr) (\log \log n)^{1/2}
   \cdot 2^{\log^* n} \cdot \Mcost(n)
\end{equation}
where
\[
D \coloneqq (2 \log 2)^{1/2} \approx 1.177.
\]
This is the best bound for $\Rcost(n)$ known to the author.
Note that $\exp(D (\log \log n)^{1/2})$
grows much faster than any fixed power of $\log \log n$,
so even this tighter bound for $\Rcost(n)$ still leaves
a considerable gap in complexity between
the restricted product and ordinary polynomial multiplication.
In this section we will sketch a proof of
\eqref{eq:restricted-advanced-complexity},
mainly to give the reader a sense of how difficult it is to
completely eliminate the $(\log \log N)^{o(1)}$ factor in
\Cref{thm:main-heuristic} (see \Cref{rem:main-heuristic-precise}).

We first reduce to computing restricted products over $\FFext$
for suitable $\lambda$.
Let $\Rbarcost_\lambda(n)$ denote the cost of computing the restricted product
of $f \in \FFext[x]_{2n}$ and $g \in \FFext[x^{-1}]_n$.
Given a restricted product problem of size $n$ over $\FF_2$,
we may apply \Cref{lem:restricted-decomposition}
with $m \coloneqq \lg n$ to reduce to
$O(n/\log n)$ restricted products over $\FF_2$ of size $\lg n$,
plus a restricted product of size $O(n/\log n)$ over $\FF_2[z]_m$.
We embed the latter into a restricted product over $\FFext$
for $\lambda \coloneqq 2 \lg n$, leading to the bound
\[
\Rbarcost(n) < \Rbarcost_\lambda(O(n / \log n))
   + O\left(\frac{n}{\log n} \Rcost(\lg n)\right) + O(n).
\]
The second term on the right is $n (\log \log n)^{1+o(1)}$
by \Cref{thm:restricted-simple},
which is negligible compared to \eqref{eq:restricted-advanced-complexity}.
To prove \eqref{eq:restricted-advanced-complexity},
it therefore suffices to show that
\begin{equation}
\label{eq:Rcost-lambda-bound}
\Rbarcost_\lambda(n) \ll \exp\bigl(D (\log \log n)^{1/2} \bigr)
   (\log \log n)^{1/2} \cdot 2^{\log^* \lambda} \cdot \Mcost(\lambda n)
\end{equation}
in the regime where $\lambda \sim 2 \lg n$.

We will first explain how to achieve the slightly weaker bound
\begin{equation}
\label{eq:Rcost-lambda-weaker-bound}
\Rbarcost_\lambda(n) \ll \exp\bigl(D (\log \log n)^{1/2} \bigr)
   (\log \log n)^{1/2} \cdot 4^{\log^* \lambda} \cdot \Mcost(\lambda n)
\end{equation}
over a range of parameters $(n,\lambda)$ that includes $\lambda \sim 2 \lg n$.
The algorithm involves two different reductions,
depending on the relative size of $n$ and~$\lambda$.
When $n$ is large enough,
we reduce to a smaller value of $n$ by following the same plan as in
the proof of \Cref{thm:restricted-simple},
except that we evaluate-interpolate directly over $\FFext$ instead of
introducing a transform scheme.
This strategy contributes the
$\exp(D (\log \log n)^{1/2}) (\log \log n)^{1/2}$ factor
in \eqref{eq:Rcost-lambda-weaker-bound}.
When $n$ becomes too small,
i.e., when the polynomial arithmetic in the ``$\lambda$ direction''
becomes relatively too expensive,
we re-encode the problem over $\FF_{2^{\lambda'}}$
with $\lambda'$ exponentially smaller than $\lambda$;
this explains the $4^{\log^* \lambda}$ factor
in \eqref{eq:Rcost-lambda-weaker-bound}.

We now give a few more details.
Let $\Phi(x) \coloneqq 2(\log_2 x + 1)^4$
(unrelated to the $\Phi(x)$ from \Cref{sec:transform-schemes}).
Then $\Phi(x)$ is logarithmically slow and admits an iterator $\Phi^*(x)$
such that $\Phi^*(x) = \log^* x + O(1)$ \cite[\S5]{HvdHL-mul}.
Define
\begin{align*}
\Fcost(n) & \coloneqq
   \exp\bigl(D (\log \log n)^{1/2}\bigr) (\log \log n)^{1/2}, \\
\Gcost(\lambda) & \coloneqq
   \left(1 - \frac{1}{\log \log \lambda}\right) 4^{\Phi^*(\lambda)}.
\end{align*}
We will show that there exist constants $C > 0$ and $n_0 \geq 1$ such that
\begin{equation}
\label{eq:Rcost-lambda-induction}
\Rbarcost_\lambda(n) < C \Fcost(n) \Gcost(\lambda) \Mcost(\lambda n)
\end{equation}
for all $(n,\lambda)$ satisfying
\[
\lg^2 \lambda < \lg n < \lambda, \qquad n \geq n_0.
\]

The argument is by induction on $(n,\lambda)$.
First consider the case that $n$ is small,
say $\lg n < 2 \lg^4 \lambda$.
Set $\lambda' \coloneqq 2 \lg^4 \lambda$;
then $\lg^2 \lambda' < \lg n < \lambda'$ (for large~$n$).
We lift the coefficient ring $\FFext \cong \FF_2[z]/h(z)$ to
$\FF_2[z]_\lambda$,
and cut up these polynomials into
$k \coloneqq \lceil 2\lambda/\lambda' \rceil$ chunks of length $\lambda'/2$.
We thereby reduce to a restricted product of size~$n$ over
$\FF_2[z]_{\lambda'/2}[y]_k$.
We then embed $\FF_2[z]_{\lambda'/2}$ in $\FF_{2^{\lambda'}}$,
thus reducing to a restricted product of size $n$ over
$\FF_{2^{\lambda'}}[y]_k$.
Next we use \Cref{prop:geometric}
to evaluate/interpolate $y$ at $2k-1$ powers of a
generator of $\FF_{2^{\lambda'}}$.
One checks that $2k-1 \leq 2^{\lambda'} - 1$ (for large $n$),
so there are certainly enough evaluation points in $\FF_{2^{\lambda'}}$.
The cost of the evaluation/interpolation steps is
$O(n (\Mcost(\lambda' k) + \Mcost(\lambda') \log \lambda'))
= O(n \Mcost(\lambda)) = O(\Mcost(\lambda n))$.
We have thus finally reduced to $2k-1$ separate restricted products
of size $n$ over $\FF_{2^{\lambda'}}$, and we obtain the estimate
\[
\Rbarcost_\lambda(n) <
   \left(\frac{4\lambda}{\lambda'} + O(1)\right) \Rbarcost_{\lambda'}(n)
   + O(\Mcost(\lambda n)).
\]
Using the inductive hypothesis \eqref{eq:Rcost-lambda-induction}
to estimate $\Rbarcost_{\lambda'}(n)$ leads to
\[
\Rbarcost_\lambda(n) <
   C \Fcost(n) \left(4 + O\left(\frac{\lambda'}{\lambda} + 
   \frac{1}{\Fcost(n) \Gcost(\lambda')} \right) \right)
   \Gcost(\lambda') \Mcost(\lambda n).
\]
One checks that
\[
\left(4 + O\left(\frac{\lambda'}{\lambda} + 
   \frac{1}{\Fcost(n) \Gcost(\lambda')} \right) \right)
\left(1 - \frac{1}{\log \log \lambda'}\right)
   < 4 \left(1 - \frac{1}{\log \log \lambda}\right)
\]
for large $n$,
and since $\lambda' \leq 2(\log_2 \lambda + 1)^4 = \Phi(\lambda)$,
also $\Phi^*(\lambda') + 1 \leq \Phi^*(\Phi(\lambda)) + 1 = \Phi^*(\lambda)$
for large $n$.
Combining these facts shows that \eqref{eq:Rcost-lambda-induction}
holds for $(n,\lambda)$.

We now consider the case that $n$ is large,
i.e., $\lg n \geq 2 \lg^4 \lambda$.
Define
\[
u \coloneqq \exp\bigl(D (\log \log n)^{1/2}\bigr),
\qquad r \coloneqq n^{1/u} + O(1),
\qquad m \coloneqq \lceil n / r \rceil.
\]
Applying \Cref{lem:restricted-decomposition},
we reduce the problem to $r$ restricted products of size $m$ over $\FFext$
together with a restricted product of size $r$ over $\FFext[x]_m$.
To handle the latter, we use \Cref{prop:geometric}
to evaluate/interpolate directly over $\FFext$;
there are enough evaluation points in $\FFext$
as $2m-1 < n < 2^\lambda - 1$ for large $n$.
The cost of the evaluations and interpolations
(and associated transpositions) is
$O(r (\Mcost(\lambda m) + \Mcost(\lambda) \log \lambda) + \lambda r m \log r)
   = O(r \Mcost(\lambda m) + \lambda r m \log r)$,
and this reduces the restricted product over $\FFext[x]_m$
to $2m-1$ restricted products of size $r$ over $\FFext$.
We thus obtain the estimate
\[
\Rbarcost_\lambda(n) < r \Rbarcost_\lambda(m) + 2m \Rbarcost_\lambda(r)
   + O(r \Mcost(\lambda m) + \lambda r m \log r).
\]

\begin{rem}
This last inequality is reminiscent of recurrences that arise
in the analysis of the fastest known algorithms
for \emph{relaxed multiplication} (or \emph{online multiplication});
see for example \cite[Eq.\,8]{vdH-relaxed}.
Indeed, the bound \eqref{eq:restricted-advanced-complexity}
(and our choice of $r$) was inspired by the analysis in \cite{vdH-relaxed}.
\end{rem}

Define the normalisation
\[
\Rbarstar_\lambda(n) \coloneqq \frac{\Rbarcost_\lambda(n)}{\lambda n \log n}.
\]
Using calculations similar to those in the proof of
\Cref{thm:restricted-simple},
and since $\log \lambda \ll \log m$ for large $n$,
the previous inequality becomes
\[
\Rbarstar_\lambda(n) <
\left(1 + \frac{O(1)}{m}\right) \biggl(
   \frac{\log m}{\log n} (\Rbarstar_\lambda(m) + A \Mstar(\lambda m))
   + \frac{\log r}{\log n} (2 \Rbarstar_\lambda(r) + B) \biggr)
\]
for suitable constants $A, B > 0$.
One checks that $\lg^2 \lambda < \lg r, \lg m < \lambda$ for large $n$,
so we may apply the inductive hypothesis \eqref{eq:Rcost-lambda-induction}
for $(r,\lambda)$ and $(m,\lambda)$.
This yields
\begin{align*}
\Rbarstar_\lambda(m)
   & < \left(1 + \frac{O(1)}{(\log m)^{1/2}}\right)
   C \Fcost(m) \Gcost(\lambda) \Mstar(\lambda m), \\
\Rbarstar_\lambda(r)
   & < \left(1 + \frac{O(1)}{(\log r)^{1/2}}\right)
   C \Fcost(r) \Gcost(\lambda) \Mstar(\lambda r).
\end{align*}
Substituting these into the previous inequality, choosing $C > 4 \max(A, B)$,
and using the fact that $\log r, \log m \gg (\log n)^{1/2}$ for large $n$,
we obtain
\[
\Rbarstar_\lambda(n)
   < \left(1 + \frac{O(1)}{(\log n)^{1/4}}\right)
      C \Gcost(\lambda) \Mstar(\lambda n)
      \biggl( \frac{\log m}{\log n} (\Fcost(m) + \tfrac14)
      + \frac{\log r}{\log n} (2\Fcost(r) + \tfrac14) \biggr).
\]
Since $\log m \leq \log n - \log r + O(1)$,
the last bracketed expression is bounded by
\[
\left(1 + \frac{O(1)}{\log n}\right) (\Fcost(m) + \tfrac14)
   + \frac{\log r}{\log n} (2 \Fcost(r) - \Fcost(m)).
\]
Now, beginning with the definitions of $r$ and $m$,
some calculation shows that
\begin{align*}
(\log \log r)^{1/2} & = (\log \log n)^{1/2} - \frac{D}{2}
   - \frac{D^2}{8 (\log \log n)^{1/2}}
   + O\left(\frac{1}{\log \log n}\right), \\
(\log \log m)^{1/2} & = (\log \log n)^{1/2} - \frac{1 + O(u^{-1})}{2 \Fcost(n)}
\end{align*}
for large $n$,
and further calculation then leads to
\begin{align}
\Fcost(r) & = \frac12 \Fcost(n) \left(
   1 - \frac{D}{2} \left(\frac{D^2}{4} + 1\right)
      \frac{1 + O((\log \log n)^{-1/2})}{(\log \log n)^{1/2}} \right),
      \notag \\
\Fcost(m) & = \Fcost(n) \left(1 - \frac{D}{2} \cdot
   \frac{1 + O((\log \log n)^{-1/2})}{\Fcost(n)} \right).
\label{eq:Fcost-m}
\end{align}
Since $\Fcost(n)$ grows much faster than $(\log \log n)^{1/2}$,
we deduce that $\Fcost(m) > 2 \Fcost(r)$ for large $n$.
Also, from \eqref{eq:Fcost-m} we get $\Fcost(m) < \Fcost(n) - \tfrac12$
for large $n$.
This leaves us with
\[
\Rbarstar_\lambda(n) < \left(1 + \frac{O(1)}{(\log n)^{1/4}}\right) C
    (\Fcost(n) - \tfrac14) \Gcost(\lambda) \Mstar(\lambda n).
\]
Since $\Fcost(n) = o((\log n)^{1/4})$,
we conclude that \eqref{eq:Rcost-lambda-induction} holds for $(n,\lambda)$.
By induction \eqref{eq:Rcost-lambda-induction} holds for all $(n,\lambda)$
in the desired range
(after increasing $C$ if necessary to make it hold for small $n$).
This completes the proof of \eqref{eq:Rcost-lambda-weaker-bound}.

To reach \eqref{eq:Rcost-lambda-bound},
i.e., to improve the $4^{\log^* \lambda}$ factor to $2^{\log^* \lambda}$,
one may proceed as in \Cref{rem:cyclotomic},
replacing the coefficient ring $\FFext = \FF_2[z]/h(z)$
by $\FF_2[z]/\phi_\alpha(z)$ for a suitable cyclotomic polynomial
$\phi_\alpha(z)$.
We do not have space to explain the details here.

\begin{rem}
\label{rem:RAM-restricted}
In the RAM model,
it is plausible that one can achieve the same bound as
\eqref{eq:restricted-advanced-complexity}
but with the $2^{\log^* n}$ factor deleted.
Indeed, each factor of $2$ arises from one of the re-encoding steps
(i.e., the small $n$ case),
but one should be able to switch to table lookups after $O(1)$ such steps.
This may yield a small improvement in \Cref{thm:main-heuristic}
for the RAM model (see \Cref{rem:main-heuristic-precise}).
\end{rem}

\section{Compression and decompression}
\label{sec:compression}

In this section we study the following problem.
We are given positive integers $T$ and $R \leq T/2$.
For a vector $a \in \FF_2^T$,
let $\wt(a)$ denote the \emph{weight} of $a$,
i.e.,
\[
\wt(a) \coloneqq \bigl| \{ 0 \leq i < T : a_i = 1 \} \bigr|.
\]
We want to construct a linear map
\[
\kappa \colon \FF_2^T \to \FF_2^S,
\]
for a suitable positive integer $S$,
with the following properties:
\begin{itemize}
\item
If $a \in \FF_2^T$ with $\wt(a) \leq R$,
then $a$ can be recovered unambiguously from $\kappa(a)$.
In other words, we want $\kappa$ to be injective when restricted
to the subset of vectors in $\FF_2^T$ of weight at most $R$.
\item
There should be efficient algorithms for computing $\kappa(a)$ from $a$
and recovering $a$ from $\kappa(a)$.
We call these algorithms respectively
\emph{compression} and \emph{decompression}.
\end{itemize}

A trivial solution to this problem is to take $S \coloneqq T$
and $\kappa$ to be the identity map.
Of course, to obtain the best possible ``compression ratio'',
we would prefer to make $S$ as small as possible.

How small can $S$ be?
The number of vectors in $\FF_2^T$ that we want to distinguish
is $n_{T,R} \coloneqq 1 + \binom{T}{1} + \cdots + \binom{T}{R}$,
so $S$ must be at least $\log_2(n_{T,R})$.
If $R$ is not too large relative to $T$,
say $R < T^c$ for some fixed $c \in (0, 1)$,
then a brief calculation shows that $\log n_{T,R} \asymp R \log T$.
Therefore $S$ must satisfy $S \gg R \log T$.
In this section we will see how to meet this bound up to a constant factor,
i.e., we will construct a map $\kappa \colon \FF_2^T \to \FF_2^S$
with the above properties, and such that $S \ll R \log T$.

\subsection{Definition of the compression map}
\label{sec:construct-kappa}

As above, let $T$ and $R \leq T/2$ be positive integers.
Let
\begin{equation}
\label{eq:defn-lambda}
\lambda \coloneqq \lg(T+1) \ll \log T.
\end{equation}
As in \Cref{sec:ffext},
we will represent the field $\FFext$ as $\FF_2[z]/h(z)$
for a fixed irreducible $h \in \FF_2[z]$ of degree $\lambda$.
We also fix a generator $\beta \in \FFext^*$, so that
\[
\ord \beta = 2^\lambda - 1 \geq T \geq 2R.
\]
By \Cref{lem:irreducible} and \Cref{lem:generator},
we may (deterministically) compute suitable $h(z)$ and $\beta$ in time
$O(\lambda^{O(1)} + (2^\lambda)^{1/4+o(1)}) = T^{1/4+o(1)} = o(T)$.

We now define $\kappa$ as follows.
Given a vector $a \in \FF_2^T$, consider the polynomial
\[
\phi_a(x) \coloneqq a_0 + a_1 x + \cdots + a_{T-1} x^{T-1} \in \FF_2[x]_T.
\]
Evaluating $\phi_a(x)$ at the points $\beta^j$ for $j = 0, \ldots, 2R-1$,
we obtain the vector
\[
\tilde\kappa(a) \coloneqq
   \bigl( \phi_a(\beta^j) \bigr)_{j=0}^{2R-1} \in (\FFext)^{2R}.
\]
Each copy of $\FFext = \FF_2[z]/h(z)$
is isomorphic to $\FF_2^\lambda$ via the map sending
$u_0 + u_1 z + \cdots + u_{\lambda-1} z^{\lambda-1}$ to
$(u_0, u_1, \ldots, u_{\lambda-1})$.
By concatenating these vectors for $j = 0, \ldots, 2R-1$
we obtain an isomorphism
$\psi \colon (\FFext)^{2R} \to \FF_2^S$ where
\begin{equation}
\label{eq:defn-S}
S \coloneqq 2 \lambda R \ll R \log T.
\end{equation}
Finally we put
\[
\kappa \colon \FF_2^T \to \FF_2^S,
\qquad \kappa(a) \coloneqq \psi(\tilde \kappa(a)).
\]

In the remainder of this section,
we explain how to efficiently evaluate $\kappa(a)$ and
recover $a$ from $\kappa(a)$.
The key injectivity property of $\kappa$ will be deduced as a byproduct
of the correctness of the decompression algorithm.

\begin{rem}
\label{rem:reed-solomon}
The above construction of $\kappa$ is inspired by the theory of
\emph{Reed--Solomon codes} \cite{RS-codes} (or \emph{BCH codes}).
In the Reed--Solomon situation,
given a message of length $k$ over $\FF_q$,
one regards it as a polynomial in $\FF_q[x]_k$,
evaluates the polynomial at $n > k$ points in $\FF_q$,
and transmits the $n$ function values to the intended recipient.
Since the number of evaluation points is larger than the input size,
this process introduces redundancy
and hence the potential to detect and possibly correct errors.
By contrast, in our situation we evaluate at \emph{fewer} points than
the input size ($2R \leq T$).
In general this loses information,
but we will see below that standard decoding methods for Reed--Solomon
codes can be adapted to obtain a decompression algorithm for $\kappa$
with the desired properties.
One might also view our setup as a discrete variant of
\emph{compressed sensing} \cite{FR-sensing}:
we are trying to recover a function from a limited number
of function values, given some sparsity constraint on the function.
\end{rem}

\subsection{The compression algorithm}
\label{sec:compression-algorithm}

The symbols $T$, $R$, $S$ and $\kappa \colon \FF_2^T \to \FF_2^S$
continue to have the same meanings as in \Cref{sec:construct-kappa}.
For the rest of \Cref{sec:compression} we fix a transform scheme
(see \Cref{sec:transform-schemes}) with associated cost functions
$\Kcost(n)$ and $\Tcost(n) \ll \Mcost(n)$.
For convenience we define
\[
\Cstar(n) \coloneqq \Mstar(n) + \Kstar(n), \qquad n \geq 1.
\]

\begin{thm}[Compression]
\label{thm:compression}
Given $a \in \FF_2^T$,
we may evaluate $\kappa(a) \in \FF_2^S$ in time
\begin{equation}
\label{eq:compression-bound}
O(T \log T \cdot \Cstar(T)) + R (\log T)^{2+o(1)},
\end{equation}
after a precomputation (depending only on $T$ and~$R$) of cost
\[
T (\log T)^{3+o(1)}.
\]
\end{thm}

\begin{rem}
It is natural to attempt to compute $\kappa(a)$
by directly applying \Cref{prop:geometric}
to evaluate $\phi_a(x)$ at the powers of $\beta$.
However, this approach requires first lifting $\phi_a(x)$
from $\FF_2[x]$ to $\FFext[x]$ and leads to the complexity bound
$O(\Mcost(\lambda T)) < \lambda T (\log \lambda T)^{1+o(1)}
< T (\log T)^{2+o(1)}$.
This is worse than the main term of \eqref{eq:compression-bound},
which is only $T (\log T)^{1+o(1)}$
(assuming \eqref{eq:ffmul} and \eqref{eq:Kcost-bound}).
Moreover, using this bound in the core algorithm would increase
the complexity of \Cref{thm:main-heuristic} to $N (\log \log N)^{2+o(1)}$.

To prove \Cref{thm:compression}
we will instead use a more elaborate method that improves the
$T (\log T)^{2+o(1)}$ bound to $T (\log T)^{1+o(1)}$
at the expense of introducing a secondary $R (\log T)^{2+o(1)}$ term.
This secondary term will turn out to be negligible
in the applications of \Cref{thm:compression} in later sections.
\end{rem}

\begin{rem}
\label{rem:improve-compression}
Even if someone finds a polynomial multiplication algorithm
achieving $\Mcost(n) \ll n \log n$, i.e., $\Mstar(n) \ll 1$
(see \Cref{rem:nlogn-conjecture}),
it is still unclear whether the main term in \Cref{thm:compression}
can be improved to $O(T \log T)$.
The problem is the $\Kstar(T)$ term.
This constitutes yet another barrier to removing the $(\log \log N)^{o(1)}$
factor in \Cref{thm:main-heuristic},
via the $N \log T \cdot \Cstar(T) \Kstar(N)$
term in \eqref{eq:core-auxiliary-bound}.
\end{rem}

\begin{proof}
As mentioned in \Cref{sec:construct-kappa},
we may compute suitable $h(z) \in \FF_2[z]$ and $\beta \in \FFext^*$
in time $o(T)$, which certainly fits within the precomputation cost bound.
We also remind the reader of the estimates
$\Mcost(n) < n (\log n)^{1+o(1)}$ and $\lambda \ll \log T$
(see \eqref{eq:ffmul} and \eqref{eq:defn-lambda}),
which we use frequently below without further comment.

Suppose that we are given as input $a \in \FF_2^T$.
We may immediately reinterpret $a$ as the polynomial
$\phi_a(x) = a_0 + \cdots + a_{T-1} x^{T-1} \in \FF_2[x]_T$.
Consider the polynomial
\[
u(x) \coloneqq \prod_{j=0}^{2R-1} (x - \beta^j) \in \FFext[x],
\]
which has degree~$2R$.
If we can compute
\[
\bar\phi_a \coloneqq \phi_a \bmod u \in \FFext[x]_{2R},
\]
then we obtain the desired values
$\phi_a(1), \ldots, \phi_a(\beta^{2R-1}) \in \FFext$
(and hence immediately $\kappa(a)$)
by applying \Cref{prop:geometric} to $\bar\phi_a(x)$,
which costs
\begin{align}
\label{eq:bar-phi-a-bluestein}
O(\Mcost(\lambda R) + \Mcost(\lambda) \log \lambda)
   & < \lambda R (\log \lambda R)^{1+o(1)} + \lambda^{1+o(1)} \\
   & < R (\log T)^{2+o(1)}. \notag
\end{align}
Our task is therefore to show how to directly compute $\bar\phi_a(x)$
from $\phi_a(x)$ in time bounded by \eqref{eq:compression-bound}.

We begin by precomputing~$u(x)$.
Using \Cref{prop:ffext-arithmetic}(b)
to compute the powers of $\beta$,
and then applying \Cref{prop:product-tree},
this requires time
\[
O( R \Mcost(\lambda) + \Mcost(\lambda R) \log R ) < R (\log T)^{3+o(1)}.
\]

Let us split up the input polynomial $\phi_a \in \FF_2[x]_T$ into
$m$ chunks of size $2R$ for
\[
m \coloneqq \lceil T/2R \rceil \ll T/R,
\]
say
\begin{equation}
\label{eq:phi-a-split}
\phi_a(x) = f_0(x) + x^{2R} f_1(x) + \cdots + x^{2(m-1)R} f_{m-1}(x),
\end{equation}
where $f_i \in \FF_2[x]_{2R}$ for $i = 0, \ldots, m-1$.
Define
\[
v_i(x) \coloneqq x^{2Ri} \bmod u(x) \in \FFext[x]_{2R},
         \qquad i = 0, \ldots, m-1,
\]
so that
\[
\phi_a \equiv v_0 f_0 + v_1 f_1 + \cdots + v_{m-1} f_{m-1} \pmod u.
\]
Each $v_i(x)$ may be obtained from the previous one by multiplying
by $x^{2R}$ and dividing by $u(x)$.
By \Cref{prop:ffextdiv} each such multiplication and division
costs $O(\Mcost(\lambda R)) < R (\log T)^{2+o(1)}$,
so the cost of precomputing all of the $v_i$ is
\[
m R (\log T)^{2+o(1)} = T (\log T)^{2+o(1)}.
\]
Recalling that $\FFext = \FF_2[z]/h(z)$,
let us further decompose each $v_i \in \FFext[x]_{2R}$ as
\[
v_i(x) = v_{i,0}(x) + v_{i,1}(x) z + \cdots
   + v_{i,\lambda-1}(x) z^{\lambda-1},
   \qquad v_{i,j} \in \FF_2[x]_{2R}.
\]
(Extracting the $v_{i,j}$ from $v_i$
involves an array transposition of size $2R \times \lambda$ for each~$i$;
by \Cref{lem:transpose},
this contributes $O(m R \lambda \log \lambda) < T (\log T)^{1+o(1)}$
to the precomputation cost.)
Then we obtain
\begin{equation}
\label{eq:phi-a-mod-u}
\phi_a \equiv \sum_{j=0}^{\lambda-1}
   \bigl( v_{0,j} f_0 + \cdots + v_{m-1,j} f_{m-1} \bigr) z^j \pmod u.
\end{equation}

Our plan is now to use a transform pair of order $2R$ over $\FF_2$
to evaluate the sums of the products $v_{i,j} f_i$,
which live entirely in $\FF_2[x]$.
Let
\[
\fwd \colon \FF_2[x]_{2R} \to \FF_2^K,
\qquad
\inv \colon \FF_2^K \to \FF_2[x]_{4R-1}
\]
be the corresponding transform maps.
The cost of evaluating $\fwd$ or $\inv$ is $\Tcost(2R) \ll \Mcost(R)$,
and for $K$ we have the bound
\[
K \leq \Kcost(2R) = 2R \cdot \Kstar(2R) \leq 2R \cdot \Kstar(T).
\]

We may precompute $\fwd(v_{i,j}) \in \FF_2^K$ for all
$i = 0, \ldots, m-1$ and $j = 0, \ldots, \lambda-1$ in time
\[
O(m \lambda \Mcost(R)) < T (\log T)^{2+o(1)}.
\]
Returning to the given input polynomial $\phi_a(x)$
and the decomposition \eqref{eq:phi-a-split},
we may compute $\fwd(f_i) \in \FF_2^K$
for all $i = 0, \ldots, m-1$ in time
\[
O(m \Mcost(R)) = O(\Mcost(T)) = O(T \log T \cdot \Mstar(T)).
\]
We may then compute all of the pointwise sums
\[
\sum_{i=0}^{m-1} \fwd(v_{i,j}) \cdot \fwd(f_i) \in \FF_2^K,
   \qquad j = 0, \ldots, \lambda-1
\]
in time
\[
O(m \lambda K) = O(m \lambda R \cdot \Kstar(T))
   = O(T \log T \cdot \Kstar(T)).
\]
Applying the inverse transform $\inv$ once for each~$j$,
we deduce the polynomials
\[
v_{0,j} f_0 + \cdots + v_{m-1,j} f_{m-1} \in \FF_2[x]_{4R-1},
   \qquad j = 0, \ldots, \lambda-1
\]
in time
\[
O(\lambda \Mcost(R)) < R (\log T)^{2+o(1)}.
\]
Assembling these polynomials into
\[
\sum_{j=0}^{\lambda-1}
   \bigl( v_{0,j} f_0 + \cdots + v_{m-1,j} f_{m-1} \bigr) z^j
   \in \FFext[x]_{4R-1}
\]
involves an array transposition of size $(4R-1) \times \lambda$,
which costs $R (\log T)^{1+o(1)}$ as before.
Finally, we perform a division by $u(x)$ in $\FFext[x]$
to obtain $\bar\phi_a(x)$ as in \eqref{eq:phi-a-mod-u};
by \Cref{prop:ffextdiv},
the cost is the same as \eqref{eq:bar-phi-a-bluestein}.
\end{proof}

\subsection{The decompression algorithm}
\label{sec:decompression-algorithm}

We turn now to decompression.
As mentioned in \Cref{rem:reed-solomon},
our approach to this problem closely follows
standard decoding methods for Reed--Solomon codes.

\begin{thm}[Decompression]
\label{thm:decompression}
There is an algorithm with the following properties.
It takes as input a vector $b \in \FF_2^S$.
If there exists a vector $a \in \FF_2^T$
such that $\kappa(a) = b$ and $\wt(a) \leq R$,
then such $a$ is unique, and the algorithm returns~$a$.
Otherwise it returns ``FAIL''.
Its running time is
\[
O(T \log T \cdot \Cstar(T)) + R (\log T)^{5+o(1)},
\]
after a precomputation (depending only on $T$ and $R$) of cost
\[
T (\log T)^{3+o(1)}.
\]
\end{thm}
\begin{proof}
We are given as input $b \in \FF_2^S$.
We may immediately reinterpret $b$ as the vector
\[
c = (c_0, \ldots, c_{2R-1}) \coloneqq \psi^{-1}(b) \in (\FFext)^{2R}.
\]
Before describing the algorithm,
we will first show that if a suitable $a \in \FF_2^T$ exists,
i.e., with $\kappa(a) = b$ and $\wt(a) \leq R$,
then there exists a pair of polynomials $f, g \in \FFext[y]$
satisfying the hypotheses of \Cref{prop:reconstruction}
(with $n \coloneqq R$)
for the input
\[
c(y) \coloneqq c_0 + c_1 y + \cdots + c_{2R-1} y^{2R-1} \in \FFext[y]_{2R}.
\]

Let $w \coloneqq \wt(a)$,
so that $\phi_a(x) = \sum_{i=1}^w x^{s_i}$ for some sequence
$0 \leq s_1 < \cdots < s_w < T$.
(The case~$w = 0$, i.e., $\phi_a = 0$, is allowed.)
By definition $c_j = \phi_a(\beta^j)$ for $j = 0, \ldots, 2R-1$.
Let $\gamma_i \coloneqq \beta^{s_i} \in \FFext^*$ for each $i$;
then we have
\[
c_j = \sum_{i=1}^w \beta^{j s_i} = \sum_{i=1}^w \gamma_i^j,
      \qquad 0 \leq j < 2R.
\]
The latter power sums satisfy
certain linear relations over~$\FFext$.
Namely, define
\[
g(y) \coloneqq \prod_{i=1}^w (y - \gamma_i^{-1})
   = y^w + g_{w-1} y^{w-1} + \cdots + g_0 \in \FFext[y].
\]
(We take $g = 1$ if $w = 0$.)
Then since $g(\gamma_i^{-1}) = 0$ for all~$i$, we have
\[
0 = \sum_{i=1}^w \gamma_i^k g(\gamma_i^{-1})
      = \sum_{i=1}^w (\gamma_i^{k-w} + g_{w-1} \gamma_i^{k-w+1}
         + \cdots + g_0 \gamma_i^k)
\]
for any $k \in \ZZ$.
In particular,
\[
c_{k-w} + g_{w-1} c_{k-w+1} + \cdots + g_0 c_k = 0,
   \qquad w \leq k < 2R.
\]
Thus the coefficient of $y^k$ in $g(y) c(y)$ is zero for $w \leq k < 2R$,
or in other words,
\[
g(y) c(y) \equiv f(y) \pmod{y^{2R}}
\]
for some $f \in \FFext[y]$ with $\deg f < w$.
We note also that $y \ndivides g(y)$
(since $\gamma_i^{-1} \neq 0$ for all~$i$),
so $g(y)$ is invertible modulo~$y^{2R}$.

We claim that $\gcd(f,g) = 1$.
If this does not hold, then dividing out $f$ and $g$ by the gcd yields
\[
g^*(y) c(y) \equiv f^*(y) \pmod{y^{2R}}
\]
for polynomials $g^*, f^* \in \FFext[y]$
with $\deg f^* < \deg g^* < w$.
A similar argument to the previous paragraph then implies that
$\sum_{i=1}^w \gamma_i^k g^*(\gamma_i^{-1}) = 0$
for $\deg g^* \leq k < 2R$,
and in particular, for at least $2R - \deg g^* \geq 2R - w + 1 \geq w$
consecutive values of~$k$.
The corresponding $w \times w$ Vandermonde determinant is nonzero
(since the $\gamma_i$ are distinct),
so we deduce that $g^*(\gamma_i^{-1}) = 0$ for all $i = 1, \ldots, w$.
But this is impossible as $\deg g^* < w$.

We conclude that $f$ and $g$ satisfy the hypotheses of
\Cref{prop:reconstruction}.
This leads to the following algorithm.

\medskip
\step{1}{recover characteristic polynomial.}
We first attempt to recover $g(y)$ via \Cref{prop:reconstruction};
this requires time
\[
O(\Mcost(\lambda R) \log \lambda R) < R (\log T)^{3+o(1)}.
\]
If this fails then the required $a \in \FF_2^T$ cannot exist,
so we return ``FAIL''.
Otherwise we have found a candidate for~$g(y)$,
which must be correct if $a$ exists.
In any case, the candidate certainly has degree at most~$R$.

\medskip
\step{2}{find roots.}
We next invoke \Cref{prop:root-finding}
to find all roots of $g(y)$ in $\FFext$;
this requires time
\[
O(\lambda^2 \Mcost(\lambda R) \log \lambda R) < R (\log T)^{5+o(1)}.
\]
(If $\deg g = 0$, this step is trivial.)
If $g(y)$ does not split completely into linear factors,
then again the required $a \in \FF_2^T$ cannot exist
and we return ``FAIL''.
Otherwise we have found candidates for the roots~$\gamma_i^{-1}$,
which must be correct if $a$ exists.

\medskip
\step{3}{discrete logarithms.}
We must now compute the discrete logarithm of each
$\gamma_i^{-1} = \beta^{-s_i}$ to recover the indices~$s_i$.
To make this fast enough,
we precompute a table of pairs $\{(\beta^{-j}, j)\}_{j=0}^{T-1}$
sorted by the first component.
This precomputation involves one inversion and
$O(T)$ multiplications in~$\FFext$,
followed by sorting an array of length~$T$ with entries of bit size
$O(\lambda + \log T) = O(\log T)$;
by \Cref{prop:ffext-arithmetic} and \Cref{lem:sort} the cost is
\[
O\bigl((T + \log \lambda) \Mcost(\lambda)
   + (T \log T) \log T \bigr) = O(T \log^2 T).
\]

Now, given the candidates for the $\gamma_i^{-1}$ from Step 2,
we sort the list of candidates in time
$O((R \lambda) \log R) = O(R \log^2 T)$
and then perform a parallel linear search through the precomputed table
and the sorted candidate list to find the desired $s_i$
in time $O(T \log T)$.
Sorting them into the correct order ($s_1 < \cdots < s_w$)
additionally requires time $O((R \log T) \log R) = O(R \log^2 T)$.
If any of the $\gamma_i^{-1}$ candidates are not found, we return ``FAIL''.
Otherwise we have found candidates for the $s_i$,
which must be correct if $a$ exists.

\medskip
\step{4}{convert to vector.}
Using \Cref{lem:convert-format} we may compute a candidate for
$\phi_a(x) = \sum_{i=1}^w x^{s_i} \in \FF_2[x]_T$,
and hence for $a \in \FF_2^T$,
in time $O(T + R \log T) = O(T \log T)$.
This candidate must be correct if $a$ exists.

\medskip
\step{5}{verify result.}
Finally, we feed our candidate for~$a$ into \Cref{thm:compression}
at a cost of $O(T \log T \cdot \Cstar(T)) + R (\log T)^{2+o(1)}$.
If the output agrees with the original $b \in \FF_2^S$,
then the candidate must be correct and we return $a$.
Otherwise we return ``FAIL''.

This argument also establishes the uniqueness of~$a$,
because if a suitable vector $a \in \FF_2^T$ does exist,
then the above algorithm is guaranteed to find it.
\end{proof}

\section{The core algorithm}
\label{sec:core}

In this section we present a deterministic algorithm that attempts to
find all odd square-primes (see \Cref{defn:square-prime})
up to a given bound~$N$.
This algorithm forms the core of all three of the
main prime enumeration algorithms of the paper
(\Crefrange{thm:main-deterministic}{thm:main-heuristic}).
We target odd square-primes rather than primes because
this is what falls out naturally from \Cref{thm:H-congruence}
(see \Cref{rem:mobius}).

The main properties of the algorithm are summarised in
\Cref{thm:core} below.
The algorithm operates on intervals of a prescribed length $T \geq 1$.
We represent the odd square-primes in the $r$\th interval by a vector
$a^r \in \FF_2^T$, defined by
\begin{equation}
\label{eq:defn-ar}
(a^r)_t \coloneqq \begin{cases}
   1 & \textn{if $rT+t$ is an odd square-prime}, \\
   0 & \textn{otherwise},
\end{cases}
\qquad r \geq 0, \quad 0 \leq t < T.
\end{equation}
In particular, $\wt(a^r)$ counts the number of odd square-primes in
the $r$\th interval.

\begin{thm}[The core algorithm]
\label{thm:core}
There is an algorithm with the following properties.
It takes as input positive integers~$N$, $T$ and $R$
such that $T \divides N$ and $R \leq T/2$.
Its output is a list of candidates
\[
c^r \in \FF_2^T, \qquad 0 \leq r < N/T
\]
with $\wt(c^r) \leq R$.
For each candidate, if $\wt(a^r) \leq R$, then $c^r = a^r$.
Assuming that
\begin{equation}
\label{eq:logT-bound}
\log T < C \log \log N
\end{equation}
for some $C > 0$, the running time of the algorithm is
\begin{equation}
\label{eq:core-bound}
O_C\left(\left(1 + \frac{R \log N}{T} \right)
   N (\log \log N)^{1+o(1)} \right).
\end{equation}
\end{thm}

We emphasise that the theorem guarantees correctness
of the computed candidate $c^r$
only for those intervals containing at most $R$ odd square-primes.
If the $r$\th interval contains more than $R$ odd square-primes,
then the candidate $c^r$ will certainly be incorrect,
but the algorithm gives no indication of this.
The difference between the heuristic, probabilistic and deterministic
algorithms is in how we detect and repair these exceptional intervals.
These correction strategies will be discussed in later sections.

The first step of the algorithm is to select two auxiliary parameters
as a function of $N$, $T$ and~$R$:
a block size~$M$, and a slicing parameter~$W$
(both positive integers).
The following theorem expresses the complexity of the algorithm
in terms of all five parameters $N$, $T$, $R$, $M$ and~$W$,
i.e., treating $M$ and $W$ as additional inputs.
\begin{thm}
\label{thm:core-auxiliary}
There is an algorithm with the following properties.
It takes as input positive integers $N$, $T$, $R$, $M$ and~$W$ such that
\begin{equation}
\label{eq:core-auxiliary-hypotheses}
T \divides W, \qquad W \divides M, \qquad M \divides N,
   \qquad R \leq T/2.
\end{equation}
Its output is a list of candidates
\[
c^r \in \FF_2^T, \qquad 0 \leq r < N/T
\]
with $\wt(c^r) \leq R$.
For each candidate, if $\wt(a^r) \leq R$, then $c^r = a^r$.
The running time of the algorithm is
\begin{multline}
\label{eq:core-auxiliary-bound}
O\biggl(
   N \tsp \frac{R \log T \log N}{T} \Mstar(N)
   + N \tsp \frac{\rho(W) \log N}{W} \Mstar(N)
   + (M + B) \log^3 N \\[3pt]
   + N \log W \cdot \Mstar(W) \Kstar(N)
   + N \log B \cdot \Rstar(B) \Kstar(W) \Kstar(N) \\[3pt]
   + \left(1 + \frac{R(\log T)^{4+o(1)}}{T}\right)
      N \log T \cdot \Cstar(T) \Kstar(N)
   + T (\log T)^{3+o(1)}
\biggr)
\end{multline}
where $B \coloneqq N/M$ and where the quantity $\rho(W)$ is defined
as in \eqref{eq:defn-rho}.
\end{thm}

\begin{rem}
In the probabilistic and deterministic algorithms
the overall complexity is driven by the first term of
\eqref{eq:core-auxiliary-bound},
which arises from the inverse transform step
(\Cref{prop:core-compute-kappa-ar}),
and the remaining terms are at most $N (\log \log N)^{1+o(1)}$
which is negligible.
In the heuristic algorithm we are able to squeeze the first term down to
$O(N \log \log N \cdot \Mstar(N)) < N (\log \log N)^{1+o(1)}$
and the other terms start to become significant.
\end{rem}

The rest of the section is structured as follows.
In \Crefrange{sec:core-prime-powers}{sec:core-inverse}
we prove \Cref{thm:core-auxiliary}.
The description of the algorithm is broken up into several steps:
\begin{itemize}
\item
\Cref{sec:core-prime-powers}:
reduce from primes to prime powers.
\item
\Cref{sec:core-blocks}:
reduce to products of blocks.
\item
\Cref{sec:core-slice}:
reduce to products of slices.
\item
\Cref{sec:core-forward}:
introduce a transform scheme and perform forward transforms.
\item
\Cref{sec:core-pointwise}:
evaluate pointwise sums.
\item
\Cref{sec:core-compression}:
apply compression.
\item
\Cref{sec:core-inverse}:
perform inverse transforms and decompression.
\end{itemize}
Then in \Cref{sec:core-main} we show how to deduce
\Cref{thm:core} from \Cref{thm:core-auxiliary}
via a suitable choice of $M$ and~$W$
(and after possibly adjusting $N$ upwards slightly to satisfy
the divisibility constraints in \eqref{eq:core-auxiliary-hypotheses}).

\begin{rem}
The reader may wish to keep in mind the following notational convention
that is used throughout this section:
parameters represented by capital letters close to~$N$ in the alphabet
($M$, $L$,~$K$)
are intended to have the same ``level of exponentiality'' as~$N$,
whereas more distant letters ($T$, $S$, $R$, $W$,~$B$)
represent quantities that are exponentially smaller.
\end{rem}

\begin{table}

{
\Crefname{thm}{Thm.}{}%
\Crefname{defn}{Defn.}{}%
\Crefname{section}{Sec.}{}%

\begin{tabular}{l@{\hspace{13pt}}l@{\hspace{13pt}}l}
\toprule
$N$ &
target prime bound &
\Cref{thm:core} \\

$T$, $R$ &
interval length, weight threshold &
\Cref{sec:compression} \& \Cref{thm:core} \\

$S$ &
length of compressed vectors &
\eqref{eq:defn-S} \& \Cref{sec:core-compression} \\

$M$, $W$ &
block size, number of slices &
\Cref{thm:core-auxiliary} \\

$B$, $L$ &
number of blocks, length of slices &
\eqref{eq:defn-B}, \eqref{eq:defn-L} \\

$K$, $\fwd$, $\inv$ &
transform size, forward/inverse transforms &
\eqref{eq:core-transforms} \\

$\Aset_d(h)$ &
index set for ordinary/restricted products &
\eqref{eq:defn-Aset} \\

$\rho(W)$ &
bound for number of nonzero slices &
\eqref{eq:defn-rho} \\

\midrule

$F$, $G_d$ &
main series multiplicands &
\Cref{defn:series} \\

$A$ &
generating function for odd square-primes &
\eqref{eq:defn-A} \\

$E_1$, $E_2$, $E$ &
error series &
\eqref{eq:defn-E1}, \eqref{eq:defn-E2}, \eqref{eq:defn-E} \\

$a^r$, $e^r$ &
target vectors, error vectors &
\eqref{eq:defn-ar}, \eqref{eq:defn-er} \\

$\hat a^r$, $\hat e^r$ &
compressed vectors &
\eqref{eq:defn-hat-ar-er} \\

$F^i$, $G_d^j$ &
blocks of $F$ and $G_d$ &
\eqref{eq:F-blocks}, \eqref{eq:Gj-blocks}, \eqref{eq:G2-blocks} \\

$A^h$, $E^h$ &
blocks of $A$ and $E$ &
\eqref{eq:AE-blocks} \\

$F^{i,u}$, $G_d^{j,v}$ &
slices of $F$ and $G_d$ &
\eqref{eq:AEFG-slices} \\

$E^{h,w}$, $A^{h,w}$ &
slices of $A$ and $E$ &
\eqref{eq:AEFG-slices} \\

\midrule

$\inv^{\sigma,\tau}$ &
$(\sigma,\tau)$-specialised inverse transforms &
\eqref {eq:inv-sigmatau} \\

$Z_{h,w}$ &
pointwise products of transforms &
\eqref{eq:defn-Z} \\

$Z_{h,w}^{\sigma,\tau}$ &
$(\sigma,\tau)$-specialisation of $Z_{h,w}$ &
\eqref{eq:defn-Z-sigmatau} \\

$Z^{\sigma,\tau}_{h,m,t}$ &
reindexed version of $Z_{h,w}^{\sigma,\tau}$ &
\eqref{eq:Z-sigmatau-reindexed} \\

$\hat Z_{h,m,s}^{\sigma,\tau}$ &
compressed version of $Z_{h,m,t}^{\sigma,\tau}$ &
\eqref{eq:compressed-Z} \\

\bottomrule
\end{tabular}
}   
\caption{Glossary of selected notation in \Cref{sec:core}}
\end{table}

\subsection{Reduction from primes to prime powers}
\label{sec:core-prime-powers}

We now embark on the proof of \Cref{thm:core-auxiliary}.
From this point until the end of \Cref{sec:core-inverse},
we assume that we are given parameters $N$, $T$, $R$, $M$ and $W$
satisfying \eqref{eq:core-auxiliary-hypotheses}.
Our goal is to compute vectors $c^r \in \FF_2^T$ for
$0 \leq r < N/T$ with the properties specified in
\Cref{thm:core-auxiliary}.

The target vectors $a^r$ may be expressed as
\[
(a^r)_t = A_{rT + t},
\qquad r \geq 0, \quad 0 \leq t < T
\]
where $A \in \FF_2\bbracket{x}$ is the generating function
\begin{equation}
\label{eq:defn-A}
A(x) \coloneqq
   \sum_{\substack{m \geq 1 \\ \textn{$m$ odd}}} \;
   \sum_{p \geq 3}
   x^{m^2 p}
   \in \FF_2\bbracket{x}.
\end{equation}
Our goal is therefore to compute (candidates for)
the first $N$ coefficients of $A(x)$.
However, since our main tool is \Cref{thm:H-congruence},
which involves prime powers instead of primes,
we will work instead with the modified series
\[
A'(x) \coloneqq
   \sum_{\substack{m \geq 1 \\ \textn{$m$ odd}}} \;
   \sum_{\substack{p \geq 3 \\ l \geq 1}}
   x^{m^2 p^l}
   \in \FF_2\bbracket{x}.
\]
Let
\begin{equation}
\label{eq:defn-E1}
E_1(x) \coloneqq A(x) - A'(x) \in \FF_2\bbracket{x}
\end{equation}
be the corresponding error.

\begin{prop}
\label{prop:compute-E1}
We may compute $E_1(x) \pmod{x^N}$ in time $O(N)$.
\end{prop}
\begin{proof}
We have
\[
E_1(x) =
   \sum_{\substack{m \geq 1 \\ \textn{$m$ odd}}} \;
   \sum_{\substack{p \geq 3 \\ l \geq 2}}
   x^{m^2 p^l}.
\]
To compute $E_1(x) \pmod{x^N}$,
we must find all pairs $(m,p^l)$
with $m \geq 1$ odd, $p \geq 3$ prime, and $l \geq 2$,
such that $m^2 p^l < N$.
In all these pairs we clearly have $l < \log_3 N$ and $m < N^{1/2}$.
For fixed $l$ and~$m$, the number of possible $p$ is at most
$(N/m^2)^{1/l} \leq (N/m^2)^{1/2} = N^{1/2}/m$,
so the total number of pairs is at most
\[
\sum_{\substack{2 \leq l < \log_3 N \\ 1 \leq m < N^{1/2}}}
   \frac{N^{1/2}}{m}
   \ll N^{1/2} \log^2 N.
\]
We may enumerate all such pairs in time $N^{1/2+o(1)}$,
and then compute the exponents $m^2 p^l$, sort them,
and remove pairs of adjacent duplicates (since we are working over $\FF_2$)
in time $N^{1/2+o(1)}$.
We then recover $E_1(x) \pmod{x^N}$ via \Cref{lem:convert-format} in time
$O(N + N^{1/2+o(1)} \log N) = O(N)$.
\end{proof}

\subsection{Reduction to products of blocks}
\label{sec:core-blocks}

Recall the series $F, G_d, H_d \in \ZZ\bbracket{x^{\pm1}}$
for $d \in \dset = \{-1, -2, 2\}$
(see \Cref{defn:series}).
For the rest of this section we only care about these series modulo $2$,
so we regard them as lying in $\FF_2\bbracket{x^{\pm1}}$.
\Cref{thm:H-congruence} then implies that
\begin{equation}
\label{eq:Aprime-sum}
A'(x) = H_{-1}(x) + H_{-2}(x) + H_2(x).
\end{equation}
The goal of this section is to express each $H_d(x)$,
and hence $A'(x)$ and $A(x)$,
in terms of products of ``blocks'' of size~$M$.

\begin{rem}
The reason for splitting into blocks is to reduce
to ordinary polynomial multiplication,
which can be tackled by FFT methods.
In particular, for the $d = 2$ case,
the author does not know how to
compute restricted products directly using FFTs,
so the block decomposition enables us to handle this case efficiently.
Strictly speaking the block decomposition is unnecessary for $d = -1, -2$,
since in these cases $H_d(x)$ already has the form of a polynomial product.
For uniformity of presentation we will use the block decomposition
for all three values of~$d$.
(In the sketch of the core algorithm given in \Cref{sec:overview},
the block decomposition was omitted in order to simplify the discussion.)
\end{rem}

The cases $d = -1, -2$ are entirely straightforward.
Recall that in these cases $H_d = F \cdot G_d$.
We split up $F(x)$ and $G_d(x)$ into blocks, writing
\begin{align}
\label{eq:F-blocks}
F(x)   & = \sum_{i \geq 0} x^{iM} F^i(x),
\\
\label{eq:Gj-blocks}
G_d(x) & = \sum_{j \geq 0} x^{jM} G_d^j(x), \qquad d = -1,-2,
\end{align}
where $F^i, G_d^j \in \FF_2[x]_M$.
Then $H_d(x) = F(x) \cdot G_d(x)$ becomes
\begin{equation}
\label{eq:Hd-blocks}
H_d(x) = \sum_{i, j \geq 0} x^{(i+j)M} F^i(x) G_d^j(x), \qquad d = -1,-2.
\end{equation}

We want to prove a formula analogous to \eqref{eq:Hd-blocks}
for the case $d = 2$.
However, the restricted product $H_2 = F \restrict G_2$
does not decompose into blocks in such a simple way.
Instead we must apply \Cref{lem:restricted-decomposition},
which tells us that
\begin{equation}
\label{eq:H2-blocks}
H_2(x) = (F \restrict G_2)(x) = F G_{2,0} + E_2(x)
   + \sum_{\substack{j \geq 1 \\ i \geq 2j}} x^{(i-j)M} F^i(x) G_2^j(x)
\end{equation}
where
\begin{alignat}{2}
\label{eq:defn-E2}
E_2(x) & \coloneqq
   \sum_{k \geq 0} x^{kM} \bigl( F_*^k \restrict (G_2)_*^k \bigr)
   & & \in \FF_2\bbracket{x}, \\
F_*^k(x) & \coloneqq \notag
   \sum_{0 \leq u < 2M} F_{2kM+u} x^u & & \in \FF_2[x]_{2M}, \\
(G_2)_*^k(x) & \coloneqq \notag
   \sum_{1 \leq v < M} G_{2,kM+v} x^{-v} & & \in \FF_2[x^{-1}]_M,
\end{alignat}
and
\begin{equation}
\label{eq:G2-blocks}
G_2(x) = \sum_{j \geq 1} x^{-jM} G_2^j(x), \qquad G_2^j \in \FF_2[x]_M.
\end{equation}
The first term in \eqref{eq:H2-blocks} vanishes as $G_{2,0} = 0$
(see \eqref{eq:defn-G2}).
The third term is analogous to \eqref{eq:Hd-blocks}.
The second term is handled by means of the following proposition.
Here and below, we will write
\begin{equation}
\label{eq:defn-B}
B \coloneqq \frac{N}{M}
\end{equation}
for the number of blocks up to~$x^N$;
this is a positive integer thanks to \eqref{eq:core-auxiliary-hypotheses}.
\begin{prop}
\label{prop:compute-E2}
We may compute $E_2(x) \pmod{x^N}$ in time
\begin{equation}
\label{eq:E2-complexity}
O(N + (M + B) \log^3 N).
\end{equation}
\end{prop}
\begin{rem}
In \Cref{sec:core-main} we will choose $M \asymp N/\log^3 N$
(and hence $B = N/M \asymp \log^3 N$)
so that \eqref{eq:E2-complexity} simplifies to $O(N)$.
\end{rem}
\begin{proof}
The only values of $k$ in \eqref{eq:defn-E2}
that contribute to $E_2(x) \pmod{x^N}$ are $k = 0, \ldots, B-1$.
We will compute each of these $F_*^k \restrict (G_2)_*^k$ separately
and sum the results.
One could use the algorithms from \Cref{sec:restricted-products}
to compute these restricted products,
but this turns out to be too slow.
To obtain a fast enough algorithm,
we must take advantage of the sparse structure of $F(x)$ and $G_2(x)$.

We first find a sparse representation for $F_*^k$.
Recall from \eqref{eq:defn-F} that $F_i$ is~$1$ if
$i = a^2$ for some odd $a \geq 1$, and is otherwise zero.
The number of integers in the interval $2kM \leq i < 2(k+1)M$
of the form $i = a^2$ ($a$ odd) is at most
\[
Q_k \coloneqq \sqrt{2(k+1)M} - \sqrt{2kM} + 1.
\]
We may enumerate these integers using at most $O(Q_k)$ word operations,
where by ``word operation'' we mean an addition, subtraction, comparison,
multiplication, or square root (rounding downwards)
of integers with $O(\log N)$ bits.
In this way we obtain a list of integers $u = i - 2kM$ corresponding to
terms $x^u$ appearing in $F_*^k$.

Similarly, from \eqref{eq:defn-G2} we see that $G_{2,j}$ is~$1$
when $j = 2b^2$ for some $b \geq 1$, and is otherwise zero.
The number of integers in the interval $kM + 1 \leq j < (k+1)M$
of the form $j = 2b^2$ is at most
\[
Q'_k \coloneqq \sqrt{(k+1)M/2} - \sqrt{kM/2} + 1.
\]
Again we may enumerate these integers,
and thus obtain a list of integers $v = j - kM$ corresponding to
terms $x^{-v}$ in $(G_2)_*^k$,
in at most $O(Q'_k)$ word operations.

Assuming classical algorithms for integer arithmetic,
the total work so far is $O((Q_k + Q'_k) \log^2 N)$ bit operations.

Next, for each pair $(u,v)$, we check whether $u \geq 2v$,
and if so append $u-v$ to a list.
This costs $O(P_k \log N)$ bit operations where $P_k \coloneqq Q_k Q'_k$,
and produces a list of integers $w$ corresponding to terms $x^w$
in $F_*^k \restrict (G_2)_*^k \in \FF_2[x]_{2M}$.
We then sort the list in time $O(P_k \log P_k \log N) = O(P_k \log^2 N)$
(\Cref{lem:sort}),
and remove pairs of adjacent duplicates (since we are working modulo~$2$)
in time $O(P_k \log N)$.
Finally we use \Cref{lem:convert-format} to convert the
sorted list of (distinct) exponents to a polynomial in $\FF_2[x]_{2M}$
in time $O(M + P_k \log M) = O(M + P_k \log N)$.

Combining all the contributions above,
we conclude that the cost of computing $F_*^k \restrict (G_2)_*^k$
for a single~$k$ is $O(M + P_k \log^2 N)$ bit operations.
Since
\[
Q_k, Q'_k \ll \frac{\sqrt{M}}{\sqrt{k+1}} + 1,
\]
we have
\begin{multline*}
\sum_{0 \leq k < B} P_k
   \ll \sum_{0 \leq k < B} \left(\frac{\sqrt{M}}{\sqrt{k+1}} + 1\right)^2
   \ll \sum_{0 \leq k < B}
      \left(\frac{M}{k+1} + \frac{\sqrt{M}}{\sqrt{k+1}} + 1 \right) \\
   \ll M \log B + \sqrt{MB} + B
   \ll M \log N + N^{1/2} + B.
\end{multline*}
The total cost over all~$k$ is thus
\[
O(BM + (M \log N + N^{1/2} + B) \log^2 N)
   = O(N + (M+B)\log^3 N).
\]
Finally, the cost of adding up the terms in \eqref{eq:defn-E2} is $O(N)$.
\end{proof}

Let us write
\begin{equation}
\label{eq:defn-E}
E(x) \coloneqq E_1(x) + E_2(x) \in \FF_2\bbracket{x}
\end{equation}
for the sum of the two error series.
Combining \eqref{eq:defn-E1}, \eqref{eq:Aprime-sum},
\eqref{eq:Hd-blocks} and \eqref{eq:H2-blocks},
and defining for convenience $G_2^0 \coloneqq 0$,
we obtain the desired expression for $A(x)$ in terms of products of blocks:
\begin{equation}
\label{eq:A-blocks-1}
A(x) = E(x)
   + \sum_{d =-1,-2} \;
   \sum_{\substack{j \geq 0 \\ i \geq 0}} x^{(i+j)M} F^i G_d^j
   + \sum_{\substack{j \geq 0 \\ i \geq 2j}} x^{(i-j)M} F^i G_2^j.
\end{equation}

Next we split up $A(x)$ and $E(x)$ into blocks, say
\begin{equation}
\label{eq:AE-blocks}
A(x) = \sum_{h \geq 0} x^{hM} A^h(x), \qquad
E(x) = \sum_{h \geq 0} x^{hM} E^h(x),
\end{equation}
where each $A^h, E^h \in \FF_2[x]_M$.
We want to use \eqref{eq:A-blocks-1} to obtain a formula for
each block $A^h$ in terms of the blocks of $F$, $G_d$ and~$E$.
To this end we introduce the following notation:
if $f \in \FF_2[x]$ is a polynomial and $t_1 < t_2$ are integers,
we write
\[
f_{[t_1,t_2]} \coloneqq \sum_{t_1 \leq t < t_2} f_t x^{t-t_1}
   \in \FF_2[x]_{t_2-t_1}.
\]
First consider the cases $d = -1, -2$ in \eqref{eq:A-blocks-1}.
The term $x^{(i+j)M} F^i G_d^j$ contributes $(F^i G_d^j)_{[0,M]}$
(the bottom half of $F^i G_d^j$) to $A^{i+j}(x)$
and $(F^i G_d^j)_{[M,2M]}$ (the top half) to $A^{i+j+1}(x)$.
Thus the contribution to $A^h(x)$ for a given $h \geq 0$ is
\[
\sum_{\substack{i, j \geq 0 \\ i+j=h}} (F^i G_d^j)_{[0,M]} +
\sum_{\substack{i, j \geq 0 \\ i+j=h-1}} (F^i G_d^j)_{[M,2M]}.
\]
Similarly, the contribution to $A^h(x)$
from the terms $x^{(i-j)M} F^i G_2^j$ in \eqref{eq:A-blocks-1} is
\[
\sum_{\substack{j \geq 0, \, i \geq 2j \\ i-j=h}} (F^i G_2^j)_{[0, M]} +
\sum_{\substack{j \geq 0, \, i \geq 2j \\ i-j=h-1}} (F^i G_2^j)_{[M, 2M]}.
\]
We may unify the two cases by defining index sets
\begin{equation}
\label{eq:defn-Aset}
\Aset_d(h) \coloneqq \begin{cases}
   \{(i,j) \in \ZZ^2 : i, j \geq 0, \hspace{26pt} i+j = h\}, & d = -1, -2, \\
   \{(i,j) \in \ZZ^2 : j \geq 0, \, i \geq 2j, \, i-j = h\}, & d = 2.
   \end{cases}
\end{equation}
With this notation \eqref{eq:A-blocks-1} becomes
\begin{equation}
\label{eq:A-blocks-2}
A^h(x) = E^h(x) +
   \sum_{\sigma=0,1} \;
   \sum_{\substack{d \in \dset \\ (i,j) \in \Aset_d(h-\sigma)}}
      (F^i G_d^j)_{[\sigma M, (\sigma+1)M]}, \qquad h \geq 0.
\end{equation}

\subsection{Reduction to products of slices}
\label{sec:core-slice}

We now introduce the slicing parameter $W \geq 1$.
Recall from \eqref{eq:core-auxiliary-hypotheses} that
$T \divides W$ and $W \divides M$.
For any polynomial $f \in \FF_2[x]$ and $0 \leq w < W$ define
\[
f_{\bangle{w}}(y) \coloneqq
   \sum_{\ell \geq 0} f_{\ell W + w} \, y^\ell \in \FF_2[y].
\]
In other words, $f_{\bangle{w}}$ extracts the monomials of $f(x)$
whose exponents are congruent to $w$ modulo~$W$.
We call $f_{\bangle{w}}(y)$ a ``slice'' of $f(x)$.
Notice that if $f \in \FF_2[x]_M$ then $f_{\bangle{w}} \in \FF_2[y]_L$,
where we write
\begin{equation}
\label{eq:defn-L}
L \coloneqq \frac{M}{W}
\end{equation}
for the length of each slice.
The identity
\[
BLW = N,
\]
which follows from \eqref{eq:defn-B} and \eqref{eq:defn-L},
will be used frequently below.
The goal of this section is to rewrite \eqref{eq:A-blocks-2}
in terms of slices,
i.e., we want to express each slice of $A^h(x)$
in terms of products of slices of $F^i(x)$ and $G_d^j(x)$.

\begin{rem}
\label{rem:core-slices}
There are two reasons for introducing slices.
First, slices are critical for the compression scheme
employed in \Cref{sec:core-compression}.
Second, many of the slices of $F(x)$ and $G_d(x)$ are automatically zero,
thanks to the quadratic exponents in \eqref{eq:defn-F}--\eqref{eq:defn-G2}.
Roughly speaking, each prime factor of $W$ cuts out about half of the
congruence classes modulo~$W$.
To control the cost of the forward transform step
(\Cref{sec:core-forward}),
it will be essential to choose $W$ with many prime factors,
so that a large fraction of the forward transforms are automatically zero.
\end{rem}

\begin{rem}
\label{rem:varphi-W-speedup}
The exponents appearing in $H_d(x)$ are also subject to certain restrictions
involving congruences modulo squares of small primes.
For example, a square-prime $m^2 p$ can never be $10$ or $15$ modulo $25$,
and the cases $5$ and $20$ modulo $25$ can only occur when $p = 5$.
It may be possible to use this observation to achieve a small saving
in some parts of the core algorithm.
There is no effect on the probabilistic or deterministic variants,
because the observation does not appreciably reduce the amount of
``entropy'' recovered by the compression scheme,
so the dominant first term of \eqref{eq:core-auxiliary-bound} is unaffected.
But for a suitable choice of~$W$,
and after some reorganisation of the algorithm,
it is plausible that \eqref{eq:compute-Z}
could be improved by a factor of $\varphi(W)/W$,
where $\varphi(W)$ denotes the totient function.
This might translate to an overall speedup for the heuristic variant
by a factor of $\log \log W \asymp \log \log \log N$
(see \Cref{rem:main-heuristic-precise}).
The author has not checked the details.
\end{rem}

If $f, g \in \FF_2[x]_M$ and $0 \leq w < W$,
then the slice $(fg)_{\bangle{w}} \in \FF_2[y]_{2L}$
of the product $fg \in \FF_2[x]_{2M}$ is given by
\begin{equation}
\label{eq:slice-product}
(fg)_{\bangle{w}}
   = \sum_{\substack{0 \leq u, v < W \\ u+v=w}}
         f_{\bangle{u}} g_{\bangle{v}} +
      \sum_{\substack{0 \leq u, v < W \\ u+v=w+W}}
         y f_{\bangle{u}} g_{\bangle{v}}
      \enspace \in \FF_2[y]_{2L}.
\end{equation}
To prove this formula,
consider a typical term $f_i g_j x^{i+j}$ in the product $f(x)g(x)$
such that $i + j \equiv w \pmod W$.
We may write $i = \ell W + u$ and $j = mW + v$ with $0 \leq u, v < W$,
so that $u + v \equiv w \pmod W$.
The contribution of this term to $(fg)_{\bangle{w}}$ is
$f_i g_j y^{(i+j-w)/W} = f_i g_j y^{\ell+m + (u+v-w)/W}$.
On the other hand, its contribution to $f_{\bangle{u}} g_{\bangle{v}}$
is $(f_i y^\ell) (g_j y^m) = f_i g_j y^{\ell+m}$.
If $u+v < W$ then $u + v = w$, so these contributions are identical.
But if $u+v \geq W$ then $u + v = w + W$,
so we must introduce an extra factor of $y$ to make the contributions agree.
These two cases correspond to the two sums on the right hand side of
\eqref{eq:slice-product}.

Let us now use \eqref{eq:slice-product} to
find a formula for the slices of $A^h(x)$.
Define
\begin{equation}
\label{eq:AEFG-slices}
\begin{aligned}
  A^{h,w}(y) & \coloneqq (A^h)_{\bangle{w}}   \in \FF_2[y]_L, & \quad
  F^{i,u}(y) & \coloneqq (F^i)_{\bangle{u}}   \in \FF_2[y]_L, \\
  E^{h,w}(y) & \coloneqq (E^h)_{\bangle{w}}   \in \FF_2[y]_L, &
G_d^{j,v}(y) & \coloneqq (G_d^j)_{\bangle{v}} \in \FF_2[y]_L.
\end{aligned}
\end{equation}
Then \eqref{eq:slice-product} implies that
\[
(F^i G_d^j)_{\bangle{w}} =
   \sum_{\tau=0,1} \; \sum_{\substack{0 \leq u,v < W \\ u+v=w+\tau W}}
   y^\tau F^{i,u} G_d^{j,v} \enspace \in \FF_2[y]_{2L}.
\]
It is clear that
\[
\bigl( (F^i G_d^j)_{[\sigma M, (\sigma+1) M]} \bigr)_{\bangle{w}}
   = \bigl( (F^i G_d^j)_{\bangle{w}} \bigr)_{[\sigma L, (\sigma+1) L]},
\]
so \eqref{eq:A-blocks-2} becomes
\begin{multline}
\label{eq:A-slices}
A^{h,w}(y) = E^{h,w}(y) + \phantom{x} \\*
   \sum_{\substack{\sigma = 0, 1 \\ \tau = 0, 1}} \;
   \sum_{\substack{d \in \dset \\ (i,j) \in \Aset_d(h-\sigma)}} \;
   \sum_{\substack{0 \leq u, v < W \\ u+v=w+\tau W}}
      (F^{i,u} G_d^{j,v})_{[\sigma L - \tau, (\sigma+1) L - \tau]}
   \enspace \in \FF_2[y]_L.
\end{multline}
(When $\sigma = 0$ and $\tau = 1$,
the starting index $\sigma L - \tau$ in \eqref{eq:A-slices} is~$-1$.
In this context, the coefficient of $y^{-1}$ in $F^{i,u}(y) G_d^{j,v}(y)$
is understood to be zero.)

\subsection{Forward transforms}
\label{sec:core-forward}

To take advantage of the repeated multiplicands in \eqref{eq:A-slices},
we will re-express the sum in terms of
a transform pair of order~$L$ over~$\FF_2$
(see \Cref{sec:transform-schemes}).
According to \Cref{thm:transform-scheme}
we may construct forward and inverse transform maps
\begin{equation}
\label{eq:core-transforms}
\fwd \colon \FF_2[y]_L \to \FF_2^K, \qquad
\inv \colon \FF_2^K \to \FF_2[y]_{2L-1}
\end{equation}
which can both be evaluated in time
\begin{equation}
\label{eq:core-eval-cost}
\Tcost(L) \ll \Mcost(L) = L \log L \cdot \Mstar(L)
   \leq L \log N \cdot \Mstar(N),
\end{equation}
and where $K$ satisfies
\begin{equation}
\label{eq:core-K}
K \leq \Kcost(L) = L \cdot \Kstar(L) \leq L \cdot \Kstar(N).
\end{equation}
We may then rewrite \eqref{eq:A-slices} as follows.
For each $(\sigma, \tau) \in \{0,1\}^2$ define
\begin{equation}
\label{eq:inv-sigmatau}
\inv^{\sigma,\tau} \colon \FF_2^K \to \FF_2[y]_L, \qquad
\inv^{\sigma,\tau}(Z) \coloneqq
   (\inv(Z))_{[\sigma L - \tau, (\sigma+1)L - \tau]}.
\end{equation}
Clearly each $\inv^{\sigma,\tau}$ is an $\FF_2$-linear map.
Also set
\begin{equation}
\label{eq:defn-Z}
Z_{h,w} \coloneqq
   \sum_{\substack{d \in \dset \\ (i,j) \in \Aset_d(h)}} \;
   \sum_{\substack{0 \leq u, v < W \\ u+v=w}}
   \fwd(F^{i,u}) \cdot \fwd(G_d^{j,v})
   \enspace \in \FF_2^K,
\qquad
\begin{gathered}[t]
h \geq -1, \\
0 \leq w < 2W,
\end{gathered}
\end{equation}
where the dot indicates pointwise multiplication in $\FF_2^K$.
Then \eqref{eq:A-slices} becomes simply
\begin{equation}
\label{eq:A-slices-eval}
A^{h,w}(y) = E^{h,w}(y) +
\sum_{\substack{\sigma = 0, 1 \\ \tau = 0, 1}}
   \inv^{\sigma,\tau}(Z_{h - \sigma, w + \tau W})
   \; \in \FF_2[y]_L,
\qquad
\begin{gathered}[t]
h \geq 0, \\
0 \leq w < W.
\end{gathered}
\end{equation}
Note that $Z_{h,w}$ is automatically zero when either $h = -1$ or $w = 2W-1$;
we allow these parameter values in order to simplify
the statement of \eqref{eq:A-slices-eval}.

We next discuss how to evaluate the transforms in \eqref{eq:defn-Z}
that contribute to $A(x) \pmod{x^N}$,
i.e., affecting the slices $A^{h,w} \in \FF_2[y]_L$ for $0 \leq h < B$.
It is clear from \eqref{eq:A-slices-eval} that to obtain these $A^{h,w}$,
it suffices to compute the $Z_{h,w}$ for $0 \leq h < B$.
The following result indicates which transforms in \eqref{eq:defn-Z}
contribute to these $Z_{h,w}$.
\begin{lem}
\label{lem:Aset-bounds}
Let $d \in \dset$ and $0 \leq h < B$.
Then for any $(i,j) \in \Aset_d(h)$ we have
$0 \leq i < 2B-1$ and $0 \leq j < B$.
\end{lem}
\begin{proof}
According to \eqref{eq:defn-Aset},
when $d \in \{-1, -2\}$ we have $i, j \geq 0$ and $i+j = h \leq B-1$.
These inequalities imply that $0 \leq i \leq B-1 \leq 2B-2$
and $0 \leq j \leq B-1$.

For $d = 2$ we have $i \geq 2j \geq 0$ and $i-j = h \leq B-1$.
Then $j = 2j - j \leq i - j \leq B-1$ and $i = j + h \leq (B-1)+(B-1) = 2B-2$.
\end{proof}

Before considering the cost of computing the actual transforms,
we discuss briefly how to obtain the polynomials
$F^{i,u}(y)$ and $G_d^{j,v}(y)$ themselves.
\begin{prop}
\label{prop:compute-FG-blocks}
In time $O(N)$, we may compute
\[
F^{i,u} \in \FF_2[y]_L, \qquad 0 \leq i < 2B-1, \quad 0 \leq u < W
\]
arranged lexicographically by $(i,u)$, and for each $d \in \dset$,
\[
G_d^{j,v} \in \FF_2[y]_L, \qquad 0 \leq j < B, \quad 0 \leq v < W
\]
arranged lexicographically by $(j,v)$.
\end{prop}
\begin{proof}
We explain the algorithm only for the $F^{i,u}(y)$ case;
the $G_d^{j,v}(y)$ are handled in a similar fashion.

Recall from \eqref{eq:defn-F} that
$F(x) = \sum_{a \geq 1, \, \textn{$a$ odd}} x^{a^2} \in \FF_2\bbracket{x}$.
For each odd~$a$ in the range $1 \leq a < ((2B-1)M)^{1/2}$,
we compute the block index $i \coloneqq \lfloor a^2 / M \rfloor$
and the slice index $u \coloneqq a^2 \bmod W$
corresponding to the term $x^{a^2}$
(note that $0 \leq i < 2B-1$ and $0 \leq u < W$),
and its position $t \coloneqq (a^2 - iM - u)/W$ within the slice
($0 \leq t < L$).
This data corresponds to the term $y^t$ appearing in
$F^{i,u}(y) = (F^i)_{\bangle{u}}(y)$.
We put all such tuples $(i,u,t)$ into a list and
sort the list lexicographically (\Cref{lem:sort}).
There are $O(N^{1/2})$ tuples, and each has bit size $O(\log N)$,
so we can certainly do all of this in time $N^{1/2+o(1)}$.
Then for each $(i,u)$ we extract the corresponding values of~$t$
and feed them into \Cref{lem:convert-format} to obtain $F^{i,u}(y)$.
The cost for each $(i,u)$ is $O(L + n_{i,u} \log L)$ where
$n_{i,u}$ is the number of tuples,
so the total over all $(i,u)$ is
\[
O\bigl(BWL + \textstyle \sum_{i,u} n_{i,u} \log L \bigr)
   = O(N + N^{1/2} \log N) = O(N). \qedhere
\]
\end{proof}

We now consider the computation of the actual transforms
$\fwd(F^{i,u})$ and $\fwd(G_d^{j,v})$.
Note that we do not have enough time to invoke $\fwd$
for all of these polynomials;
the total amount of data to transform is $\Theta(N)$ bits,
and the FFTs impose an overhead of $\log N$,
but $\Theta(N \log N)$ grossly exceeds our time budget
in \Crefrange{thm:main-deterministic}{thm:main-heuristic}.
Instead, as mentioned in \Cref{rem:core-slices},
we must take advantage of the fact that many of the slices
vanish automatically:
\begin{lem}
\label{lem:zero-slices}
Let $i, j \geq 0$, $d \in \dset$ and $0 \leq u, v < W$. 
\begin{enumabc}
\item If $u$ is not a square in $\ZZ/W\ZZ$ then $F^{i,u} = 0$.
\item If $v$ is not $-d$ times a square in $\ZZ/W\ZZ$ then $G_d^{j,v} = 0$.
\end{enumabc}
\end{lem}
\begin{proof}
These facts follow more or less immediately from
\eqref{eq:defn-F}--\eqref{eq:defn-G2}.
For (a), recall that
\[
F^i(x) = \hspace{-7pt}
         \sum_{\substack{a \geq 1, \, \textn{$a$ odd} \\
                        iM \leq a^2 < (i+1)M}}
                        \hspace{-7pt} x^{a^2 - iM}
         \enspace \in \FF_2[x]_M.
\]
The monomials in $F^{i,u}(y) = (F^i)_{\bangle{u}}$
correspond to those terms for which $a^2 - iM \equiv u \pmod W$.
Since $W \divides M$, this condition is equivalent to $a^2 \equiv u \pmod W$.
If $u$ is not a square modulo~$W$ then no such terms exist,
so $F^{i,u} = 0$.

For (b), an almost identical argument works for $G_d^{j,v}(x)$
for $d = -1, -2$.
For $d = 2$ the definition of the blocks is ``reversed'';
from \eqref{eq:G2-blocks} we have
\[
G_2^j(x) = \hspace{-7pt}
           \sum_{(j-1)M < 2b^2 \leq jM}
           \hspace{-7pt} x^{-2b^2 + jM}
   \enspace \in \FF_2[x]_M, \qquad j \geq 1,
\]
so the same argument works in this case as well.
(The result also holds for $j = 0$ as we defined $G_2^0 \coloneqq 0$.)
\end{proof}

To estimate the cost of the forward transforms,
it is convenient to introduce the quantity
\begin{equation}
\label{eq:defn-rho}
\rho(W) \coloneqq \max_{d \in \dset} \rho_d(W)
\end{equation}
where
\[
\rho_d(W) \coloneqq \bigl| \{w \in \ZZ/W\ZZ :
   w = -d\alpha^2 \textn{ for some } \alpha \in \ZZ/W\ZZ \} \bigr|.
\]
Then we have:
\begin{prop}
\label{prop:compute-transforms}
Given as input the polynomials $F^{i,u}, G_d^{j,v} \in \FF_2[y]_L$
computed in \Cref{prop:compute-FG-blocks},
we may compute their transforms
$\fwd(F^{i,u}), \fwd(G_d^{j,v}) \in \FF_2^K$,
arranged lexicographically by $(i,u)$ (respectively $(j,v)$ for each~$d$),
in time
\[
O\left( N \tsp \frac{\rho(W) \log N}{W} \Mstar(N) + N \cdot \Kstar(N) \right).
\]
\end{prop}
\begin{rem}
In \Cref{sec:core-main} we will choose $W$ divisible by
sufficiently many small primes to ensure that $\rho(W) \ll W/\log N$.
The complexity bound in \Cref{prop:compute-transforms}
then becomes simply $O(N (\Mstar(N) + \Kstar(N))) < N (\log \log N)^{o(1)}$.
\end{rem}
\begin{proof}
We walk through the input blocks $F^{i,u}(y)$ and $G_d^{j,v}(y)$.
For each block, we first check if it is identically zero.
If it is not zero, we invoke~$\fwd$.
If it is zero, we simply write $K$ zeroes to the output.

The total number of blocks is $O(BW)$,
so by \eqref{eq:core-K} the cost of processing the zero blocks is
\[
O(BWK) = O(BWL \cdot \Kstar(N)) = O(N \cdot \Kstar(N)).
\]
Let us now estimate the cost of processing the nonzero blocks.
\Cref{lem:zero-slices} implies that for each~$i$
there are at most $\rho_{-1}(W) \leq \rho(W)$
values of $u \in \{0, \ldots, W-1\}$
for which $F^{i,u}$ can possibly be nonzero,
and similarly that for each $(d,j)$
there are at most $\rho_d(W) \leq \rho(W)$
values of $v \in \{0, \ldots, W-1\}$
for which $G_d^{j,v}$ can possibly be nonzero.
The total number of invocations of $\fwd$ is therefore at most
$O(B \rho(W))$,
and by \eqref{eq:core-eval-cost}
the cost of processing the nonzero blocks is
\[
O(B \rho(W) L \log N \cdot \Mstar(N))
   = O\left(\frac{\rho(W) \log N}{W} \, N \cdot \Mstar(N) \right).
   \qedhere
\]
\end{proof}

\subsection{Pointwise sums}
\label{sec:core-pointwise}

In this section we discuss how to compute those vectors
$Z_{h,w} \in \FF_2^K$ (defined in \eqref{eq:defn-Z})
that contribute to $A(x) \pmod{x^N}$.
\begin{prop}
\label{prop:compute-Z}
Given as input the transforms $\fwd(F^{i,u}), \fwd(G_d^{j,v}) \in \FF_2^K$
computed in \Cref{prop:compute-transforms},
we may compute
\[
Z_{h,w} \in \FF_2^K, \qquad 0 \leq h < B, \quad 0 \leq w < 2W-1,
\]
arranged lexicographically by $(h,w)$, in time
\begin{equation}
\label{eq:compute-Z}
O\bigl( N \log W \cdot \Mstar(W) \Kstar(N)
   + N \log B \cdot \Rstar(B) \Kstar(W) \Kstar(N) \bigr).
\end{equation}
\end{prop}
\begin{proof}
Computing $Z_{h,w}$ amounts to evaluating the pointwise sums
\[
(Z_{h,w})_k =
   \sum_{\substack{d \in \dset \\ (i,j) \in \Aset_d(h)}} \;
   \sum_{\substack{0 \leq u, v < W \\ u+v=w}}
   \fwd(F^{i,u})_k \cdot \fwd(G_d^{j,v})_k
   \in \FF_2, \qquad 0 \leq k < K.
\]
According to \Cref{lem:Aset-bounds},
all of the $\fwd(F^{i,u})$ and $\fwd(G_d^{j,v})$ appearing in this formula
are available as part of the input.

Let us first transpose the data so that we can work on each
value of~$k$ separately.
The input vectors $\fwd(F^{i,u})$ and $\fwd(G_d^{j,v})$
consist of $O(BW)$ bit arrays of length~$K$;
by \Cref{lem:transpose} and \eqref{eq:core-K},
transposing these to obtain $K$ bit arrays of size $O(BW)$ requires time
\begin{equation}
\label{eq:Z-transpose-cost}
O(BWK \log(BW)) = O(N \log(BW) \Kstar(N)).
\end{equation}
Similarly, after computing all of the $(Z_{h,w})_k$ for each~$k$ separately
(as discussed in the following paragraphs),
the cost of transposing back to obtain
the output vectors $Z_{h,w}$ is also given by \eqref{eq:Z-transpose-cost}.
Moreover \eqref{eq:Z-transpose-cost} is clearly dominated by
\eqref{eq:compute-Z}.
   
Now fix one $k \in \{0, \ldots, K-1\}$ and define
\[
f^{i,u} \coloneqq \fwd(F^{i,u})_k \in \FF_2, \qquad
g_d^{j,v} \coloneqq \fwd(G_d^{j,v})_k \in \FF_2.
\]
Thanks to the previous paragraph,
we may assume that the $f^{i,u}$ are stored as one contiguous block
of size $O(BW)$ (arranged lexicographically by $(i,u)$),
and similarly for the~$g_d^{j,v}$.
Our goal is to compute the sums
\[
z_{h,w} \coloneqq (Z_{h,w})_k =
   \sum_{\substack{d \in \dset \\ (i,j) \in \Aset_d(h)}} \;
   \sum_{\substack{0 \leq u, v < W \\ u+v = w}}
   f^{i,u} g_d^{j,v} \in \FF_2,
\qquad
\begin{aligned}[t]
0 & \leq h < B, \\
0 & \leq w < 2W-1,
\end{aligned}
\]
arranged lexicographically by $(h,w)$.
It suffices in turn to compute the sums
\begin{equation}
\label{eq:Z-inner-sum}
z_{d,h,w} \coloneqq
   \sum_{(i,j) \in \Aset_d(h)} \;
   \sum_{\substack{0 \leq u, v < W \\ u+v = w}}
   f^{i,u} g_d^{j,v} \in \FF_2
\end{equation}
for each $d \in \dset$,
since we may then recover $z_{h,w}$ via
$z_{h,w} = \sum_{d \in \dset} z_{d,h,w}$.

The sum over $(u,v)$ in \eqref{eq:Z-inner-sum} corresponds to
evaluating a certain polynomial product.
To make this explicit let us define polynomials
\begin{align*}
f^i(t) & \coloneqq \sum_{0 \leq u < W} f^{i,u} t^u \in \FF_2[t]_W,
   \qquad 0 \leq i < 2B-1, \\
g_d^j(t) & \coloneqq \sum_{0 \leq v < W} g_d^{j,v} t^v \in \FF_2[t]_W,
   \qquad d \in \dset, \quad 0 \leq j < B.
\end{align*}
Then $z_{d,h,w}$ is exactly the coefficient of~$t^w$ in the sum
\[
y_{d,h} \coloneqq
   \sum_{(i,j) \in \Aset_d(h)} f^i(t) g_d^j(t) \in \FF_2[t]_{2W-1},
   \qquad d \in \dset, \quad 0 \leq h < B.
\]

To evaluate the sums $y_{d,h}$ efficiently
we invoke \Cref{thm:transform-scheme} to introduce
yet another transform pair over $\FF_2$, this time of order~$W$.
This yields maps
\[
\fwd' \colon \FF_2[t]_W \to \FF_2^{K'}, \qquad
\inv' \colon \FF_2^{K'} \to \FF_2[t]_{2W-1}
\]
that may be evaluated in time
\[
\Tcost(W) \ll \Mcost(W) = W \log W \cdot \Mstar(W)
\]
and where $K'$ satisfies
\[
K' \leq \Kcost(W) = W \cdot \Kstar(W).
\]
We thus obtain
\[
y_{d,h} = \inv'\biggl( \, \sum_{(i,j) \in \Aset_d(h)}
   \fwd'(f^i) \cdot \fwd'(g_d^j) \biggr)
\]
where the dot indicates pointwise multiplication in~$\FF_2^{K'}$.

We now use the following algorithm to compute the $y_{d,h}$.
(We remind the reader that $k$ is held fixed throughout this discussion.)
   
First we evaluate $\fwd'(f^i), \fwd'(g_d^j) \in \FF_2^{K'}$
for all $i$ and $(d,j)$.
(Note that each $f^i, g_d^j \in \FF_2[t]_W$ is already stored as
a contiguous block of $W$ bits.)
There are $O(B)$ such transforms so the cost is
\[
O(B W \log W \cdot \Mstar(W)).
\]
   
Next, for each $k' \in \{0, \ldots, K'-1\}$
we need to compute the inner sums
\begin{equation}
\label{eq:Z-inner-sum2}
\sum_{(i,j) \in \Aset_d(h)}
   \fwd'(f^i)_{k'} \cdot \fwd'(g_d^j)_{k'} \in \FF_2
\qquad 0 \leq h < B.
\end{equation}
To access the relevant data for each~$k'$,
we must perform a transposition of size $K' \times O(B)$.
The cost of this transposition is
\[
O(B K' \log B) = O(BW \log B \cdot \Kstar(W)).
\]
Then for each $k'$ we must compute the sum in \eqref{eq:Z-inner-sum2}
for all $h \in \{0, \ldots, B-1\}$.
For $d \in \{-1, -2\}$ this reduces to an ordinary product
of two polynomials in $\FF_2[t]_B$.
For $d = 2$ it becomes a restricted product of polynomials in
$\FF_2[t]_{2B-1}$ and $\FF_2[t^{-1}]_B$.
The cost in the first case is $\Mcost(B)$,
and in the second case
(after zero-padding the first argument to $\FF_2[t]_{2B}$)
is $\Rcost(B)$.
Using \eqref{eq:MR-reduction}, we see that the cost of this step
across all values of~$k'$ is
\[
O(K' \Rcost(B)) = O(BW \log B \cdot \Rstar(B) \Kstar(W)).
\]
After performing these (ordinary or restricted) products,
we must transpose back to the original ordering,
which has the same cost as before.
   
Finally, we must evaluate the inverse transforms $\inv'$
to obtain the~$y_{d,h}$.
There are $O(B)$ such transforms, so the cost is again
$O(B W \log W \cdot \Mstar(W))$.
   
At this stage we have obtained all of the $y_{d,h}$,
and hence the $z_{d,h,w}$ (ordered by $(h,w)$ for each~$d$).
We may then compute $z_{h,w} = \sum_{d \in \dset} z_{d,h,w}$
for all $(h,w)$ in time $O(BW)$.
We conclude that the cost of computing all of the $z_{h,w}$
for a single value of~$k$ is
\[
O\bigl( BW (
   \log W \cdot \Mstar(W) + \log B \cdot \Rstar(B) \Kstar(W) ) \bigr).
\]
Summing over all $K \leq L \cdot \Kstar(N)$ values of $k$
yields the final bound \eqref{eq:compute-Z}.
\end{proof}

To simplify notation later,
it will be convenient to slightly reorganise the output
of \Cref{prop:compute-Z}.
Let us define
\begin{equation}
\label{eq:defn-Z-sigmatau}
Z^{\sigma,\tau}_{h,w} \coloneqq Z_{h - \sigma, w + \tau W} \in \FF_2^K,
\qquad (\sigma, \tau) \in \{0, 1\}^2, \quad h \geq 0, \quad 0 \leq w < W,
\end{equation}
so that the main sum \eqref{eq:A-slices-eval} becomes
\begin{equation}
\label{eq:A-slices-eval-2}
A^{h,w}(y) = E^{h,w}(y) +
   \sum_{\substack{\sigma = 0, 1 \\ \tau = 0, 1}}
   \inv^{\sigma,\tau}(Z^{\sigma,\tau}_{h,w})
   \enspace \in \FF_2[y]_L,
\qquad
\begin{gathered}[t]
h \geq 0, \\
0 \leq w < W.
\end{gathered}
\end{equation}

\begin{prop}
\label{prop:compute-Z2}
Given as input the vectors $Z_{h,w}$ computed in \Cref{prop:compute-Z},
we may compute
\[
Z^{\sigma,\tau}_{h,w} \in \FF_2^K, \qquad
(\sigma, \tau) \in \{0, 1\}^2, \quad 0 \leq h < B, \quad 0 \leq w < W,
\]
arranged lexicographically by $(h,w)$ for each $(\sigma,\tau)$, in time
\[
O(N \cdot \Kstar(N)).
\]
\end{prop}
\begin{proof}
This is just a simple copy-and-paste operation,
requiring time $O(BWK) = O(BWL \cdot \Kstar(N)) = O(N \cdot \Kstar(N))$.
\end{proof}

\subsection{Apply compression}
\label{sec:core-compression}

At this stage it would be possible to simply evaluate the inverse transforms
in \eqref{eq:A-slices-eval-2} to compute the desired $A^{h,w}(y)$.
However, as noted previously in connection with the forward transforms,
this approach is too slow because the total bit size is $\Theta(N)$.
Instead we will first use the compression framework
of \Cref{sec:compression} to reduce the number of transforms required.

To begin with, let us see how to rewrite \eqref{eq:A-slices-eval-2}
directly in terms of the target intervals~$a^r$.
By definition we have
\begin{align*}
(A^{h,w})_\ell & = ((A^h)_{\bangle{w}})_\ell \\
               & = (A^h)_{\ell W + w} = A_{hM +\ell W + w},
\qquad
0 \leq h < B, \quad 0 \leq w < W, \quad 0 \leq \ell < L.
\end{align*}
Now consider a typical entry of some~$a^r$, say
\[
(a^r)_t = A_{rT + t}, \qquad 0 \leq r < N/T, \quad 0 \leq t < T.
\]
By \eqref{eq:core-auxiliary-hypotheses}
we may write $rT + t$ uniquely in the form
\begin{equation}
\label{eq:r-decomposition}
rT + t = h_r M + \ell_r W + m_r T + t,
\end{equation}
with
\[
0 \leq h_r    < N/M = B, \qquad
0 \leq \ell_r < M/W = L, \qquad
0 \leq m_r    < W/T.
\]
Therefore we obtain
\[
(a^r)_t = A_{h_r M + \ell_r W + m_r T + t} = (A^{h_r,m_rT + t})_{\ell_r}.
\]
Similarly, if we define vectors $e^r \in \FF_2^T$ encoding
the coefficients of $E(x)$, i.e.,
\begin{equation}
\label{eq:defn-er}
(e^r)_t \coloneqq E_{rT+t},
\qquad r \geq 0, \quad 0 \leq t < T,
\end{equation}
then we obtain the analogous formula $(e^r)_t = (E^{h_r,m_rT + t})_{\ell_r}$.
Putting everything together, \eqref{eq:A-slices-eval-2} becomes
\begin{equation}
\label{eq:ar-formula}
(a^r)_t = (e^r)_t +
\sum_{\substack{\sigma = 0, 1 \\ \tau = 0, 1}}
\bigl(\inv^{\sigma,\tau}(Z^{\sigma,\tau}_{h_r,m_r,t})\bigr)_{\ell_r},
\qquad 0 \leq r < N/T, \quad 0 \leq t < T
\end{equation}
where
\begin{equation}
\label{eq:Z-sigmatau-reindexed}
Z^{\sigma,\tau}_{h,m,t} \coloneqq Z^{\sigma,\tau}_{h,mT+t} \in \FF_2^K,
\qquad
\begin{aligned}
& (\sigma, \tau) \in \{0, 1\}^2, & & 0 \leq h < B, \\
& 0 \leq m < W/T, & & 0 \leq t < T.
\end{aligned}
\end{equation}

We are now in a position to introduce the compression framework.
Let $R \geq 1$ be the parameter specified in the input to
\Cref{thm:core-auxiliary}.
Let $S \geq 1$ be determined from $T$ and $R$ via
\eqref{eq:defn-lambda} and \eqref{eq:defn-S},
and let $\kappa \colon \FF_2^T \to \FF_2^S$
be the corresponding compression map constructed
in \Cref{sec:construct-kappa}.
Let $\kappa_{s,t} \in \FF_2$ be the entries of the matrix of~$\kappa$
with respect to the standard bases for $\FF_2^T$ and $\FF_2^S$,
i.e., so that for any $a \in \FF_2^T$ we have
\[
\kappa(a)_s = \sum_{0 \leq t < T} \kappa_{s,t} a_t,
   \qquad 0 \leq s < S.
\]
Let us define
\begin{equation}
\label{eq:defn-hat-ar-er}
\hat a^r \coloneqq \kappa(a^r) \in \FF_2^S,
\qquad
\hat e^r \coloneqq \kappa(e^r) \in \FF_2^S,
\qquad 0 \leq r < N/T,
\end{equation}
and also compressed versions of the $Z^{\sigma,\tau}_{h,m,t}$ given by
\begin{equation}
\label{eq:compressed-Z}
\hat Z^{\sigma,\tau}_{h,m,s} \coloneqq
\sum_{0 \leq t < T} \kappa_{s,t} Z^{\sigma,\tau}_{h,m,t} \in \FF_2^K,
\qquad
\begin{aligned}
& (\sigma, \tau) \in \{0, 1\}^2, & & 0 \leq h < B, \\
& 0 \leq m < W/T, & & 0 \leq s < S.
\end{aligned}
\end{equation}
Applying $\kappa$ to both sides of \eqref{eq:ar-formula},
and using the $\FF_2$-linearity of~$\inv^{\sigma,\tau}$,
we then obtain the following ``compressed'' version of \eqref{eq:ar-formula}:
\begin{equation}
\label{eq:kappa-ar}
(\hat a^r)_s = (\hat e^r)_s +
   \sum_{\substack{\sigma = 0, 1 \\ \tau = 0, 1}}
   \bigl(\inv^{\sigma,\tau}(\hat Z^{\sigma,\tau}_{h_r,m_r,s})\bigr)_{\ell_r},
   \qquad 0 \leq r < N/T, \quad 0 \leq s < S.
\end{equation}

\begin{rem}
The algebra just carried out is one of the most important steps in the paper.
As mentioned in \Cref{rem:key-idea},
the key insight is that the compression and transform maps ``commute'',
essentially because they are linear maps.
This is what enables us to reduce the number of inverse transforms,
and is ultimately the source of the overall savings in the main results.
\end{rem}

The next result estimates the complexity of computing the compressed vectors.
\begin{prop}
\label{prop:compute-hatZ}
Given as input
\begin{itemize}
\item
the polynomial $E_1(x) \pmod{x^N}$ computed in
\Cref{prop:compute-E1},
\item
the polynomial $E_2(x) \pmod{x^N}$ computed in
\Cref{prop:compute-E2}, and
\item the vectors $Z^{\sigma,\tau}_{h,w} \in \FF_2^K$
computed in \Cref{prop:compute-Z2},
\end{itemize}
we may compute the compressed vectors
\[
\hat e^r \in \FF_2^S, \qquad 0 \leq r < N/T
\]
and the compressed vectors
\[
\hat Z^{\sigma,\tau}_{h,m,s} \in \FF_2^K, \qquad
\begin{aligned}
& (\sigma, \tau) \in \{0, 1\}^2, & & 0 \leq h < B, \\
& 0 \leq m < W/T, & & 0 \leq s < S,
\end{aligned}
\]
arranged lexicographically by $(h,m,s)$ for each $(\sigma,\tau)$,
in time
\[
O\left(\left(1 + \frac{R(\log T)^{1+o(1)}}{T}\right)
      N \log T \cdot \Cstar(T) \Kstar(N)\right) + T (\log T)^{3+o(1)}.
\]
\end{prop}
\begin{proof}
We first find $E(x) = E_1(x) + E_2(x) \pmod{x^N}$ in time $O(N)$
to obtain the vectors $e^r \in \FF_2^T$.
We then compute each $\hat e^r \in \FF_2^S$ by applying
the compression algorithm of \Cref{thm:compression}.
After a precomputation of cost $T (\log T)^{3+o(1)}$,
the cost per interval is
$O(T \log T \cdot \Cstar(T)) + R (\log T)^{2+o(1)}$,
so the total cost over all $N/T$ intervals is
\[
O(N \log T \cdot \Cstar(T)) + \frac{N R (\log T)^{2+o(1)}}{T}
   \ll \left(1 + \frac{R(\log T)^{1+o(1)}}{T}\right) N \log T \cdot \Cstar(T).
\]

Now consider the $Z^{\sigma,\tau}_{h,w}$.
We must first reorganise the input data to enable contiguous access to
the relevant vectors of length~$T$.
Writing $w = mT + t$, the inputs
$(Z^{\sigma,\tau}_{h,mT+t})_k = (Z^{\sigma,\tau}_{h,m,t})_k \in \FF_2$
are provided as bit arrays of size
$B \times (W/T) \times T \times K$ for each $(\sigma,\tau)$,
i.e., arranged lexicographically by $(h,m,t,k)$.
We must transpose the last two coordinates to obtain arrays of size
$B \times (W/T) \times K \times T$.
This can be done in time
\[
O(B (W/T) \cdot K T \log T) = O(N \log T \cdot \Kstar(N)).
\]
After this rearrangement,
we evaluate \eqref{eq:compressed-Z} for each $(\sigma,\tau,h,m,k)$,
again via \Cref{thm:compression}.
The number of invocations is
\begin{equation}
\label{eq:compression-count}
O(B (W/T) K) = O((N/T) \Kstar(N)),
\end{equation}
and the cost per interval is again
$O(T \log T \cdot \Cstar(T)) + R (\log T)^{2+o(1)}$,
so the total cost is
\[
O\left(\left(1 + \frac{R(\log T)^{1+o(1)}}{T}\right)
      N \log T \cdot \Cstar(T) \Kstar(N) \right).
\]
The $\hat Z^{\sigma,\tau}_{h,m,s} \in \FF_2^K$ (the ``compressed'' data)
now form bit arrays of size
$B \times (W/T) \times K \times S$ for each $(\sigma,\tau)$,
and we must transpose the last two coordinates
to obtain the final output arrays of size
$B \times (W/T) \times S \times K$.
The cost of this step is
\[
O(B (W/T) \cdot S K \log S) =
   O\left(\frac{S}{T} \, N \log S \cdot \Kstar(N) \right).
\]
By \eqref{eq:defn-S} we have $S \ll R \log T$,
and the hypotheses of \Cref{thm:core-auxiliary}
include that $R \leq T/2$, so $\log S \ll \log T$.
The transposition cost thus becomes
\[
O\left(\frac{R \log T}{T} \, N \log T \cdot \Kstar(N) \right). \qedhere
\]
\end{proof}

\subsection{Inverse transforms and decompression}
\label{sec:core-inverse}

The next step is to perform the inverse transforms
in \eqref{eq:kappa-ar} to deduce the compressed vectors~$\hat a^r$.
\begin{prop}
\label{prop:core-compute-kappa-ar}
Given as input the vectors $\hat e^r \in \FF_2^S$
and $\hat Z^{\sigma,\tau}_{h,m,s} \in \FF_2^K$ computed in
\Cref{prop:compute-hatZ},
we may compute
\[
\hat a^r \in \FF_2^S, \qquad 0 \leq r < N/T
\]
in time
\[
O\left(N \tsp \frac{R \log T \log N}{T} \Mstar(N) \right).
\]
\end{prop}
\begin{proof}
We first compute the inverse transforms
\[
\inv(\hat Z^{\sigma,\tau}_{h,m,s}) \in \FF_2[y]_{2L-1}, \qquad
\begin{aligned}
& (\sigma, \tau) \in \{0, 1\}^2, & & 0 \leq h < B, \\
& 0 \leq m < W/T, & & 0 \leq s < S.
\end{aligned}
\]
The cost of each transform is given by \eqref{eq:core-eval-cost},
and the number of transforms is $O(B (W/T) S)$,
so the cost of this step is
\[
O(B (W/T) S \cdot L \log N \cdot \Mstar(N))
   = O\left( \frac{S}{T} \, N \log N \cdot \Mstar(N) \right).
\]

Next we compute the sums
\[
\sum_{\substack{\sigma = 0, 1 \\ \tau = 0, 1}}
   \inv^{\sigma,\tau}(\hat Z^{\sigma,\tau}_{h,m,s}) \in \FF_2[y]_L,
   \qquad 0 \leq h < B, \quad 0 \leq m < W/T, \quad 0 \leq s < S.
\]
Recall from \eqref{eq:inv-sigmatau} that $\inv^{\sigma,\tau}(Z)$
simply extracts from $\inv(Z)$ a certain subinterval of length~$L$,
parameterised by $\sigma$ and~$\tau$.
This can be achieved in linear time,
and similarly for the summation over $(\sigma, \tau)$.
The cost of this step is thus
\[
O(B (W/T) S \cdot L) = O\left(\frac{S}{T} \, N \right).
\]

At this point we have computed the quantities
\[
\sum_{\substack{\sigma = 0, 1 \\ \tau = 0, 1}}
\bigl(\inv^{\sigma,\tau}(\hat Z^{\sigma,\tau}_{h,m,s})\bigr)_\ell \in \FF_2,
\qquad
\begin{aligned}[t]
   & 0 \leq h < B, & & 0 \leq m < W/T, \\
   & 0 \leq s < S, & & 0 \leq \ell < L
\end{aligned}
\]
as a bit array of size $B \times (W/T) \times S \times L$,
ordered lexicographically by $(h,m,s,\ell)$.
These are exactly the quantities appearing in \eqref{eq:kappa-ar},
but to recover the $\hat a^r \in \FF_2^S$ ordered correctly by $r$,
we must reorganise the data.
Namely, according to \eqref{eq:r-decomposition} we have
$r = h_r (M/T) + \ell_r (W/T) + m_r$,
so we want the data ordered by $(h,\ell,m,s)$.
This may be achieved by transposing $(m,s)$ and~$\ell$,
which costs
\[
O\bigl(B \cdot (SW/T) L \log(SW/T)\bigr)
   = O\left(\frac{S}{T} \, N \log(SW/T) \right).
\]
We have $W \leq N$ and $S/T \ll (R \log T)/T \ll \log T \leq \log N$
(since we assumed that $R \leq T/2$),
so $\log(SW/T) \ll \log N$, and the last cost estimate becomes
\[
O\left(\frac{S}{T} \, N \log N \right).
\]
After this transposition, we have successfully computed
\[
\sum_{\substack{\sigma = 0, 1 \\ \tau = 0, 1}}
\bigl(\inv^{\sigma,\tau}(\hat Z^{\sigma,\tau}_{h_r,m_r,s})\bigr)_{\ell_r},
\qquad 0 \leq r < N/T, \quad 0 \leq s < S
\]
ordered correctly by $(r,s)$.
Adding in the supplied $\hat e^r$ terms costs an additional $O((S/T) N)$,
and finally yields the desired vectors $\hat a^r$
according to \eqref{eq:kappa-ar}.
\end{proof}

Finally, we run the decompression algorithm to obtain
the desired $c^r \in \FF_2^T$.
\begin{prop}
\label{prop:core-compute-cr}
Given the vectors $\hat a^r \in \FF_2^S$ computed in
\Cref{prop:core-compute-kappa-ar},
we may compute candidates
\[
c^r \in \FF_2^T, \qquad 0 \leq r < N/T
\]
satisfying the conclusion of \Cref{thm:core-auxiliary} in time
\[
O\left(\left(1 + \frac{R(\log T)^{4+o(1)}}{T}\right)
      N \log T \cdot \Cstar(T)\right) + T (\log T)^{3+o(1)}.
\]
\end{prop}
\begin{proof}
We apply \Cref{thm:decompression} to each $\hat a^r = \kappa(a^r)$,
taking $c^r$ to be the output of the decompression algorithm,
or simply zero if it returns ``FAIL''.
(The algorithm could optionally return additional information to report
such failures, but we will not make use of this.)
After a precomputation of cost $T (\log T)^{3+o(1)}$,
the cost per interval is
$O(T \log T \cdot \Cstar(T)) + R (\log T)^{5+o(1)}$.
The number of intervals is~$N/T$,
leading to the claimed complexity bound.
\end{proof}

This concludes the proof of \Cref{thm:core-auxiliary}.
The complexity bound \eqref{eq:core-auxiliary-bound} is obtained by summing
the contributions from Propositions
\ref{prop:compute-E1},
\ref{prop:compute-E2},
\ref{prop:compute-FG-blocks},
\ref{prop:compute-transforms},
\ref{prop:compute-Z},
\ref{prop:compute-Z2},
\ref{prop:compute-hatZ},
\ref{prop:core-compute-kappa-ar} and
\ref{prop:core-compute-cr}.

\subsection{Proof of \Cref{thm:core}}
\label{sec:core-main}

Assume that we are given $N$, $T$ and $R$ satisfying the hypotheses
of \Cref{thm:core}.
Our first task is to choose suitable parameters $M$ and $W$
to use in \Cref{thm:core-auxiliary}.
For $W$ we will take
\begin{equation}
\label{eq:defn-W}
W \coloneqq T \prod_{3 \leq p \leq X} p
\end{equation}
where $X \geq 3$ is an integer defined by
\[
X \coloneqq 2 \log \log N \log \log \log N + O(1),
\]
and for $M$ we choose
\begin{equation}
\label{eq:defn-M}
M \coloneqq \frac{N}{\log^3 N} + O(W)
\qquad
\textn{ such that } W \divides M.
\end{equation}
To ensure that \eqref{eq:core-auxiliary-hypotheses} holds,
we round $N$ up to the nearest multiple of $M$,
setting $N' \coloneqq \lceil N/M \rceil M$.
We then invoke \Cref{thm:core-auxiliary} with the
parameters $N'$, $T$, $R$, $M$ and $W$,
to obtain a list of candidates $c^r \in \FF_2^T$ for $0 \leq r < N'/T$.
Any surplus vectors, i.e., for $r \geq N/T$, may be discarded.

The output of the algorithm is clearly correct,
and it remains to analyse its complexity.
We begin by estimating the ratio $\rho(W)/W$ appearing
in \eqref{eq:core-auxiliary-bound},
where we recall that $\rho(W)$ is defined by \eqref{eq:defn-rho}.
\begin{lem}
For $W$ defined as in \eqref{eq:defn-W} we have
\begin{equation}
\label{eq:rhoW-estimate}
\frac{\rho(W)}{W} \ll \frac{1}{\log N}.
\end{equation}
\end{lem}
\begin{proof}
Let $d \in \dset$.
We use the following facts about the function
$\rho_d(\cdotspace{1pt})$:
\begin{itemize}
\item
Let $p \geq 2$ be a prime and let $e \geq 1$.
Then clearly
\[
\rho_d(p^e) \leq p^e.
\]
\item 
Let $p \geq 3$ be a prime and let $e \geq 1$.
We claim that
\[
\rho_d(p^e) \leq \frac{p+1}{2} \, p^{e-1}.
\]
To prove this, consider a nonzero $w \in \ZZ/p^e\ZZ$.
Write $w = p^f u$ where $0 \leq f < e$ and $u$ is a unit.
Since $d$ is also a unit,
Hensel lifting implies that the congruence $w \equiv -d \alpha^2 \pmod{p^e}$
can be solved for $\alpha$ if and only if $f$ is \emph{even}
and the congruence $u \equiv -d \alpha^2 \pmod p$ has a solution.
The latter is solvable for exactly $(p-1)/2$ choices of~$u \pmod p$.
Therefore
\[
\rho_d(p^e) = 1 + \sum_{\substack{0 \leq f < e \\ \textn{$f$ even}}}
   \frac{p-1}{2} \, p^{e-1-f}.
\]
If $e = 1$ then $\rho_d(p^e) = 1 + (p-1)/2 = (p+1)/2$ and we are done.
Suppose now that $e \geq 2$. Then
\begin{multline*}
\hspace{20pt} \rho_d(p^e) - \frac{p-1}{2} \, p^{e-1}
   = 1 + \sum_{\substack{2 \leq f < e \\ \textn{$f$ even}}}
      \frac{p-1}{2} \, p^{e-1-f}
   < 1 + \frac{p-1}{2} \, p^{e-1} \, \frac{1}{p^2-1} \\
   = p^{e-1} \left( \frac{1}{p^{e-1}} + \frac{1}{2(p+1)}\right)
   \leq p^{e-1} \left(\frac{1}{3} + \frac{1}{8}\right) < p^{e-1},
\end{multline*}
and the claim follows immediately.
\item
For any integer $W \geq 1$,
let $W = \prod_{p \divides W} p^{e_p}$ be the factorisation of $W$
into prime powers.
Then the Chinese remainder theorem implies that
\[
\rho_d(W) = \prod_{p \divides W} \rho_d(p^{e_p}).
\]
\end{itemize}

Now let $W$ be given by \eqref{eq:defn-W}
and let $W = \prod_{p \divides W} p^{e_p}$
be its factorisation into prime powers.
We have $e_p \geq 1$ for $3 \leq p \leq X$.
The above facts therefore imply that
\begin{multline*}
\frac{\rho_d(W)}{W}
= \prod_{p \divides W} \frac{\rho_d(p^{e_p})}{p^{e_p}}
\leq \prod_{3 \leq p \leq X} \frac{\rho_d(p^{e_p})}{p^{e_p}} \\
\leq \prod_{3 \leq p \leq X} \frac{(p+1)/2}{p}
= \frac{1}{2^{\pi(X)-1}}
   \prod_{3 \leq p \leq X} (1 + p^{-1})
\end{multline*}
where $\pi(X)$ denotes the usual prime-counting function,
i.e., $\pi(X) = \sum_{p \leq X} 1$.
We have $\pi(X) \sim X/\log X$ by the prime number theorem,
and $\prod_{p \leq X} (1 + p^{-1}) \asymp \log X$
by Mertens' theorem \cite[Thm.\,2.7(e)]{MV-mult-nt}.
Now observe that
\[
\frac{X}{\log X} =
   \frac{2 \log \log N \log \log \log N + O(1)}
      {\log \log \log N + \log \log \log \log N + O(1)}
   = (2 + o(1)) \log \log N,
\]
so
\[
\pi(X) = (1 + o(1)) \frac{X}{\log X} = (2 + o(1)) \log \log N
   > \frac{1}{\log 2}(\log \log N + \log \log X)
\]
for large $N$.
This implies that $2^{\pi(X)} > \log N \log X$ for large $N$, and hence that
\[
\frac{\rho_d(W)}{W} \ll \frac{\log X}{\log N \log X} = \frac{1}{\log N}.
\]
Since $\rho(W) = \max_{d \in \dset} \rho_d(W)$,
we conclude that also $\rho(W)/W \ll 1 / \log N$.
\end{proof}

Let us now analyse the complexity.
Let $C > 0$ be the constant appearing in \eqref{eq:logT-bound}.
For $N \leq \exp(\exp(\exp(C)))$ the running time is clearly $O_C(1)$,
so we may assume that $N > \exp(\exp(\exp(C)))$,
i.e., that $C < \log \log \log N$.
By the prime number theorem we have
$\log W = \log T + \sum_{3 \leq p \leq X} \log p \ll \log T + X$.
The hypothesis \eqref{eq:logT-bound} then implies that
\begin{equation}
\label{eq:logW-estimate}
\log W \ll \log \log N \log \log \log N.
\end{equation}
This estimate also shows that the $O(W)$ term in \eqref{eq:defn-M}
is negligible compared to the main $N/\log^3 N$ term, so
\begin{equation}
\label{eq:M-estimate}
M \sim \frac{N}{\log^3 N}.
\end{equation}
Moreover, we may clearly compute $X$, $W$ and $M$ in time $(\log N)^{O(1)}$.

After rounding $N$ up to $N'$ we have $N' = (1 + O(1/\log^3 N)) N$.
It follows easily that \eqref{eq:rhoW-estimate}, \eqref{eq:logW-estimate},
\eqref{eq:M-estimate} and \eqref{eq:logT-bound}
continue to hold with $N$ replaced by $N'$.
If we can establish \eqref{eq:core-bound} for $N'$,
then it certainly holds for the original~$N$.
Therefore for the rest of the discussion
there is no harm in writing $N$ instead of~$N'$.

Note that $B = N/M \sim \log^3 N$ by \eqref{eq:M-estimate}.
Using \eqref{eq:logT-bound}, \eqref{eq:rhoW-estimate},
\eqref{eq:logW-estimate} and \eqref{eq:M-estimate},
the complexity bound \eqref{eq:core-auxiliary-bound} simplifies to
\begin{multline}
\label{eq:core-bound-progress}
O_C \biggl(
   \frac{R \log N}{T} N \log \log N \cdot \Mstar(N) + N \cdot \Mstar(N) \\
   \shoveright{+ N \log \log N \,
      \Bigl(\log \log \log N \cdot \Mstar(W) \Kstar(N)
      + \Rstar(B) \Kstar(W) \Kstar(N) \Bigr) } \\
   + \left( 1 + \frac{R(\log \log N)^{4+o(1)}}{T} \right)
      N \log \log N \cdot \Cstar(T) \Kstar(N) \biggr).
\end{multline}
By \Cref{thm:restricted-simple} we may take
\[
\Rstar(B) < \exp((\log \log B)^{1/2+o(1)})
   < (\log B)^{o(1)} < (\log \log N)^{o(1)}.
\]
Assuming that we use a sufficiently fast multiplication algorithm
and a transform scheme with sufficiently low multiplicative complexity
(see \eqref{eq:ffmul} and \eqref{eq:Kcost-bound}), then
\begin{align*}
\Mstar(W) & \leq \Mstar(N) < (\log \log N)^{o(1)}, \\
\Kstar(W) & \leq \Kstar(N) < (\log \log N)^{o(1)},
\end{align*}
and
\[
\Cstar(T) = \Mstar(T) + \Kstar(T) \leq \Mstar(N) + \Kstar(N)
   < (\log \log N)^{o(1)}.
\]
The bound \eqref{eq:core-bound-progress}
thus simplifies to \eqref{eq:core-bound},
completing the proof of \Cref{thm:core}.

\begin{rem}
\label{rem:W-bigger-than-T}
It is clear from the above discussion that our choice of $W$
is essentially optimal,
i.e., in order to ensure that $\rho(W)/W \ll 1/\log N$,
we are forced to take $\log W \gg \log \log N \log \log \log N$.
On the other hand,
to minimise the overall complexity of the
probabilistic and deterministic algorithms,
we must take $\log T \ll \log \log N$
to optimise the dominant first term in \eqref{eq:core-auxiliary-bound}.
This discrepancy explains why we use two slicing parameters $T < W$
rather than just one.
(It may be possible to modify the heuristic version to manage with
$W = T = (\log N)^{\Theta(\log \log \log N)}$
without changing the overall complexity,
but the author has not checked this carefully.)
\end{rem}

\begin{rem}
\label{rem:W-logloglogN}
The choice $\log W \asymp \log \log N \log \log \log N$
introduces yet another obstruction to removing the
$(\log \log N)^{o(1)}$ factor in \Cref{thm:main-heuristic},
via the $N \log W \cdot \Mstar(W) \Kstar(N)$ term
in \eqref{eq:core-auxiliary-bound}.
\end{rem}

\section{The heuristic algorithm}
\label{sec:heuristic}

Recall that the core algorithm (\Cref{thm:core})
outputs an incorrect candidate $c^r \in \FF_2^T$
whenever the number of odd square-primes in the $r$\th interval
exceeds the threshold parameter~$R$.
In this section we pursue the simplest possible strategy
for dealing with these problematic intervals:
we simply \emph{assume} that they do not exist.
Our precise requirement is stated as the following conjecture.
Let $\tilde\pi(x)$ denote the counting function for the square-primes, i.e.,
\begin{equation}
\label{eq:tilde-pi-defn}
\tilde\pi(x) \coloneqq
   \sum_{m \geq 1} \;
   \sum_{\substack{m^2 p \leq x \\ \text{$p$ prime}}} 1
      = \sum_{1 \leq m \leq (x/2)^{1/2}} \pi\left(\frac{x}{m^2}\right).
\end{equation}
\begin{conj}
\label{conj:square-primes}
For all $x \geq 2$ we have
\[
\tilde\pi(x + \log^3 x) - \tilde\pi(x) < 3 \log^2 x,
\]
i.e., the interval $(x, x + \log^3 x]$ contains
fewer than $3 \log^2 x$ square-primes.
\end{conj}
We have stated the conjecture in this form for the sake of definiteness,
but to prove \Cref{thm:main-heuristic}
it would suffice to assume any inequality of the following type:
there exist constants $A > 2$ and $C > 0$ such that
\begin{equation}
\label{eq:conj-square-primes-weaker}
\tilde\pi(x + \log^A x) - \tilde\pi(x) < C \log^{A-1} x
\end{equation}
for all sufficiently large $x$.
(We could also of course replace $\tilde\pi(x)$ by the counting function
for the \emph{odd} square-primes,
but for aesthetic reasons we prefer this more general formulation.)

A proof of \Cref{conj:square-primes}
or any estimate of the form \eqref{eq:conj-square-primes-weaker}
appears to be way beyond current technology.
The best we can offer are heuristic arguments and numerical evidence,
which are presented in
\Crefrange{sec:primes-very-short}{sec:square-primes-very-short}.
Then in \Cref{sec:heuristic-algorithm} we describe
the main heuristic prime enumeration algorithm that relies on the conjecture,
and complete the proof of \Cref{thm:main-heuristic}.

\medskip
Before going on, let us determine what conditions must be imposed on
$T$ and~$R$ to ensure that the core algorithm is fast enough
to prove \Cref{thm:main-heuristic}.
Examining the second-last term of \eqref{eq:core-auxiliary-bound},
it is clear that to obtain an overall complexity bound
as sharp as \eqref{eq:main-heuristic}
we must take $\log T < (\log \log N)^{1+o(1)}$.
Similarly, by considering the first term of \eqref{eq:core-auxiliary-bound}
we see that we must have $R \log T \log N < T (\log \log N)^{1+o(1)}$.
The latter implies that $T > (\log N)^{1+o(1)}$
and hence that $\log T \gg \log \log N$,
so we are forced to take $R < T (\log \log N)^{o(1)} / \log N$.
To simplify the following discussion,
we will delete the $(\log \log N)^{o(1)}$ factors
and insist on the slightly stronger bounds
$\log T \ll \log \log N$ and $R \ll T / \log N$.
(Note that the first of these already appeared
as the hypothesis \eqref{eq:logT-bound} in \Cref{thm:core}.)

On the other hand,
we also need $R$ to be as least as large as the ``typical'' number of
odd square-primes in an interval of length~$T$.
Since the density of odd square-primes around $N$ is $\Theta(1/\log N)$
(see \Cref{rem:odd-square-primes-asymptotic}),
we are forced to take $R \gg T / \log N$, and hence 
\[
R \asymp \frac{T}{\log N}.
\]

Armed with these constraints on $T$ and $R$,
we can now briefly explain why we are led to consider a hypothesis of
the form \eqref{eq:conj-square-primes-weaker}.
To satisfy the constraints, suppose that we take for instance
\[
T = \log^A N + O(1), \qquad
R = C \log^{A-1} N + O(1)
\]
for some fixed $A > 1$ and $C > 0$.
For the algorithm to succeed on a given interval $(x, x+T]$,
we need to ensure that $\tilde\pi(x + T) - \tilde\pi(x) \leq R$
(here for simplicity we are again ignoring the ``odd'' condition),
or roughly speaking
\begin{equation}
\label{eq:tilde-pi-condition}
\tilde\pi(x + \log^A N) - \tilde\pi(x) < C \log^{A-1} N.
\end{equation}
Of course, it is unreasonable to expect this inequality to hold for
values of $x$ that are too small relative to~$N$,
as the density of square-primes will be higher for such intervals.
So we should only require \eqref{eq:tilde-pi-condition} to hold
for~$x$ in some wide range near~$N$, say $N/\log N < x < N$.
For such~$x$ we may approximate $\log N$ by $\log x$,
leading immediately to a hypothesis of the form
\eqref{eq:conj-square-primes-weaker}.

The distribution of square-primes in such short intervals
does not seem to have been studied before,
but the corresponding question for \emph{primes} in short intervals
has received considerable attention.
In the next section,
we begin our heuristic discussion by reviewing
what is known (or believed) for the prime case.

\subsection{Primes in very short intervals}
\label{sec:primes-very-short}

The prime number theorem suggests that
the number of primes in the interval $(x, x+y]$, for $y = o(x)$,
should be roughly $y / \log x$.
Let us consider the function
\[
D(x,y) \coloneqq \frac{\pi(x+y) - \pi(x)}{y / \log x},
\]
which measures the deviation from this predicted value.
For large enough $y = y(x)$,
we do indeed have $D(x,y) \to 1$ as $x \to \infty$.
For example, 
it is known that ${D(x,y) \to 1}$ holds for $y = x^{7/12 + \eps}$
\cite[\S10.5]{IK-analytic}.
Even for $y$ as small as $x^{0.525}$,
it is known at least that $D(x,y) \asymp 1$
(see \cite{BHP-diff} and \cite[Thm.\,3.9]{MV-mult-nt}
for the lower and upper bounds respectively).
But these intervals are far too large for our purposes;
as indicated above, we are interested in the case $y = \log^A x$
for fixed $A > 1$.
In this regime, very little is known.

It is believed that the cases $A \leq 2$ and $A > 2$
behave quite differently,
with a phase change in behaviour occurring around $A = 2$ \cite{GL-short}.
Henceforth we consider only the $A > 2$ case.
By analogy with our requirement \eqref{eq:conj-square-primes-weaker}
for the square-prime situation,
what we hope to be true is that $D(x, \log^A x) \ll 1$
for each fixed $A > 2$.

For $A > 2$,
Selberg showed (assuming the Riemann Hypothesis) that indeed
$D(x, \log^A x) \to 1$ as $x \to \infty$,
except possibly for $x$ lying in a sparse exceptional set
\cite{Selberg-normal-density}.
Unfortunately, ``almost all'' results are useless
for proving a result like \Cref{thm:main-heuristic}.
We need an estimate that holds for \emph{every} $x$,
as even a single exceptional interval renders
the output of the algorithm incorrect.

Analysing $D(x, y)$ rigorously seems hopeless for $y = \log^A x$,
so let us ask instead what a heuristic argument might look like.
A natural approach is to consider a Cram\'er-style probabilistic model
\cite{Cramer-model}.
Let us assume that each integer in $(x,x+y]$ is ``prime''
with probability $1/\log x$.
Under this model, $\pi(x + y) - \pi(x)$ is a random variable
following a binomial distribution,
with expected value $y/\log x = \log^{A-1} x$.
The tail probabilities may be estimated using Chernoff's inequality
\cite[Thm.\,4.4(1)]{MU-probability-computing};
for each $\delta > 0$, we obtain
\[
P(D(x,y) \geq 1+\delta) \leq \exp(-c_\delta \log^{A-1} x)
\]
for some constant $c_\delta > 0$.
In particular, if $A > 2$,
we see that the right hand side decays to zero rapidly as $x \to \infty$,
much faster than~$1/x$.
This suggests that for each $\delta > 0$,
with ``probability 1'' we should have
$D(x, \log^A x) < 1 + \delta$ for all sufficiently large~$x$.
A similar argument suggests that for each $\delta \in (0,1)$ we should have
$D(x, \log^A x) > 1 - \delta$ for all sufficiently large~$x$.
We might therefore be tempted to conclude that
for the ``real'' primes,
$D(x, \log^A x) \to 1$ (for \emph{all} $x$ without exception).

Rather surprisingly, this conclusion turns out to be incorrect.
Indeed, it was proven unconditionally by Maier \cite{Maier-short}
that for any fixed $A > 1$,
\[
\limsup_{x \to \infty} \, D(x, \log^A x) > 1,
\qquad
\liminf_{x \to \infty} \, D(x, \log^A x) < 1.
\]
In other words, for each $A > 1$,
there is some $\delta > 0$ (in principle explicitly computable)
such that $D(x, \log^A x) > 1 + \delta$
(respectively $D(x, \log^A x) < 1 - \delta$)
for a sequence of arbitrarily large values of~$x$.
The values of $\delta$ constructed by Maier are extremely small,
so this result does not immediately invalidate our hope that
$D(x, \log^A x) \ll 1$,
but it should perhaps shake the reader's confidence
in the reasonableness of the model.

One obvious way in which the model is unrealistic is that it does
not incorporate any information about congruences modulo small primes.
For example, it predicts that even numbers are occasionally prime.
The question of how this affects the distribution of $D(x,y)$
was investigated in detail by Granville and Lumley \cite{GL-short}.
Let us briefly sketch their explanation for
the unexpected behaviour of $D(x,y)$ when $y = \log^A x$.
For an interval $I$ of length $y$ and a given bound $z < y$,
let $I_z$ denote the set of integers in $I$ not divisible
by any prime $p \leq z$.
We expect $|I_z|$ to be well approximated by
$y \prod_{p \leq z} (p-1)/p$.
But there is some variation in~$|I_z|$,
and occasionally we encounter intervals $I$ for which $|I_z|$
is unusually large or small compared to this prediction,
due to the residues modulo small primes not being
completely independent on such short intervals.
For $y = \log^A x$
this effect is strong enough to bias the total number of primes in $I$
in such a way that $D(x, y)$ occasionally drifts away from~$1$.

Granville and Lumley propose a modified Cram\'er model
that explicitly takes into account the effect of small primes.
They show that under their model,
for $A > 2$ it is reasonable to expect that
\begin{equation}
\label{eq:GL-model}
\limsup_{x \to \infty} \, D(x, \log^A x) = \sigma_+(A),
\qquad
\liminf_{x \to \infty} \, D(x, \log^A x) = \sigma_-(A)
\end{equation}
for certain explicitly defined functions $\sigma_-(A) < 1 < \sigma_+(A)$.
(See \cite[\S2.2]{GL-short} for the definitions of $\sigma_\pm(A)$,
and \cite[\S9.3]{GL-short} for a statement of the above heuristic.
Strictly speaking \eqref{eq:GL-model} is slightly different to
the statement in \cite{GL-short},
but it is in the same spirit.)
In particular, under their model it is still reasonable to believe that
$D(x, \log^A x) \ll 1$,
which is our primary concern in this section.

Using the model from \cite{GL-short},
it is even possible to predict specific upper bounds
for $D(x, \log^A x)$.
Both $\sigma_+(A)$ and $\sigma_-(A)$ are continuous in $A$,
with $\sigma_+(A)$ decreasing towards $1$
and $\sigma_-(A)$ increasing towards $1$ as $A \to \infty$.
One also has the explicit bound
$1.015 \leq \sigma_+(2) \leq e^\gamma \approx 1.78107$
\cite[\S3.1]{GL-short}.
Since $\sigma_+(A)$ is decreasing, this implies that
$\limsup_{x \to \infty} \, D(x, \log^A x) < 1.79$
for all $A > 2$.
In particular, for each fixed $A > 2$,
we should expect that $D(x, \log^A x) < 1.79$
for all sufficiently large~$x$.
Tighter bounds could be calculated for specific values of $A$ if desired.

\begin{rem}
If we hold $y$ fixed and let $x \to \infty$,
then as explained in \cite[\S4]{GL-short},
the Hardy--Littlewood prime tuple conjecture \cite{HL-prime-tuples}
predicts that $\pi(x + y) - \pi(x)$
is occasionally as large as about $y / \log y$.
This extreme possibility is what makes it so difficult
to bound the number of primes in short intervals;
we would need to prove that such a fantastic coincidence
cannot occur ``too early''.
\end{rem}

\subsection{Modelling the density of square-primes}
\label{sec:modelling-density}

We would like to adapt the discussion in the previous section
from primes to square-primes.
The first step is to find the density of square-primes near $x$.
This turns out to be roughly $(\pi^2/6)/\log x$,
so the square-primes are denser than the primes by a factor of about
$\pi^2/6 \approx 1.6449$.
However, the ``correct'' density function also includes
some lower order terms.
We will take a few moments to work out these terms explicitly,
as they help explain some features of the
experimental data presented in \Cref{sec:square-primes-very-short}.

Consider the function
\[
\ell(x) \coloneqq
   \begin{cases}
   0 & x < 2,  \\[5pt]
   \displaystyle \int_2^x \frac{dt}{\log t} & x \geq 2,
   \end{cases}
\]
a slight modification of the usual logarithmic integral $\li(x)$.
It is continuous on~$\RR$
and differentiable everywhere except at $x = 2$.
It is well known that
\begin{equation}
\label{eq:pi-approximation}
\pi(x) = \ell(x) + O_k\left(\frac{x}{\log^k x}\right), \qquad x > 2
\end{equation}
for any $k \geq 2$ \cite[Thm.\,6.9]{MV-mult-nt}.

We may derive an analogue of \eqref{eq:pi-approximation}
for the square-primes as follows.
Taking our cue from \eqref{eq:tilde-pi-defn},
let us define
\[
\tilde\ell(x) \coloneqq \sum_{m \geq 1} \ell\left(\frac{x}{m^2}\right).
\]
This sum is finite for any~$x$,
as the terms for $m > (x/2)^{1/2}$ are all zero.
Hence $\tilde\ell(x)$ is continuous,
and it is differentiable everywhere except at $x = 2m^2$
for integers $m \geq 1$.
The analogue of \eqref{eq:pi-approximation} for the square-primes
is the following estimate.
\begin{lem}
\label{lem:tilde-pi-approximation}
For any $k \geq 2$,
\begin{equation}
\label{eq:tilde-pi-approximation}
\tilde\pi(x) = \tilde\ell(x) + O_k\left(\frac{x}{\log^k x}\right),
   \qquad x > 2.
\end{equation}
\end{lem}
\begin{proof}
We have
\begin{align*}
\tilde\pi(x)
   & = \sum_{1 \leq m \leq (x/2)^{1/2}} \pi\left(\frac{x}{m^2}\right) \\
   & = \sum_{1 \leq m \leq (x/2)^{1/2}}
       \left(\ell\left(\frac{x}{m^2}\right)
       + O_k\left(\frac{x/m^2}{\log^k(x/m^2)}\right)\right) \\
   & = \tilde\ell(x)
       + \sum_{1 \leq m \leq (x/2)^{1/2}}
         O_k\left(\frac{x/m^2}{\log^k(x/m^2)}\right).
\end{align*}
For $m \leq x^{1/4}$ we have $\log(x/m^2) \gg \log x$,
so the error term is at most
\[
\sum_{1 \leq m \leq x^{1/4}} \hspace{-8pt}
   O_k\left(\frac{x/m^2}{\log^k x}\right)
   + \sum_{m > x^{1/4}} \hspace{-5pt}
   O_k\left(\frac{x}{m^2}\right)
   \ll_k \frac{x}{\log^k x} + \frac{x}{x^{1/4}}
   \ll_k \frac{x}{\log^k x}. \qedhere
\]
\end{proof}

\begin{rem}
Under the Riemann Hypothesis,
the error term in \eqref{eq:pi-approximation} can be improved to
$x^{1/2+o(1)}$ \cite[Thm.\,13.1]{MV-mult-nt}.
It follows easily that the error term
in \eqref{eq:tilde-pi-approximation} also drops to $x^{1/2+o(1)}$,
providing some evidence that $\tilde\ell(x)$ is the correct analogue
of $\ell(x)$ for the square-primes.
\end{rem}

It is natural to model the density of the square-primes
by the derivative of~$\tilde\ell(x)$, which is given by
\begin{equation}
\label{eq:tilde-ell-derivative}
\tilde\ell'(x) =
   \sum_{1 \leq m \leq (x/2)^{1/2}} \frac{1}{m^2 \log(x/m^2)}.
\end{equation}
(For convenience we take $\tilde\ell'(x)$
to be defined by \eqref{eq:tilde-ell-derivative}
even at those points where $\tilde\ell(x)$ is not differentiable,
i.e., at the points $x = 2m^2$ for integers $m \geq 1$.)
To better understand the behaviour of $\tilde\ell'(x)$ for large~$x$,
we observe that it has the following expansion
in terms of powers of $\log x$.
\begin{lem}
\label{lem:tilde-ellprime-expansion}
For any $k \geq 2$ we have
\begin{equation}
\label{eq:tilde-ellprime-expansion}
\tilde\ell'(x) = \sum_{j=0}^{k-2} \frac{c_j}{(\log x)^{j+1}}
      + O_k\left(\frac{1}{\log^k x}\right),  \qquad x > 2
\end{equation}
where the constants $c_j$ are given by
\[
c_j \coloneqq (-2)^j \zeta^{(j)}(2)
   = \sum_{m \geq 1} \frac{(2 \log m)^j}{m^2}, \qquad j \geq 0.
\]
\end{lem}
Here $\zeta^{(j)}(s)$ denotes the $j$\th derivative
of the Riemann zeta function.
Thus for example
\[
c_0 = \zeta(2) = \pi^2/6 \approx 1.64493, \quad
c_1 = -2\zeta'(2) \approx 1.87510, \quad
c_2 = 4\zeta''(2) \approx 7.95712,
\]
and hence
\[
\tilde\ell'(x) = \frac{1.644\ldots}{\log x} + \frac{1.875\ldots}{\log^2 x}
   + \frac{7.957\ldots}{\log^3 x} + O\left(\frac{1}{\log^4 x}\right),
   \qquad x > 2.
\]
\begin{proof}
The contribution from the terms with $m > x^{1/4}$
in \eqref{eq:tilde-ell-derivative} is at most
\[
\sum_{m > x^{1/4}} O(m^{-2}) = O(x^{-1/4})
   = O_k\left(\frac{1}{\log^k x}\right).
\]
For the terms with $m \leq x^{1/4}$
we have $|{2\log m/\log x}| \leq \tfrac12$ and hence
\begin{multline*}
\frac{1}{\log(x/m^2)}
   = \frac{1}{\log x} \left(1 - \frac{2\log m}{\log x}\right)^{-1} \\
   = \frac{1}{\log x} \left(\sum_{j=0}^{k-2}
      \left(\frac{2\log m}{\log x}\right)^j
      + O_k\left(\left(\frac{2\log m}{\log x}\right)^{k-1}\right)\right).
\end{multline*}
Thus
\begin{multline*}
\tilde\ell'(x) = \sum_{j=0}^{k-2}
   \left( \sum_{1 \leq m \leq x^{1/4}}
      \frac{(2\log m)^j}{m^2} \right) \frac{1}{(\log x)^{j+1}} \\
   + \sum_{1 \leq m \leq x^{1/4}}
      O_k\left( \frac{(2\log m)^{k-1}}{m^2} \cdot \frac{1}{\log^k x} \right)
      + O_k\left(\frac{1}{\log^k x}\right).
\end{multline*}
In this expression, the second term simplifies to $O_k(\log^{-k} x)$.
As for the first term, observe that for each $j = 0, \ldots, k-2$ we have
\begin{multline*}
\sum_{1 \leq m \leq x^{1/4}} \frac{(2\log m)^j}{m^2}
   = c_j - \sum_{m > x^{1/4}} \frac{(2\log m)^j}{m^2}
   = c_j - \sum_{m > x^{1/4}} O_k(m^{-3/2}) \\
   = c_j + O_k(x^{-1/8})
   = c_j + O_k\left(\frac{1}{\log^k x}\right).   \qedhere
\end{multline*}
\end{proof}

\begin{rem}
It is not true that $\tilde\ell'(x) = \sum_{j \geq 0} c_j / (\log x)^{j+1}$,
because this sum does not converge for any~$x$.
Instead, this series must be interpreted as an asymptotic expansion,
i.e., one obtains a good approximation by summing finitely many terms,
where the optimal number of terms depends on~$x$.
The situation is analogous to the well-known asymptotic expansion
$\sum_{j \geq 0} j! x / (\log x)^{j+1}$ for $\li(x)$
\cite[\S6.12(i)]{NIST-DLMF},
which does not converge for any~$x$.
\end{rem}

An easy consequence of the preceding results is
the following estimate for $\tilde\pi(x)$
(which could also be proved directly from \eqref{eq:tilde-pi-defn}).
\begin{cor}
\label{cor:square-prime-density}
As $x \to \infty$,
\[
\tilde\pi(x) \sim \frac{\pi^2}{6} \cdot \frac{x}{\log x}.
\]
\end{cor}
\begin{proof}
By \Cref{lem:tilde-pi-approximation} and \Cref{lem:tilde-ellprime-expansion},
\begin{align*}
\tilde\pi(x)
   & = \tilde\ell(x) + O\left(\frac{x}{\log^2 x}\right) \\
   & = \int_2^x \tilde\ell'(t) \, dt + O\left(\frac{x}{\log^2 x}\right) \\
   & = \int_2^x
      \left(\frac{c_0}{\log t} + O\left(\frac{1}{\log^2 t}\right)\right) dt
      + O\left(\frac{x}{\log^2 x}\right) \\
   & = \frac{\pi^2}{6} \cdot \frac{x}{\log x}
      + O\left(\frac{x}{\log^2 x}\right). \qedhere
\end{align*}
\end{proof}

\begin{rem}
\label{rem:odd-square-primes-asymptotic}
The corresponding estimate for the number of \emph{odd} square-primes
up to $x$ is $(\pi^2/8) (x / \log x)$.
\end{rem}

\subsection{Square-primes in very short intervals}
\label{sec:square-primes-very-short}

Let us now return to our discussion of square-primes in very short intervals.
Using the proposed density function $\tilde\ell'(x)$
we may define an analogue of $D(x,y)$ for the square-primes:
\[
\tilde D(x,y) \coloneqq
   \frac{\tilde\pi(x+y) - \tilde\pi(x)}
      {\displaystyle y \cdot \tilde\ell'(x)}.
\]
For large intervals, $\tilde D(x,y)$ behaves as expected.
For example, one can show that $\tilde D(x,y) \to 1$
for $y = x^{7/12 + \eps}$,
using the analogous result for primes mentioned earlier.
However, we are chiefly interested in the regime where $y = \log^A x$;
to obtain our target bound \eqref{eq:conj-square-primes-weaker},
we want to show that
\begin{equation}
\label{eq:tilde-delta-bound}
\tilde D(x, \log^A x) \ll 1.
\end{equation}
Unfortunately, just as in the prime case,
we have no idea how to prove such a bound.
(It does not even seem possible to infer it from the corresponding
conjectured bound $D(x, \log^A x) \ll 1$ for the primes.)

What heuristic justification can we give for \eqref{eq:tilde-delta-bound}?
The author suspects that a model constructed
along similar lines to \cite{GL-short} would predict that
\begin{equation}
\label{eq:GL-model-squares}
\limsup_{x \to \infty} \, \tilde D(x, \log^A x) = \tilde\sigma_+(A),
\qquad
\liminf_{x \to \infty} \, \tilde D(x, \log^A x) = \tilde\sigma_-(A)
\end{equation}
for certain functions $\tilde\sigma_-(A) < 1 < \tilde\sigma_+(A)$.
Indeed, the arguments in \cite{GL-short} mainly rely on the fact
that primes can be detected by sieving;
square-primes can also be detected by sieving,
but working with congruences modulo~$p^k$ instead of modulo~$p$
(see \Cref{lem:sieve-criterion} and \Cref{rem:adjusted-Delta}).

We do not have the space (or frankly the inclination)
to carry out this plan in detail here.
Instead, we will conclude our discussion by presenting
the results of some computations of values of $\tilde D(x, \log^A x)$.
We hope that this data will convince the reader that
our definition of $\tilde D(x,y)$ is reasonable,
and that \eqref{eq:GL-model-squares} and \Cref{conj:square-primes} are
at the very least highly plausible.

\medskip
The computations were run on the HPC system \textsc{Katana}
at UNSW Sydney \cite{katana}.
We wrote a C++ program that uses a straightforward sieving procedure
to find all square-primes in a large interval,
typically of length $2^{34} \approx 1.7 \times 10^{10}$,
and then counts the square-primes in each subinterval of the form
$(x, x + \log^A x]$ for $A \in \{3, 4\}$.
The code is available upon request.
The computation checked all~$x$ up to $2^{52} \approx 4.5 \times 10^{15}$,
consuming roughly $92{,}000$ CPU hours altogether,
with each job requiring up to 2.6 GB of RAM.

The code found exactly $219{,}505{,}912{,}823{,}599$
square-primes less than $2^{52}$.
This matches the value computed using \eqref{eq:tilde-pi-defn}
together with the \texttt{prime\_pi} routine included in
\textsc{Sage} \cite{sage-10.6} for computing $\pi(x)$.
The latter calculation took about 9~minutes on the author's laptop.

\Cref{tab:tildeD} shows the maximum and minimum values
of $\tilde D(x, \log^A x)$ for $A \in \{3, 4\}$
over dyadic intervals from $2^{30}$ up to $2^{52}$.
The main interesting feature is that the values in each column
are very close to $1$ and remain highly stable as $x$ increases,
exactly as predicted by the square-prime analogue of the
Granville--Lumley heuristics mentioned above.
The simplest Cram\'er model would predict instead that the values
rapidly approach~$1$ as $x$ increases.

Let us examine some extreme values in more detail.
Consider the last row of the table,
covering the interval $2^{51} < x < 2^{52}$.
\begin{itemize}
\item
The maximum value of $\tilde D(x, \log^4 x)$
over this interval is $1.0185$.
\begin{itemize}
\item
This value occurs for $x = 3{,}783{,}664{,}593{,}821{,}531$.
\item
The expected number of square-primes in $(x, x + \log^4 x)$
is $\tilde\ell'(x) \log^4 x \approx 78676$.
\item
The actual number $80134$ overshoots this estimate by $1.85\%$.
\item
If we had estimated the density using only the leading term of
\eqref{eq:tilde-ellprime-expansion},
the expected number would have been
$(\pi^2/6) \log^3 x \approx 75914$.
The observed value $80134$ overshoots this estimate by $5.56\%$.
\end{itemize}
\item
The minimum value of $\tilde D(x, \log^4 x)$
over this interval is $0.9824$.
\begin{itemize}
\item
This value occurs for $x = 3{,}953{,}337{,}990{,}092{,}939$.
\item
The expected number of square-primes in $(x, x + \log^4 x)$
is $\tilde\ell'(x) \log^4 x \approx 78961$.
\item
The actual number $77571$ undershoots this estimate by $1.76\%$.
\item
If we had estimated the density using only the leading term of
\eqref{eq:tilde-ellprime-expansion},
the expected number would have been
$(\pi^2/6) \log^3 x \approx 76193$.
The observed value $77571$ \emph{overshoots} this estimate by $1.81\%$.
\end{itemize}
\end{itemize}
In this example we see strong evidence that the secondary terms
in \Cref{lem:tilde-ellprime-expansion} carry useful information:
including these terms shifts the ``error window'' for $\tilde D(x,y)$
from the unbalanced interval $(0.0181, 0.0556)$ to the
much more symmetric interval $(-0.0176, 0.0185)$.

As for \Cref{conj:square-primes} itself,
the data in \Cref{tab:tildeD} suggests that the quantity
$(\tilde\pi(x + \log^3 x) - \tilde\pi(x)) / \log^2 x$
should in the long run be bounded by a number close to
$1.13 \times \pi^2/6 \approx 1.86$.
In fact, this quantity does take somewhat larger values for small~$x$
(not shown in the table),
but even for these~$x$ it never exceeds~$2.1$,
and is certainly nowhere near the value~$3$ proposed by
\Cref{conj:square-primes}.






\begin{table}
\begin{tabular}{lcccc}
\toprule
& \multicolumn{2}{c}{$y = \log^3 x$} & \multicolumn{2}{c}{$y = \log^4 x$} \\ \cmidrule(lr){2-3} \cmidrule(lr){4-5}
& $\min \tilde D(x,y)$ & $\max \tilde D(x,y)$
& $\min \tilde D(x,y)$ & $\max \tilde D(x,y)$ \\
\midrule
$2^{30} < x < 2^{31}$ & 0.8760 & 1.1207 & 0.9797 & 1.0199 \\
$2^{31} < x < 2^{32}$ & 0.8758 & 1.1278 & 0.9813 & 1.0183 \\
$2^{32} < x < 2^{33}$ & 0.8655 & 1.1344 & 0.9806 & 1.0194 \\
$2^{33} < x < 2^{34}$ & 0.8730 & 1.1204 & 0.9809 & 1.0199 \\
$2^{34} < x < 2^{35}$ & 0.8789 & 1.1285 & 0.9805 & 1.0183 \\
$2^{35} < x < 2^{36}$ & 0.8741 & 1.1234 & 0.9802 & 1.0208 \\
$2^{36} < x < 2^{37}$ & 0.8646 & 1.1257 & 0.9809 & 1.0185 \\
$2^{37} < x < 2^{38}$ & 0.8711 & 1.1220 & 0.9793 & 1.0204 \\
$2^{38} < x < 2^{39}$ & 0.8719 & 1.1224 & 0.9796 & 1.0190 \\
$2^{39} < x < 2^{40}$ & 0.8729 & 1.1264 & 0.9811 & 1.0198 \\
$2^{40} < x < 2^{41}$ & 0.8776 & 1.1277 & 0.9790 & 1.0214 \\
$2^{41} < x < 2^{42}$ & 0.8718 & 1.1290 & 0.9797 & 1.0191 \\
$2^{42} < x < 2^{43}$ & 0.8738 & 1.1330 & 0.9804 & 1.0193 \\
$2^{43} < x < 2^{44}$ & 0.8632 & 1.1299 & 0.9807 & 1.0209 \\
$2^{44} < x < 2^{45}$ & 0.8696 & 1.1296 & 0.9803 & 1.0196 \\
$2^{45} < x < 2^{46}$ & 0.8762 & 1.1262 & 0.9808 & 1.0186 \\
$2^{46} < x < 2^{47}$ & 0.8760 & 1.1249 & 0.9822 & 1.0182 \\
$2^{47} < x < 2^{48}$ & 0.8775 & 1.1332 & 0.9810 & 1.0190 \\
$2^{48} < x < 2^{49}$ & 0.8760 & 1.1271 & 0.9808 & 1.0192 \\
$2^{49} < x < 2^{50}$ & 0.8768 & 1.1266 & 0.9825 & 1.0178 \\
$2^{50} < x < 2^{51}$ & 0.8772 & 1.1225 & 0.9822 & 1.0181 \\
$2^{51} < x < 2^{52}$ & 0.8757 & 1.1239 & 0.9824 & 1.0185 \\
\bottomrule
\end{tabular}
\caption{Minimum and maximum values of $\tilde D(x,\log^A x)$
   over dyadic intervals from $2^{30}$ up to $2^{52}$,
   for $A = 3$ and $A = 4$.}
\label{tab:tildeD}
\end{table}

\subsection{The main heuristic algorithm}
\label{sec:heuristic-algorithm}

In this section we prove \Cref{thm:main-heuristic}.
We first discuss a few subroutines.

\begin{prop}[\cite{LO-computing-pi}]
\label{prop:count-primes}
There is a deterministic algorithm that counts the exact number
of primes $p < N$, for a given integer $N \geq 2$,
in time $N^{1/2 + o(1)}$.
\end{prop}
\begin{rem}
For our purposes, any algorithm that solves this problem
in time $O(N^\alpha)$ for some fixed $\alpha < 1$ would suffice.
This includes for instance the ``combinatorial'' $N^{2/3+o(1)}$
method of \cite{LMO-computing-pi}.
\end{rem}

\begin{prop}
\label{prop:count-square-primes}
There is a deterministic algorithm that counts the exact number
of odd square-primes $n < N$, for a given integer $N \geq 3$,
in time $N^{1/2 + o(1)}$.
\end{prop}
\begin{proof}
Let $\pi_0(x)$ denote the number of primes \emph{strictly} less than~$x$.
Then the number of odd square-primes less than~$N$ is given by
\[
\sum_{\substack{1 \leq m < (N/3)^{1/2} \\ \textn{$m$ odd}}}
   \left( \pi_0\left(\frac{N}{m^2}\right) - 1 \right).
\]
Invoking \Cref{prop:count-primes} for each~$m$,
this quantity may be computed in time
\begin{align*}
\sum_{m < (N/3)^{1/2}} (N/m^2)^{1/2+o(1)}
   & < N^{1/2+o(1)} \sum_{m < (N/3)^{1/2}} m^{-1} \\
   & < N^{1/2+o(1)} \log N < N^{1/2+o(1)}.   \qedhere
\end{align*}
\end{proof}

\begin{prop}
\label{prop:primes-to-square-primes}
Let $N \geq 2$.
Given as input a sorted list of all primes $p < N$,
we may compute a sorted list of all odd square-primes $n < N$ in time $O(N)$.
\end{prop}
\begin{proof}
The set of odd square-primes less than~$N$ is
the (disjoint) union of the sets
\[
S_m \coloneqq \{ m^2 p : \textn{$p$ prime}, \, 3 \leq p < N/m^2 \},
   \qquad 1 \leq m < (N/3)^{1/2}, \quad \textn{$m$ odd}.
\]
The idea of the algorithm is to use the provided list of primes to generate
the $S_m$ as sorted lists, and then merge the lists via \Cref{lem:merge}.

To achieve the $O(N)$ complexity bound,
it is important to merge the lists in the right order.
Let us group the $S_m$ dyadically, defining
\begin{equation}
\label{eq:Jk-defn}
J_k \coloneqq \{2^{k-1} \leq m < 2^k : \textn{$m$ odd}\},
   \qquad 1 \leq k \leq K, \quad K \coloneqq \lg((N/3)^{1/2}).
\end{equation}
Computing the products $m^2 p = m(mp)$ via classical long multiplication,
the cost of generating all the lists $\{S_m\}_{m \in J_k}$ for a given $k$ is
\[
\sum_{m \in J_k} |S_m| \cdot O(\log m \log N)
   = O\biggl(k \log N \sum_{m \in J_k} |S_m|\biggr).
\]
(Here for convenience we have put $S_m \coloneqq \emptyset$
for $m \geq (N/3)^{1/2}$.)
We then merge together the $S_m$ for $m \in J_k$ via
a sequence of pairwise merges forming a binary tree.
The tree depth is $\log_2 |J_k| + O(1) = k + O(1)$
so the cost of this step is also
\[
O\biggl(k \log N \sum_{m \in J_k} |S_m|\biggr).
\]
Now observe that
\[
|S_m| = \pi(N/m^2) + O(1) \ll \frac{N}{m^2 \log(N/m^2)}.
\]
To simplify the subsequent calculations
it will be convenient to replace this by the weaker bound
\[
|S_m| \ll \frac{N}{m^{3/2} \log N}.
\]
To justify this, one easily checks that $m^{1/2} \log(N/m^2) \gg \log N$
holds for $1 \leq m \leq (N/3)^{1/2}$,
by considering separately the cases $m \leq N^{1/4}$ and $m > N^{1/4}$.
Thus the union $S^k \coloneqq \cup_{m \in J_k} S_m$ has cardinality
\[
|S^k| \ll \sum_{m \in J_k} \frac{N}{2^{3k/2} \log N}
   \ll \frac{N}{2^{k/2} \log N}
\]
and the cost of computing $S^k$ as a sorted list is
\[
O\left(k \log N \sum_{m \in J_k} \frac{N}{2^{3k/2} \log N}\right)
   = O\left(\frac{k N}{2^{k/2}}\right).
\]
Since $\sum_{k \geq 1} k \cdot 2^{-k/2}$ converges,
the total cost over all~$k$ is $O(N)$.

Finally we must merge the~$S^k$ into a single sorted list.
We process these in reverse order:
we begin with~$S^K$, then merge with $S^{K-1}$ to obtain $S^{K-1} \cup S^K$,
then merge with~$S^{K-2}$, and so on until reaching
$S^1 \cup \cdots \cup S^K$.
For each~$k$ we have
\[
|S^{k+1} \cup \cdots \cup S^K|
   \ll \sum_{j=k+1}^K \frac{N}{2^{j/2} \log N}
   \ll \frac{N}{2^{k/2} \log N},
\]
so the cost of merging $S^k$ with $S^{k+1} \cup \cdots \cup S^K$ is
$O(2^{-k/2} N)$,
and the total over all~$k$ is again $O(N)$.
\end{proof}

\begin{prop}
\label{prop:square-primes-to-primes}
Let $N \geq 2$.
Given as input a sorted list of all odd square-primes $n < N$,
we may compute a sorted list of all primes $p < N$ in time $O(N)$.
\end{prop}
\begin{proof}
Let~$\tilde S$ be the provided list of odd square-primes less than~$N$.
Our goal is to compute the set
$S \coloneqq \{\textn{$p$ prime}: 3 \leq p < N\}$
as a sorted list.
(The prime $p = 2$ may be inserted separately at the end.)

Let
\[
U(x) \coloneqq \sum_{p \geq 3} x^p \in \ZZ\bbracket{x}
\]
and
\[
\tilde U(x) \coloneqq \sum_{\substack{m \geq 1 \\ \textn{$m$ odd}}} U(x^{m^2})
   = \sum_{\substack{\textn{$n$ odd,} \\ \textn{square-prime}}}
      \hspace{-10pt} x^n
   \in \ZZ\bbracket{x}.
\]
M\"obius inversion yields
\[
\sum_{\substack{m \geq 1 \\ \textn{$m$ odd}}} \mu(m) \tilde U(x^{m^2}) = U(x),
\]
so
\begin{equation}
\label{eq:mobius}
\sum_{\substack{m \geq 1 \\ \textn{$m$ odd, square-free}}}
   \hspace{-20pt} \tilde U(x^{m^2})
   \equiv U(x) \pmod 2.
\end{equation}
Now define
\[
\tilde S_m \coloneqq \{m^2 n : n \in \tilde S, \; 3 \leq n < N/m^2\},
   \qquad 1 \leq m < (N/3)^{1/2}, \quad \textn{$m$ odd}.
\]
Let~$S'$ be the union of the~$\tilde S_m$
\emph{as multisets}, i.e., allowing multiplicities,
taken over odd square-free values of $m < (N/3)^{1/2}$.
Then \eqref{eq:mobius} shows that a given integer~$n$
lies in~$S$ if and only if the multiplicity of~$n$ in~$S'$ is odd.

These observations lead to the following algorithm.
We first compute a list of all positive odd
square-free integers $m < (N/3)^{1/2}$;
this is easily done in time $N^{1/2+o(1)}$ using a sieve.
We next compute the relevant~$\tilde S_m$, and their multiset union~$S'$,
using essentially the same algorithm as in the proof
of \Cref{prop:primes-to-square-primes},
i.e., merging the sets using the same dyadic grouping strategy.
The complexity analysis is essentially identical,
as \Cref{cor:square-prime-density} implies that
$|\tilde S_m| \ll (N/m^2) / \log(N/m^2)$.
The only difference is that the sorted lists may contain repeated elements,
but this does not affect the complexity.
Finally, at the end we scan through $S'$ and
delete pairs of repeated integers to recover~$S$;
the cost of this step is certainly at most $O(N)$.
\end{proof}

Now we present the main heuristic algorithm.
We are given some large~$N$ as input.
The main steps of the algorithm are:
\begin{enumalgo}
\item
Choose suitable parameters $T$, $R$ and $N_0 \approx N/\log N$.
\item
Find all primes up to~$N_0$ using Sergeev's algorithm,
and deduce the list of odd square-primes up to~$N_0$
(\Cref{prop:primes-to-square-primes}).
\item
Use the core algorithm (\Cref{thm:core})
to compute candidates for the odd square-primes between $N_0$ and~$N$.
\item
Determine the exact number of odd square-primes less than~$N$
(\Cref{prop:count-square-primes}).
If the count disagrees with the proposed list from Steps 1--3,
return ``FAIL''.
(This never occurs if \Cref{conj:square-primes} is true.)
\item
Otherwise we have correctly found the odd square-primes up to~$N$.
Conclude by recovering the list of primes up to~$N$
(\Cref{prop:square-primes-to-primes}).
\end{enumalgo}
We now explain and analyse each of these steps in more detail.

\medskip
\step{1}{choose parameters.}
Choose positive integers
\begin{align*}
   T & \coloneqq \tfrac12 \log^3 N + O(1), \\
   R & \coloneqq 4 \log^2 N + O(1), \\
 N_0 & \coloneqq \frac{N}{\log N} + O(T), \qquad T \divides N_0.
\end{align*}
Increasing $N$ slightly if necessary,
we may also assume that $T \divides N$.
Clearly these parameters may be computed in time $(\log N)^{O(1)}$.

\medskip
\step{2}{find odd square-primes up to~$N_0$.}
Using Sergeev's algorithm
(see \cite{Ser-prime-turing,Ser-prime-turing-arxiv},
or the presentation in \Cref{sec:sergeev})
compute the list of primes $p < N_0$.
This requires time
\[
O(N_0 \log N_0) = O(N).
\]
We may then deduce the list of odd square-primes $n < N_0$
via \Cref{prop:primes-to-square-primes}
in time $O(N_0) = O(N / \log N)$.

\medskip
\step{3}{find candidates for the odd square-primes between $N_0$ and~$N$.}
We invoke \Cref{thm:core},
noting that the hypothesis $R \leq T/2$ holds for sufficiently large~$N$,
and \eqref{eq:logT-bound} certainly holds for our choice of~$T$.
Since $(R \log N)/T \ll 1$ the cost is
\[
N (\log \log N)^{1+o(1)}.
\]
This produces candidates $c^r \in \FF_2^T$
for the true vectors $a^r \in \FF_2^T$ for $0 \leq r < N/T$.

Let $r_0 \coloneqq N_0 / T$.
We are only interested in the candidates $c^r$
for $r$ in the range $r_0 \leq r < N/T$,
as the corresponding $a^r$ indicate the square-primality of all
integers in the interval $N_0 \leq n < N$.
For such~$r$, the number of odd square-primes in the
$r$\th interval is at most
$\tilde\pi(x + T) - \tilde\pi(x)$ where $x \coloneqq rT - 1$.
We have $x \geq r_0 T - 1 = N_0 - 1 = N/\log N + O(\log^3 N)$,
so $\log x \geq \log N + O(\log \log N)$.
It follows easily that
\[
T = \tfrac12 \log^3 N + O(1) < \log^3 x
\]
for sufficiently large~$N$.
Similarly, $R > 3 \log^2 x$ for sufficiently large~$N$.
Therefore, if \Cref{conj:square-primes} is true, we find that
\[
\tilde\pi(x + T) - \tilde\pi(x)
   \leq \tilde\pi(x + \log^3 x) - \tilde\pi(x) < 3 \log^2 x < R.
\]
This means that $\wt(a^r) \leq R$
and hence the candidate $c^r$ is correct, i.e., $c^r = a^r$,
for all such~$r$.

\begin{rem}
The constants in the definitions of $T$ and~$R$ were chosen
somewhat wastefully to simplify the analysis.
In a real implementation,
one could tune these constants to improve the performance.
\end{rem}

\step{4}{compare against the actual number of odd square-primes
   less than~$N$.}
In time $O(N)$
we may count the number of odd square-primes $n < N_0$ found in Step~2
and the number of odd square-prime candidates $n \in [N_0, N)$
found in Step~3.
We compare the total to the actual number of odd square-primes less than~$N$,
computed via \Cref{prop:count-square-primes} in time $N^{1/2+o(1)}$.
If all of the candidates are correct, these numbers will obviously agree.
But if even a single candidate is incorrect,
then the total count of odd square-prime candidates will be
strictly \emph{smaller} than the correct number,
because for such $r$ we necessarily have $\wt(a^r) > R \geq \wt(c^r)$.
If this occurs then we (correctly) return ``FAIL''.
(We stress that this argument does \emph{not} depend on
the truth or otherwise of \Cref{conj:square-primes}.)

\medskip
\step{5}{recover list of primes up to~$N$.}
Using \Cref{lem:convert-format},
we may convert the vectors $c^r$ ($ = a^r$) into a list of odd square-primes
less than~$N$ in time $O(N)$.
We then invoke \Cref{prop:square-primes-to-primes}
to recover the list of primes less than~$N$, again in time $O(N)$.
This completes the proof of \Cref{thm:main-heuristic}.

\begin{rem}
\label{rem:main-heuristic-precise}
In the above analysis, the dominant contribution to the running time
was the invocation of \Cref{thm:core}.
The remaining steps have negligible $O(N)$ complexity.
It is possible to derive a more precise complexity bound by analysing
the terms in \eqref{eq:core-auxiliary-bound}.
Given our current state of knowledge,
the most expensive term in \eqref{eq:core-auxiliary-bound}
is the fifth term $N \log B \cdot \Rstar(B) \Kstar(W) \Kstar(N)$,
which arises from the restricted products
in the pointwise multiplication step.
If we use the estimates \eqref{eq:Kcost-bound-improved},
\eqref{eq:restricted-advanced-complexity} and \eqref{eq:ffmul}
for $\Kstar$, $\Rstar$ and~$\Mstar$,
then this term, and hence the overall cost of \Cref{thm:main-heuristic},
becomes
\begin{equation}
\label{eq:heuristic-precise}
O\bigl(N \log \log N \exp\bigl(D (\log \log \log N)^{1/2}\bigr)
   (\log \log \log N)^{1/2} \cdot 32^{\log^* N}\bigr)
\end{equation}
where $D \coloneqq (2 \log 2)^{1/2} \approx 1.177$.
It is plausible that the constant $32$ could be reduced by fusing together
the evaluation strategies used at different levels of the algorithm.
But apart from this,
and possibly the small speedup hinted at in \Cref{rem:varphi-W-speedup}
(at most a factor of $\log \log \log N$),
the complicated formula \eqref{eq:heuristic-precise}
represents the best complexity bound known to the author
for enumerating the primes up to~$N$ under any
hypothesis of comparable strength to \Cref{conj:square-primes}.
\end{rem}

\section{Sieve estimates for square-rough integers}
\label{sec:sieve}

If one believes \Cref{conj:square-primes},
then there do not exist any ``short'' intervals
containing ``too many'' square-primes.
Unfortunately, we are unable to prove this conjecture.
The aim of this section is to prove \Cref{thm:sieve-main},
which guarantees that such intervals occur \emph{very rarely}.
This result will play a key role in the analysis of the
rigorous versions of the main algorithm
(Theorems \ref{thm:main-deterministic} and \ref{thm:main-probabilistic}).

In fact, for the deterministic version
we will need an even more general formulation involving
\emph{square-rough} integers, which we now define.
\begin{defn}
Let $y > 1$ be a real number and let $n$ be a positive integer.
\begin{itemize}
\item We say that $n$ is \emph{$y$-rough}
if it has no prime divisors $p \leq y$.
(This terminology is not as common in the literature 
as the complementary notion of a \emph{$y$-smooth} integer,
but see for example \cite[Ch.\,14]{Kou-primes}.)
\item We say that $n$ is \emph{square-$y$-rough} if it is
of the form $n = m^2 n'$ where $m \geq 1$ is an integer
and $n'$ is $y$-rough.
Equivalently, for all primes $p \leq y$,
the $p$-adic valuation of $n$ is \emph{even}.
\end{itemize}
\end{defn}
In particular, if $n \leq N$ and $n$ is square-$\sqrt{N}$-rough
then $n$ is either a square-prime or a perfect square.
(The converse is almost true: see \Cref{lem:square-prime-exceptions}.)

The main result of this section is the following theorem.
\begin{thm}
\label{thm:sieve-main}
There exists an absolute constant $C_0 > 1$ with the following property.
Let $g$, $U$ and $N$ be positive integers such that $U \divides N$
and let $y$ be a real number satisfying
\begin{equation}
\label{eq:y-hypothesis}
\exp(C_0 g^4 \log 5U \log(4g^2 \log 5U)) < y \leq N^{1/2}.
\end{equation}
Then the number of integers $k \in [0, N/U)$ such that
there are at least $g$ square-$y$-rough integers in the interval
$[kU, (k+1)U)$ is at most
\[
O\left( \frac{N}{U} \left(\frac{\alpha U}{\log y}\right)^g \right)
\]
where
\begin{equation}
\label{eq:alpha-defn}
\alpha \coloneqq 6 \log(4g^2 \log 5U)
   \prod_{p^2 \divides U} \left(1 - \frac{1}{p} \right).
\end{equation}
\end{thm}

\begin{rem}
The constant $C_0$ is effectively computable,
and could be determined by working out explicit values for
various big-$O$ constants in this section.
\end{rem}


\subsection{Setting up the sieve}
\label{sec:sieve-setup}

We will prove \Cref{thm:sieve-main} using sieve methods,
working one ``constellation'' at a time.
Let $g$ and $U$ be positive integers,
and fix a $g$-tuple of pairwise distinct integers
$u = (u_1, \ldots, u_g)$ such that $0 \leq u_i < U$ for all $i$.
The bulk of \Cref{sec:sieve} is devoted to the proof of
\Cref{thm:sieve-tuple} (below),
which bounds the number of $k \in [0, X)$ such that
$Uk + u_1, \ldots, Uk + u_g$ are simultaneously square-$z$-rough,
for given parameters $X$ and $z$.
Afterwards, in \Cref{sec:sieve-main-proof},
we will show how to deduce \Cref{thm:sieve-main} from \Cref{thm:sieve-tuple}
by summing the contributions from all possible tuples $u$.

To state \Cref{thm:sieve-tuple} we must introduce several auxiliary quantities.
For any prime $p$ and $i \in \{1, \ldots, g\}$ define
\begin{equation}
\label{eq:Bpi-defn}
B_{p,i} \coloneqq \{k \in \ZZ/p^2\ZZ :
   Uk + u_i \equiv 0 \tpmod p \textn{ and }
   Uk + u_i \not\equiv 0 \tpmod{p^2} \},
\end{equation}
and put
\[
B_p \coloneqq \bigcup_{1 \leq i \leq g} B_{p,i}
   \,\, \subseteq \ZZ/p^2\ZZ.
\]
Note that $B_p$ depends of course on $u$.
The following result shows that the target integers $k$ specified
in \Cref{thm:sieve-tuple} survive after sieving out
the residue classes in $B_p$ for $p \leq z$.
\begin{lem}
\label{lem:sieve-criterion}
Let $z > 1$ be a real number and let $k \geq 0$ be an integer.
If $Uk + u_1$, $\ldots, Uk + u_g$ are all square-$z$-rough
then $(k \bmod p^2) \notin B_p$ for all primes $p \leq z$.
\end{lem}
\begin{proof}
If $(k \bmod p^2) \in B_p$ for some $p \leq z$,
then $(k \bmod p^2) \in B_{p,i}$ for some $i$,
so $p \divides Uk + u_i$ and $p^2 \ndivides Uk + u_i$.
By hypothesis $Uk + u_i$ is square-$z$-rough,
say $Uk + u_i = m^2 n$ where $n$ is $z$-rough.
But then $p \ndivides n$ (as $p \leq z$),
so $p \divides m$, and then $p^2 \divides Uk + u_i$, a contradiction.
\end{proof}
\begin{rem}
Zero is not considered to be square-$z$-rough (or $z$-rough),
so \Cref{lem:sieve-criterion} holds vacuously when $k = 0$ and
$u_i = 0$ for some $i$.
\end{rem}

We say that a prime $p$ is \emph{bad}
(with respect to the parameters $g$, $U$ and $u$)
if any of the following hold:
\begin{enumerate}[label={(\roman*)}]
\item
$p \divides U$,
\item
$p$ divides some $u_i - u_j$ (for $i \neq j$), or
\item
$p \leq 4g^2$.
\end{enumerate}
Otherwise $p$ is \emph{good}.
Clearly there are only finitely many bad primes.
Define
\[
Q = Q(u) \coloneqq \prod_{\textn{$p$ bad}} p.
\]
Certainly $Q \geq 6$, as (iii) implies that $p = 2, 3$ are bad.
\begin{lem}
\label{lem:p-good}
If $p$ is good then
\[
|B_p| = g(p-1) < p^2.
\]
\end{lem}
\begin{proof}
Let $i \in \{1, \ldots, g\}$.
Since $p \ndivides U$,
there are $p$ values of $k \pmod{p^2}$ such that $Uk + u_i \equiv 0 \pmod p$,
and exactly one of these satisfies $Uk + u_i \equiv 0 \pmod{p^2}$.
Thus $|B_{p,i}| = p-1$.
The condition $p \ndivides u_i - u_j$ implies that the $B_{p,i}$ are disjoint,
so $|B_p| = g(p-1)$.
Since $p > 4g^2 > g$ we obtain $g(p-1) < p^2$.
\end{proof}
Finally we set
\begin{equation}
\label{eq:Delta-defn}
\Delta = \Delta(u) \coloneqq \prod_p
   \left(1 - \frac{|B_p|}{p^2}\right) \left(1 - \frac1p \right)^{-g}.
\end{equation}
\Cref{lem:p-good} implies that the factor $1 - |B_p|/p^2$ does not vanish
for the good primes.
If $|B_p| = p^2$ for any bad prime~$p$ then $\Delta = 0$.
Otherwise, as we will see later (\Cref{cor:Fp-converges}),
the product in \eqref{eq:Delta-defn} converges absolutely and $\Delta \neq 0$.

We may now state the main estimate for a fixed tuple $u = (u_1, \ldots, u_g)$.
We stress that the implied big-$O$ constants
in the following result are absolute,
and do not depend on $g$, $U$, or $u$.

\begin{thm}
\label{thm:sieve-tuple}
There exists an absolute constant $C_1 > 1$ with the following property.
Let $g$, $U$, $u = (u_1, \ldots, u_g)$,
$Q = Q(u)$ and $\Delta = \Delta(u)$ be as above.
Let
\begin{equation}
\label{eq:theta-defn}
\theta = \theta(g,Q) \coloneqq g^2 \log 5Q \log \log 5Q.
\end{equation}
Then for any integer $X \geq 1$ and any real number $z > e^{C_1 \theta}$,
the number of integers $k \in [0,X)$ such that
$Uk + u_1, \ldots, Uk + u_g$ are all square-$z$-rough is at most
\[
\frac{g! \Delta X}{(\log z)^g}
   \left(1 + O\left(\frac{\theta}{\log z} \right)\right)
   + O\left( \frac{z^2 (\log z)^{4g}}{(2g)!^2} \right).
\]
\end{thm}

Note that the hypothesis $z > e^{C_1 \theta}$
(along with $C_1 > 1$ and $Q \geq 6$)
implies that $z > Q$ and hence that $p \leq z$ for all bad primes~$p$.
If $|B_p| = p^2$ for any of these primes,
then \Cref{lem:sieve-criterion} implies that
$Uk + u_1, \ldots, Uk + u_g$ can never be simultaneously square-$z$-rough,
so in this case \Cref{thm:sieve-tuple} holds trivially
(and of course $\Delta = 0$).
Henceforth we may assume that
\begin{equation}
\label{eq:Bp-assumption}
|B_p| < p^2 \qquad \text{for all primes $p$}.
\end{equation}

\begin{rem}
\label{rem:adjusted-Delta}
It is likely that the constant $\Delta$ in \Cref{thm:sieve-tuple}
can be replaced by a smaller value $\Delta'$
that takes into account congruences modulo higher powers of~$p$,
i.e., instead of just sieving out residues
$\{x \equiv 0 \pmod p, \, x \not\equiv 0 \pmod{p^2}\}$,
we could also sieve out
$\{x \equiv 0 \pmod{p^3}, \, x \not\equiv 0 \pmod{p^4}\}$ and so on.
A rough analogy is that the density of squarefree integers is given by
$\prod_p (1 + p^{-2} + p^{-4} + \cdots)^{-1} \approx 0.6079$
rather than $\prod_p (1 + p^{-2})^{-1} \approx 0.6580$.
The author has not checked the details.
\end{rem}

\begin{rem}
Our proof of \Cref{thm:sieve-tuple} will use only
standard techniques from sieve theory,
but the author was unable to find existing statements
in the literature that could be used ``off-the-shelf''
in place of the theorem.
The closest available results are of the following type:
given an integer tuple $r = (r_1, \ldots, r_g)$,
the number of $n < X$ such that $n + r_1, \ldots, n + r_g$ are
simultaneously \emph{prime} is at most
\begin{equation}
\label{eq:prime-tuples}
\frac{2^g g! D_r X}{(\log X)^g} (1 + o_r(1)),
   \qquad X \to \infty,
\end{equation}
where $D_r$ is a product of local densities depending on~$r$.
Examples of statements along these lines include
\cite[Thm.\,5.7]{HR-sieves},
\cite[Thm.\,4, \S2.3.3]{Greaves-sieves}
(note that the definition of $\pi(f,X)$ is incorrect;
the condition ``$f(n)$ is prime'' should read
``$f(n)$ is a product of $k$ primes''),
\cite[Thm.\,6.7]{IK-analytic},
\cite[Eq.\,3.46]{MV-mult-nt},
\cite[Thm.\,7.16]{FI-opera}
and \cite[Cor.\,21.3]{Kou-primes}.
These sources also (implicitly) give asymptotic formulas for
the more general situation in which ``prime'' is replaced by ``$z$-rough''.
The most famous case of \emph{twin primes}
corresponds to taking $g = 2$, $r = (0, 2)$.

There are two main differences between \Cref{thm:sieve-tuple}
and the results just listed.
First, as we are interested in square-primes rather than primes,
we must sieve out residue classes modulo $p^2$ instead of modulo~$p$
(see \Cref{lem:sieve-criterion}).
Second, the $o_r(1)$ error terms in the abovementioned results
are not uniform in~$g$.
In the applications in
\Crefrange{sec:probabilistic}{sec:deterministic},
the value of $g$ will grow slowly as a function of~$X$
(roughly like $\log \log X$),
so we require an error term that is fully explicit in terms of~$g$.

There are very few works that keep track of the error term for varying~$g$.
One example is \cite{Elsholtz-tuples},
which gives an upper bound for the number of prime $g$\nobreakdash-tuples
up to $X$ when $g \asymp \log X$.
This choice of $g$ is much larger than in our setting,
and the resulting bounds in \cite{Elsholtz-tuples} are
not strong enough for our purposes.
Another example is \cite[Lem.\,5.1]{FKL-chains},
which is essentially a variant of \cite[Thm.\,5.7]{HR-sieves}
with a more explicit error term.
The argument in \cite{FKL-chains} includes a sketch of
what modifications must be made to \cite[pp.\,147--152]{HR-sieves}
to obtain error terms uniform in~$g$.
Similar issues will arise in our proof of \Cref{thm:sieve-tuple}.
\end{rem}

\begin{rem}
\label{rem:remove-g!}
The Hardy--Littlewood conjecture \cite[p.\,61]{HL-prime-tuples}
predicts that for fixed~$g$,
the true number of prime $g$-tuples up to $X$ is $\sim D_r X / (\log X)^g$,
i.e., with the same constant as in \eqref{eq:prime-tuples}
but with the $2^g g!$ factor removed.
Similarly, if we take say $z = X^{1/2} / (\log X)^{3g}$
in \Cref{thm:sieve-tuple},
we find that the number of $k \in [0,X)$
such that $Uk + u_1, \ldots, Uk + u_g$
are all square-prime is bounded above by
$\sim 2^g g! \Delta X / (\log X)^g$.
The author expects that the true number is asymptotically
$\Delta' X / (\log X)^g$
where $\Delta'$ is as described in \Cref{rem:adjusted-Delta}.
\end{rem}

Our proof of \Cref{thm:sieve-tuple} relies on
\emph{Selberg's sieve} \cite{Selberg-sieve}.
Let us briefly recall the general setup for the sieve,
following the variant described in \cite[\S7.2]{CM-sieve}.

Let $\Aset$ be a finite set of cardinality $X \geq 1$,
let $\Pset$ be a (possibly infinite) set of prime numbers,
and for each $p \in \Pset$ let $\Aset_p$ be a subset of $\Aset$.
For $z > 1$, our goal is to find an upper bound for the cardinality of the set
\[
\Aset(z) \coloneqq
\Aset \mathbin{\Big\backslash}
   \bigcup_{\substack{p \in \Pset \\ p \leq z}} \Aset_p,
\]
i.e., how many elements of $\Aset$ remain after sieving out the elements
of $\Aset_p$ for all $p \in \Pset$, $p \leq z$.

We require that the $\Aset_p$ are ``almost independent''
in the following sense.
For each $p$ we are given a real number $\delta(p) \in (0,1)$
approximating the density of $\Aset_p$ in~$\Aset$,
i.e., so that $|\Aset_p| \approx \delta(p) X$.
Let $\Dset$ be the set of finite products of distinct elements of $\Pset$,
i.e., the set of positive squarefree integers $d$ whose prime divisors
all lie in $\Pset$.
Define
\[
\Aset_d \coloneqq \bigcap_{p \divides d} \Aset_p, \qquad
   \delta(d) \coloneqq \prod_{p \divides d} \delta(p), \qquad d \in \Dset.
\]
(For $d = 1$ this means $\Aset_1 = \Aset$ and $\delta(1) = 1$.)
Then for all $d \in \Dset$ we want $\delta(d)$ to approximate the density
of $\Aset_d$ in $\Aset$, i.e., we want the remainders
\[
r(d) \coloneqq |\Aset_d| - \delta(d) X, \qquad d \in \Dset
\]
to be in some sense ``small''.

With the setup just described,
the following theorem provides an upper bound for $|\Aset(z)|$,
expressed in terms of the quantities
\[
H(z) \coloneqq \sum_{\substack{d \in \Dset \\ d \leq z}} h(d),
\qquad
R(z) \coloneqq
   \sum_{\substack{d_1, d_2 \in \Dset \\ d_1, d_2 \leq z}} |r([d_1,d_2])|,
\]
where
\begin{equation}
\label{eq:defn-h}
h(d) \coloneqq \prod_{p \divides d} \frac{\delta(p)}{1 - \delta(p)} > 0,
   \qquad d \in \Dset,
\end{equation}
and where $[d_1,d_2]$ denotes the least common multiple of $d_1$ and $d_2$.
\begin{thm}[Selberg's sieve]
\label{thm:selberg}
For any $z > 1$,
\[
|\Aset(z)| \leq \frac{X}{H(z)} + O(R(z)).
\]
\end{thm}
\begin{proof}
This is essentially \cite[Thm.\,7.2.1]{CM-sieve}.
The translation to our notation is as follows:
\begin{itemize}
\item
The function $f(d)$ in \cite{CM-sieve} corresponds to our
$1/\delta(d) \in (1,\infty)$.
\item
The function $f_1(d)$ in \cite{CM-sieve} is defined by
$f_1(d) \coloneqq \sum_{e \divides d} \mu(e) f(d/e)$.
Since $f$ and $\mu$ are multiplicative, for $d \in \Dset$ this yields
\[
f_1(d) = \prod_{p \divides d} \,
         \sum_{e \divides p} \frac{\mu(e)}{\delta(p/e)}
   = \prod_{p \divides d} \left(\frac{1}{\delta(p)} - 1\right)
   = \frac{1}{h(d)}.
\]
Thus $f_1(d)$ corresponds to our $1/h(d) \in (0,\infty)$ for $d \in \Dset$.
\item
The function $V(z)$ in \cite{CM-sieve} is given by
\[
V(z) \coloneqq
   \sum_{\substack{d \leq z \\ d \divides P(z)}} \frac{\mu^2(d)}{f_1(d)}
   = \sum_{\substack{d \leq z \\ d \divides P(z)}} h(d)
\]
where $P(z) \coloneqq \prod_{p \in \Pset, p < z} p$.
This is exactly our $H(z)$,
except that $P(z)$ only includes primes $p < z$,
whereas we allow primes $p \leq z$.
We can make these agree by replacing $z$ with $z + \eps$ if $z \in \ZZ$.
(Alternatively, one can check that the proof of \cite[Thm.\,7.2.1]{CM-sieve}
still works if we change the definition of $P(z)$ to
$\prod_{p \in \Pset, p \leq z} p$.)
\item
The remainders $R_d$ in \cite{CM-sieve} are exactly our $r(d)$,
for $d \in \Dset$.
\end{itemize}
The conclusion of \cite[Thm.\,7.2.1]{CM-sieve} is that
$|\Aset(z)| \leq X/H(z) + O(R(z))$.
(Again, to identify the latter term as $R(z)$ we may have to replace
$z$ by $z + \eps$ if $z \in \ZZ$.)
\end{proof}
\begin{rem}
Incidentally, it is clear from the proof of \cite[Thm.\,7.2.1]{CM-sieve}
that the error term $O(R(z))$ could be replaced by simply $R(z)$.
\end{rem}

Let us now return to the proof of \Cref{thm:sieve-tuple}.
We will apply \Cref{thm:selberg} with the following data:
\begin{itemize}
\item
$\Aset \coloneqq \{0, 1, \ldots, X-1\}$.
\item
$\Pset$ is the set of primes such that $|B_p| \neq 0$.
(By \Cref{lem:p-good} this includes all good primes.)
\item
$\Aset_p \coloneqq \{k \in \Aset : (k \bmod p^2) \in B_p\}$
and $\delta(p) \coloneqq |B_p|/p^2$ for $p \in \Pset$.
\end{itemize}
According to the assumption \eqref{eq:Bp-assumption}
we have $\delta(p) \in (0, 1)$ for every $p \in \Pset$.

By \Cref{lem:sieve-criterion},
the integers $k$ described in \Cref{thm:sieve-tuple} all lie in $\Aset(z)$,
so our goal is to find an upper bound for $|\Aset(z)|$.
As input to \Cref{thm:selberg}
we will need the following estimates for $H(z)$ and $R(z)$.
\begin{prop}
\label{prop:H-estimate}
Let $\theta = \theta(g,Q)$ be as in \eqref{eq:theta-defn}.
For $z > e^{3\theta}$,
\[
H(z) = \frac{(\log z)^g}{g! \Delta}
   \left(1 + O\left(\frac{\theta}{\log z}\right) \right).
\]
\end{prop}
\begin{prop}
\label{prop:R-bound}
For $z > e^{g^2}$,
\[
R(z) \ll \frac{z^2 (\log z)^{4g}}{(2g)!^2}.
\]
\end{prop}
These results will be proved in \Cref{sec:sieve-H-bound}
and \Cref{sec:sieve-R-bound} respectively,
after a detour in \Cref{sec:sieve-preliminary} to develop some
preliminary estimates involving sums of the $g$-fold divisor function.

To finish the proof of \Cref{thm:sieve-tuple},
observe that by \Cref{prop:H-estimate} there is a constant $C > 0$ such that
for any $z > e^{3\theta}$,
\[
H(z) = \frac{(\log z)^g}{g! \Delta} (1 + t(z)), \qquad
   |t(z)| \leq \frac{C \theta}{\log z}.
\]
Taking $C_1 \coloneqq \max(2C, 3)$,
we find that for $z > e^{C_1 \theta}$ we have $|t(z)| \leq \tfrac12$
and hence $(1 + t(z))^{-1} = 1 + O(|t(z)|)$.
Therefore
\[
\frac{X}{H(z)} = \frac{g! \Delta X}{(\log z)^g}
   \left(1 + O\left(\frac{\theta}{\log z} \right)\right),
   \qquad z > e^{C_1 \theta}.
\]
Combining this with \Cref{thm:selberg} and \Cref{prop:R-bound}
(whose hypothesis is satisfied as certainly $3\theta > g^2$)
completes the proof of \Cref{thm:sieve-tuple}.

\begin{rem}
It is likely that a result similar to \Cref{thm:sieve-tuple}
could be proved using alternative sieve techniques
such as the \emph{large sieve} as described in \cite[Thm.\,7.7]{IK-analytic}.
Note that applying the latter theorem in the most straightforward way
leads to suboptimal results:
since we are working with congruences modulo~$p^2$ instead of modulo~$p$,
the points $\alpha_r$ (in the notation of \cite{IK-analytic})
are $\delta$-spaced for $\delta \approx z^{-4}$
instead of the usual~$z^{-2}$.
The factor on the right hand side of \cite[Thm.\,7.7]{IK-analytic}
thus becomes $N + O(z^4)$,
so we can only sieve up to $z \leq N^{1/4}$.
It may be possible to sieve up to $\approx N^{1/2}$
by a more careful application of the large sieve,
but we have not investigated this.
Sieving up to $N^{1/4}$ would be adequate for the applications
in \Crefrange{sec:probabilistic}{sec:deterministic},
but we prefer to state \Cref{thm:sieve-tuple} in this stronger form.
\end{rem}

\subsection{Divisor sum estimates}
\label{sec:sieve-preliminary}

For integers $g \geq 1$ and $n \geq 1$ define
\[
\tau_g(n) \coloneqq \sum_{n_1 \cdots n_g = n} 1,
\]
i.e., $\tau_g$ is the $g$-fold Dirichlet convolution
of the constant~$1$ function.
Define also, for any integer $Q \geq 1$,
\[
\Lambda(g,Q) \coloneqq \left(\frac{Q}{\varphi(Q)}\right)^g \geq 1.
\]
In this section we prove the following estimate.
\begin{prop}
\label{prop:tau-Q-estimate}
Let $g \geq 1$ and $Q \geq 1$ be integers,
let $\theta = \theta(g,Q)$ be as in \eqref{eq:theta-defn},
and let $\Lambda = \Lambda(g,Q)$ be as above.
Then for any $z > e^\theta$,
\begin{equation}
\label{eq:tau-Q-estimate}
\sum_{\substack{1 \leq n \leq z \\ (n,Q) = 1}} \frac{\tau_g(n)}{n}
   = \frac{(\log z)^g}{g! \Lambda}
      \left(1 + O\left(\frac{\theta}{\log z}\right)\right).
\end{equation}
\end{prop}

\begin{rem}
A frequently quoted statement in the literature is that for each $g \geq 1$,
\[
\sum_{1 \leq n \leq z} \tau_g(n)
   = z P_g(\log z) + O_{g,\eps}(z^{\alpha_g+\eps})
\]
for a certain polynomial $P_g(x)$ of degree $g-1$
and a constant $\alpha_g < 1$.
(See for instance \cite[Eq.\,12.1.3]{Titchmarsh-zeta}
or \cite[Thm.\,8.31]{IK-analytic}.
These works consider $\sum \tau_g(n)$,
but similar techniques can be applied to $\sum \tau_g(n)/n$.)
Considerable effort has gone into optimising the error term,
i.e., minimising the exponent~$\alpha_g$.
For our purposes we do not need such a small error term,
but we do need an error term that is completely explicit
with respect to~$g$;
the above results do not specify the value of the big-$O$ constant,
or the lower order terms in $P_g(x)$, as a function of~$g$.
Our \Cref{prop:tau-Q-estimate} is more in the spirit of the estimate
used in the proof of \cite[Thm.\,6.5]{Lenstra-algorithms-ANT}
(for a stronger version see \cite{Bordelles-inequality}).
\end{rem}

The following rough upper bound for $Q/\varphi(Q)$ will be
occasionally useful.
\begin{lem}
\label{lem:Q-phiQ-bound}
For $Q \geq 1$,
\[
\frac{Q}{\varphi(Q)} < 3 \log \log 5Q.
\]
\end{lem}
\begin{proof}
According to \cite[Thm.\,15]{RS-primes},
\[
\frac{Q}{\varphi(Q)} < e^\gamma \log \log Q + \frac{2.50637}{\log \log Q},
   \qquad Q \geq 3,
\]
where $\gamma \approx 0.5772$ is Euler's constant.
Since $e^\gamma < 1.782$ and
$\log \log Q \log \log 5Q > 2.06$ for $Q \geq 34$ we obtain
\[
\frac{Q}{\varphi(Q)}
   < 1.782 \log \log 5Q + \frac{2.50637 \log \log 5Q}{2.06}
   < 2.999 \log \log 5Q, \qquad Q \geq 34.
\]
For $1 \leq Q \leq 33$ the inequality may be checked directly.
\end{proof}

To prove \Cref{prop:tau-Q-estimate},
we begin with an estimate for the $Q = 1$ case, i.e., for the sum
\begin{equation}
\label{eq:tau-sum}
\sum_{1 \leq n \leq z} \frac{\tau_g(n)}{n}
   = \sum_{1 \leq n_1 \cdots n_g \leq z} \frac{1}{n_1 \cdots n_g}.
\end{equation}
\begin{lem}
\label{lem:tau-sum-estimate}
For any $g \geq 1$ and $z \geq 1$,
\[
\frac{(\log z)^g}{g!}
   \leq \sum_{1 \leq n \leq z} \frac{\tau_g(n)}{n}
   \leq \frac{(\log z + g)^g}{g!}.
\]
\end{lem}
\begin{proof}
We prove the lower bound by induction on~$g$.
It holds for $g = 1$ as $\log z \leq \sum_{1 \leq n \leq z} n^{-1}$.
For $g \geq 2$ we have by induction
\[
\sum_{1 \leq n \leq z} \frac{\tau_g(n)}{n}
   = \sum_{1 \leq n_1 \leq z} \frac{1}{n_1}
      \sum_{1 \leq n_2 \cdots n_g \leq \tfrac{z}{n_1}}
      \frac{1}{n_2 \cdots n_g}
   \geq \sum_{1 \leq n \leq z} \frac1n \cdot \frac{\log(z/n)^{g-1}}{(g-1)!}.
\]
The function $t \mapsto t^{-1} \log(z/t)^{g-1}$
is non-negative and decreasing on $[1,z]$,
so the previous expression is at least
\[
\int_1^z \frac{1}{t} \cdot \frac{\log(z/t)^{g-1}}{(g-1)!} \, dt
   = \int_1^z \frac{(\log u)^{g-1}}{(g-1)!} \cdot \frac{du}{u}
   = \frac{(\log z)^g}{g!}.
\]

For the upper bound we use a similar argument to the proof of
\cite[Thm.\,6.5]{Lenstra-algorithms-ANT}.
Let $S = \{i_1, \ldots, i_k\}$ be a subset of $\{1, \ldots, g\}$,
and let $X_S$ be the sum of those terms in \eqref{eq:tau-sum}
for which $n_i \geq 2$ for $i \in S$ and $n_i = 1$ for $i \notin S$,
i.e.,
\[
X_S =
   \sum_{\substack{1 \leq n_{i_1} \cdots n_{i_k} \leq z \\ n_{i_t} \geq 2}}
   \frac{1}{n_{i_1} \cdots n_{i_k}}.
\]
Replacing each $(n_{i_1}, \ldots, n_{i_k})$ by the box
$\prod_{t=1}^k (n_{i_t}-1, n_{i_t}]$
and comparing with an integral over the region
$E_S \coloneqq
\{1 \leq x_{i_1} \cdots x_{i_k} \leq z\} \subset [1,\infty)^k$ yields
\[
X_S \leq \int_{E_S} \frac{dx_{i_1} \cdots dx_{i_k}}{x_{i_1} \cdots x_{i_k}}
   = \int_{\substack{y_{i_1}, \ldots, y_{i_k} \geq 0 \\
   y_{i_1} + \cdots + y_{i_k} \leq \log z}}
   \, dy_{i_1} \cdots dy_{i_k}
   = \frac{(\log z)^k}{k!}.
\]
(This inequality holds even when $k = 0$,
as in this case $X_S$ contains the single term $n_1 = \cdots = n_g = 1$.)
Summing over all subsets $S \subseteq \{1, \ldots, g\}$ we obtain
\[
\sum_{1 \leq n \leq z} \frac{\tau_g(n)}{n} = \sum_S X_S
   \leq \sum_{k=0}^g \binom{g}{k} \frac{(\log z)^k}{k!},
\]
and since $g!/k! \leq g^{g-k}$,
\[
\sum_{1 \leq n \leq z} \frac{\tau_g(n)}{n}
   \leq \sum_{k=0}^g \binom{g}{k} \frac{(\log z)^k g^{g-k}}{g!}
   = \frac{(\log z + g)^g}{g!}. \qedhere
\]
\end{proof}

We now examine the $Q \geq 2$ case. Here the sum is given by
\begin{equation}
\label{eq:tau-Q-sum}
\sum_{\substack{1 \leq n \leq z \\ (n,Q) = 1}} \frac{\tau_g(n)}{n}
   = \sum_{\substack{1 \leq n_1 \cdots n_g \leq z \\ (n_i,Q) = 1}}
      \frac{1}{n_1 \cdots n_g}.
\end{equation}

\begin{lem}
\label{lem:tau-Q-bounds}
For any $g \geq 1$, $Q \geq 2$ and $z > Q^g$,
\begin{equation}
\label{eq:tau-Q-bounds}
\frac{\log(z/Q^g)^g}{g! \Lambda} \leq
   \sum_{\substack{1 \leq n \leq z \\ (n,Q) = 1}} \frac{\tau_g(n)}{n}
   \leq \frac{(\log z)^g}{g! \Lambda}
      \left(1 + \frac{4 g \log 5Q \log \log 5Q}{\log z} \right)^g.
\end{equation}
\end{lem}
\begin{proof}
To prove the lower bound,
consider the set of integer tuples $(s_1, \ldots, s_g)$
such that $1 \leq s_1 \cdots s_g \leq z/Q^g$.
Each such tuple yields $\varphi(Q)^g$ tuples
$(n_1, \ldots, n_g) = (Qs_1 - r_1, \ldots, Qs_g - r_g)$
(for $0 \leq r_i < Q$ with $(r_i, Q) = 1$)
such that $(n_i,Q) = 1$ and $1 \leq n_1 \cdots n_g \leq z$.
Therefore
\begin{align*}
\sum_{\substack{1 \leq n_1 \cdots n_g \leq z \\ (n_i,Q) = 1}}
      \frac{1}{n_1 \cdots n_g}
& \geq \sum_{\substack{1 \leq s_1 \cdots s_g \leq \tfrac{z}{Q^g}}} \;
   \sum_{\substack{0 \leq r_i < Q \\ (r_i, Q) = 1}}
   \frac{1}{(Qs_1 - r_1) \cdots (Qs_g - r_g)} \\
& \geq \varphi(Q)^g
   \sum_{\substack{1 \leq s_1 \cdots s_g \leq \tfrac{z}{Q^g}}}
   \frac{1}{Q^g s_1 \cdots s_g}
= \Lambda^{-1} \sum_{1 \leq n \leq \tfrac{z}{Q^g}} \frac{\tau_g(n)}{n}.
\end{align*}
The lower bound in \eqref{eq:tau-Q-bounds}
then follows from \Cref{lem:tau-sum-estimate}.

We now prove the upper bound.
For any subset $S = \{i_1, \ldots, i_k\} \subseteq \{1, \ldots, g\}$
let $X_S$ denote the sum of those terms in \eqref{eq:tau-Q-sum}
for which $n_i > Q$ for $i \in S$ and $n_i \leq Q$ for $i \notin S$.
We will write the complement of $S$ as $\{j_1, \ldots, j_m\}$
with $k + m = g$.
Thus
\[
X_S = \sum_{\substack{1 \leq n_{j_t} \leq Q \\ (n_{j_t},Q) = 1}}
   \frac{1}{n_{j_1} \cdots n_{j_m}}
   \sum_{\substack{1 \leq n_{i_1} \cdots n_{i_k}
      \leq \tfrac{z}{n_{j_1} \cdots n_{j_m}} \\
      (n_{i_t},Q) = 1 \\ n_{i_t} > Q}}
   \frac{1}{n_{i_1} \cdots n_{i_k}}.
\]
For $k = 0$, i.e., $S = \emptyset$,
the inner sum is taken to have the value~$1$.
To estimate the inner sum for $k \geq 1$,
observe that each $(n_{i_1}, \ldots, n_{i_k})$ appearing in the sum
may be written uniquely in the form
$(Q s_{i_1} + r_{i_1}, \ldots, Q s_{i_k} + r_{i_k})$ with
\[
s_{i_t} \geq 1, \qquad 1 \leq r_{i_t} \leq Q, \qquad (r_{i_t}, Q) = 1,
\qquad s_{i_1} \cdots s_{i_k} \leq \frac{z}{Q^k n_{j_1} \cdots n_{j_m}}.
\]
Thus the inner sum is bounded above by
\begin{multline*}
\sum_{1 \leq s_{i_1} \cdots s_{i_k} \leq
   \tfrac{z}{Q^k n_{j_1} \cdots n_{j_m}}} \;
\sum_{\substack{1 \leq r_{i_t} \leq Q \\ (r_{i_t}, Q) = 1}}
   \frac{1}{(Q s_{i_1} + r_{i_1}) \cdots (Q s_{i_k} + r_{i_k})} \\
\leq \frac{\varphi(Q)^k}{Q^k}
   \sum_{1 \leq s_{i_1} \cdots s_{i_k} \leq
   \tfrac{z}{Q^k n_{j_1} \cdots n_{j_m}}} \;
   \frac{1}{s_{i_1} \cdots s_{i_k}}
\leq \frac{\varphi(Q)^k}{Q^k}
   \sum_{1 \leq s_{i_1} \cdots s_{i_k} \leq z}
   \frac{1}{s_{i_1} \cdots s_{i_k}} \\
\leq \frac{\varphi(Q)^k}{Q^k} \cdot \frac{(\log z + k)^k}{k!},
\end{multline*}
the last inequality following from \Cref{lem:tau-sum-estimate}.
Substituting this bound (which also holds for $k = 0$)
into the formula for $X_S$ yields
\[
X_S \leq
   \frac{1}{k!}
   \left( \frac{\varphi(Q)}{Q} (\log z + k) \right)^k
   \sum_{\substack{1 \leq n_{j_t} \leq Q \\ (n_{j_t},Q) = 1}}
   \frac{1}{n_{j_1} \cdots n_{j_m}}.
\]
Using the estimate
\[
\sum_{\substack{1 \leq n \leq Q \\ (n,Q) = 1}} n^{-1}
   \leq \sum_{1 \leq n \leq Q} n^{-1} \leq \log Q + 1
\]
we obtain
\[
X_S \leq \frac{1}{k!}
   \left( \frac{\varphi(Q)}{Q} (\log z + g) \right)^k (\log Q + 1)^{g-k}.
\]
Summing over all subsets~$S$,
and again using the estimate $g!/k! \leq g^{g-k}$,
\begin{align*}
\sum_{\substack{1 \leq n \leq z \\ (n,Q) = 1}} \frac{\tau_g(n)}{n}
   & \leq \sum_{k=0}^g \binom{g}{k} \frac{1}{k!}
   \left( \frac{\varphi(Q)}{Q} (\log z + g) \right)^k (\log Q + 1)^{g-k} \\[-9pt]    
   & \leq \frac{1}{g!} \sum_{k=0}^g \binom{g}{k}
   \left( \frac{\varphi(Q)}{Q} (\log z + g) \right)^k (g(\log Q + 1))^{g-k} \\
   & = \frac{1}{g!} \left( \frac{\varphi(Q)}{Q}(\log z + g)
                              + g (\log Q + 1)\right)^g \\
   & = \frac{(\log z)^g}{g! \Lambda} \left(1 +
      \frac{g \bigl(\frac{Q}{\varphi(Q)} (\log Q + 1) + 1\bigr)}{\log z}
      \right)^g.
\end{align*}
The proof is concluded by invoking \Cref{lem:Q-phiQ-bound}
and observing that for $Q \geq 2$,
\[
3 \log \log 5Q \, (\log Q + 1) + 1
   < 3 \log 5Q \log \log 5Q + 1 < 4 \log 5Q \log \log 5Q.
   \qedhere
\]
\end{proof}

\begin{proof}[Proof of \Cref{prop:tau-Q-estimate}]
First suppose that $Q \geq 2$.
We will show that the upper and lower bounds in \eqref{eq:tau-Q-bounds}
both simplify to \eqref{eq:tau-Q-estimate}.
(The hypothesis of \Cref{lem:tau-Q-bounds} holds as
we are assuming that $z > e^\theta$;
it is enough to check that $\log 5Q \log \log 5Q > \log Q$ for $Q \geq 2$.)
For the lower bound, first observe that
\[
\log(z/Q^g)^g = (\log z)^g \left(1 - \frac{g \log Q}{\log z}\right)^g.
\]
The hypothesis $z > e^\theta$ implies that
$g \log Q / \log z \ll 1/g$, so
\[
\left(1 - \frac{g \log Q}{\log z}\right)^g
   = 1 + O\left(\frac{g^2 \log Q}{\log z}\right)
   = 1 + O\left(\frac{\theta}{\log z}\right).
\]
For the upper bound, the hypothesis $z > e^\theta$ likewise yields
\[
\left(1 + \frac{4 g \log 5Q \log \log 5Q}{\log z} \right)^g
   = 1 + O\left(\frac{\theta}{\log z}\right).
\]
If $Q = 1$ then the result follows directly from
\Cref{lem:tau-sum-estimate} in a similar way.
\end{proof}

\subsection{The estimate for $H(z)$}
\label{sec:sieve-H-bound}

Returning to the notation of \Cref{sec:sieve-setup},
in this section we prove \Cref{prop:H-estimate},
showing along the way that the defining product \eqref{eq:Delta-defn}
for $\Delta$ converges absolutely.
Our strategy for estimating $H(z)$ will be to approximate $h(n)$
by the multiplicative function
\[
f(n) \coloneqq
\begin{cases}
   \tau_g(n)/n & \textn{if } (n,Q) = 1, \\
   0           & \textn{otherwise}.
\end{cases}
\]
The point is that the summatory function of $f(n)$ is easy to estimate
thanks to \Cref{prop:tau-Q-estimate}.
(This line of attack is inspired by the proof of
\cite[Thm.\,3.10]{MV-mult-nt}.)

Let us examine the formal Dirichlet series associated to $h(n)$ and $f(n)$.
For convenience we extend $h(n)$ to a multiplicative function
defined for all $n \geq 1$ by setting $h(n) \coloneqq 0$ for $n \notin \Dset$.
In particular, $h(p) = 0$ for $p \notin \Pset$,
and $h(p^\ell) = 0$ for any prime $p$ and any $\ell \geq 2$.
We then have
\[
\sum_{n \geq 1} h(n) \tsp n^{-s} =
   \prod_p (1 + h(p) \tsp p^{-s}).
\]
To compute the Dirichlet series for $f(n)$,
observe that the Euler factor at $p$ for $\tau_g(n)$ is
$(1 + p^{-s} + p^{-2s} + \cdots)^g = (1 - p^{-s})^{-g}$,
so
\[
\sum_{n \geq 1} f(n) \tsp n^{-s} =
   \prod_{\textn{$p$ good}} (1 - p^{-1} p^{-s})^{-g}.
\]
Now let $\psi(n)$ be the multiplicative function defined by $h = \psi * f$,
where $*$ denotes Dirichlet convolution.
Then the Dirichlet series for $\psi(n)$ is given by
\begin{equation}
\label{eq:psi-series}
\sum_{n \geq 1} \psi(n) \tsp n^{-s} =
   \prod_{\textn{$p$ bad}} (1 + h(p) \tsp p^{-s})
   \prod_{\textn{$p$ good}} (1 + h(p) \tsp p^{-s}) (1 - p^{-1} p^{-s})^g.
\end{equation}
This expression leads to the following bound for $\psi(p^\ell)$,
for good primes $p$.
\begin{lem}
\label{lem:psi-bound}
For $p$ good we have
\[
|\psi(p)| < \frac{g^2}{p^2}
\]
and
\[
|\psi(p^\ell)| < \frac{g^\ell}{(\ell-1)! \tsp p^\ell},
   \qquad \ell \geq 2.
\]
\end{lem}
\begin{proof}
Consider the Euler factor at $p$ in \eqref{eq:psi-series}, i.e.,
\[
1 + \psi(p) \tsp p^{-s} + \psi(p^2) \tsp p^{-2s} + \cdots
   = (1 + h(p) \tsp p^{-s}) (1 - p^{-1} p^{-s})^g.
\]
Expanding out the last factor using the binomial theorem,
we obtain the formula
\[
\psi(p^\ell) = (-1)^\ell \left(
   \binom{g}{\ell} p^{-\ell}
      - \binom{g}{\ell-1} h(p) \tsp p^{-\ell+1} \right),
   \qquad \ell \geq 1.
\]
Clearly $\psi(p^\ell) = 0$ when $\ell \geq g+2$,
so the result holds in this case.
Henceforth we assume that $1 \leq \ell \leq g+1$.
For such $\ell$ we have
$\binom{g}{\ell} = \binom{g}{\ell-1} \cdot (g-\ell+1)/\ell$,
so
\begin{equation}
\label{eq:psi-p-ell}
|\psi(p^\ell)| =
   \binom{g}{\ell-1} p^{-\ell} \left|\frac{g - \ell + 1}{\ell}
      - h(p) p \right|,
   \qquad 1 \leq \ell \leq g+1.
\end{equation}
We also recall from \eqref{eq:defn-h} and \Cref{lem:p-good} that
\[
h(p) = \frac{\delta(p)}{1 - \delta(p)}
   = \frac{g(p-1)}{p^2 - g(p-1)} = \frac{gp-g}{p^2 - gp + g},
\]
Moreover, since $p$ is good we know that $p > 4g^2$.

Let us first consider the special case $g = 1$.
Here we have
\[
h(p) p = \frac{p^2 - p}{p^2 - p + 1}
   = 1 - \frac{1}{p^2 - p + 1}.
\]
Applying \eqref{eq:psi-p-ell},
for $\ell = 1$ we obtain
\[
|\psi(p)| = p^{-1} \bigl\lvert 1 - h(p) p \bigr\rvert
   = \frac{1}{p(p^2 - p + 1)} < \frac{1}{p^2} = \frac{g^2}{p^2},
\]
and for $\ell = 2$,
\[
|\psi(p^2)| = p^{-2} \bigl\lvert 0 - h(p) p \bigr\rvert
   = \frac{1}{p^2} \cdot \frac{p^2-p}{p^2 - p + 1} < \frac{1}{p^2}
   = \frac{g^2}{1! \tsp p^2}.
\]
In both cases the desired result holds.
Henceforth we assume that $g \geq 2$.

We claim that
\begin{equation}
\label{eq:hp-p-estimate}
g < h(p) p < g + \frac{g^2}{p}.
\end{equation}
To prove this, first observe that
\[
h(p) p - g = \frac{gp^2 - gp}{p^2 - gp + g} - g
   = \frac{g(gp - p - g)}{p^2 - gp + g}.
\]
The factor $gp - p - g$ is positive,
as $gp - p - g \geq 2p - p - g = p - g > 4g^2 - g > 0$.
This establishes the first inequality in \eqref{eq:hp-p-estimate},
and the second one follows from
\[
h(p) p - g < \frac{g(gp - 4g^2 - g)}{p^2 - gp}
   < \frac{g(gp - g^2)}{p^2 - gp} = \frac{g^2}{p}.
\]

Using \eqref{eq:hp-p-estimate}, we may rewrite \eqref{eq:psi-p-ell} as follows:
\begin{align}
|\psi(p^\ell)|
   & \leq \binom{g}{\ell-1} p^{-\ell} \left(
      \left|\frac{g - \ell + 1}{\ell} - g \right|
      + |g - h(p) p| \right) \notag \\
   & < \binom{g}{\ell-1} p^{-\ell}
      \left( \frac{(\ell-1)(g+1)}{\ell} + \frac{g^2}{p} \right),
      \qquad 1 \leq \ell \leq g+1.
      \label{eq:psi-p-ell-2}
\end{align}
We now apply \eqref{eq:psi-p-ell-2} with various values of $\ell$.
For $\ell = 1$, we get
\[
|\psi(p)| < \frac{1}{p} \cdot \frac{g^2}{p} = \frac{g^2}{p^2}.
\]
For $\ell = 2$,
\[
|\psi(p^2)| < \frac{g}{p^2} \left( \frac{g+1}{2} + \frac{g^2}{p} \right)
   < \frac{g}{p^2} \left( \frac{g+1}{2} + \frac{1}{4} \right)
   < \frac{g}{p^2} \cdot g = \frac{g^2}{1! \tsp p^2}.
\]
Finally, for $\ell \geq 3$, since
\[
\frac{(\ell-1)(g+1)}{\ell} = (g+1) - \frac{g+1}{\ell} \leq (g+1) - 1 = g
\]
we find that
\begin{multline*}
|\psi(p^\ell)|
   < \frac{g(g-1)(g-2) \cdots (g-\ell+2)}{(\ell-1)! \tsp p^\ell}
      \left(g + \frac14\right) \\
   < \frac{g^{\ell-2} (g-1)}{(\ell-1)! \tsp p^\ell} (g+1)
   = \frac{g^{\ell-2}(g^2 - 1)}{(\ell-1)! \tsp p^\ell}
   < \frac{g^\ell}{(\ell-1)! \tsp p^\ell}. \qedhere
\end{multline*}
\end{proof}

\begin{rem}
The proof of \Cref{lem:psi-bound} is essentially due to ChatGPT 5.4
(Extended Pro).
The chatbot was given \eqref{eq:psi-p-ell} and the definition of $h(p)$
and asked to prove the two inequalities in \Cref{lem:psi-bound}.
The argument was found after about 45 minutes of ``thinking'',
including some back-and-forth discussion regarding the
strength of the condition $p > 4g^2$ in the definition of ``good prime''.
The text of the proof was written by the author,
after rearranging the chatbot output and
making some minor simplifications.
\end{rem}

Let $E_p(s) \coloneqq
1 + \psi(p) \tsp p^{-s} + \psi(p^2) \tsp p^{-2s} + \cdots$
denote the Euler factor at $p$ in \eqref{eq:psi-series} and let
\[
F_p(\sigma) \coloneqq
   1 + |\psi(p)| \tsp p^{-\sigma} + |\psi(p^2)| \tsp p^{-2\sigma} + \cdots
\]
be the corresponding series with non-negative coefficients.
For good $p$ we obtain the following estimate for $F_p(\sigma)$.
\begin{lem}
\label{lem:Fp-estimate}
Let $0 \leq \eps \leq \tfrac12$.
For any good prime $p$ we have
\[
|F_p(-\eps) - 1| < \frac{3 g^2}{p^{2-2\eps}} < \frac{3}{4}.
\]
\end{lem}
\begin{proof}
By \Cref{lem:psi-bound},
\begin{align*}
|F_p(-\eps) - 1|
   & \leq |\psi(p)| \tsp p^{\eps} + |\psi(p^2)| \tsp p^{2\eps} + \cdots \\
   & < \frac{g^2}{p^{2-\eps}} + \frac{g^2}{1! \tsp p^{2-2\eps}}
      + \frac{g^3}{2! \tsp p^{3-3\eps}}
      + \frac{g^4}{3! \tsp p^{4-4\eps}} + \cdots \\
   & = \frac{g^2}{p^{2-2\eps}}
      \left(p^{-\eps} + \frac{1}{1!} + \frac{g}{2! \tsp p^{1-\eps}}
      + \frac{g^2}{3! \tsp p^{2-2\eps}} + \cdots \right).
\end{align*}
Since $p > 4g^2$ we have $p^{1-\eps} \geq p^{1/2} > 2g$, so we get
\begin{align*}
|F_p(-\eps) - 1|
   & < \frac{g^2}{p^{2-2\eps}}
   \left( 1 + \frac{1}{1!} + \frac{2^{-1}}{2!} + \frac{2^{-2}}{3!}
      + \cdots \right) \\
   & = \frac{g^2}{p^{2-2\eps}} (2 e^{1/2} - 1)
      < \frac{3g^2}{p^{2-2\eps}} < \frac34. \qedhere
\end{align*}
\end{proof}

\begin{cor}
\label{cor:log-Fp-bound}
Let $0 \leq \eps < \tfrac12$.
Then $F_p(-\eps) > 0$ for all good $p$ and
\[
\sum_{\textn{$p$ good}} \log F_p(-\eps) \ll \frac{g^{4\eps}}{1 - 2\eps}.
\]
\end{cor}
\begin{proof}
By \Cref{lem:Fp-estimate} the terms $F_p(-\eps)$
are positive and bounded away from zero, so
\begin{multline*}
\sum_{\textn{$p$ good}} \log F_p(-\eps)
   \ll \sum_{\textn{$p$ good}} \frac{g^2}{p^{2-2\eps}}
   < \sum_{n > 4g^2} \frac{g^2}{n^{2-2\eps}} \\
   < g^2 \int_{4g^2}^\infty \frac{dt}{t^{2-2\eps}}
   = \frac{g^2}{1 - 2\eps} (4g^2)^{-1+2\eps}
   \ll \frac{g^{4\eps}}{1 - 2\eps}.
   \qedhere
\end{multline*}
\end{proof}

\begin{cor}
\label{cor:Fp-converges}
The product $\prod_{\textn{$p$ good}} E_p(s)$ converges absolutely
on the half-plane $\real(s) > -\tfrac12$.
In particular, the defining product \eqref{eq:Delta-defn} for $\Delta$
converges absolutely and $\Delta \neq 0$.
\end{cor}
\begin{proof}
\Cref{cor:log-Fp-bound} implies that $\prod_{\textn{$p$ good}} F_p(\sigma)$
converges for $\sigma > -\tfrac12$.
The first statement follows as $|E_p(s) - 1| \leq |F_p(\sigma) - 1|$
for $\sigma \coloneqq \real(s)$.

For the second statement, take $s = 0$.
We find that
\[
\prod_{\textn{$p$ good}} E_p(0)
   = \prod_{\textn{$p$ good}} (1 + h(p))(1 - p^{-1})^g
\]
converges absolutely.
But $(1 + h(p))^{-1} = 1 - \delta(p)$, so the product
\[
\prod_{\textn{$p$ good}} E_p(0)^{-1}
   = \prod_{\textn{$p$ good}} (1 - \delta(p))(1 - p^{-1})^{-g}
\]
also converges absolutely.
This is exactly the product in \eqref{eq:Delta-defn}
with the bad primes omitted.
Moreover \eqref{eq:Bp-assumption} implies that $\delta(p) \neq 1$ for bad $p$,
so $\Delta \neq 0$.
\end{proof}

Next we give a preliminary estimate for $H(z)$ with the correct main term.
\begin{lem}
For any $z > e^{2\theta}$,
\begin{multline}
\label{eq:H-estimate-prelim}
H(z) = \frac{(\log z)^g}{g! \Delta} + \mathord{\phantom{x}} \\
   O\Biggl( \frac{(\log z)^{g-1}}{g! \Lambda}
      \biggl(
         \theta \sum_{k \geq 1} |\psi(k)|
         + g \sum_{k \geq 1} |\psi(k)| \log k +
         \Lambda \theta \sum_{k \geq e^{-\theta} z} |\psi(k)|
      \biggr) \Biggr).
\end{multline}
\end{lem}
\begin{proof}
By the definition of $\psi(n)$ we have
\[
H(z) = \sum_{1 \leq n \leq z} h(n)
   = \sum_{1 \leq k \leq z} \psi(k) \sum_{1 \leq n \leq \tfrac{z}{k}} f(n).
\]
We consider first the sum over $1 \leq k < k_0$ where
\[
k_0 \coloneqq e^{-\theta} z.
\]
For such $k$ we have $z/k > e^\theta$,
and hence by \Cref{prop:tau-Q-estimate},
\begin{align*}
\sum_{1 \leq k < k_0} \psi(k) \sum_{1 \leq n \leq \tfrac{z}{k}} f(n)
   & = \sum_{1 \leq k < k_0} \psi(k) \, \frac{\log(z/k)^g}{g! \Lambda}
      \left(1 + O\left(\frac{\theta}{\log(z/k)}\right)\right) \\
   & = \sum_{1 \leq k < k_0} \psi(k) \, \frac{\log(z/k)^g}{g! \Lambda}
      + O\left( \frac{(\log z)^{g-1}}{g! \Lambda} \, \theta
         \sum_{k \geq 1} |\psi(k)| \right).
\end{align*}
We split up the first sum into subintervals
$1 \leq k < k_1$ and $k_1 \leq k < k_0$
where
\[
k_1 \coloneqq z^{1/(2g)};
\]
this decomposition is valid as the inequality
$z^{1 - 1/(2g)} \geq z^{1/2} > e^\theta$ ensures that ${k_1 < k_0}$.
The sum over the second subinterval is bounded by
\begin{align}
\label{eq:tail-estimate}
\frac{(\log z)^g}{g! \Lambda} \sum_{k_1 \leq k < k_0} |\psi(k)|
   & \leq \frac{(\log z)^g}{g! \Lambda} \cdot
      \frac{1}{\log k_1} \sum_{k \geq k_1} |\psi(k)| \log k \\
   & \ll \frac{(\log z)^{g-1}}{g! \Lambda} \,
      g \sum_{k \geq 1} |\psi(k)| \log k.   \notag
\end{align}
For $k$ in the first subinterval we have
$\log k \ll g^{-1} \log z$ and hence
\[
\log(z/k)^g = (\log z)^g \left(1 - \frac{\log k}{\log z}\right)^g
   = (\log z)^g \left(1 + O\left(\frac{g \log k}{\log z}\right)\right),
\]
so the sum over this subinterval becomes
\begin{equation}
\label{eq:first-interval}
\frac{(\log z)^g}{g! \Lambda} \sum_{1 \leq k < k_1} \psi(k)
   + O\left( \frac{(\log z)^{g-1}}{g! \Lambda} \,
      g \sum_{k \geq 1} |\psi(k)| \log k \right).
\end{equation}
The series $\sum_{k \geq 1} \psi(k) = \prod_p E_p(0)$
converges thanks to \Cref{cor:Fp-converges}
(and \cite[Thm.\,1.9]{MV-mult-nt}).
Estimating the tail
$|\sum_{k \geq k_1} \psi(k)| \leq \sum_{k \geq k_1} |\psi(k)|$
via the same method used in \eqref{eq:tail-estimate},
the first term of \eqref{eq:first-interval} becomes
\[
\frac{(\log z)^g}{g! \Lambda} \sum_{k \geq 1} \psi(k)
   + O\left( \frac{(\log z)^{g-1}}{g! \Lambda} \,
      g \sum_{k \geq 1} |\psi(k)| \log k \right).
\]
Moreover, by \eqref{eq:Delta-defn} and \eqref{eq:psi-series},
\begin{align*}
\prod_p E_p(0)
   & = \prod_{\textn{$p$ bad}} (1 + h(p))
   \prod_{\textn{$p$ good}} (1 + h(p)) (1 - p^{-1})^g \\
   & = \prod_{\textn{$p$ bad}} (1 - p^{-1})^{-g}
   \prod_p (1 + h(p))(1 - p^{-1})^g = \Lambda \Delta^{-1},
\end{align*}
yielding the main term in \eqref{eq:H-estimate-prelim}.

Finally we consider the sum over $k \geq k_0$.
Here we discard the condition ${(n,Q) = 1}$
and apply \Cref{lem:tau-sum-estimate} to obtain
\begin{align*}
\Biggl| \sum_{k_0 \leq k \leq z} \psi(k)
   \sum_{1 \leq n \leq \tfrac{z}{k}} f(n) \Biggr|
   & \leq \sum_{k_0 \leq k \leq z} |\psi(k)|
      \sum_{1 \leq n \leq \tfrac{z}{k}} \frac{\tau_g(n)}{n} \\
   & \leq \sum_{k_0 \leq k \leq z} |\psi(k)| \,
      \frac{(\log(z/k) + g)^g}{g!} \\
   & \leq \frac{(\log z)^{g-1}}{g! \Lambda} \cdot
      \Lambda \, \frac{(\theta + g)^g}{(\log z)^{g-1}}
      \sum_{k \geq k_0} |\psi(k)|.
\end{align*}
Since $Q \geq 6$ we clearly have $\theta > g$,
so the last expression is bounded by
\[
\frac{(\log z)^{g-1}}{g! \Lambda} \cdot
   \Lambda \, \frac{(2\theta)^g}{(2\theta)^{g-1}}
   \sum_{k \geq k_0} |\psi(k)|
   \ll \frac{(\log z)^{g-1}}{g! \Lambda} \,
   \Lambda \theta \sum_{k \geq k_0} |\psi(k)|.
   \qedhere
\]
\end{proof}

We now proceed to bound the error terms in \eqref{eq:H-estimate-prelim}.
\begin{lem}
\label{lem:psi-eps-bound}
Fix $c \in (0, \tfrac12)$.
Then for $0 \leq \eps \leq c$,
\[
\sum_{n \geq 1} |\psi(n)| \tsp n^\eps
   \ll \Delta^{-1} \Lambda \cdot Q^\eps \exp(O_c(g^{4\eps})).
\]
\end{lem}
\begin{proof}
By \Cref{cor:log-Fp-bound} we have
\begin{multline*}
\sum_{n \geq 1} |\psi(n)| \tsp n^\eps
   = \prod_p F_p(-\eps)
   = \prod_{\textn{$p$ bad}} (1 + h(p) \tsp p^\eps)
       \prod_{\textn{$p$ good}} F_p(-\eps) \\
   \leq \prod_{\textn{$p$ bad}} p^\eps (1 + h(p)) \cdot
       \exp(O_c(g^{4\eps}))
   = \prod_{\textn{$p$ bad}} (1 + h(p)) \cdot
       Q^\eps \exp(O_c(g^{4\eps})).
\end{multline*}
To estimate $\prod_{\textn{$p$ bad}} (1 + h(p))$ observe that
\[
\Delta^{-1} = \prod_p (1 + h(p)) (1 - p^{-1})^g
   = \prod_{\textn{$p$ bad}} (1 + h(p)) (1 - p^{-1})^g
      \prod_{\textn{$p$ good}} E_p(0),
\]
so
\[
\prod_{\textn{$p$ bad}} (1 + h(p))
   = \Delta^{-1} \Lambda \prod_{\textn{$p$ good}} E_p(0)^{-1}.
\]
By \Cref{cor:log-Fp-bound} again with $\eps = 0$,
we find that $\prod_{\textn{$p$ good}} E_p(0)^{-1} \ll 1$.
\end{proof}

\begin{lem}
\label{lem:psi-log-bound}
We have
\[
\sum_{n \geq 1} |\psi(n)| \log n \ll \Delta^{-1} \Lambda \cdot \log Q.
\]
\end{lem}
\begin{proof}
The identity $\log n = \sum_{q \text{ prime}, \, q \divides n} v_q(n) \log q$
implies that
\begin{multline*}
\sum_{n \geq 1} |\psi(n)| \log n = \\
   \sum_q \left(|\psi(q)| \log q + |\psi(q^2)| \, 2 \log q + \cdots \right)
   \prod_{p \neq q} \left( 1 + |\psi(p)| + |\psi(p^2)| + \cdots \right).
\end{multline*}
(Alternatively, this identity could be obtained by differentiating
$\sum_{n \geq 1} |\psi(n)| \tsp n^{-\sigma} = \prod_p F_p(\sigma)$
and evaluating at $\sigma = 0$.)
Recalling from \eqref{eq:psi-series} that
$F_p(\sigma) = 1 + h(p) \tsp p^{-\sigma}$ for $p$ bad and
$F_p(\sigma) = 1 + |\psi(p)| \tsp p^{-\sigma}
   + |\psi(p^2)| \tsp p^{-2\sigma} + \cdots$
for $p$ good,
we obtain
\begin{multline*}
\sum_{n \geq 1} |\psi(n)| \log n
   = \sum_{\textn{$q$ bad}} h(q) \log q
      \prod_{\substack{\textn{$p$ bad} \\ p \neq q}} (1 + h(p))
      \prod_{\textn{$p$ good}} F_p(0) \\
   + \sum_{\textn{$q$ good}}
      \bigl( |\psi(q)| \log q + |\psi(q^2)| \, 2 \log q + \cdots \bigr)
      \prod_{\textn{$p$ bad}} (1 + h(p))
      \prod_{\substack{\textn{$p$ good} \\ p \neq q}} F_p(0).
\end{multline*}
We saw in the proof of \Cref{lem:psi-eps-bound} that
$\prod_{\textn{$p$ bad}} (1 + h(p)) \ll \Delta^{-1} \Lambda$,
and \Cref{cor:log-Fp-bound} implies that
$\prod_{\textn{$p$ good}} F_p(0) \ll 1$.
By \Cref{lem:psi-bound}, for good $q$ we have
\[
|\psi(q)| \log q + |\psi(q^2)| \, 2 \log q + \cdots
   \ll \left( \frac{g^2}{q^2} + \frac{2g^2}{1! \tsp q^2}
      + \frac{3g^3}{2! \tsp q^3} + \cdots \right) \log q
   \ll \frac{g^2 \log q}{q^2}.
\]
Combining these observations we find that
\[
\sum_{n \geq 1} |\psi(n)| \log n \ll
   \left(\sum_{\textn{$q$ bad}} \frac{h(q) \log q}{1 + h(q)}
      + \sum_{\textn{$q$ good}} \frac{g^2 \log q}{q^2} \right)
   \Delta^{-1} \Lambda.
\]
Since $h(q)/(1+h(q)) < 1$
the first sum is bounded by $\sum_{\textn{$q$ bad}} \log q = \log Q$.
The second sum is $O(g^2)$,
and we certainly have $g^2 \ll \log Q$
as $Q$ is divisible by all primes $p \leq 4g^2$.
\end{proof}

\begin{proof}[Proof of \Cref{prop:H-estimate}]
To estimate the first error term in \eqref{eq:H-estimate-prelim},
we apply \Cref{lem:psi-eps-bound} with $\eps = 0$ (and say $c = \tfrac13$)
to obtain
\[
\theta \sum_{k \geq 1} |\psi(k)| \ll \theta \cdot \Delta^{-1} \Lambda.
\]
For the second term, \Cref{lem:psi-log-bound} yields
\[
g \sum_{k \geq 1} |\psi(k)| \log k \ll g \log Q \cdot \Delta^{-1} \Lambda
   \ll \theta \cdot \Delta^{-1} \Lambda.
\]
To handle the third term,
observe that for any $y > 1$ and $0 \leq \eps \leq \tfrac13$,
\Cref{lem:psi-eps-bound} implies that
\[
\sum_{k \geq y} |\psi(k)|
   \leq y^{-\eps} \sum_{k \geq y} |\psi(k)| \tsp k^\eps
   \ll \Delta^{-1} \Lambda \cdot y^{-\eps} Q^\eps \exp(O(g^{4\eps})).
\]
Taking $\eps \coloneqq (\log 5Q)^{-1}$
(and noting that $\eps \leq (\log 30)^{-1} < \tfrac13$ as $Q \geq 6$),
we obtain $Q^\eps \ll 1$ and also $\exp(O(g^{4\eps})) \ll 1$
as certainly $g < Q$.
The third term thus becomes
\[
\Lambda \theta \sum_{k \geq e^{-\theta} z} |\psi(k)|
   \ll \theta \cdot \Delta^{-1} \Lambda \cdot
      (e^\theta / z)^\eps \Lambda.
\]
By \Cref{lem:Q-phiQ-bound} we have
\[
\eps^{-1} \log \Lambda < \log 5Q \cdot g \log(3 \log \log 5Q)
   < 2 g \log 5Q \log \log 5Q \leq 2\theta,
\]
since $\log(3x) < 2x$ for all $x > \log \log 30 \approx 1.224$.
Using the hypothesis that $z > e^{3\theta}$ it follows that
\[
(e^\theta / z)^\eps \Lambda \leq
   \left(\frac{e^\theta \cdot e^{2\theta}}{e^{3 \theta}}\right)^\eps = 1,
\]
so our bound for the third term becomes simply
\[
\Lambda \theta \sum_{k \geq e^{-\theta} z} |\psi(k)|
   \ll \theta \cdot \Delta^{-1} \Lambda.
\]
Putting everything together, we conclude that
\[
H(z) = \frac{(\log z)^g}{g! \Delta}
      + O\Biggl( \frac{(\log z)^{g-1}}{g! \Lambda}
      \cdot \theta \cdot \Delta^{-1} \Lambda \Biggr)
   = \frac{(\log z)^g}{g! \Delta}
   \left(1 + O\left(\frac{\theta}{\log z}\right) \right). \qedhere
\]
\end{proof}

\subsection{The estimate for $R(z)$}
\label{sec:sieve-R-bound}

In this section we prove \Cref{prop:R-bound}.
Our main task is to bound the quantity
$r(d) = |\Aset_d| - \delta(d) X$ for $d \in \Dset$.
By definition
\[
\Aset_d = \{0 \leq k < X : (k \bmod p^2) \in B_p
   \textn{ for all } p \divides d\},
   \qquad d \in \Dset.
\]
We begin by analysing the structure of $B_p$ in more detail.
Recall that \eqref{eq:Bp-assumption} remains in force,
so $\delta(p) \in (0,1)$ for all $p \in \Pset$.
\begin{lem}
\label{lem:Bp-structure}
Let $p \in \Pset$.
Then there exist $\alpha_p \in \{0, 1\}$
and integers $n_p^1, n_p^2 \in [0,g]$
such that $B_p$ is of the form $B_p = B_p^1 \setminus B_p^2$ where
\begin{itemize}
\item
$B_p^1 \subseteq \ZZ/p^2\ZZ$ is a union of
exactly $n_p^1$ residue classes modulo~$p^{\alpha_p}$, and
\item
$B_p^2 \subseteq B_p^1$ is a union of
exactly $n_p^2$ residue classes modulo~$p^{\alpha_p+1}$.
\end{itemize}
\end{lem}
\begin{proof}
We first consider the case that $p \divides U$.
Fix $i \in \{1, \ldots, g\}$.
We claim that $B_{p,i}$ (see \eqref{eq:Bpi-defn})
is a union of either zero, $p-1$ or $p$ residue classes modulo~$p$.
Indeed, if $p \ndivides u_i$ then $B_{p,i} = \emptyset$.
Now suppose that $p \divides u_i$. Then
\[
B_{p,i} = \bigl\{ k \in \ZZ/p^2\ZZ :
   \tfrac{U}{p}k \not\equiv -\tfrac{u_i}{p} \tpmod p \bigr\}.
\]
If $p^2 \divides U$, then $B_{p,i} = \emptyset$ or $\ZZ/p^2\ZZ$
according to whether $p^2 \divides u_i$ or not;
whereas if $p^2 \ndivides U$, then $B_{p,i}$ consists of exactly $p-1$
residue classes modulo $p$.
In all cases the claim is proved.

It follows that $B_p = \cup_i B_{p,i}$
is itself a union of either zero, $p-1$ or $p$ residue classes modulo~$p$.
The first and last possibilities are ruled out because
we know that $\delta(p) \in (0,1)$,
so the second option must hold.
We may therefore take $B_p^1 \coloneqq \ZZ/p^2\ZZ$
and $B_p^2 \coloneqq \{k \in \ZZ/p^2\ZZ : k \equiv c \pmod p\}$
for a suitable integer~$c$,
and the result holds with $\alpha_p \coloneqq 0$,
$n_p^1 \coloneqq 1$ and $n_p^2 \coloneqq 1$.

Now suppose that $p \ndivides U$.
Then $B_{p,i} = B_{p,i}^1 \setminus B_{p,i}^2$ where
\begin{align*}
B_{p,i}^1 & \coloneqq \{k \in \ZZ/p^2\ZZ : Uk \equiv -u_i \tpmod{p}\}, \\
B_{p,i}^2 & \coloneqq \{k \in \ZZ/p^2\ZZ : Uk \equiv -u_i \tpmod{p^2}\}
   \subseteq B_{p,i}^1.
\end{align*}
Note that $B_{p,i}^1$ consists of exactly one residue class modulo~$p$,
and $B_{p,i}^2$ of exactly one residue class modulo~$p^2$.
Let $B_p^1 \coloneqq \cup_i B_{p,i}^1$;
this set is a union of $n_p^1$ residue classes modulo~$p$
for some $n_p^1 \leq g$.
We have $B_p = \cup_i B_{p,i} \subseteq \cup_i B_{p,i}^1 = B_p^1$,
so $B_p = B_p^1 \setminus B_p^2$ where $B_p^2 \coloneqq B_p^1 \setminus B_p$.
Now observe that
\[
B_p^2 = (\cup_i B_{p,i}^1) \setminus (\cup_i B_{p,i})
   \subseteq \cup_i (B_{p,i}^1 \setminus B_{p,i})
   = \cup_i B_{p,i}^2,
\]
so $B_p^2$ is a union of $n_p^2$ residue classes modulo $p^2$
for some $n_p^2 \leq g$.
The result therefore holds with $\alpha_p \coloneqq 1$
and with $n_p^1$ and $n_p^2$ as described above.
\end{proof}

\begin{prop}
\label{prop:r-bound}
For any $d \in \Dset$ we have
\[
|r(d)| \leq (2g)^{\omega(d)}
\]
where $\omega(d)$ denotes the number of (distinct) prime divisors of $d$.
\end{prop}
\begin{proof}
We continue with the notation of \Cref{lem:Bp-structure}.
For $p \divides d$ let
\begin{align*}
\Aset_p^1 & \coloneqq \{ 0 \leq k < X : (k \bmod p^2) \in B_p^1 \}, \\
\Aset_p^2 & \coloneqq \{ 0 \leq k < X : (k \bmod p^2) \in B_p^2 \}
   \subseteq \Aset_p^1,
\end{align*}
so that $\Aset_p = \Aset_p^1 \setminus \Aset_p^2$ and hence
\begin{equation}
\label{eq:Aset-d-formula}
\Aset_d = \bigcap_{p \divides d} {(\Aset_p^1 \setminus \Aset_p^2)}.
\end{equation}
The cardinality of $\Aset_d$ is given by
\begin{equation}
\label{eq:Aset-d-cardinality}
|\Aset_d| = \sum_{d_1 d_2 = d} (-1)^{\omega(d_2)} N_{d_1,d_2},
   \qquad 
   N_{d_1,d_2} \coloneqq \biggl|
      \bigcap_{p \divides d_1} \Aset_p^1 \cap
      \bigcap_{p \divides d_2} \Aset_p^2
   \biggr|.
\end{equation}
(To prove this formula from \eqref{eq:Aset-d-formula},
one may argue by induction on~$\omega(d)$.)

Let us now estimate $N_{d_1,d_2}$ for a given pair $(d_1, d_2)$.
Each $\Aset_p^1$ (for $p \divides d_1$)
is defined by congruences modulo~$p^{\alpha_p}$,
and each $\Aset_p^2$ (for $p \divides d_2$)
by congruences modulo~$p^{\alpha_p+1}$.
Combining these congruences by the Chinese remainder theorem,
we obtain exactly
\[
n \coloneqq
   \prod_{p \divides d_1} n_p^1 \cdot
   \prod_{p \divides d_2} n_p^2
\]
residues modulo
\[
L \coloneqq
   \prod_{p \divides d_1} p^{\alpha_p} \cdot
   \prod_{p \divides d_2} p^{\alpha_p+1}.
\]
The quantity $N_{d_1,d_2}$ counts the number of integers $k \in [0,X)$
satisfying these congruences,
so by splitting up the interval into blocks of length~$L$ we obtain
\[
\left\lfloor \frac{X}{L} \right\rfloor n
   \leq N_{d_1,d_2} \leq \left\lceil \frac{X}{L} \right\rceil n
\]
or in other words
\begin{equation}
\label{eq:Nd1d2-estimate}
\left| N_{d_1,d_2} - \frac{X}{L} \cdot n \right| \leq n.
\end{equation}
Recalling the bounds $n_p^1 \leq g$ and $n_p^2 \leq g$
from \Cref{lem:Bp-structure},
and defining
\[
\delta^1(p) \coloneqq \frac{|B_p^1|}{p^2} = \frac{n_p^1}{p^{\alpha_p}},
\qquad
\delta^2(p) \coloneqq \frac{|B_p^2|}{p^2} = \frac{n_p^2}{p^{\alpha_p+1}},
\]
the inequality \eqref{eq:Nd1d2-estimate} becomes
\[
\Biggl| N_{d_1,d_2} - X
   \prod_{p \divides d_1} \delta^1(p) \prod_{p \divides d_2} \delta^2(p)
   \Biggr| \leq g^{\omega(d_1)} g^{\omega(d_2)} = g^{\omega(d)}.
\]
Substituting this estimate into \eqref{eq:Aset-d-cardinality},
we find that
\[
\Biggl| |\Aset_d| - X
   \sum_{d_1 d_2 = d} (-1)^{\omega(d_2)}
   \prod_{p \divides d_1} \delta^1(p) \prod_{p \divides d_2} \delta^2(p)
   \Biggr| \leq \sum_{d_1 d_2 = d} g^{\omega(d)} = (2g)^{\omega(d)}.
\]
The proof is completed by observing that
\[
\sum_{d_1 d_2 = d} (-1)^{\omega(d_2)}
   \prod_{p \divides d_1} \delta^1(p) \prod_{p \divides d_2} \delta^2(p)
   = \prod_{p \divides d} (\delta^1(p) - \delta^2(p))
   = \prod_{p \divides d} \delta(p) = \delta(d).
   \qedhere
\]
\end{proof}

\begin{proof}[Proof of \Cref{prop:R-bound}]
By \Cref{prop:r-bound} we know that
\[
R(z) \leq
   \sum_{\substack{d_1, d_2 \in \Dset \\ d_1, d_2 \leq z}}
      (2g)^{\omega([d_1,d_2])}
   \leq
   \sum_{\substack{d_1, d_2 \in \Dset \\ d_1, d_2 \leq z}}
      (2g)^{\omega(d_1) + \omega(d_2)}
   \leq
   \left( \sum_{d \in \Dset, \, d \leq z} (2g)^{\omega(d)} \right)^2.
\]
For $d \in \Dset$ we have
$(2g)^{\omega(d)} = \prod_{p \divides d} 2g
   = \prod_{p \divides d} \tau_{2g}(p) = \tau_{2g}(d)$
(where $\tau_g(n)$ is the multiplicative function
defined in \Cref{sec:sieve-preliminary}), so
\[
R(z) \leq
   \left( \sum_{d \in \Dset, \, d \leq z} \tau_{2g}(d) \right)^2
   \leq \left( \sum_{1 \leq d \leq z} \tau_{2g}(d) \right)^2.
\]
By \Cref{lem:tau-sum-estimate},
\[
\sum_{1 \leq d \leq z} \tau_{2g}(d)
   \leq z \sum_{1 \leq d \leq z} \frac{\tau_{2g}(d)}{d}
   \leq z \, \frac{(\log z + 2g)^{2g}}{(2g)!}
   = z \, \frac{(\log z)^{2g}}{(2g)!}
      \left(1 + \frac{2g}{\log z}\right)^{2g}.
\]
The hypothesis $z > e^{g^2}$ implies that
$(1 + 2g/\log z)^{2g} \ll 1$, so we are done.
\end{proof}

\subsection{Proof of \Cref{thm:sieve-main}}
\label{sec:sieve-main-proof}

Let $g$, $U$, $N$ and $y$ be as in \Cref{thm:sieve-main},
let $X \coloneqq N/U \geq 1$, and define
\[
E \coloneqq \{ k \in [0, X): \textn{there are at least $g$
   square-$y$-rough integers in } [kU, (k+1)U) \}.
\]
Our goal is to bound $|E|$ from above.
Let
\[
\Gamma = \Gamma(g, U) \coloneqq \{0, 1, \ldots, U-1\}^g,
\]
and for $u \in \Gamma$ define
\[
E(u) \coloneqq \{ k \in [0, X): Uk + u_1, \ldots, Uk + u_g
   \textn{ are all square-$y$-rough} \}.
\]
Writing $\Gamma^*$ for the set of tuples
$u = (u_1, \ldots, u_g) \in \Gamma$ whose
entries are pairwise distinct,
we clearly have $E = \cup_{u \in \Gamma^*} E(u)$.
Taking into account the possible permutations of each
$(u_1, \ldots, u_g) \in \Gamma^*$,
we find that
\begin{equation}
\label{eq:E-average}
|E| \leq \frac{1}{g!} \sum_{u \in \Gamma^*} |E(u)|.
\end{equation}

We will use \Cref{thm:sieve-tuple} to bound each $|E(u)|$, taking
\[
z \coloneqq \frac{y}{U^{1/2} (\log y)^{3g}}.
\]
To apply \Cref{thm:sieve-tuple} we must check that the hypothesis
$z > e^{C_1 \theta}$ is satisfied.
We claim that this follows from \eqref{eq:y-hypothesis},
provided that $C_0$ is taken sufficiently large.
To prove this, first observe that
\[
\frac{\log z}{\log y}
   = 1 - \frac{\log U}{2\log y} - \frac{3g \log \log y}{\log y}.
\]
For $C_0$ sufficiently large,
\eqref{eq:y-hypothesis} certainly implies that $\log y > 2 g \log U$ and that
\[
\frac{\log y}{\log \log y} > (\log y)^{1/2} > 12g^2,
\]
and therefore that
\begin{equation}
\label{eq:logz-logy-ratio}
\frac{\log z}{\log y} > 1 - \frac{1}{4g} - \frac{1}{4g}
   = 1 - \frac{1}{2g} \geq \frac12.
\end{equation}
We now observe the following estimates for $Q$ and~$\theta$.
\begin{lem}
\label{lem:Q-theta-bounds}
Let $u \in \Gamma^*$,
and let $Q \coloneqq Q(u)$ and $\theta = \theta(g, Q)$
be as in \Cref{sec:sieve-setup}.
Then
\[
\log 5Q < 4g^2 \log 5U
\]
and
\[
\theta < 4 g^4 \log 5U \log(4g^2 \log 5U).
\]
\end{lem}
\begin{proof}
By definition of~$Q(u)$,
\[
Q \leq U \cdot \prod_{1 \leq i < j \leq g} |u_i - u_j|
   \cdot \prod_{p \leq 4g^2} p
\]
so
\[
\log 5Q \leq \log 5U + \binom{g}{2} \log U + \sum_{p \leq 4g^2} \log p.
\]
By \cite[Thm.\,9]{RS-primes} we have $\sum_{p \leq x} \log p < 1.02x$
for~$x > 0$ so
\[
\log 5Q < \bigl(1 + \tbinom{g}{2}\bigr) \log 5U + 4.08g^2
   < g^2 \log 5U + \frac{4.08 g^2 \log 5U}{\log 5} < 4 g^2 \log 5U.
\]
It follows immediately that
\[
\theta = g^2 \log 5Q \log \log 5Q < 4 g^4 \log 5U \log(4 g^2 \log 5U).
\qedhere
\]
\end{proof}
If \eqref{eq:y-hypothesis} holds,
then $\log y > (C_0/4) \theta$ by \Cref{lem:Q-theta-bounds},
and then $\log z > (C_0/8) \theta$ by \eqref{eq:logz-logy-ratio}.
Taking $C_0 > 8 C_1$ we conclude that $z > e^{C_1 \theta}$ as claimed.

Applying \Cref{thm:sieve-tuple}, we obtain an upper bound for the number of
$k \in [0,X)$ such that $Uk + u_1, \ldots, Uk + u_g$ are all square-$z$-rough.
Any square-$y$-rough integer is automatically square-$z$-rough
(as $z \leq y$),
so the same bound holds for~$E(u)$, i.e.,
\[
|E(u)| \leq \frac{N}{U} \cdot \frac{g! \Delta(u)}{(\log z)^g}
   \left(1 + O\left(\frac{\theta}{\log z} \right)\right)
   + O\left( \frac{z^2 (\log z)^{4g}}{(2g)!^2} \right).
\]
Clearly $(\log z)^g \leq (\log y)^g$,
and from \eqref{eq:logz-logy-ratio} we have
\[
(\log z)^g = \left(\frac{\log z}{\log y}\right)^g (\log y)^g
   > \left(1 - \frac{1}{2g}\right)^g (\log y)^g \gg (\log y)^g,
\]
so the bound for $|E(u)|$ becomes
\[
|E(u)| \ll \frac{N}{U} \cdot \frac{g! \Delta(u)}{(\log y)^g}
   + \frac{y^2}{(2g)!^2 U (\log y)^{2g}}.
\]

We now sum over $u \in \Gamma^*$ and apply \eqref{eq:E-average}.
Since $|\Gamma^*| \leq |\Gamma| = U^g$,
and using the assumption that $y \leq N^{1/2}$
(see \eqref{eq:y-hypothesis}), we obtain
\begin{equation}
\label{eq:E-bound}
|E| \ll \frac{N}{U} \cdot \frac{1}{(\log y)^g}
      \sum_{u \in \Gamma^*} \Delta(u)
   + \frac{N}{U} \cdot \frac{1}{g! (2g)!^2} \cdot \frac{U^g}{(\log y)^{2g}}.
\end{equation}
The key issue is now to estimate $\sum_{u \in \Gamma^*} \Delta(u)$,
i.e., to bound the average value of $\Delta(u)$ over all possible tuples~$u$.

Recall that by definition
\[
\Delta(u) = \prod_p \left(1 - \frac{|B_p(u)|}{p^2}\right)
   \left(1 - \frac{1}{p}\right)^{-g},
\]
where $B_p(u)$ means the set $B_p \subseteq \ZZ/p^2\ZZ$
corresponding to the given~$u$.
By \Cref{cor:log-Fp-bound} with $\eps = 0$
the contribution to the above product from the good primes
(for the given $u$) is $\Theta(1)$, so
\[
\Delta(u) \ll \prod_{p \divides Q(u)}
   \left(1 - \frac{|B_p(u)|}{p^2}\right) \left(1 - \frac{1}{p}\right)^{-g}.
\]
By \Cref{lem:Q-phiQ-bound} and \Cref{lem:Q-theta-bounds} we have
\begin{equation}
\label{eq:alpha1-inequality}
\prod_{p \divides Q} \left(1 - \frac{1}{p}\right)^{-1} =
   \frac{Q}{\varphi(Q)} \leq 3 \log \log 5Q
   \leq 3 \log(4g^2 \log 5U).
\end{equation}
Therefore
\[
\Delta(u) \ll \alpha_1^g \prod_{p \divides Q(u)}
   \left(1 - \frac{|B_p(u)|}{p^2}\right),
   \qquad \alpha_1 \coloneqq 3 \log(4g^2 \log 5U).
\]

This last inequality remains valid after omitting from the product
any desired subset of the bad primes.
We will retain only those primes for which $p^2 \divides U$,
as these are the only ones that appear to be susceptible to
a straightforward averaging argument.
We thus have
\[
\Delta(u) \ll \alpha_1^g \prod_{p^2 \divides U}
   \left(1 - \frac{|B_p(u)|}{p^2}\right).
\]
For these primes we may calculate $|B_p(u)|$ as follows.
Define
\[
S_p \coloneqq \{x \in \ZZ/p^2\ZZ : x \equiv 0 \tpmod{p}
   \textn{ and } x \not\equiv 0 \tpmod{p^2}\}.
\]
One checks from \eqref{eq:Bpi-defn} that $B_{p,i}(u) = \ZZ/p^2\ZZ$
if $(u_i \bmod p^2) \in S_p$,
and otherwise $B_{p,i}(u) = \emptyset$.
As $B_p(u) = \cup_i B_{p,i}(u)$, we get
\[
|B_p(u)| = \begin{cases}
   p^2 & \textn{if $(u_i \bmod p^2) \in S_p$
      for some $i \in \{1, \ldots, g\}$}, \\
   0   & \textn{otherwise}.
\end{cases}
\]
It follows that
\begin{align*}
\sum_{u \in \Gamma^*} \Delta(u)
   & \ll \alpha_1^g
   \cdot \bigl\lvert \{u \in \Gamma^* : (u_i \bmod p^2) \notin S_p
      \textn{ for all $i = 1, \ldots, g$ and all $p^2 \divides U$} \} \bigr\rvert \\
   & \leq \alpha_1^g
   \cdot \bigl\lvert \{u \in \Gamma : (u_i \bmod p^2) \notin S_p
      \textn{ for all $i = 1, \ldots, g$ and all $p^2 \divides U$} \} \bigr\rvert \\
   & = \alpha_1^g \cdot U^g \prod_{p^2 \divides U}
      \left(1 - \frac{|S_p|}{p^2}\right)^g
   = \alpha_1^g \cdot U^g \prod_{p^2 \divides U}
      \left(1 - \frac{p-1}{p^2}\right)^g.
\end{align*}
Combining with \eqref{eq:E-bound} we obtain
\begin{equation}
\label{eq:E-bound2}
|E| \ll \frac{N}{U} \cdot \frac{U^g}{(\log y)^g}
   \left( \alpha_1^g \alpha_2^g + \frac{1}{g! (2g)!^2 (\log y)^g} \right)
\end{equation}
where
\[
\alpha_2 \coloneqq \prod_{p^2 \divides U} \left(1 - \frac{p-1}{p^2}\right).
\]
By \eqref{eq:alpha1-inequality} we have
\[
\alpha_2 \geq \prod_{p^2 \divides U} \left(1 - \frac{1}{p}\right)
   \geq \prod_{p \divides Q(u)} \left(1 - \frac{1}{p}\right)
   \geq \alpha_1^{-1}
\]
so the second term in \eqref{eq:E-bound2} is negligible
and we conclude that
\[
|E| \ll \frac{N}{U} \cdot \frac{(\alpha_1 \alpha_2 U)^g}{(\log y)^g}.
\]

It remains to show that $\alpha_1 \alpha_2 < \alpha$
where $\alpha$ is as defined in \eqref{eq:alpha-defn}.
For this it suffices to observe that
\[
\alpha_2
   = \prod_{p^2 \divides U} \left(1 - \frac{1}{p} \right)
      \left(1 + \frac{1}{p^2 - p}\right)
   < \prod_{p^2 \divides U} \left(1 - \frac{1}{p} \right)
      \prod_p \left(1 + \frac{1}{p^2 - p}\right),
\]
and to check numerically that
$\prod_p (1 + 1/(p^2 - p)) \approx 1.9436 < 2$.
This completes the proof of \Cref{thm:sieve-main}.

\section{The probabilistic algorithm}
\label{sec:probabilistic}

The goal of this section is to prove \Cref{thm:main-probabilistic}.
Recall that the principal difficulty in the core algorithm
is that if a vector $a \in \FF_2^T$ has weight larger than
the compression threshold $R$,
we cannot reliably detect this by examining $\kappa(a)$
because $\kappa(a)$ might coincide with $\kappa(a')$
for some ``impostor'' $a' \in \FF_2^T$ with $\wt(a') \leq R$.
The key idea of this section is to add a randomisation stage
to the compression map,
in such a way that for any \emph{fixed} $a$,
an impostor is unlikely to exist.

\subsection{The core algorithm with permuted vectors}
\label{sec:permuted-core}

Let $T$ and $R \leq T/2$ be positive integers,
and let $S$ and $\kappa \colon \FF_2^T \to \FF_2^S$
be defined as in \Cref{sec:compression}.

Let $\Pi_T$ be the set of permutations of $\{0, \ldots, T-1\}$,
i.e., the set of bijections from this set to itself.
For any $\pi = (\pi_0, \ldots, \pi_{T-1}) \in \Pi_T$
define a corresponding injective linear map
\[
Q_\pi \colon \FF_2^T \to \FF_2^T
\]
by sending the basis element $e_t$ to $e_{\pi_t}$ for each
$t = 0, \ldots, T-1$.
Equivalently,
$Q_\pi(a_0, \ldots, a_{T-1}) \coloneqq (a_{\rho_0}, \ldots, a_{\rho_{T-1}})$
where $\rho \coloneqq \pi^{-1}$ is the inverse permutation.
Clearly $Q_\pi$ preserves weights, i.e.,
$\wt(Q_\pi(a)) = \wt(a)$ for all $a \in \FF_2^T$.
Define also a modified compression map
\[
\kappa_\pi \colon \FF_2^T \to \FF_2^S, \qquad
   \kappa_\pi \coloneqq \kappa \circ Q_\pi.
\]
We then have the following variant of \Cref{thm:core} (the core algorithm)
in which the compression map $\kappa$ is replaced by~$\kappa_\pi$.
The vectors $a^r \in \FF_2^T$ have the same meaning as in \Cref{sec:core}.

\begin{thm}[Core algorithm with permuted vectors]
\label{thm:core-permuted}
There is a (deterministic) algorithm with the following properties.
It takes as input positive integers $N$, $T$ and~$R$
such that $T \divides N$ and $R \leq T/2$,
and a permutation $\pi = (\pi_0, \ldots, \pi_{T-1}) \in \Pi_T$.
Its output is the list of vectors
\[
\kappa_\pi(a^r) \in \FF_2^S, \qquad 0 \leq r < N/T.
\]
Assuming that \eqref{eq:logT-bound} holds for some $C > 0$,
the running time of the algorithm is given by \eqref{eq:core-bound}.
\end{thm}
\begin{proof}
The algorithm is exactly the same as in \Cref{thm:core},
except that we must apply $Q_\pi$ before each invocation of~$\kappa$,
and we leave out the final decompression.
Therefore the only additional thing to check is the cost of
the applications of~$Q_\pi$.

Given an integer $n \geq 1$ and a list of vectors
$u_0, \ldots, u_{n-1} \in \FF_2^T$,
we claim that the permuted vectors
$Q_\pi(u_0), \ldots, Q_\pi(u_{n-1}) \in \FF_2^T$
may be computed in time
\[
O(n T \log T + T \log^2 T).
\]
To achieve this,
we first use \Cref{lem:transpose} to transpose the data,
obtaining a list of vectors $v_0, \ldots, v_{T-1} \in \FF_2^n$
in time $O(n T \log T)$.
We then construct a list of pairs $(v_j, \pi_j)_{0 \leq j < T}$
and use \Cref{lem:sort} to sort the list by the second component
in time $O(T(n + \log T) \log T)$;
this produces the list $v_{\rho_0}, \ldots, v_{\rho_{T-1}}$
where $\rho \coloneqq \pi^{-1}$.
Transposing back again yields the desired array
$Q_\pi(u_0), \ldots, Q_\pi(u_{n-1})$ in time $O(n T \log T)$,
completing the proof of the claim.

The relevant invocations of $\kappa$ all occur in
the proof of \Cref{prop:compute-hatZ},
namely, to compute the $\hat e^r$ from $e^r$
and to compute the $\hat Z^{\sigma,\tau}_{h,m,s}$
from~$Z^{\sigma,\tau}_{h,m,t}$.
We have $n = N/T$ in the first case and
$n \ll (N/T) \Kstar(N)$ in the second (see \eqref{eq:compression-count}),
so the total cost of all applications of $Q_\pi$ is
\[
O(N \log T \cdot \Kstar(N) + T \log^2 T),
\]
which is negligible compared to the cost estimate stated in
\Cref{prop:compute-hatZ}.
Thus this additional work does not change the final cost bound
given by \eqref{eq:core-bound}.
\end{proof}

The motivation for introducing permutations is the following result.
It says that if $\wt(a) > R$, and if $\pi \in \Pi_T$ is selected randomly,
then $\kappa_\pi(a)$ is unlikely to be confused with $\kappa(b)$
for any vector $b \in \FF_2^T$ such that $\wt(b) \leq R/2$.
The point is that these $\kappa(b)$ form a very sparse subset of~$\FF_2^S$,
and $\kappa_\pi(a)$ will almost certainly miss it.
\begin{prop}
\label{prop:kappa-unlucky}
Assume that $R < T/16$ and let
\[
\Gset \coloneqq \{b \in \FF_2^T : \wt(b) \leq R/2\}.
\]
Fix $a \in \FF_2^T$ with $R < \wt(a) \leq T/2$.
If $\pi \in \Pi_T$ is chosen uniformly randomly then
\[
P\bigl(\kappa_\pi(a) \in \kappa(\Gset)\bigr) < \frac{1}{2^{R-1}}.
\]
\end{prop}
\begin{proof}
Since $\pi$ is chosen uniformly randomly,
$Q_\pi(a)$ has the same distribution as a vector
selected uniformly randomly from all vectors in $\FF_2^T$
of weight $w \coloneqq \wt(a)$.
Therefore, if we set
\[
S_w \coloneqq \{b \in \FF_2^T : \wt(b) = w, \;
   \kappa(b) \in \kappa(\Gset) \},
\]
then
\[
P\bigl(\kappa_\pi(a) \in \kappa(\Gset)\bigr) = |S_w| \big\slash \binom{T}{w}.
\]
   
Next, consider the set
\begin{multline*}
S'_w \coloneqq \bigl\{ (b_0, b_1) \in \FF_2^T \times \FF_2^T :
   \wt(b_0) = w-R, \; \wt(b_1) = R, \\
   \textn{$b_0$ and $b_1$ have disjoint support}, \;
   \kappa(b_0 + b_1) \in \kappa(\Gset) \bigr\}.
\end{multline*}
By ``disjoint support'', we mean that $b_0$ and $b_1$ do not share
any coordinates $t$ where $(b_0)_t = (b_1)_t = 1$.
Then we have
\[
|S'_w| = \binom{w}{R} |S_w|.
\]
To see this, observe that the map $(b_0, b_1) \mapsto b_0 + b_1$
sends $S'_w$ to~$S_w$,
and hits each $b \in S_w$ exactly $\binom{w}{R}$ times,
once for each cardinality-$R$ subset of the support of~$b$.
   
Now let us count the elements of $S'_w$ in a different way.
The number of $b_0 \in \FF_2^T$ with $\wt(b_0) = w-R$ is $\binom{T}{w-R}$.
Fix one such $b_0$, and let us count the $b_1 \in \FF_2^T$
such that $(b_0, b_1) \in S'_w$.
The condition $\kappa(b_0 + b_1) \in \kappa(\Gset)$ is equivalent
to $\kappa(b_1) \in \kappa(\Gset) - \kappa(b_0)$.
We know from \Cref{thm:decompression} that $\kappa$ is injective
on the set of vectors of weight at most~$R$,
so the number of possible $b_1$ is at most
\[
|\kappa(\Gset) - \kappa(b_0)| = |\kappa(\Gset)|.
\]
But again by \Cref{thm:decompression} we know that
$|\kappa(\Gset)| = |\Gset|$,
so we conclude that
\[
|S'_w| \leq \binom{T}{w - R} |\Gset|.
\]
   
Finally we estimate~$|\Gset|$.
By definition
\[
|\Gset| = \sum_{0 \leq k \leq R/2} \binom{T}{k}.
\]
For $k = 1, \ldots, \lfloor R/2 \rfloor$, since $T \geq 2R$ we have
\[ 
\binom{T}{k-1} \Big\slash \binom{T}{k} = \frac{k}{T-k+1}
   < \frac{R/2}{2R - R/2} = \frac{1}{3},
\]
so the terms decrease rapidly and we find that
\[
|\Gset| < 2 \binom{T}{\lfloor R/2 \rfloor}.
\]
   
Putting everything together we obtain
\begin{align*}
P\bigl(\kappa_\pi(a) \in \kappa(\Gset)\bigr)
   & < 2 \binom{T}{w-R} \binom{T}{\lfloor R/2 \rfloor}
      \Big\slash \binom{T}{w} \binom{w}{R} \\
   & = 2 \cdot \frac{T!}{(T - \lfloor R/2 \rfloor)!}
      \cdot \frac{(T-w)!}{(T-w+R)!} \cdot \frac{R!}{\lfloor R/2 \rfloor!} \\
   & \leq 2 \cdot T^{\lfloor R/2 \rfloor}
      \cdot \frac{1}{(T - w)^R} \cdot R^{\lceil R/2 \rceil}.
\end{align*}
Since we assumed that $w \leq T/2$ and $T \geq 16R$ this becomes
\begin{multline*}
P\bigl(\kappa_\pi(a) \in \kappa(\Gset)\bigr)
   < 2 \cdot \frac{T^{\lfloor R/2 \rfloor} R^{\lceil R/2 \rceil}}{(T/2)^R}
   = 2^{R+1} (R/T)^{\lceil R/2 \rceil} \\
   \leq 2^{R+1} (1/16)^{\lceil R/2 \rceil}
   \leq 2^{R+1} (1/16)^{R/2}
   = 2^{-R+1}.
   \qedhere
\end{multline*}
\end{proof}

\subsection{Testing for square-primality}
\label{sec:modified-AKS}

In the main probabilistic algorithm,
we need a fallback method to directly test square-primality
of integers lying in the problematic intervals,
i.e., for those $a^r$ with $\wt(a^r) > R$.
For testing \emph{primality}, we have available the following famous result.
\begin{prop}[\protect{\cite[Thm.\,5.1]{AKS-primes}}]
\label{prop:AKS}
There is a deterministic algorithm that tests primality of a given
integer $n \geq 2$ in time $(\log n)^{10.5 + o(1)}$.
\end{prop}
\begin{rem}
We have cited here the simplest version of the AKS algorithm.
Exponents better than 10.5 are known,
but this is irrelevant for our purposes; our main requirement
is that the complexity is polynomial in $\log n$.
(A better exponent would allow us to replace \Cref{prop:rejected-intervals}
by a slightly weaker statement.)
\end{rem}

The author does not know how to test an integer for square-primality
in polynomial time, deterministically or otherwise.
We will instead deploy the following workaround that
enables us to efficiently test many integers simultaneously.
\begin{prop}[Testing square-primality in bulk]
\label{prop:square-prime-bulk}
There is a (deterministic) algorithm with the following properties.
It takes as input an integer $N \geq 1$
and a list $\Rset$ of (distinct) integers in $[1,N)$.
It returns a sorted list consisting of all square-primes in~$\Rset$.
Its running time is
\[
O(N) + |\Rset| (\log N)^{11.5+o(1)}.
\]
\end{prop}
\begin{proof}
Choose an integer parameter
\[
N_0 \coloneqq \frac{N}{\log^2 N} + O(1).
\]
We first compute the list of all primes $p < N_0$
using Sergeev's algorithm \cite{Ser-prime-turing}
(or see \Cref{thm:sergeev} of this paper).
This requires time $O(N_0 \log N_0) = O(N / \log N)$.

Next, we generate a list of all integers of the form $m^2 p$
(with $p$ prime) such that $p < N_0$ and $m^2 p < N$.
For each $p < N_0$ there are at most $(N/p)^{1/2}$ possible values of $m$,
so the number of pairs is at most
\[
\sum_{p < N_0} \frac{N^{1/2}}{p^{1/2}}
   \asymp N^{1/2} \int_2^{N_0} \frac{dt}{t^{1/2} \log t}
   \asymp N^{1/2} \cdot \frac{N_0^{1/2}}{\log N_0}
   \asymp \frac{N}{\log^2 N}.
\]
Using classical multiplication,
each $m^2 p$ may be computed in time $O(\log^2 N)$,
so the total cost of computing the list is
$O((N / \log^2 N) \log^2 N) = O(N)$.
We then sort the list using \Cref{lem:sort}
in time $O((N / \log^2 N) \log N \cdot \log N) = O(N)$.
Let $\Sset$ denote the resulting sorted list.
Since $|\Rset| \leq N$,
we may also sort $\Rset$ in time $O(|\Rset| \log^2 N)$.
Walking through $\Sset$ and $\Rset$ in parallel,
we can find any matches,
and hence identify any square-primes in~$\Rset$
of the form $m^2 p$ with $p < N_0$,
in time $O((|\Sset| + |\Rset|) \log N) = O(N/\log N + |\Rset| \log N)$.

It remains to find those square-primes in~$\Rset$ of the form $m^2 p$
with $N_0 \leq p < N$.
In such cases we must have $m \leq m_0$ where
\[
m_0 \coloneqq \bigl\lceil (N/N_0)^{1/2} \bigr\rceil \asymp \log N.
\]
As we walk through $\Rset$,
whenever we encounter some $n \notin \Sset$,
it suffices to check, for each $m \in \{1, \ldots, m_0\}$,
whether $n$ is divisible by $m^2$, and if so, whether $n/m^2$ is prime
(via \Cref{prop:AKS}).
The overall cost is
\[
|\Rset| \tsp m_0 (\log N)^{10.5+o(1)} < |\Rset| (\log N)^{11.5+o(1)}.
\qedhere
\]
\end{proof}

\subsection{The main probabilistic algorithm}
\label{sec:probabilistic-algorithm}

We are given some large $N$ as input.
The main steps of the algorithm are:
\begin{enumalgo}
\item
Choose suitable parameters $T$ and~$R$.
\item 
Randomly select $\pi \in \Pi_T$,
and use the modified core algorithm (\Cref{thm:core-permuted})
to compute $\kappa_\pi(a^r)$ for all~$r$.
\item
For each~$r$, attempt to decompress $\kappa_\pi(a^r)$
(\Cref{thm:decompression}).
If decompression succeeds,
and if the resulting candidate $c^r$ for $a^r$ satisfies $\wt(c^r) \leq R/2$,
then ``accept'' this candidate as correct.
Otherwise mark this interval as ``rejected''.
\item
Find the odd square-primes in the rejected intervals
directly using the bulk square-primality test (\Cref{prop:square-prime-bulk}).
\item
Determine the exact number of odd square-primes less than~$N$
(\Cref{prop:count-square-primes}).
If the count disagrees with the proposed list from Steps 1--4,
return ``FAIL''.
\item
Otherwise, we have correctly found the odd square-primes up to~$N$.
Conclude by recovering the list of primes up to~$N$
(\Cref{prop:square-primes-to-primes}).
\end{enumalgo}
We now explain and analyse each of these steps in more detail.

\medskip
\step{1}{choose parameters}.
Let
\[
\tilde A \coloneqq \frac13 \log\left(\frac{\log N}{80} \right),
\]
noting that $\tilde A > 1$ for large enough~$N$.
Choose a positive integer $A = \tilde A + O(1)$ with $A < \tilde A$
and define
\[
U' \coloneqq \prod_{p \leq A} p^2.
\]
By the prime number theorem $\log U' \sim 2A$,
so for large enough $N$ we have $\log U' < 3A < 3\tilde A$ and hence
\[
U' < \frac{\log N}{80}.
\]
Choose also an integer $\beta \geq 0$ such that
\[
\frac13 < 2^\beta \cdot \frac{U'}{\tfrac{1}{80} \log N} < 1
\]
and set
\begin{equation}
\label{eq:U-defn}
U \coloneqq 2^\beta U' = 2^\beta \prod_{p \leq A} p^2.
\end{equation}
Then by construction
\begin{equation}
\label{eq:U-interval}
\frac{\log N}{240} < U < \frac{\log N}{80}.
\end{equation}
Next choose a positive integer $T$ such that
\[
T = \frac{\log^2 N}{7000 \log \log N} + O(U), \qquad U \divides T.
\]
Thus
\begin{equation}
\label{eq:T-estimate}
T = \left(1 + O\left(\frac{\log \log N}{\log N}\right)\right)
      \frac{\log^2 N}{7000 \log \log N},
\end{equation}
and certainly $T$ is even.
Finally, choose a positive integer $R$ such that
\[
R = \log_2 N + O(1).
\]
Increasing $N$ slightly if necessary, we may assume that $T \divides N$.
All of the above parameters may be computed in time $(\log N)^{O(1)}$.

\begin{rem}
\label{rem:U-special}
The reason for choosing $U$ of this special form is to take advantage of the
$\prod_{p^2 \divides U} (1 - p^{-1})$ factor in \eqref{eq:alpha-defn}.
The algorithm can be made to work without this choice,
but we would need to reduce $U$ and $T$ by
a factor of $\Theta(\log \log \log N)$ to compensate
(in order to make the proof of \Cref{prop:rejected-intervals} go through).
The effect would be to increase the complexity in
\Cref{thm:main-probabilistic} by a factor of $\log \log \log N$.
This would of course be absorbed by the $(\log \log N)^{o(1)}$ factor
in \Cref{thm:main-probabilistic},
but see \Cref{rem:main-probabilistic-precise}.
\end{rem}

\step{2}{compute $\kappa_\pi(a^r)$ for all $r$.}
We now select a random permutation $\pi \in \Pi_T$
and invoke \Cref{thm:core-permuted}
to compute $\kappa_\pi(a^r) \in \FF_2^S$ for $0 \leq r < N/T$.
The hypotheses $R \leq T/2$ and \eqref{eq:logT-bound} (with $C = 2$)
are certainly satisfied for large enough~$N$.
The cost of invoking \Cref{thm:core-permuted} is
\begin{equation}
\label{eq:step2-cost}
\left(1 + \frac{R \log N}{T}\right) N (\log \log N)^{1+o(1)}
   = N (\log \log N)^{2+o(1)}.
\end{equation}

\begin{rem}
See \cite[Sec.\,3.4.2]{Knu-TAOCP2} for a discussion of methods
for selecting random permutations.
As mentioned in \Cref{sec:new-results},
the number of random bits needed to generate~$\pi$ is
$O(\log T!) = O(T \log T) = O(\log^2 N)$.
\end{rem}

\step{3}{decompression and rejection tests.}
For each $r$ we perform the following steps.
We first apply \Cref{thm:decompression} (the decompression algorithm)
to $\kappa_\pi(a^r) = \kappa(Q_\pi(a^r))$.
The cost over all $r$ is the same as in the proof of \Cref{thm:core},
and so is covered by \eqref{eq:step2-cost}.

If decompression fails, we mark this interval as ``rejected''.
If it succeeds then we have a candidate for $Q_\pi(a^r)$.
We apply $(Q_\pi)^{-1} = Q_{\pi^{-1}}$ using the same method as described
in the proof of \Cref{thm:core-permuted}.
This produces a candidate for~$a^r$, say $c^r \in \FF_2^T$.
Finally, we check whether $\wt(c^r) \leq R/2$.
If so, we accept the candidate $c^r$ as correct.
Otherwise, we reject this interval.
We analyse the complexity of these rejection tests below,
but first we prove the following two results.
\begin{prop}
\label{prop:accepted-correct}
For sufficiently large~$N$,
the probability that all accepted intervals are correct
is at least~$\tfrac12$.
\end{prop}
\begin{prop}
\label{prop:rejected-intervals}
The number of rejected intervals is at most $O(N / \log^{14} N)$.
\end{prop}

\begin{proof}[Proof of \Cref{prop:accepted-correct}]
Let us examine what happens for each~$r$, depending on the weight of $a^r$:
\begin{itemize}
\item
\textit{Case 1: $\wt(a^r) \leq R/2$}.
\newline
\Cref{thm:decompression} implies that decompression succeeds
and returns the correct vector, so $c^r = a^r$.
Thus $\wt(c^r) = \wt(a^r) \leq R/2$ and the interval is accepted.
\item
\textit{Case 2: $R/2 < \wt(a^r) \leq R$.}
\newline
Again \Cref{thm:decompression} implies that decompression
succeeds and that $c^r = a^r$.
Then $\wt(c^r) = \wt(a^r) > R/2$, so the interval is rejected.
\item
\textit{Case 3: $\wt(a^r) > R$.}
\newline
Here the decompression succeeds only if there exists a vector
$b \in \FF_2^T$ such that $\wt(b) \leq R$ and $\kappa_\pi(a^r) = \kappa(b)$.
Moreover, we only accept this interval if additionally $\wt(b) \leq R/2$.
\Cref{prop:kappa-unlucky} shows that the probability
of this occurring is at most $2^{1-R} \ll 1/N$.
(The hypothesis $R < T/16$ holds for sufficiently large~$N$,
and the hypothesis $\wt(a^r) \leq T/2$ holds as $a^r$ only includes
those bits corresponding to \emph{odd} square-primes.)
\end{itemize}
Therefore, the probability that even a single interval is incorrectly
accepted is at most
\[
O\left( \frac{N}{T} \cdot \frac{1}{N}\right) = O(1/T) < \frac12
\]
for sufficiently large~$N$.
\end{proof}

\begin{rem}
On heuristic grounds,
similar to the arguments discussed in \Cref{sec:heuristic},
we expect that \emph{all} intervals are accepted (correctly),
provided that $N$ is large enough.
In other words, cases~2 and~3 should never occur.
Indeed, if we adopt a simple Cram\'er-type model,
then in an interval of length $T \asymp \log^2 N / \log \log N$,
the expected number of square-primes is $\asymp \log N / \log \log N$,
and the variance is also $\asymp \log N / \log \log N$.
Observing more than $R/2 \asymp \log N$ square-primes
is an event lying $\asymp \sqrt{\log N \log \log N}$ standard deviations
above the mean,
which should occur much less than $1/N$ of the time.
\end{rem}

For the proof of \Cref{prop:rejected-intervals}
we first make a preliminary observation:
\begin{lem}
\label{lem:square-prime-exceptions}
Let $N$ and $Y \leq N$ be positive integers.
Then there are at most $O(N^{1/2} Y^{1/2})$ positive integers $n < N$
that are square-prime but not square-$Y$-rough.
\end{lem}
\begin{proof}
Suppose that $n < N$ is square-prime, say $n = m^2 p$
for some $m \geq 1$ and prime $p$.
The only way that $n$ can fail to be square-$Y$-rough
is for $p$ itself to satisfy $p \leq Y$.
For each such $p$ there are at most $(N/p)^{1/2}$ possible choices for $m$,
so the number of possibilities is at most
\[
\sum_{p \leq Y} (N/p)^{1/2}
   \ll N^{1/2} \sum_{1 \leq k \leq Y} k^{-1/2}
   \ll N^{1/2} Y^{1/2}. \qedhere
\]
\end{proof}

\begin{proof}[Proof of \Cref{prop:rejected-intervals}]
In the proof of \Cref{prop:accepted-correct},
we saw that if an interval is rejected then
it necessarily contains more than $R/2$ square-primes.
Let us split up each interval into $T/U$ subintervals of length $U$,
for $U$ defined as in \eqref{eq:U-defn}.
If an interval contains more than $R/2$ square-primes,
then by the pigeonhole principle at least one of its $T/U$ subintervals
contains more than $(R/2)/(T/U)$ square-primes.
Now observe, by \eqref{eq:U-interval} and \eqref{eq:T-estimate}, that
\begin{align*}
\frac{R/2}{T/U} & >
   \frac{\log_2 N + O(1)}{2} \cdot \frac{\log N}{240}
   \cdot \left(1 + O\left(\frac{\log \log N}{\log N}\right)\right)
   \frac{7000 \log \log N}{\log^2 N} \\
   & = \left(\frac{1}{480} + o(1)\right) 7000 \log_2 \log N
      > 13 \log_2 \log N
\end{align*}
for large enough~$N$.
Therefore it suffices to obtain an upper bound for the number of
subintervals (of length~$U$) containing at least
\[
g \coloneqq \lceil 13 \log_2 \log N \rceil
\]
square-primes.

By \Cref{lem:square-prime-exceptions},
the number of such subintervals that do \emph{not} contain at least~$g$
square-$\sqrt{N}$-rough integers is $O(N^{3/4})$,
which is negligible compared to the target $O(N/\log^{14} N)$ bound.
So it suffices in turn to bound the number of subintervals containing at least
$g$ square-$\sqrt{N}$-rough integers.

We will estimate the number of such subintervals by applying
\Cref{thm:sieve-main} with $y \coloneqq N^{1/2}$.
The hypothesis \eqref{eq:y-hypothesis} certainly holds for large enough~$N$
as $g \asymp \log \log N$ and $U \asymp \log N$.
Therefore by \eqref{eq:U-interval}
the number of such subintervals is at most
\[
O\biggl( \frac{N}{U} \left(\frac{\alpha U}{\log y}\right)^g\biggr)
   = O\biggl( \frac{N}{\log N}
      \left(\frac{\tfrac1{80} \alpha \log N}{\tfrac12 \log N}\right)^g \biggr)
   = O\left(\frac{N}{\log N} \left(\frac{\alpha}{40}\right)^g \right)
\]
where $\alpha$ is given by \eqref{eq:alpha-defn}.
Now, by \eqref{eq:U-defn} and \cite[Eq.\,3.26]{RS-primes} we have
\[
\prod_{p^2 \divides U} \left(1 - \frac{1}{p}\right)
   \leq \prod_{p \leq A} \left(1 - \frac{1}{p}\right)
   < \frac{1}{\log A}
\]
for large~$N$.
Since $g \asymp \log \log N$, $\log 5U \asymp \log \log N$
and $A \asymp \log \log N$ we obtain
\begin{align}
\label{eq:alpha-calculation}
\alpha < \frac{6 \log(4g^2 \log 5U)}{\log A}
   & = \frac{6 \log(\Theta((\log \log N)^3))}{\log(\Theta(\log \log N))} \\
   & = \frac{18 \log \log \log N + O(1)}{\log \log \log N + O(1)}
      = 18 + o(1). \notag
\end{align}
Therefore, for large enough $N$,
\[
\left(\frac{\alpha}{40}\right)^g < (1/2)^g
   \leq 2^{-13 \log_2 \log N} = \frac{1}{(\log N)^{13}}. \qedhere
\]
\end{proof}

Let us write $\Rset$ for the set of rejected integers,
i.e., those integers lying in some rejected interval.
The cost of the rejection tests may be estimated as follows.
As shown in the proof of \Cref{thm:core-permuted},
the cost over all~$r$ of applying $Q_{\pi^{-1}}$
is already covered by \eqref{eq:step2-cost}.
To check the condition $\wt(c^r) \leq R/2$,
we simply walk through the array,
counting down from $R/2$ in each interval,
at a cost of $O(N)$.
Writing down the rejected integers as we proceed
costs an additional $O(|\Rset| \log N)$.
Each rejected interval contains at most $T \ll \log^2 N$ integers,
so \Cref{prop:rejected-intervals} implies that
\begin{equation}
\label{eq:Rset-cardinality}
|\Rset| \ll N / \log^{12} N,
\end{equation}
and the cost of the rejection tests is only $O(N)$.

\medskip
\step{4}{handle rejected intervals.}
We apply \Cref{prop:square-prime-bulk} to determine which
(odd) integers in $\Rset$ are square-prime.
By \eqref{eq:Rset-cardinality} the cost of this step is
\begin{equation}
\label{eq:bulk-invocation-cost}
O(N) + \frac{N}{\log^{12} N} \cdot (\log N)^{11.5+o(1)} = O(N).
\end{equation}
At this point, if there were no incorrectly accepted intervals,
then we know the correct value of $a^r$ for all~$r$.

\medskip
\step{5}{check against actual number of odd square-primes.}
Using \Cref{lem:convert-format},
in time $O(N)$ we may combine the results of the previous steps
into a single proposed list of odd square-primes less than $N$.
We count these integers and compare to the correct value,
obtained via \Cref{prop:count-square-primes} in time $N^{1/2+o(1)}$.
If even a single candidate $c^r$ is incorrect,
then it must have arisen from an incorrectly accepted interval,
so we must have $\wt(a^r) > R \geq \wt(c^r)$.
As in the proof of \Cref{thm:main-heuristic} (\Cref{sec:heuristic-algorithm}),
our overall tally will then be too small, and we return ``FAIL''.

\medskip
\step{6}{recover primes up to $N$.}
Otherwise our list must be correct,
and we finish up by applying \Cref{prop:square-primes-to-primes}
to deduce the list of primes less than $N$ in time $O(N)$.
This completes the proof of \Cref{thm:main-probabilistic}.

\begin{rem}
\label{rem:main-probabilistic-precise}
The complexity in \Cref{thm:main-probabilistic} is given more precisely by
\[
O(N (\log \log N)^2 \Mstar(N)).
\]
This arises from the first term of \eqref{eq:core-auxiliary-bound}
(the inverse transform step in the core algorithm).
All other terms in \eqref{eq:core-auxiliary-bound},
and all other steps in the proof of \Cref{thm:main-probabilistic},
contribute at most $N (\log \log N)^{1+o(1)}$.
Combining with \eqref{eq:ffmul},
we obtain the best rigorous randomised (Las Vegas) complexity bound
known to the author for enumerating the primes up to $N$:
\[
O\bigl(N (\log \log N)^2 \cdot 4^{\log^* N}\bigr).
\]
\end{rem}

\subsection{On the gap in complexity between
   the probabilistic and heuristic algorithms}
\label{sec:probabilistic-complexity-gap}

In this section we briefly indicate some reasons for the $\log \log N$
complexity gap between the probabilistic and heuristic algorithms
(Theorems \ref{thm:main-probabilistic} and \ref{thm:main-heuristic}).
This has apparently very little to do with our randomisation strategy,
but is due rather to limitations in the sieve methods from \Cref{sec:sieve}.

Whatever randomisation strategy we pursue,
to get a rigorous probabilistic algorithm based on the existing core algorithm
it seems necessary to bound the number of intervals of length $T$ up to~$N$
containing more than $R$ square-primes,
since these are the intervals that the core algorithm
cannot reliably resolve.
Moreover, because of the cost of checking the integers in these intervals
(\Cref{prop:square-prime-bulk}),
this bound must be at least as good as $N / (\log N)^{O(1)}$
for suitable $O(1)$.

To achieve any nontrivial bound at all via \Cref{thm:sieve-main},
we require at the very least that $\alpha U / \log y < 1$.
We saw in the proof of \Cref{thm:sieve-main} that
$\alpha > \alpha_1 \alpha_2 \geq 1$,
so this requirement implies that $U < \log y$.
For the choice $y \coloneqq N^{1/2}$, we see that \Cref{thm:sieve-main}
can only cope with intervals of length $U < \tfrac12 \log N$.

Let us assume that if the target length $T$ is larger than $\tfrac12 \log N$,
then we use the pigeonhole principle approach
(as in the proof of \Cref{thm:main-probabilistic})
to recover a bound for the length-$T$ intervals from
some bound for length-$U$ intervals for a suitable $U \divides T$
(possibly incurring a $(\log N)^{O(1)}$ loss in the quality of the bound).
Thus our problem becomes to bound the number of
length-$U$ intervals containing more than $g \coloneqq (U/T)R$ square-primes.
If we try to do this using \Cref{thm:sieve-main},
and if we want to achieve the desired $N/(\log N)^{O(1)}$ bound,
then we need $(\alpha U / \log y)^g \ll 1/(\log N)^C$ for some $C > 0$.
For $y = N^{1/2}$ as above,
this implies that $((\log N)/2U)^g \gg (\log N)^C$ and hence that
\[
g \log\left(\frac{\log N}{2U}\right) > C \log \log N + O(1).
\]

On the other hand, the cost of the core algorithm
is driven mainly by the ratio $R/T$,
via the first term of \eqref{eq:core-auxiliary-bound}.
Now observe that
\[
\frac{R}{T} = \frac{g}{U} = \frac{2g}{\log N} \cdot \frac{\log N}{2U}.
\]
Since $(\log N)/2U > 1$ and $x > \log x$ for $x > 1$ we obtain
\[
\frac{R}{T} > \frac{2g}{\log N} \log\left(\frac{\log N}{2U}\right)
   > \frac{2C \log \log N + O(1)}{\log N} \asymp \frac{\log \log N}{\log N}.
\]
We conclude that the first term of \eqref{eq:core-auxiliary-bound}
grows at least as rapidly as
\[
N \tsp \frac{\log \log N}{\log N} \log \log N \log N \cdot \Mstar(N)
   \asymp N (\log \log N)^2 \Mstar(N),
\]
which matches the upper bound actually achieved
by the probabilistic algorithm (see \Cref{rem:main-probabilistic-precise}).

By constrast, for the heuristic algorithm we are able to take
$R/T \asymp 1/\log N$,
so that the first term of \eqref{eq:core-auxiliary-bound} is only
$N (\log \log N)^{1+o(1)}$.
Of course this choice of $R/T$ corresponds to
the optimal possible compression ratio,
as the density of the square-primes within the integers is $\Theta(1/\log N)$.

The foregoing discussion suggests that even a slight improvement
in \Cref{thm:sieve-main} might be enough to make the probabilistic
algorithm competitive with the heuristic algorithm.
Let us give an example.
We already mentioned in \Cref{rem:remove-g!} that
estimates like \eqref{eq:prime-tuples} are believed to hold
with the $g!$ factor deleted, at least for \emph{fixed}~$g$.
Suppose that we could similarly remove the $g!$ factor
from \Cref{thm:sieve-tuple}, even for $g$ as large as $\Theta(\log \log N)$.
The corresponding bound in \Cref{thm:sieve-main} would then be
\[
O\left(\frac{N}{U} \cdot \frac{1}{g!}
   \left(\frac{\alpha U}{\log y}\right)^g \right).
\]
Taking $y = N^{1/2}$, $R \asymp \log N$ and $\alpha \asymp 1$ as before,
but increasing $T$ and $U$ by a factor of $\log \log N$ to
$T \asymp \log^2 N$ and $U \asymp \log N \log \log N$,
the bound would become
\[
O\left(\frac{N}{U} \cdot \frac{(\Theta(\log \log N))^g}{g!} \right)
   = O\left(\frac{N}{U} \left(\frac{\Theta(\log \log N)}{g/e}\right)^g\right).
\]
By taking a large enough constant in $g \asymp UR/T \asymp \log \log N$,
the bound would simplify to $N / (\log N)^{\Theta(1)}$,
so \Cref{prop:rejected-intervals} would continue to hold.
The ratio $R/T$ would fall to $1/\log N$,
and the overall complexity would match the heuristic algorithm.
This provides further tantalising evidence that the primes up to~$N$
may be enumerated in time $N (\log \log N)^{1+o(1)}$.

\section{The deterministic algorithm}
\label{sec:deterministic}

The goal of this section is to prove \Cref{thm:main-deterministic}.
The general strategy from \Cref{sec:probabilistic} of modifying
the compression map to avoid impostors cannot possibly work in this context,
because for any fixed choice of modified compression map,
we cannot rule out the existence of some exotic pattern of square-primes
that defeats the choice.
Instead we must find some way of locating the problematic intervals
\emph{in advance}.
In this section we will show how to adapt
Sergeev's prime enumeration algorithm \cite{Ser-prime-turing}
for this purpose.
Sergeev's algorithm in its original form is too slow,
but by terminating the algorithm early,
we obtain a cheaper algorithm that nevertheless
(for suitable parameter choices) yields enough information
to correctly recognise \emph{almost all} intervals as being non-problematic.

\subsection{A generalisation of Sergeev's algorithm}
\label{sec:sergeev}

As mentioned in \Cref{sec:introduction},
Sergeev \cite{Ser-prime-turing} gave the first algorithm that
finds the primes $p < N$ in time $O(N \log N)$ in the Turing model.
The following result adapts Sergeev's method to the
more general problem of finding all $Y$-rough integers up to~$N$.

\begin{thm}
\label{thm:sergeev}
There is a deterministic algorithm with the following properties.
It takes as input positive integers $N$ and $Y$ such that
$2 \leq Y \leq N^{1/2}$.
Its output is a bit array $(u_1, \ldots, u_{N-1})$ such that
$u_n = 1$ if and only if $n$ is $Y$-rough.
Its running time is $O(N \log Y)$.
\end{thm}

\begin{rem}
Taking $Y \coloneqq \lfloor N^{1/2} \rfloor$ in \Cref{thm:sergeev}
essentially recovers Sergeev's result, with two caveats.
First, the output only includes the primes in the range $N^{1/2} < p < N$.
We must also find the primes $p \leq N^{1/2}$,
which can be done easily in negligible time.
Second, the algorithm returns a bit array rather than a list.
If a list is preferred, we can use \Cref{lem:convert-format}
to perform the conversion in time $O(N)$.
\end{rem}

Before giving the proof of \Cref{thm:sergeev},
we introduce a data structure used by the algorithm
to represent subsets $S \subseteq \{0, \ldots, N-1\}$.
Let $r \geq 1$ be an integer.
The \emph{$r$-compressed representation} of $S$
is a string defined as follows.
We break up $\{0, \ldots, N-1\}$ into blocks of length~$2^r$,
writing $S = \cup_i S^{[i]}$ where
\[
S^{[i]} \coloneqq S \cap \bigl[i \cdot 2^r, (i+1) \cdot 2^r\bigr),
   \qquad i = 0, \ldots, \lceil N/2^r \rceil - 1.
\]
To encode a single block $S^{[i]}$,
we walk through the $n \in S^{[i]}$ in increasing order,
writing out the $r$ least significant bits of each~$n$,
with a separator symbol (say \texttt{:}) between each group of $r$ bits.
We encode the entire set $S$ by simply concatenating the encodings
of the blocks $S^{[i]}$ for $i = 0, 1, \ldots, \lceil N/2^r \rceil - 1$,
using a terminating symbol (say \texttt{\#}) to mark the end of each $S^{[i]}$.
For example, if $N = 50$ and $r = 3$,
the set $S = \{7, 8, 9, 10, 32, 49\}$ is encoded as the string
``\texttt{111\#000:001:010\#\#\#000\#\#001\#}''.
In this representation,
$S$ occupies $O(r|S| + \lceil N/2^r \rceil)$ cells on the tape.
(Strictly speaking the \texttt{:} separator is unnecessary
as the algorithm always knows the value of~$r$.)

\begin{lem} Let $r \geq 1$.
\label{lem:compressed-ops}
\begin{enumabc}
\item
Given the $r$-compressed representations of disjoint sets $S_1$ and $S_2$,
we may compute the $r$-compressed representation of $S_1 \cup S_2$ in time
\[
O(r|S_1| + r|S_2| + \lceil N/2^r \rceil).
\]
\item
Given the $r$-compressed representations of
pairwise disjoint sets $S_1, \ldots, S_m$
(assumed to be stored consecutively on the tape),
we may compute the $r$-compressed representation of
$S \coloneqq S_1 \cup \cdots \cup S_m$ in time
\[
O(r|S| \log m + m \lceil N / 2^r \rceil).
\]
\end{enumabc}
\end{lem}
\begin{proof}
For (a) we invoke \Cref{lem:merge}, one block at a time,
working from left to right.
For (b), we apply (a) to the $S_j$ in pairs,
repeating the process until all the lists are merged.
By (a) the cost of the $k$\th pass is
$O(r|S| + (m/2^k) \lceil N / 2^r \rceil)$,
so the total over all $k$ is $O(r|S| \log m + m \lceil N / 2^r \rceil)$.
\end{proof}

We will also need the following estimate.
For $x \geq y \geq 2$ let $\Phi(x,y)$ denote the
number of $y$-rough integers $n \in [1,x]$.
\begin{lem}
\label{lem:y-rough-estimate}
For any $x \geq y \geq 2$,
\begin{equation}
\label{eq:Phi-upper-bound}
\Phi(x,y) \ll \frac{x}{\log y}.
\end{equation}
If additionally $y < Cx$ for fixed $C \in (0,1)$ then
\begin{equation}
\label{eq:Phi-lower-bound}
\Phi(x,y) \asymp_C \frac{x}{\log y}.
\end{equation}
\end{lem}
\begin{proof}
According to \cite[Thm.\,III.6.4]{Tenenbaum-NT},
\[
\Phi(x,y) = \frac{x \omega(u) - y}{\log y} + O\left(\frac{x}{\log^2 y}\right)
\]
where $u \coloneqq \log x / \log y \geq 1$
and where $\omega(u)$ denotes the Buchstab function.
(For the definition of $\omega(u)$ see the discussion preceding
\cite[Thm.\,III.6.4]{Tenenbaum-NT}.)

By \cite[Eq.\,III.6.23]{Tenenbaum-NT}
we have $\omega(u) \in [\tfrac12,1]$ for all $u \geq 1$.
The upper bound \eqref{eq:Phi-upper-bound} follows immediately.

To prove the lower bound, suppose first that $y < x^{1/2}$. Then
\[
\Phi(x,y)
   > \frac{x}{2 \log y} + O\left(\frac{x^{1/2} \log y + x}{\log^2 y}\right)
   = \frac{x}{2 \log y} + O\left(\frac{x}{\log^2 y}\right).
\]
This implies \eqref{eq:Phi-lower-bound} for $y \geq y_0$,
provided $y_0$ is chosen large enough.
But \eqref{eq:Phi-lower-bound} also holds for $y < y_0$,
as trivially $\Phi(x,y) \asymp x \asymp x / \log y$ for bounded~$y$.

Finally, suppose that $x^{1/2} \leq y < Cx$ for fixed $C \in (0,1)$.
Then the $y$-rough integers $n \leq x$ are exactly
the primes $y < p \leq x$ together with $n = 1$,
so $\Phi(x,y) = \pi(x) - \pi(y) + 1$.
By the prime number theorem, the latter quantity is
\begin{multline*}
\frac{x}{\log x} - \frac{y}{\log y} + O\left(\frac{x}{\log^2 x}\right)
> \frac{x}{\log x} - \frac{Cx}{\log Cx} + O\left(\frac{x}{\log^2 x}\right) \\
= \frac{x}{\log x}\left(1 - \frac{C}{1 + \frac{\log C}{\log x}}\right) 
+ O\left(\frac{x}{\log^2 x}\right)
\gg_C \frac{x}{\log x} \gg \frac{x}{\log y}. \qedhere
\end{multline*}
\end{proof}

\begin{proof}[Proof of \Cref{thm:sergeev}]
For each prime $p \leq Y$ let
\[
L_p \coloneqq \{1 \leq n < N :
   \textn{the smallest prime divisor of $n$ is $p$}\}.
\]
For example,
\begin{align*}
L_2 & = \{2, 4, 6, 8, 10, 12, \ldots\}, \\
L_3 & = \{3, 9, 15, 21, 27, 33, \ldots\}, \\
L_5 & = \{5, 25, 35, 55, 65, 85, \ldots\}.
\end{align*}
These sets are disjoint and their union is exactly the set of
positive integers $n < N$ that are \emph{not} $Y$-rough.
The algorithm has three main steps.
In Step 1 we generate each $L_p$ as an $r_p$-compressed list
for $r_p \coloneqq \lg p = \lceil \log_2 p \rceil$.
These lists are written to a single tape successively
for $p = 2, 3, 5, \ldots$ up to~$Y$.
(The cost of finding these primes is $Y^{1+o(1)} < N^{1/2+o(1)}$
which is negligible.)
In Step 2 we merge together the lists corresponding to
each possible value of~$r$,
and in Step 3 we write the output array.
We now describe these steps in more detail.

\medskip
\step{1}{generate compressed lists.}
Consider the $p$\th iteration, i.e., the iteration that generates~$L_p$.
The key observation (which goes back to \cite{Mai-primes}) is that
$L_p$ consists of exactly those integers of the form $n = pm$
where $m$ lies in the set
\[
V_p \coloneqq \{1 \leq m < N/p :
   \textn{$m$ not divisible by any prime $q < p$}\}.
\]

At the beginning of the $p$\th iteration
we assume that $V_p$ is already known,
and that it is stored on a separate tape as a bit array
$v = (v_1, \ldots, v_{N_p-1})$
where $N_p \coloneqq \lceil N / p \rceil$,
i.e., with $v_m = 1$ if and only if $m \in V_p$.
(Prior to the first iteration,
since $V_2 = \{1, 2, \ldots, N_2 - 1\}$
it is trivial to generate the required array $v = (1, \ldots, 1)$.)
To generate $L_p$,
the basic idea is to examine $v_m$ for each
$m = 1, 2, \ldots, N_p - 1$ in turn,
computing $pm \in L_p$ whenever we see $v_m = 1$,
and writing it in the desired $r_p$-compressed format as we proceed.
In preparation for the next iteration,
we must also generate the bit array $v'$ corresponding to $V_{p'}$,
where $p'$ is the next prime after~$p$.

The bottleneck in this algorithm is the computation of the products $pm$.
If these are calculated from scratch for each $m$,
the resulting algorithm turns out to be too slow,
even using the fastest known methods for integer multiplication
\cite{HvdH-nlogn}.
To achieve the target complexity bound
we instead adopt the following strategy suggested by Sergeev.

As we iterate over $m = 1, 2, \ldots, N_p-1$,
let $m' < m$ denote the most recent value such that $v_{m'} = 1$,
taking $m' \coloneqq 0$ if no such $m'$ has yet occurred.
We keep track of the difference $m - m'$ on a separate tape;
for each~$m$, the amortised cost of incrementing $m - m'$ is $O(1)$
(see \Cref{sec:counting}).

Let $a_m$ denote the quantity $pm \pmod{2^{r_p}}$
with $0 \leq a_m < 2^{r_p}$.
During the loop over $m$,
we retain the value of $a_{m'}$ for $m'$ as defined above.
Whenever we encounter $v_m = 1$,
we compute the product $p(m - m')$ using
the classical multiplication algorithm
and add it to the previous value $a_{m'}$.
This yields $2^{r_p} b + a_m$ for some integer $b \geq 0$.
We advance $b$ blocks by writing $b$ terminating \texttt{\#} symbols,
and then write the new~$a_m$.
Finally we reset $m - m'$ to zero and proceed to the next value of~$m$.

We can decide when the loop over $m$ has finished
by counting downwards from $N_p$ on another tape,
noting that $N_p$ may be computed in time $O(\log^2 N)$
using classical long division.
After concluding the loop over~$m$,
we must also write sufficiently many \texttt{\#} symbols
to cover any remaining empty blocks;
again, this can be achieved by counting down from
$\lceil N/2^{r_p} \rceil$ on another tape.

While performing these steps, it is easy to generate the array $v'$
on a separate tape.
We simply set $v'_m \coloneqq v_m$ for all~$m$,
with the exception that $v'_m \coloneqq 0$ whenever $p \divides m$.
The latter condition can be recognised by simply counting down from $p-1$
repeatedly on another tape.
Thus we can generate $v'$ in amortised $O(1)$ cost for each value of $m$.

Let us now analyse the overall complexity for a given~$p$.
For each $m$ there are several operations of cost $O(1)$,
whose total contribution is $O(N_p) = O(N/p)$.
The number of \texttt{\#} symbols written is
$\lceil N / 2^{r_p} \rceil \asymp N/p$, contributing another $O(N/p)$.
The $O(\log^2 N)$ division cost is certainly also $O(N/p)$
as $N/p \geq N/Y \geq N^{1/2}$.

Consider now the operations performed each time we encounter $v_m = 1$.
Let $d_i \geq 1$ be the difference $m - m'$ for the $i$\th such pair,
indexed over $1 \leq i \leq |V_p|$.
The cost of computing the product $p d_i$ and adding it to $a_{m'}$
is $O(\log p \log d_i + \log p)$,
and the cost of writing the resulting $a_m$ is $O(r_p) = O(\log p)$.
The total is thus
\begin{equation}
\label{eq:sum-log-di}
O\biggl(\log p \sum_i \log d_i + |V_p| \log p \biggr).
\end{equation}
To analyse this expression,
we consider two different cases depending on the size of~$p$.
We will need the formula
\[
|V_p| = \Phi(N/p, p-1) + O(1),
\]
which follows from the fact that $V_p$ is exactly the set of
$(p-1)$-rough integers (strictly) less than~$N/p$.

\emph{Case 1: $p \geq \tfrac12 N^{1/2}$.}
Then by \Cref{lem:y-rough-estimate} we have $|V_p| \ll N / (p \log p)$,
so \eqref{eq:sum-log-di} becomes
\[
O(\log p \cdot |V_p| \log N + |V_p| \log p) = O(N^{1/2} \log N).
\]
This bound also absorbs the $O(N/p) = O(N^{1/2})$
contribution mentioned earlier.
Since there are at most $\pi(Y) \leq \pi(N^{1/2}) \asymp N^{1/2} / \log N$
such primes,
the total contribution from this case is at most
\[
O\left( \frac{N^{1/2}}{\log N} \cdot N^{1/2} \log N \right) = O(N).
\]

\emph{Case 2: $p < \tfrac12 N^{1/2}$.}
By the convexity of the $\log$ function
(equivalently, the inequality of arithmetic and geometric means),
\[
\sum_i \log d_i \leq |V_p| \log\biggl( |V_p|^{-1} \sum_i d_i \biggr).
\]
Since $p-1 < \tfrac12 N^{1/2} < \tfrac14 (N/p)$,
we find from \Cref{lem:y-rough-estimate} that
\[
|V_p| \asymp \frac{N}{p \log p}.
\]
(\Cref{lem:y-rough-estimate} is inapplicable when $p = 2$,
but the same bound holds trivially in this case.)
As $\sum_i d_i \leq N/p$ we obtain
\[
\sum_i \log d_i \ll \frac{N}{p \log p}
   \log\biggl( \frac{p \log p}{N} \cdot \frac{N}{p}\biggr)
   = \frac{N \log \log p}{p \log p}
\]
so the cost estimate \eqref{eq:sum-log-di} becomes
\[
O\left(\frac{N \log \log p}{p}\right).
\]
Again this absorbs the earlier $O(N/p)$ contribution.
Summing over all $p \leq Y$, the total contribution from this case is
\[
\sum_{p \leq Y} O\left( \frac{N \log \log p}{p} \right)
   = O\biggl(N \log \log Y \sum_{p \leq Y} p^{-1} \biggr)
   = O(N (\log \log Y)^2).
\]

In summary, the total cost of Step~1, including primes from both cases, is
\[
O(N (\log \log Y)^2).
\]
We also note for later use that in both cases we have proved the upper bound
\begin{equation}
\label{eq:Vp-upper-bound}
|V_p| \ll \frac{N}{p \log p}.
\end{equation}

\begin{rem}
Sergeev's analysis of the multiplication step is essentially the same as ours,
but his analysis of the cost of writing the compressed list is not as tight.
To compensate, he includes a separate loop at the beginning
to handle the primes ${p < \log N}$.
The cost of this loop is $O(N \log N / \log \log N)$,
which is negligible compared to our target bound $O(N \log Y)$
when $Y \approx N^{1/2}$.
However, it is too slow if $Y$ is much smaller,
for instance when $Y \approx \exp(\sqrt{\log N})$.
This is exactly the regime of interest
in \Cref{sec:deterministic-algorithm} below.
\end{rem}

\step{2}{merge compressed lists for each $r$.}
The values of~$r_p$ that occur in Step~1 lie in the range
$1 \leq r_p \leq \beta$ where
$\beta \coloneqq \lg Y = \lceil \log_2 Y \rceil$.
For each $r \in \{1, 2, \ldots, \beta\}$ let $T_r$ be the union
of those $L_p$ for which $r_p = r$, i.e.,
\[
T_r \coloneqq \bigcup_{\substack{2^{r-1} < p \leq 2^r \\ p \leq Y}} L_p.
\]
The number of primes contributing to $T_r$ is at most $\pi(2^r) \ll 2^r/r$,
so the cost of invoking \Cref{lem:compressed-ops}(b) to compute
the $r$-compressed representation of $T_r$ is
\[
O\bigl(r |T_r| \log(2^r / r) + (2^r / r) \lceil N/2^r \rceil\bigr).
\]
Recalling that $|L_p| = |V_p|$, by \eqref{eq:Vp-upper-bound} we obtain
\[
|T_r| \ll \sum_{2^{r-1} < p \leq 2^r} \frac{N}{p \log p}
   \ll \frac{N}{2^r \cdot r} \sum_{2^{r-1} < p \leq 2^r} 1
   \ll \frac{N}{2^r \cdot r} \cdot \frac{2^r}{r} = \frac{N}{r^2}.
\]
Also, since $r \leq \lceil \log_2 Y \rceil < \frac12 \log_2 N + O(1)$
we have $2^r \ll N^{1/2}$ and hence $\lceil N / 2^r \rceil \asymp N / 2^r$.
The cost estimate for this $r$ thus becomes
\[
O\left(r \cdot \frac{N}{r^2} \cdot r
   + \frac{2^r}{r} \cdot \frac{N}{2^r}\right)
   = O(N).
\]
Summing over the $\beta \asymp \log Y$ values of~$r$,
the total cost of Step~2 is $O(N \log Y)$.

\medskip
\step{3}{write output array.}
For each $r \in \{1, \ldots, \beta\}$
we must convert the $r$-compressed representation of $T_r$
to a bit vector of length~$N$.
This may be achieved by invoking \Cref{lem:convert-format}
separately for each block of length~$2^r$.
The cost for the $i$\th block is $O(2^r + n_i r)$
where $n_i$ is the number of elements in the $i$\th block.
The total over all $\lceil N / 2^r \rceil$ blocks is
\[
O\bigl( \lceil N / 2^r \rceil 2^r + |T_r| r \bigr) = O(N + N/r) = O(N)
\]
so the total over all~$r$ is $O(N \log Y)$.
We must then sum the resulting vectors
(i.e., compute the union of the $T_r$),
again at a cost of $O(N \log Y)$,
and perform a final trivial $O(N)$ pass to compute the complement.
\end{proof}

\begin{rem}
Noting that Steps~2 and~3 in the above algorithm are
asymptotically more expensive than Step~1,
it is tempting to wonder if the overall complexity could be improved
by somehow pushing some work from Steps~2 and~3 into Step~1.
For example, perhaps a more complicated version of Step~1 could
generate a union of several $L_p$ directly,
in a bid to reduce the cost of the merging step.
The author has so far been unable to find
any asymptotic improvement along these lines.
\end{rem}

\Cref{thm:sergeev} finds the $Y$-rough integers up to~$N$,
but what is actually needed in the main deterministic algorithm
is the set of (odd) \emph{square}-$Y$-rough integers up to~$N$.
The next result shows how to compute these integers.
\begin{prop}
\label{prop:rough-to-square-rough}
Let $N$ and $Y$ be as in \Cref{thm:sergeev}.
Given the output of \Cref{thm:sergeev},
we may compute a sorted list of all odd square-$Y$-rough integers
$n < N$ in time
\[
O\left(\frac{N \log N}{\log Y}\right).
\]
\end{prop}
\begin{proof}
The number of $Y$-rough integers less than $N$ is $O(N / \log Y)$
(\Cref{lem:y-rough-estimate}),
so by \Cref{lem:convert-format} we may obtain these integers
from the output of \Cref{thm:sergeev} as a sorted list in time
\[
O\left(N + \frac{N}{\log Y} \log N \right).
\]
Call this list $S$.
We now apply the same strategy as in the proof of
\Cref{prop:primes-to-square-primes}:
we generate the sorted lists
\[
S_m \coloneqq \{ m^2 n : n \in S, \; 1 \leq n < N/m^2 \},
   \qquad 1 \leq m < N^{1/2}, \quad \textn{$m$ odd},
\]
and then merge the $S_m$ together.
(Unlike the situation in \Cref{prop:primes-to-square-primes},
the $S_m$ are \emph{not} in general disjoint.)
To estimate the cost, observe first that by \Cref{lem:y-rough-estimate},
\[
|S_m| \ll
   \begin{cases}
      N/(m^2 \log Y), & Y \leq N/m^2, \\
      N/m^2,          & Y > N/m^2.
   \end{cases}
\]
In both cases it follows that
\[
|S_m| \ll \frac{N}{m^{3/2} \log Y};
\]
the argument in the second case is that the condition $Y > N/m^2$
implies that $m^{1/2} \gg (N/Y)^{1/4} \geq Y^{1/4} \gg \log Y$.
Therefore the union $S^k \coloneqq \cup_{m \in J_k} S_m$
(where $J_k$ is defined by \eqref{eq:Jk-defn})
has cardinality
\[
|S^k| \ll \sum_{m \in J_k} \frac{N}{2^{3k/2} \log Y}
   = \frac{N}{2^{k/2} \log Y},
\]
and a similar argument to the proof of \Cref{prop:primes-to-square-primes}
shows that the cost of generating $S^k$ as a sorted list is
\[
O\left(k \log N \sum_{m \in J_k} \frac{N}{2^{3k/2} \log Y}\right)
   = O\left(\frac{k N \log N}{2^{k/2} \log Y}\right).
\]
The total over all $k$ is thus $O(N \log N / \log Y)$.
Finally we merge together the $S^k$ to obtain a single sorted list;
the cost of this step is also $O(N \log N / \log Y)$,
by a similar argument to the last part of the proof of
\Cref{prop:primes-to-square-primes}.
\end{proof}

\subsection{The main deterministic algorithm}
\label{sec:deterministic-algorithm}

We are given some large $N$ as input.
The main steps of the algorithm are:
\begin{enumalgo}
\item
Choose suitable parameters $T$, $R$ and~$Y$.
\item
Use the core algorithm (\Cref{thm:core})
to compute candidates $c^r$ for all~$r$.
\item
Find all $Y$-rough integers less than $N$ using the
generalisation of Sergeev's algorithm (\Cref{thm:sergeev}).
\item
Deduce the list of all square-$Y$-rough integers less than~$N$
(\Cref{prop:rough-to-square-rough}).
\item
Mark as ``rejected'' any interval containing
more than $R$ odd square-$Y$-rough integers,
or containing any square-primes that are not square-$Y$-rough.
\item
Find the odd square-primes in the rejected intervals directly
using the bulk square-primality test (\Cref{prop:square-prime-bulk}).
\item
We have correctly found the odd square-primes up to~$N$.
Conclude by recovering the list of primes up to~$N$
(\Cref{prop:square-primes-to-primes}).
\end{enumalgo}
We now explain and analyse each of these steps in more detail.

\medskip
\step{1}{choose parameters}.
We follow here a similar procedure to
the proof of \Cref{thm:main-probabilistic}.
Let
\[
\tilde A \coloneqq \tfrac13 \log\bigl((\log N)^{1/2} \log \log N \bigr),
\]
noting that $\tilde A > 1$ for large enough~$N$.
Choose a positive integer $A = \tilde A + O(1)$ with $A < \tilde A$
and define
\[
T' \coloneqq \prod_{p \leq A} p^2.
\]
By the prime number theorem $\log T' \sim 2A$,
so for large enough $N$ we have $\log T' < 3A < 3\tilde A$ and hence
\[
T' < (\log N)^{1/2} \log \log N.
\]
Choose also an integer $\beta \geq 0$ such that
\[
\frac13 < 2^\beta \cdot \frac{T'}{(\log N)^{1/2} \log \log N} < 1
\]
and set
\[
T \coloneqq 2^\beta T' = 2^\beta \prod_{p \leq A} p^2.
\]
Then by construction
\begin{equation}
\label{eq:T-interval}
\tfrac13 (\log N)^{1/2} \log \log N < T < (\log N)^{1/2} \log \log N.
\end{equation}
Finally, choose integers $R \geq 1$ and $Y \geq 2$ such that
\begin{align*}
R & = 13 \log_2 \log N + O(1), \\
Y & = \exp\bigl(40 (\log N)^{1/2} \log \log N\bigr) + O(1).
\end{align*}
Increasing $N$ slightly if necessary, we may assume that $T \divides N$.
All of the above parameters may be computed in time $(\log N)^{O(1)}$.

\begin{rem}
As in \Cref{rem:U-special},
the reason for choosing $T$ of the above form is to take advantage of
the $\prod_{p^2 \divides U} (1 - p^{-1})$ factor in \eqref{eq:alpha-defn}.
Without this choice the overall complexity would increase by
$(\log \log \log N)^{1/2}$.
See also \Cref{rem:main-deterministic-precise}.
\end{rem}

\step{2}{compute candidates $c^r$ for all~$r$.}
We invoke \Cref{thm:core} to compute candidates $c^r \in \FF_2^T$
for the vectors $a^r \in \FF_2^T$ for $0 \leq r < N/T$.
The hypotheses $R \leq T/2$ and \eqref{eq:logT-bound} (with $C = 1$)
are certainly satisfied for large enough $N$.
The cost of applying \Cref{thm:core} is
\[
\left(1 + \frac{R \log N}{T}\right) N (\log \log N)^{1+o(1)}
   = N (\log N)^{1/2} (\log \log N)^{1+o(1)}.
\]

\smallskip
\step{3}{find all $Y$-rough integers up to~$N$.}
We apply \Cref{thm:sergeev} to find the $Y$-rough integers less than~$N$,
represented as a bit array.
The hypothesis $2 \leq Y \leq N^{1/2}$ is satisfied for large enough~$N$,
and the cost is
\[
O(N \log Y) = O(N (\log N)^{1/2} \log \log N).
\]

\smallskip
\step{4}{find all odd square-$Y$-rough integers up to~$N$.}
We apply \Cref{prop:rough-to-square-rough} to the output of Step~3
to find the odd square-$Y$-rough integers less than~$N$,
as a sorted list. The cost is
\begin{equation}
\label{eq:square-rough-cost}
O\left(\frac{N \log N}{\log Y}\right)
   = O\left(\frac{N (\log N)^{1/2}}{\log \log N} \right).
\end{equation}

\smallskip
\step{5}{find all rejected intervals.}
We ``reject'' an interval if either
\begin{enumabc}
\item
it contains more than $R$ odd square-$Y$-rough integers, or
\item
it contains a square-prime that is not square-$Y$-rough.
\end{enumabc}
Observe that if an interval is accepted (i.e., not rejected),
then the candidate $c^r$ must be correct.
For if $c^r$ is incorrect,
then the corresponding interval contains more than $R$ odd square-primes,
so either it contains at least one square-prime
that is not square-$Y$-rough (and is hence rejected by condition (b)),
or it contains more than $R$ odd square-$Y$-rough integers
(and so is rejected by condition (a)).

To analyse the cost of identifying the rejected intervals,
we first establish the following estimate.
\begin{prop}
\label{prop:rejected-intervals-deterministic}
The number of rejected intervals is at most $O(N / \log^{13} N)$.
\end{prop}
\begin{proof}
To estimate the number of intervals of type (a)
we apply \Cref{thm:sieve-main} with
$U \coloneqq T$, $g \coloneqq R + 1$ and $y \coloneqq Y$.
Since $g \asymp \log \log N$ and $\log 5U = \log 5T \asymp \log \log N$
by \eqref{eq:T-interval},
the left hand side of \eqref{eq:y-hypothesis} is
$\exp(O((\log \log N)^5 \log \log \log N))$,
whereas $y = \exp(\Theta((\log N)^{1/2} \log \log N))$,
so the hypothesis \eqref{eq:y-hypothesis}
certainly holds for large enough~$N$.
Therefore the number of intervals satisfying (a) is at most
\[
O\left(\frac{N}{T} \left(\frac{\alpha T}{\log Y}\right)^g\right)
\]
where $\alpha$ is given by \eqref{eq:alpha-defn}.
A calculation similar to \eqref{eq:alpha-calculation} shows that
$\alpha < 18 + o(1)$,
and by \eqref{eq:T-interval} we have
\[
\log Y = 40(\log N)^{1/2} \log \log N + O(1) > 40 T + O(1).
\]
Thus $\alpha T / \log Y < 18/40 + o(1) < \tfrac12$ for large~$N$,
and the previous estimate becomes
\[
O\left(\frac{N}{T} (1/2)^g\right)
   = O\left(\frac{N}{(\log N)^{1/2}} \, 2^{-13 \log_2 \log N + O(1)}\right)
   = O(N / (\log N)^{13.5}).
\]

As for the intervals of type (b),
by \Cref{lem:square-prime-exceptions} the number of such intervals
is at most $O(N^{1/2} Y^{1/2}) < N^{1/2+o(1)}$,
which is negligible.
\end{proof}

We now discuss the complexity of the rejection tests.
Let $\Rset$ denote the set of rejected integers,
i.e., those integers lying in some rejected interval.
By \Cref{prop:rejected-intervals-deterministic} we have
$|\Rset| \ll N / \log^{12} N$ as certainly $T \ll \log N$.
First consider the intervals of type (a).
The proof of \Cref{prop:rough-to-square-rough} shows that
the number of square-$Y$-rough integers less than $N$ is $O(N / \log Y)$,
so we may convert the output of Step~4 to a bit array in time
$O(N \log N / \log Y)$ via \Cref{lem:convert-format}.
We then walk through this array,
counting down from $R$ in each interval,
at a cost of $O(N)$.
Writing down the rejected integers as we proceed
costs an additional $O(|\Rset| \log N) = O(N)$.
To find the intervals of type (b)
we use the obvious algorithm based on the proof of
\Cref{lem:square-prime-exceptions}, whose complexity is $N^{1/2+o(1)}$.
We must then merge the lists arising from the intervals of type (a) and (b).
Overall, the dominant contribution is the
$O(N \log N / \log Y)$ conversion cost,
which is already covered by \eqref{eq:square-rough-cost}.

\medskip
\step{6}{handle rejected intervals.}
We apply \Cref{prop:square-prime-bulk} to determine which
elements of $\Rset$ are odd square-primes.
As in \eqref{eq:bulk-invocation-cost} the cost is $O(N)$.

As noted earlier,
the accepted candidates $c^r$ are already known to be correct.
We may thus easily combine the output of \Cref{prop:square-prime-bulk}
with the $c^r$ array from Step~2
to obtain a list of correct values of $a^r$ for $0 \leq r < N/T$
in time $O(N)$.

\medskip
\step{7}{recover list of primes up to $N$.}
We use \Cref{lem:convert-format} to convert the bit array
to a sorted list of all odd square-primes less than~$N$
in time $O(N)$,
and then finally recover the list of primes up to~$N$ via
\Cref{prop:square-primes-to-primes}, again in time $O(N)$.
This completes the proof of \Cref{thm:main-deterministic}.

\begin{rem}
\label{rem:main-deterministic-precise}
The complexity in \Cref{thm:main-deterministic} is given more precisely by
\[
O\bigl(N (\log N)^{1/2} \log \log N \cdot \Mstar(N)\bigr).
\]
As in \Cref{rem:main-probabilistic-precise}
this bound arises from the first term of \eqref{eq:core-auxiliary-bound},
with the other terms contributing only $N (\log \log N)^{1+o(1)}$.

If the algorithm ``knows'' the function $\Mstar(n)$
then we can do slightly better.
Indeed, increasing $T$ and $\log Y$ by a factor of $(\Mstar(N))^{1/2}$,
the cost of Step~2 falls and the cost of Step~3 grows by the same factor.
The argument in \Cref{prop:rejected-intervals-deterministic} remains valid,
and the resulting complexity bound is
\begin{equation}
\label{eq:deterministic-optimal}
O\bigl(N (\log N)^{1/2} \log \log N \cdot (\Mstar(N))^{1/2} \bigr).
\end{equation}
Combining with \eqref{eq:ffmul},
we obtain the best rigorous, deterministic complexity bound
known to the author for enumerating the primes up to~$N$:
\[
O\bigl(N (\log N)^{1/2} \log \log N \cdot 2^{\log^* N}\bigr).
\]
\end{rem}

\subsection{On the gap in complexity between
   the deterministic and other two algorithms}
\label{sec:deterministic-complexity-gap}

There is unfortunately a rather large $(\log N)^{1/2+o(1)}$ gap
in complexity between the deterministic algorithm
(\Cref{thm:main-deterministic})
and the probabilistic and heuristic versions
(Theorems \ref{thm:main-probabilistic} and \ref{thm:main-heuristic}).
In this section we explain briefly
(along similar lines to \Cref{sec:probabilistic-complexity-gap})
why this gap seems to be unavoidable,
at least if we insist on using our generalisation of Sergeev's method
(\Cref{thm:sergeev}) to find the problematic intervals.

As in \Cref{sec:probabilistic-complexity-gap},
assume that we need to bound the number of $U$-intervals containing
at least $g \coloneqq (U/T)R$ square-$Y$-rough integers,
for some $U \divides T$.
Using \Cref{thm:sieve-main},
we need $(\alpha U / \log Y)^g \ll 1/(\log N)^C$ for some $C > 0$, and thus
\[
g \log\left( \frac{\log Y}{U} \right) > C \log \log N + O(1).
\]
But then, again as in \Cref{sec:probabilistic-complexity-gap},
\[
\frac{R}{T} = \frac{g}{\log Y} \cdot \frac{\log Y}{U}
   > \frac{g}{\log Y} \log\left(\frac{\log Y}{U}\right)
   > \frac{C \log \log N + O(1)}{\log Y} \asymp \frac{\log \log N}{\log Y}.
\]
This places a lower limit on the first term of
\eqref{eq:core-auxiliary-bound},
and implies that to minimise this term we must take $Y$ as large as possible.
But this increases the $O(N \log Y)$ cost of \Cref{thm:sergeev}.
Balancing these contributions leads exactly to the bound
\eqref{eq:deterministic-optimal}.

\printbibliography

\end{document}